\pdfoutput=1
\documentclass[12pt,a4paper,dvipsnames]{amsart}

\usepackage{Mang2024a}

\title[Anick resolution for $U_n^+$]{Anick resolution for the \\free unitary quantum group}
\author{Alexander Mang}
\address{Hamburg University, Mathematics Department, Bundesstraße 55, Room 316, 20146 Hamburg, Germany, \href{mailto:alexander.mang@uni-hamburg.de}{alexander.mang@uni-hamburg.de}}
\email{\href{mailto:alex@alexandermang.net}{alex@alexandermang.net}}
\urladdr{\href{https://alexandermang.net}{https://alexandermang.net}}
\date{\today}

\begin{document}
\begin{abstract}
  A resolution $P$ of the co-unit of the Hopf $\ast$-algebra $\mathcal{O}(U_n^+)$ of representative functions on  van Daele and Wang's free unitary quantum  group $U_n^+$ in terms of free $\mathcal{O}(U_n^+)$-modules is computed for arbitrary $n$.
  \par A different such resolution was recently found by Baraquin, Franz, Gerhold, Kula and Tobolski. While theirs has  desirable properties which $P$ lacks, $P$ is still good enough to compute the (previously known) quantum group co\-ho\-mo\-lo\-gy and comes instead with an important advantage: $P$ can be arrived at without the clever combination of certain results potentially very particular to  $U_n^+$ that enabled the aforementioned authors to find their resolution. Especially, $P$ relies neither on the resolution for $O_n^+$ obtained by Collins, Härtel and Thom nor the one for $SL_2(q)$ found by Hadfield and Krähmer.
  \par
  Rather, as shown in the present article, the recursion defining the Anick resolution of the co\-/unit of $\mathcal{O}(U_n^+)$ can be solved in closed form. That suggests a potential strategy for determining the co\-ho\-mo\-lo\-gies of arbitrary easy quantum groups.  
\end{abstract}

\maketitle

This is a case study in computing the quantum group co\-ho\-mo\-lo\-gy of the discrete duals of so-called easy compact quantum groups.
  \par
  \vspace{0.5em}
  More precisely, the Anick resolution of the trivial right module of the algebra of representative functions on the free unitary quantum group $U_n^+$ is computed for arbitrary $n$ and used to deduce the (previously known) co\-ho\-mo\-lo\-gy of the dual of  $U_n^+$.
  \vspace{1.em}
  \begin{itemize}
  \item Which findings did the case study produce? -- Section~\ref{section:main-result}
  \item What are \enquote{easy} quantum groups and $U_n^+$? -- Section~\ref{section:on_free_unitary_quantum_group_and_easy_quantum_groups}
  \item What are Anick resolutions? -- Section~\ref{section:anick}
  \item Why do the results of the study hold? -- Sections~\ref{section:groebner_computation}--\ref{section:cohomology}
  \item Why do the case study in the first place? -- Section~\ref{section:motivation}
  \item Which conclusions can be drawn from the results? -- Section~\ref{section:conclusions}
  \end{itemize}
  \section{Main result}
  \label{section:main-result}
The main statement of the article is greatly shortened by the use of the following abbreviations.

  \begin{notation*}
  \newcommand{\indi}{i}
  \newcommand{\indj}{j}
    In the context of any set $\thegens$ with a specified involution $\Sstar$, for any $\thedim\in\Sintegersp$ and any $\thedim\times\thedim$-matrix  $\anymatA=(\anymatAX{\indj}{\indi})_{(\indj,\indi)\in\SYnumbers{\thedim}^{\Ssetmonoidalproduct 2}}$ of elements of $\thegens$ consider besides $\anymatA$ also the three matrices $\anymatAT$, $\anymatB$, $\anymatBT$
defined by for any $(\indj,\indi)\in \SYnumbers{\thedim}^{\Ssetmonoidalproduct 2}$,
\begin{IEEEeqnarray*}{C}
\anymatATX{\indj}{\indi}\Seqpd\anymatAX{\indi}{\indj},\hspace{3em}\anymatBX{\indj}{\indi}\Seqpd\anymatAX{\Sexchanged(\indj)}{\Sexchanged(\indi)}{}\SstarP,\hspace{3em}\anymatBTX{\indj}{\indi}\Seqpd\anymatAX{\Sexchanged(\indj)}{\Sexchanged(\indi)}{}\SstarP,
\end{IEEEeqnarray*}
where
 $\Sxfromto{\Sexchanged}{\Sintegers}{\Sintegers},\,\indi\mapsto\thedim-\indi+1$ is the reflection at $\thedim$.
\end{notation*}
\begin{main*}
  \label{main-result}
  \newcommand{\indi}{i}
  \newcommand{\indj}{j}
  \newcommand{\indk}{k}
  \newcommand{\indl}{\ell}
  \newcommand{\orderind}{\ell}    
  \newcommand{\inds}{s}
  \newcommand{\indt}{t}
  \newcommand{\indiX}[1]{\indi_{#1}}
  \newcommand{\indjX}[1]{\indj_{#1}}  
  \newcommand{\anycolor}{\mathfrak{c}}
  \newcommand{\colc}{\mathfrak{c}}
  \newcommand{\cold}{\mathfrak{d}}
  \newcommand{\colcX}[1]{\colc_{#1}}  
  
  \newcommand{\anylen}{m}  
  \newcommand{\anymatX}[2]{\anymat_{#1,#2}}
  \newcommand{\anyelement}{a}
  \newcommand{\otherelement}{a'}  

For any field $\thefield$,  any $\thedim\in\Sintegersp$ with $2\leq \thedim$ and any $2\thedim$-elemental set $\thegens=\{\thewuniX{\indj}{\indi},\thebuniX{\indj}{\indi}\}_{\indi,\indj=1}^\thedim$,
  if $\thewuniX{\indj}{\indi}{}\SstarP\Seqpd \thebuniX{\indj}{\indi}$ and    $\thebuniX{\indj}{\indi}{}\SstarP\Seqpd \thewuniX{\indj}{\indi}$ for any
 $(\indj,\indi)\in\SYnumbers{\thedim}^{\Ssetmonoidalproduct 2}$,
  if   $\theuni\Seqpd (\thewuniX{\indj}{\indi})_{(\indj,\indi)\in\SYnumbers{\thedim}^{\Ssetmonoidalproduct 2}}$, if $\thematrices\Seqpd \{\theuni,\theuni\StransposeP,\theuni\SskewstarP,\theuni\SskewdaggerP\}$ (which means $|\thematrices|=4$), if 
  $\thealgebra$ is the quotient of the free unital $\thefield$-algebra $\thefreealg$  generated by $\thegens$ by the two-sided ideal $\theideal$ generated by 
  \begin{IEEEeqnarray*}{l}
    \therels\Seqpd \{ \textstyle\sum_{\inds=1}^\thedim \anymatAX{\indj}{\Sexchanged(\inds)}\anymatBTX{\inds}{\Sexchanged(\indi)}-\Skronecker{\indj}{\indi}\Saction\theone\Ssetbuilder \anymat\in\thematrices\Sand (\indj,\indi)\in \SYnumbers{\thedim}^{\Ssetmonoidalproduct 2}\},
  \end{IEEEeqnarray*}
  and if  $\Sxfromto{\thecounit}{\thealgebra}{\thefield}$ is the unique morphism of complex algebras with for any $\anymat\in\thematrices$ and $(\indj,\indi)\in \SYnumbers{\thedim}^{\Ssetmonoidalproduct 2}$,
  \begin{IEEEeqnarray*}{rCl}
    \thecounit(\anymatAX{\indj}{\indi}+\theideal)\Seqpd \Skronecker{\indj}{\indi},
  \end{IEEEeqnarray*}
  which turns $\thefield$ into a right $\thealgebra$-module $\themodule$,
  then the following are true:
  \par
  If for any $\anymat\in\thematrices$, any $(\indj,\indi)\in\SYnumbers{\thedim}^{\Ssetmonoidalproduct 2}$ and  any $\orderind\in\Sintegersp$,
  \begin{IEEEeqnarray*}{rCl}
    \Sbv{\anymat}{\indj}{\indi}\equiv    \Sbvfull{\anymat}{\orderind}{\indj}{\indi}\Seqpd
    \begin{cases}
   \anymatAX{\indj}{\Sexchanged(\indi)}   &\Scase \orderind=1\\
\anymatAX{\indj}{\thedim}(\anymatBTX{1}{\thedim}\anymatAX{1}{\thedim})^{\frac{\orderind-2}{2}}\anymatBTX{1}{\Sexchanged(\indi)}      &\Scase  \orderind\text{ even}\\
\anymatAX{\indj}{\thedim}(\anymatBTX{1}{\thedim}\anymatAX{1}{\thedim})^{\frac{\orderind-3}{2}}\anymatBTX{1}{\thedim}\anymatAX{1}{\Sexchanged(\indi)}      &\Scase  \orderind\neq 1\Sand \orderind\text{ odd},
    \end{cases}
  \end{IEEEeqnarray*}
   if for any $\orderind\in\Sintegersp$,
  \begin{IEEEeqnarray*}{rCl}
    \thechains{\orderind}&\Seqpd& \{\Sbvfull{\anymat}{\orderind}{\indj}{\indi}\Ssetbuilder \anymat\in\thematrices\Sand (\indj,\indi)\in\SYnumbers{\thedim}^{\Ssetmonoidalproduct 2}\}\subseteq \thefreealg
  \end{IEEEeqnarray*}
    (which means in particular $|\thechains{1}|=2\thedim^2$ and $|\thechains{2}|=4\thedim^2-2$ and $|\thechains{\indl}|=4\thedim^2$ for any $\indl\in\Sintegersp$ with $3\leq \indl$),
    if for any $\orderind\in\Sintegers$ with $-1\leq \orderind$,
    \begin{IEEEeqnarray*}{rCl}
      \thechainsmodule{\orderind}&\Seqpd&
      \begin{cases}
      \themodule&\Scase \orderind=-1\\
 \thealgebra&\Scase \orderind=0\\
 \thefield\thechains{\indl}\SmonoidalproductC{\thefield}\thealgebra&\Scase 1\leq \orderind
      \end{cases}
\end{IEEEeqnarray*}
(i.e., if $\thechainsmodule{\orderind}$ is the free right $\thealgebra$-module over $\thechains{\orderind}$ for each $\orderind\in\Sintegersp$), and if $(\thedifferential{\orderind})_{\orderind\in\Sintegersnn}$ is such that for each $\orderind\in\Sintegersnn$ the mapping $\thedifferential{\orderind}$ is a morphism $\Sfromto{\thechainsmodule{\orderind}}{\thechainsmodule{\orderind-1}}$ of right $\thealgebra$-modules
and such that  for any $\anymat\in\thematrices$ and any $(\indj,\indi)\in \SYnumbers{\thedim}^{\Ssetmonoidalproduct 2}$,
\begin{IEEEeqnarray*}{rCl}
  \thedifferential{0}( \theone+\theideal)&\Seqpd&1
\end{IEEEeqnarray*}
and
\begin{IEEEeqnarray*}{rCl}
  \thedifferential{1}(\Sbv{\anymatA}{\indj}{\indi}\Smonoidalproduct (\theone+\theideal))&\Seqpd& \anymatAX{\indj}{\Sexchanged(\indi)}-\Skronecker{\indj}{\Sexchanged(\indi)}\Saction\theone+\theideal
\end{IEEEeqnarray*}
and
\begin{IEEEeqnarray*}{rCl}
  \IEEEeqnarraymulticol{3}{l}{\textstyle\thedifferential{2}(\Sbv{\anymatA}{\indj}{\indi}\Smonoidalproduct (\theone+\theideal))
  }\\
\hspace{1em}  &\Seqpd&  \textstyle\sum_{\inds=1}^\thedim\Sbv{\anymatA}{\indj}{\inds}\Smonoidalproduct(\anymatBTX{\inds}{\Sexchanged(\indi)}+\theideal)+\Sbv{\anymatBT}{\Sexchanged(\indj)}{\indi}\Smonoidalproduct(\theone+\theideal)\\
&&\textstyle{}-\Skronecker{\indj}{\thedim}\Skronecker{\indi}{\thedim}\sum_{\inds=2}^{\thedim}(\sum_{\indt=1}^{\thedim}\Sbv{\anymatA}{\indt}{\inds}\Smonoidalproduct(\anymatBTX{\inds}{\Sexchanged(\indt)}+\theideal)+\Sbv{\anymatBT}{\inds}{\Sexchanged(\inds)}\Smonoidalproduct(\theone+\theideal))
\end{IEEEeqnarray*}
and  for any $\orderind\in\Sintegersnn$ with $3\leq \orderind$, if for any $\othermatA\in\thematrices$ the symbol $\othermatPA{\orderind}$ is short for $\othermatA$ if $\orderind$ is odd and for $\othermatBT$ if $\orderind$ is even, then 
\begin{IEEEeqnarray*}{rCl}
  \IEEEeqnarraymulticol{3}{l}{
    \thedifferential{\orderind}(\Sbv{\anymatA}{\indj}{\indi}\Smonoidalproduct (\theone+\theideal))
  }\\
\hspace{1em}  &\Seqpd &\textstyle  \sum_{\inds=1}^\thedim\Sbv{\anymatA}{\indj}{\inds}\Smonoidalproduct(\anymatPAX{\orderind}{\inds}{\Sexchanged(\indi)}+\theideal)+(-1)^{\orderind}\Saction\Sbv{\anymatBT}{\Sexchanged(\indj)}{\indi}\Smonoidalproduct(\theone+\theideal)\\
&&\textstyle{}+\Skronecker{\indj}{\thedim}((\Sbv{\anymatAT}{1}{1}+\Skronecker{\orderind}{3}\sum_{\inds=2}^{\thedim-1}\Sbv{\anymatAT}{\inds}{\inds})\Smonoidalproduct(\anymatPATX{\orderind}{\Sexchanged(\indi)}{\thedim}+\theideal)\\
&&\textstyle\hspace{3em}\hfill{}-\Skronecker{\indi}{\thedim}(\sum_{\inds=1}^\thedim\Sbv{\anymatAT}{1}{\inds}\Smonoidalproduct(\anymatPATX{\orderind}{\inds}{\thedim}+\theideal)+(-1)^{\orderind}\Saction\Sbv{\anymatB}{\thedim}{1}\Smonoidalproduct(\theone+\theideal)))\\
&&\textstyle{}+\Skronecker{\indj}{1}\Skronecker{\indi}{\thedim}(-1)^{\orderind}(\Sbv{\anymatB}{1}{1}+\Skronecker{\orderind}{3}\sum_{\inds=2}^{\thedim-1}\Sbv{\anymatB}{\inds}{\inds})\Smonoidalproduct(\theone+\theideal),
\end{IEEEeqnarray*}
then the complex
\begin{IEEEeqnarray*}{C}
  \begin{tikzpicture}
    \node (p0) at (0,0) {$\themodule$};
    \node (p1) [left = 1cm of p0] {$\thealgebra$};
    \node (p2) [left = 1cm of p1] {$\thefield\thechains{1}\Smonoidalproduct\thealgebra$};
    \node (p3) [left = 1cm of p2] {$\thefield\thechains{2}\Smonoidalproduct\thealgebra$};
    \node (p4) [left = 1cm of p3] {$\hdots$};
    \draw [->] (p4) to node [above] {$\thedifferential{3}$} (p3) ;
    \draw [->] (p3) to node [above] {$\thedifferential{2}$} (p2);
    \draw [->] (p2) to node [above] {$\thedifferential{1}$} (p1);
    \draw [->] (p1) to node [above] {$\thedifferential{0}$} (p0);
  \end{tikzpicture}
\end{IEEEeqnarray*}
of right $\thealgebra$-modules is exact. (In other words,  $(\thechainsmodule{\indl},\thedifferential{\indl})_{\indl\in\Sintegersnn}$ is a free resolution of $\themodule$ in terms of right $\thealgebra$-modules.)
In fact, it is precisely the Anick resolution of $\themodule$ with respect to the degreewise lexicographic extension of the total order on $\thegens$ defined by, for any $\{\colc,\cold\}\subseteq \Scolors$ and any $\{(\indj,\indi),(\indt,\inds)\}\subseteq \SYnumbers{\thedim}^{\Ssetmonoidalproduct 2}$ with respect to the lexicographic order on $\SYnumbers{\thedim}^{\Ssetmonoidalproduct 2}$,
    \begin{IEEEeqnarray*}{rClCrCl}
      \thecuniX{\cold}{\indt}{\inds}&\leq& \thecuniX{\colc}{\indj}{\indi} &\Siff& (\cold,\colc)&=&{\mathnormal\Sblack\Swhite}\\
&&&&      {}\Sor(\cold,\colc)&=&{\mathnormal\Sblack\Sblack}\Sand (\indj,\indi)<(\indt,\inds)\\
&&&&      {}\Sor (\cold,\colc)&=&{\mathnormal\Swhite\Swhite}\Sand (\indt,\inds)<(\indj,\indi).
\end{IEEEeqnarray*}
Finally, if $\thefield=\Scomps$, then   $\thealgebra$ is  the underlying algebra and $\thecounit$ the co-unit of the CQG algebra $\mathcal{O}(U_n^+)$ of representative functions on the free unitary quantum group $U_n^+$.
\end{main*}
This is an inferior resolution compared to the one previously obtained by Baraquin, Franz, Gerhold, Kula and Tobolski in \cite{BaraquinFranzGerholdKulaTobolski2023} (see Appendix~\ref{section:relationship} for a qua\-si\-/iso\-mor\-phism). However, it can still be used to compute the (previously known) quantum group co\-ho\-mo\-lo\-gy
\begin{IEEEeqnarray*}{rCl}
  H^{\Sargph}(\widehat{U_n^+})&\cong&\SextX{\thealgebra}{\themodulecomps}{\themodulecomps}{\Sargph},
\end{IEEEeqnarray*}
as demonstrated in Section~\ref{section:cohomology}.
\par
The significance of the Main result is that it is evidence that perhaps Anick resolutions can be used to compute homological invariants of easy quantum groups. For a discussion in that respect see Section~\ref{section:discussion}.

\section{Reminder on Anick resolutions}
\label{section:anick}
{
  \newcommand{\smallelement}{b}
  \newcommand{\bigelement}{b'}
  \newcommand{\leftelement}{b_1}
  \newcommand{\rightelement}{b_2}
  \newcommand{\smalllen}{m}
  \newcommand{\biglen}{m'}
  \newcommand{\smallgen}{e}
  \newcommand{\smallgenX}[1]{\smallgen_{#1}}
  \newcommand{\biggen}{e'}
  \newcommand{\biggenX}[1]{\biggen_{#1}}
  \newcommand{\indi}{i}
  \newcommand{\indj}{j}
  \newcommand{\posi}{i}
  \newcommand{\distposi}{j}
  \newcommand{\anyvector}{v}
  \newcommand{\anymonomial}{b}
  \newcommand{\othermonomial}{b'}  
  \newcommand{\anyset}{G}
  \newcommand{\smallgroebner}{g}
  \newcommand{\biggroebner}{g'}
  \newcommand{\anygroebner}{g}
  \newcommand{\leftgroebner}{g_1}
  \newcommand{\rightgroebner}{g_2}
  \newcommand{\leftreduction}{\varphi_1}
  \newcommand{\rightreduction}{\varphi_2}
  \newcommand{\chainlen}{m}
  \newcommand{\orderind}{\ell}
  \newcommand{\orderindsmall}{\orderind'}  
  \newcommand{\chaindegree}{k}
  \newcommand{\chaindegreesmall}{k'}  
  \newcommand{\chainstarts}{\alpha}
  \newcommand{\chainstartsX}[1]{\chainstarts_{#1}}
  \newcommand{\chainends}{\beta}
  \newcommand{\chainendsX}[1]{\chainends_{#1}}
  \newcommand{\chaingen}{e}
  \newcommand{\anymonoid}{B}
  \newcommand{\anychain}{c}
  \newcommand{\otherchain}{c'}  
  \newcommand{\anyreduced}{b}
  \newcommand{\otherreduced}{b'}
  \newcommand{\monideal}{I}
  \newcommand{\anymodule}{\thefield_\thecounit}  
  \newcommand{\chaingenX}[1]{\chaingen_{#1}}
  \newcommand{\anyspace}{F}
  \newcommand{\anybasis}{B}
  \newcommand{\anygen}{e}  
  \newcommand{\anybv}{b}
  \newcommand{\otherbv}{b'}
  \newcommand{\anysubspace}{J}
  \newcommand{\anysubset}{J}
  \newcommand{\smallchain}{a}
  \newcommand{\chainrest}{b}
  \newcommand{\somenonunit}{\tilde{b}}  
  \newcommand{\anykernelvector}{v}
  \newcommand{\bigchain}{g}
  \newcommand{\prodrest}{h}  
  \newcommand{\anyextlen}{t}
  \newcommand{\anytipscal}{\lambda}
  \newcommand{\anyextra}{d}
  \newcommand{\posind}{s}      
  Anick presented the resolutions named after him in \cite{Anick1986}.  They are commonly explained within the framework of non\-/com\-mu\-ta\-ti\-ve Gröb\-ner bases developed by Mora in \cite{Mora1985}. Berg\-man's diamond lem\-ma from \cite{Bergman1978} also fits in this framework. All three are recalled below. 
  \par
  \subsection{Tips}
\label{section:anick-tips}  
  For any vector space $\anyspace$ over any field $\thefield$, any $\thefield$-basis $\anybasis$, any vector $\anyvector\in\anyspace$ and any basis element $\anybv\in\anybasis$ let $\SevaluateX{\anyvector}{\anybv}$ be the coefficient of $\anyvector$ with respect to $\anybv$. Given any total order $\leq$ on $\anybasis$ and any $\anyvector\in\anyspace$ with $\anyvector\neq 0$, there exists a unique element of  $\anybasis$, called the \emph{tip} of $\anyvector$ with respect to $\anybasis$ and $\leq$ and denoted by the (non-standard) symbol $\Stip\anyvector$, which is $\leq$-maximal among the elements $\anybv\in\anybasis$ with the property that $\SevaluateX{\anyvector}{\anybv}\neq 0$. Moreover, for any subset $\anysubset$ of $\anyspace$ let $\StipO\anysubset\Seqpd \{\Stip\anyvector\Ssetbuilder\anyvector\in\anysubset\Sand \anyvector\neq 0\}$.
  \par
  A \emph{well\-/order} is any total order with respect to which any non-empty subset has a minimal element. By a trans\-fi\-nite ge\-ne\-ra\-li\-za\-tion of Gaus\-sian e\-li\-mi\-na\-tion, if $\leq$ is a well\-/order on $\anybasis$, then any $\thefield$-vector subspace $\anysubspace$ induces a decomposition of $\anyspace$ into a direct sum $\anyspace=\anysubspace\Sdirectsum\Sspan(\anybasis\backslash \Stip\anysubspace)$ of $\thefield$-vector spaces, where $\anybasis\backslash \Stip\anysubspace$ is the set complement of $\Stip\anysubspace$ in $\anybasis$, sometimes referred to as the \emph{non-tips} of $\anysubspace$ with respect to $\anybasis$ and $\leq$. For any $\anyvector\in \anyspace$ the unique $\anyreduced\in \Sspan(\anybasis\backslash \Stip\anysubspace)$ such that $\anyvector-\anyreduced\in \anysubspace$ is called the \emph{normal form} of $\anyvector$ with respect to $\anysubspace$, $\anybasis$ and $\leq$.
  \par
  In particular, for any $\{\anybv,\otherbv\}\subseteq \anybasis\backslash\StipO\anysubspace$ with $\anybv\neq \otherbv$ the classes $\anybv+\anysubspace$ and $\otherbv+\anysubspace$ are then distinct. And, the set $\{\anybv+\anysubspace\Ssetbuilder\anybv\in \anybasis\backslash \StipO\anysubspace\}$ is a $\thefield$-basis of the quotient $\thefield$-vector space $\anyspace/\anysubspace$ of $\anyspace$ by $\anysubspace$. The mapping which for any $\anyvector\in\anyspace$ sends the class $\anyvector+\anysubspace$ to the normal form of $\anyvector$ is a well-defined $\thefield$-linear splitting $\Sfromto{\anyspace/\anysubspace}{\anyspace}$ of the projection map $\Sfromto{\anyspace}{\anyspace/\anysubspace}$.
  \subsection{Monoidal ideals}
\label{section:anick-monoidal_ideals}  
  In the following, a \emph{monoid} $\anymonoid$ is any set equipped with an associative binary operation with a neutral element. Any $\smallelement\in\anymonoid$ is said to \emph{divide} any $\bigelement\in\anymonoid$ if there exist $\{\leftelement,\rightelement\}\subseteq \anymonoid$ with $\leftelement\smallelement\rightelement=\bigelement$. Equivalently, $\bigelement$ is called a \emph{multiple} of $\smallelement$.
  \par
  A \emph{monoidal ideal} of $\anymonoid$ is any subset of $\anymonoid$ which also contains all multiples of its elements. Given any monoidal ideal $\monideal$ of $\anymonoid$, there may or may not exist a subset $\thecore$ which generates $\monideal$ as a monoidal ideal and which is minimal with that property with respect to the subset inclusion partial order $\subseteq$. If such a set $\thecore$ exists, it is unique and will, for brevity, be referred to as the \emph{core} of $\monideal$, which is non-standard terminology.
  \par
  \subsection{Monomial orders}
  \label{section:anick-monomial_orders}  
A well-order on a monoid $\anymonoid$ is called \emph{admissible} if for any $\{\smallelement,\bigelement,\leftelement,\rightelement\}\subseteq \anymonoid$, whenever $\smallelement<\bigelement$, then also $\leftelement\smallelement<\leftelement\bigelement$ and  $\smallelement\rightelement<\bigelement\rightelement$. If there exists at least one admissible order on $\anymonoid$, then either $\anymonoid$ is a trivial monoid, consisting only of a single element, or $\anymonoid$ has no left or right absorbing elements. Moreover, it can be seen that then any monoidal ideal of $\anymonoid$ has a core.
  \par
Let $\thegens$ be any set.
If $\anymonoid=\themonoid$ is the free monoid over $\thegens$, then any admissible order on $\themonoid$ is called  a \emph{monomial order} for $\thegens$.
One way of obtaining monomial  orders for $\thegens$ is to extend a given well-order on $\thegens$ \emph{de\-gree\-wise le\-xi\-co\-gra\-phi\-cal\-ly}: That means that for any $\{\smalllen,\biglen\}\subseteq \Sintegersnn$ and any $\smallgen\in\thegens^{\Ssetmonoidalproduct\smalllen}$ and $\biggen\in\thegens^{\Ssetmonoidalproduct\biglen}$ it is defined that $\smallgenX{1}\smallgenX{2}\ldots\smallgenX{\smalllen}<\biggenX{1}\biggenX{2}\ldots\biggenX{\biglen}$ if and only if either $\smalllen<\biglen$ or simultaneously $\smalllen=\biglen$ and there exists $\distposi\in\SYnumbers{\smalllen}$ such that $\smallgenX{\posi}=\biggenX{\posi}$ for any $\posi\in\SYnumbers{\distposi-1}$ and  $\smallgenX{\distposi}<\biggenX{\distposi}$.
  \par
  \subsection{Gröbner bases}
  \label{section:anick-groebner_bases}  
  Let $\thegens$ be any set, $\thefield$  any field, $\anyspace=\thefreealg$ the free $\thefield$-algebra over $\thegens$, moreover $\theone$ its unit, $\leq$ any monomial order for $\thegens$ and $\theideal$ any (two-sided) ideal of $\thefreealg$. Then, $\anymonoid=\themonoid$ is a basis of $\thefreealg$ and, with respect to that basis and $\leq$, the tips $\monideal=\StipO\theideal$ of $\theideal$ form a monoidal ideal in $\themonoid$. The elements of $\themonoid$ are usually called \emph{monomials} in $\thegens$ and the non-tips $\themonoid\backslash \StipO\theideal$  \emph{reduced}  monomials. Because $\leq$ is a well-order the monoidal ideal $\StipO\theideal$ admits a core $\thecore$, sometimes called the \emph{obstructions} of $\theideal$ with respect to $\leq$.
  \par
  A \emph{Gröbner basis} of  $\theideal$  with respect to  $\leq$ is any subset $\thegroebner$ of $\theideal$ such that 
$\StipO\thegroebner$ generates  $\StipO\theideal$ as a monoidal ideal of $\themonoid$ or, equivalently, such that $\thecore\subseteq \StipO\thegroebner$.
  Any Gröbner basis $\thegroebner$ is said to be \emph{minimal} if $0\notin\thegroebner$ and if there are no $\{\smallgroebner,\biggroebner\}\subseteq\thegroebner$ with $\smallgroebner\neq\biggroebner$ such that $\Stip\smallgroebner$ divides $\Stip\biggroebner$. That is the case if and only if $\StipO\thegroebner=\thecore$. 
  Even further, any minimal Gröbner basis is called \emph{reduced} if $\SevaluateX{\anygroebner}{\Stip\anygroebner}=1$ for any $\anygroebner\in\thegroebner$ and if there exist no $\{\smallgroebner,\biggroebner\}\subseteq\thegroebner$ and  $\anymonomial\in\themonoid$ such that  $\smallgroebner\neq\biggroebner$ and $\SevaluateX{\biggroebner}{\anymonomial}\neq 0$ and such that $\Stip\smallgroebner$ divides $\anymonomial$.
  In fact, with respect to $\leq$ there is exactly one reduced Gröbner basis of $\theideal$. It is in particular a $\thefield$-basis of $\theideal$.
  \subsection{Diamond lemma}
\label{section:anick-diamond_lemma}    
  For any $\anyvector\in\thefreealg$ with $\anyvector\neq 0$  a \emph{one-step reduction} by $\anyvector$ with respect to $\leq$ is any  $\thefield$-vector space endomorphism of $\thefreealg$ for which there exist  $\{\leftelement,\rightelement\}\subseteq \themonoid$ such that $\leftelement(\Stip\anyvector)\rightelement$ is mapped to $\leftelement(\Stip\anyvector-\frac{1}{\SevaluateX{\anyvector}{\Stip\anyvector}}\anyvector)\rightelement$ and any other element of $\themonoid$ to itself.
  \par
  Let $\thegroebner$ be any subset of $\thefreealg$. A \emph{reduction} by $\thegroebner$  with respect to $\leq$ is any element of a certain submonoid of the monoid formed by the  $\thefield$-vector space endomorphisms of $\thefreealg$ under composition, namely the one generated by the one-step reductions by arbitrary non-zero elements of $\thegroebner$.
  \par
An \emph{inclusion ambiguity} of  $\thegroebner$ with respect to $\leq$ is any quadruple $(\smallgroebner,\biggroebner,\leftelement,\rightelement)$ such that $\{\smallgroebner,\biggroebner\}\subseteq \thegroebner$ and $\{\leftelement,\rightelement\}\subseteq \themonoid$ and $\smallgroebner\neq 0$ and $\biggroebner\neq 0$ and $\smallgroebner\neq\biggroebner$ and $\leftelement(\Stip\smallgroebner)\rightelement=\Stip\biggroebner$. And it is then said to \emph{resolve} if there exist reductions $\leftreduction$ and $\rightreduction$ by $\thegroebner$ with respect to $\leq$ such that $\leftreduction(\Stip\biggroebner-\frac{1}{\SevaluateX{\biggroebner}{\Stip\biggroebner}}\biggroebner)=\rightreduction(\leftelement(\Stip\smallgroebner-\frac{1}{\SevaluateX{\smallgroebner}{\Stip\smallgroebner}}\smallgroebner)\rightelement)$.
  \par
  Similarly, an \emph{overlap ambiguity} of $\thegroebner$ with respect to $\leq$ is any tuple $(\leftgroebner,\rightgroebner,\leftelement,\rightelement
  )$ with $\{\leftgroebner,\rightgroebner\}\subseteq \thegroebner$ and $\{\leftelement,\rightelement\}\subseteq \themonoid$  and $\leftgroebner\neq 0$ and $\rightgroebner\neq 0$ such that $(\leftelement,\rightelement)\neq (1,1)$, such that $(\Stip\leftgroebner)\rightelement=\leftelement(\Stip\rightgroebner)$ and such that neither is $\rightelement$ divided by $\Stip\rightgroebner$  nor  $\leftelement$ by $\Stip\leftgroebner$. It is  said to \emph{resolve} if there exist reductions $\leftreduction$ and $\rightreduction$ by $\thegroebner$ such that $\leftreduction((\Stip\leftgroebner-\frac{1}{\SevaluateX{\leftgroebner}{\Stip\leftgroebner}}\leftgroebner)\rightelement)=\rightreduction(\leftelement(\Stip\rightgroebner-\frac{1}{\SevaluateX{\rightgroebner}{\Stip\rightgroebner}}\rightgroebner))$.
  \begin{proposition}[Bergman's diamond lemma]
    \label{proposition:bergman}
With respect to $\leq$, the set $\thegroebner$ is a Gröbner basis of $\theideal$ if and only if $\thegroebner$ generates $\theideal$ as an ideal and all inclusion and overlap ambiguities of $\thegroebner$ resolve. 
  \end{proposition}
  \subsection{Chains}
\label{section:anick-chains}  
Suppose $\thegens\cap\StipO\theideal=\emptyset$ and let $\theprechains{0}\Seqpd\thechains{0}\Seqpd\{1\}$ be the set containing the unit of $\thefreealg$ and let $\theprechains{1}\Seqpd\thechains{1}\Seqpd \thegens$.
\par
Suppressing the dependency on $\theideal$ and $\leq$ in the notation, for any $\orderind\in\Sintegersnn$ with $2\leq \orderind$, any $\chainlen\in\Sintegersnn$ and any $\chaingen \in\thegens^{\Ssetmonoidalproduct\chainlen}$  the monomial $\anychain=\chaingenX{1}\chaingenX{2}\hdots\chaingenX{\chainlen}$ is called 
an \emph{$\orderind$-pre-chain}, in symbols:\ $\anychain\in\theprechains{\orderind}$, if and only if  there exist $\chainstarts\in\SYnumbers{\chainlen}^{\Ssetmonoidalproduct (\orderind-1)}$ and $\chainends\in\SYnumbers{\chainlen}^{\Ssetmonoidalproduct (\orderind-1)}$ such that either $\orderind=2$ and $1=\chainstartsX{1}\leq \chainendsX{1}=\chainlen$ or $3\leq \orderind$ and   $1=\chainstartsX{1}<\chainstartsX{2}\leq \chainendsX{1}$ and $\chainstartsX{\orderind-1}\leq\chainendsX{\orderind-2}<\chainendsX{\orderind-1}= \chainlen$ and $\chainstartsX{\posind+1}\leq \chainendsX{\posind}<\chainstartsX{\posind+2}$ for any $\posind\in\SYnumbers{\orderind-3}$ and such that $\chaingenX{\chainstartsX{\posind}}\chaingenX{\chainstartsX{\posind}+1}\ldots\chaingenX{\chainendsX{\posind}}\in \thecore$ for any $\posind\in\SYnumbers{\orderind-1}$.
\par
Furthermore, $\anychain$ is called an \emph{$\orderind$-chain}, an element of $\thechains{\orderind}$, if and only if there exist $\chainstarts$ and $\chainends$ as above such that, additionally, for any $\orderindsmall\in\SYnumbers{\orderind}$ with $2\leq \orderindsmall$ and any $\posind\in\SYnumbers{\chainendsX{\orderindsmall-1}-1}$ the monomial $\chaingenX{1}\chaingenX{2}\hdots\chaingenX{\posind}$ is not an element of  $\theprechains{\orderindsmall}$. If so, then $(\chainstarts,\chainends)$ is unique with those extended properties and, in the present article, will be referred to as the \emph{$\orderind$-chain indices} of $\anychain$. Beware that compared to \cite{Anick1986} the numbering of $(\thechains{\orderind})_{\orderind\in\Sintegersnn}$ has been shifted by $1$. An $\orderind$\-/chain here would be an $(\orderind-1)$\-/chain there.
\subsection{Combinatorial differentials and splittings}
\label{section:anick-combinatorial_differentials_and_splitting}  
For any $\orderind \in\Sintegersp$, any $\chainlen\in \Sintegersnn$, any $\chaingen\in\thegens^{\Ssetmonoidalproduct \chainlen}$ and any $\orderind$\-/chain $\anychain=\chaingenX{1}\chaingenX{2}\ldots\chaingenX{\chainlen}\in \thechains{\orderind}$ let $\thedeconcatenator{\orderind}(\anychain)$ be defined as $(\theone,\anychain)$ if $\orderind=1$, as $(\chaingenX{1},\chaingenX{2}\chaingenX{3}\ldots\chaingenX{\chainlen})$ if $\orderind=2$, and otherwise, if  $(\chainstarts,\chainends)$  are the chain indices of $\anychain$, as  $ (\chaingenX{1}\chaingenX{2}\hdots\chaingenX{\chainendsX{\orderind-2}},\chaingenX{\chainendsX{\orderind-2}+1}\chaingenX{\chainendsX{\orderind-2}+2}\hdots\chaingenX{\chainlen})$. It can be proved that $\thedeconcatenator{\orderind}$ is a mapping $\Sfromto{\thechains{\orderind}}{\thechains{\orderind-1}\Ssetmonoidalproduct(\themonoid\backslash \StipO\theideal)}$ if $2\leq\orderind$.
\par
Moreover, let $\thechaincycles{0}\Seqpd (\thechains{0}\Ssetmonoidalproduct(\themonoid\backslash\StipO\theideal))\backslash \{(\theone,\theone)\}$ and for any $\orderind\in\Sintegersp$ define $\thechaincycles{\orderind}$ to be the set of all $(\anychain,\anyextra)\in\thechains{\orderind}\Ssetmonoidalproduct(\themonoid\backslash \StipO\theideal)$ such that, if $(\smallchain,\chainrest)=\thedeconcatenator{\orderind}(\anychain)$, then $\chainrest\anyextra\in\StipO\theideal$.
\par 
For any $\orderind\in\Sintegersnn$ and any $(\anychain,\anyextra)\in\thechaincycles{\orderind}$, if $\{\chainlen,\anyextlen\}\subseteq\Sintegersnn$ and  $\chaingen \in\thegens^{\Ssetmonoidalproduct(\chainlen+\anyextlen)}$ are such that $\anychain=\chaingenX{1}\chaingenX{2}\hdots\chaingenX{\chainlen}$ and $\anyextra=\chaingenX{\chainlen+1}\chaingenX{\chainlen+2}\hdots\chaingenX{\chainlen+\anyextlen}$, then let $\thereconcatenator{\orderind}(\anychain,\anyextra)\Seqpd(\chaingenX{1}\chaingenX{2}\hdots\chaingenX{\chainendsX{\orderind}},\allowbreak\chaingenX{\chainendsX{\orderind}+1}\chaingenX{\chainendsX{\orderind}+2}\hdots\chaingenX{\chainlen+\anyextlen})$, where $\chainendsX{\orderind}\in\SYnumbers{\chainlen+\anyextlen}$  is such that $\chainlen<\chainendsX{\orderind}$, such that, if $\orderind=0$, then  $\chainendsX{\orderind}=1$, such that, if $1\leq \orderind$, then  $\chaingenX{\chainstartsX{\orderind}}\chaingenX{\chainstartsX{\orderind}+1}\hdots\chaingenX{\chainendsX{\orderind}}\in\thecore$, where $\chainstartsX{\orderind}\in \SYnumbers{\chainlen}$ is given by $1$ if $\orderind=1$, by $\min\{\indi\in\Sintegers\Sand\exists\indj\in \Sintegers\Ssuchthat 1<\indi\leq\chainlen<\indj<\chainlen+\anyextlen\Sand \chaingenX{\indi}\chaingenX{\indi+1}\hdots\chaingenX{\indj}\in\thecore\}$ if $\orderind=2$,   and by $\min\{\indi\in\Sintegers\Sand\exists\indj\in \Sintegers\Ssuchthat\chainendsX{\orderind-2}<\indi\leq\chainlen<\indj<\chainlen+\anyextlen\Sand \chaingenX{\indi}\chaingenX{\indi+1}\hdots\chaingenX{\indj}\in\thecore\}$ if $3\leq \orderind$, where  $((\chainstartsX{1},\chainstartsX{2},\ldots, \chainstartsX{\orderind-1}),(\chainendsX{1},\chainendsX{2},\ldots,\chainendsX{\orderind-1}))$ are the $\orderind$\-/chain indices of $\anychain$.
One proves that $\thereconcatenator{\orderind}$ is a mapping $\Sfromto{\thechaincycles{\orderind}}{\thechains{\orderind+1}\Ssetmonoidalproduct\themonoid}$.
\par
Note that $\thedeconcatenator{\orderind}$, $\thechaincycles{\orderind}$ and $\thereconcatenator{\orderind}$ are left implicit in \cite{Anick1986}. They can be interpreted as the combinatorial base of, respectively, the differentials, cycles and splittings of the resolution given below.
\subsection{Chain modules}
\label{section:anick-chain_modules}  
Let $\thealgebra$ denote the quotient $\thefield$-algebra $\thefreealg/\theideal$ of $\thefreealg$ by $\theideal$ and let $\thecounit$ be any morphism $\Sfromto{\thealgebra}{\thefield}$ of $\thefield$-algebras. Then $\thecounit$ induces a right $\thealgebra$-module structure $\anymodule$ on $\thefield$. Let $\thechainsmodule{-1}\Seqpd \anymodule$ and for any $\orderind\in\Sintegersnn$ let $\thechainsmodule{\orderind}\Seqpd \thefield\thechains{\orderind}\SmonoidalproductC{\thefield}\thealgebra$, where $\thefield\thechains{\orderind}$ is the free $\thefield$-vector space over $\thechains{\orderind}$.
\par
There then exists a unique total order $\thechainsorder{\orderind}$ on the $\thefield$\-/ba\-sis  $\thechainsbasis{\orderind}\Seqpd \{\anychain\Smonoidalproduct(\anyreduced+\theideal)\Ssetbuilder(\anychain,\anyreduced)\in\thechains{\orderind}\Ssetmonoidalproduct(\themonoid\backslash \StipO\theideal)\}$ of $\thechainsmodule{\orderind}$ such that for any $\{\anychain,\otherchain\}\subseteq\thechains{\orderind}$ and any $\{\anyreduced,\otherreduced\}\subseteq\themonoid\backslash \StipO\theideal$ the inequality $\anychain\Smonoidalproduct\anyreduced\thechainsorderR{\orderind}\otherchain\Smonoidalproduct\otherreduced$ holds if and only if $\anychain\anyreduced\leq \otherchain\otherreduced$. The tips of non-zero vectors in $\thechainsmodule{\orderind}$ with respect to $\thechainsorder{\orderind}$ will be indicated by the symbol $\thechainstip{\orderind}$.
\subsection{Anick resolution}
\label{section:anick-anick_resolution}  
There exist unique families $(\thedifferential{\orderind})_{\orderind\in \Sintegersnn}$ and $(\thesplitting{\orderind})_{\orderind\in \Sintegersnn}$ of mappings such that the following conditions are met. For each $\orderind\in\Sintegersnn$ the mapping $\thedifferential{\orderind}$ is a morphism $\Sfromto{\thechainsmodule{\orderind}}{\thechainsmodule{\orderind-1}}$ of right $\thealgebra$-modules and $\thesplitting{\orderind}$ is a $\thefield$-linear map $\Sfromto{\Sker(\thedifferential{\orderind})}{\thechainsmodule{\orderind+1}}$, where $\Sker(\thedifferential{\orderind})\subseteq \thechainsmodule{\orderind}$ is the kernel of $\thedifferential{\orderind}$ as a $\thefield$-linear map. The map $\thedifferential{0}$ satisfies $\thedifferential{0}(\theone\Smonoidalproduct(\theone+\theideal))= 1$ and for any $\anygen\in\thegens=\thechains{1}$ it holds $\thedifferential{1}(\anygen\Smonoidalproduct(\theone+\theideal))=\theone\Smonoidalproduct(\anygen-\thecounit(\anygen)+\theideal)$. For each $\orderind\in\Sintegersnn$ with $2\leq \orderind$ and any $\anychain\in \thechains{\orderind}$, if $\thedeconcatenator{\orderind}(\anychain)=(\smallchain,\chainrest)$, then 
$\thedifferential{\orderind}(\anychain\Smonoidalproduct(\theone+\theideal))=\smallchain\Smonoidalproduct(\chainrest+\theideal)-\thesplitting{\orderind-2}(\thedifferential{\orderind-1}(\smallchain\Smonoidalproduct(\chainrest+\theideal)))$. For any $\orderind\in\Sintegersnn$ and any $\anykernelvector\in\Sker(\thedifferential{\orderind})$ with $\anykernelvector\neq 0$, if $\anytipscal=\SevaluateX{\anykernelvector}{\thechainstip{\orderind}\anykernelvector}$, if $\anychain\in \thechains{\orderind}$ and $\anyextra\in \themonoid\backslash\StipO\theideal$ are such that $\thechainstip{\orderind}\anykernelvector=\anychain\Smonoidalproduct(\anyextra+\theideal)$, in which case it can be proved that $(\anychain,\anyextra)\in \thechaincycles{\orderind}$,  and if $(\bigchain, \prodrest)=\thereconcatenator{\orderind}(\anychain,\anyextra)$, then $\thesplitting{\orderind}(\anykernelvector)=\anytipscal\Saction\bigchain\Smonoidalproduct(\prodrest+\theideal)+\thesplitting{\orderind}(\anykernelvector-\anytipscal\Saction\thedifferential{\orderind+1}(\bigchain\Smonoidalproduct(\prodrest+\theideal)))$.
  \begin{proposition}[Anick's resolution]
    $(\thechainsmodule{\orderind},\thedifferential{\orderind})_{\orderind\in \Sintegersnn}$ is a resolution of $\anymodule$ in terms of free right $\thealgebra$-modules.
  \end{proposition}
  }

\section{\texorpdfstring{Gröbner basis for $U_n^+$}{Gröbner basis for the free unitary quantum group}}
\label{section:groebner_computation}
From now on let $\thefield$, $\thedim$, $\thegens$, $\thematrices$, $\therels$, $\theideal$, $\thealgebra$ and $\theorder$ be as in the  \hyperref[main-result]{Main result}. 
  In this section the reduced Gröbner basis of $\theideal$ with respect to $\theorder$ is computed.
  {
    \newcommand{\indi}{i}
    \newcommand{\indj}{j}
    \newcommand{\inds}{s}
    \newcommand{\indt}{t}
    \newcommand{\anymatX}[2]{\anymat_{#1,#2}}
  \begin{definition}
    For any $\anymat\in\thematrices$ and any $(\indj,\indi)\in \SYnumbers{\thedim}^{\Ssetmonoidalproduct 2}$ let
      \begin{IEEEeqnarray*}{rCl}
        \thebv{\anymat}{\indj}{\indi}&\Seqpd& \anymatAX{\indj}{\thedim}\anymatBTX{1}{\Sexchanged(\indi)}-(\textstyle-\sum_{\inds=2}^{\thedim}\anymatAX{\indj}{\Sexchanged(\inds)}\anymatBTX{\inds}{\Sexchanged(\indi)}+\Skronecker{\indj}{\indi}\theone)\in\thefreealg \\
        \specbv{\anymat}&\Seqpd& \anymatAX{\thedim}{\thedim}\anymatBTX{1}{1}-(\textstyle \sum_{\inds=2}^{\thedim}\sum_{\indt=1}^{\thedim-1} \anymatAX{\indt}{\Sexchanged(\inds)}\anymatBTX{\inds}{\Sexchanged(\indt)}-(\thedim-2)\theone)\in\thefreealg.
      \end{IEEEeqnarray*}
      and then let
      \begin{IEEEeqnarray*}{rCl}
        \thegroebner&\Seqpd&\{\specbv{\anymat},\thebv{\anymat}{\indj}{\indi}\Ssetbuilder\anymat\in\thematrices\Sand (\indj,\indi)\in\SYnumbers{\thedim}^{\Ssetmonoidalproduct 2}\backslash \{(\thedim,\thedim)\}\}.
      \end{IEEEeqnarray*}
  \end{definition}
With these abbreviations,  $\therels=\{\thebv{\anymat}{\indj}{\indi}\Ssetbuilder\anymat\in\thematrices\Sand (\indj,\indi)\in\SYnumbers{\thedim}^{\Ssetmonoidalproduct 2}\}$. Eventually, $\thegroebner$ will be shown to be the reduced Gröbner basis.
}

\subsection{Auxiliary results}
During the computation of the Gröbner basis the following fundamental consequences of the definition of the monomial order $\theorder$ will be used frequently.
{
    \newcommand{\anymatX}[2]{\anymat_{#1,#2}}
    \newcommand{\othermatX}[2]{\othermat_{#1,#2}}  
    \newcommand{\indi}{i}
    \newcommand{\indj}{j}
    \newcommand{\indk}{k}
    \newcommand{\indl}{\ell}
    \newcommand{\inds}{s}
    \newcommand{\indt}{t}    
    \begin{lemma}
      \label{lemma:any_two_generators}
 $\othermat\in \{\anymat,\anymat\StransposeP,\anymat\SskewstarP,\anymat\SskewdaggerP\}$ for any $\{\anymat,\othermat\}\subseteq\thematrices$.    
\end{lemma}
\begin{proof} The below tables gives the argument for each of the $4\times 4$ cases.
  \begin{center}
    \begin{tabular}{c||cccc|cccc|cccc|cccc}    $\anymat$&$\theuni$&$\theuni$&$\theuni$&$\theuni$&$\theuni\StransposeP$&$\theuni\StransposeP$&$\theuni\StransposeP$&$\theuni\StransposeP$&$\theuni\SskewstarP$&$\theuni\SskewstarP$&$\theuni\SskewstarP$&$\theuni\SskewstarP$&$\theuni\SskewdaggerP$&$\theuni\SskewdaggerP$&$\theuni\SskewdaggerP$&$\theuni\SskewdaggerP$\\
      $\othermat$&$\theuni$&$\theuni\StransposeP$&$\theuni\SskewstarP$&$\theuni\SskewdaggerP$&$\theuni$&$\theuni\StransposeP$&$\theuni\SskewstarP$&$\theuni\SskewdaggerP$&$\theuni$&$\theuni\StransposeP$&$\theuni\SskewstarP$&$\theuni\SskewdaggerP$&$\theuni$&$\theuni\StransposeP$&$\theuni\SskewstarP$&$\theuni\SskewdaggerP$ \\\hline
\Tstrut$\othermat$ & $\anymat$ &$\anymat\StransposeP$ &$\anymat\SskewstarP$ &$\anymat\SskewdaggerP$ & $\anymat\StransposeP$ &$\anymat$ &$\anymat\SskewdaggerP$ &$\anymat\SskewstarP$& $\anymat\SskewstarP$ &$\anymat\SskewdaggerP$ &$\anymat$ &$\anymat\StransposeP$& $\anymat\SskewdaggerP$ &$\anymat\SskewstarP$ &$\anymat\StransposeP$ &$\anymat$
  \end{tabular}    
  \end{center}
That proves the claim.
\end{proof}
}
Insofar as case distinctions are necessary, Lem\-ma~\ref{lemma:any_two_generators} provides the four pairwise exclusive possibilities which need to be considered.
{
    \newcommand{\anymatX}[2]{\anymat_{#1,#2}}
    \newcommand{\othermatX}[2]{\othermat_{#1,#2}}
    \newcommand{\indi}{i}
    \newcommand{\indj}{j}
    \newcommand{\indk}{k}
    \newcommand{\indl}{\ell}
    \newcommand{\inds}{s}
    \newcommand{\indt}{t}    
    \begin{lemma}
      \label{lemma:generators_reformulation}
      For any $\{\anymat,\othermat\}\subseteq\thematrices$ and any $\{(\indj,\indi),(\indt,\inds)\}\subseteq \SYnumbers{\thedim}^{\Ssetmonoidalproduct 2}$,
      \begin{IEEEeqnarray*}{rCl}
        \othermatAX{\indt}{\inds}=\anymatAX{\indj}{\indi}
      \end{IEEEeqnarray*}
      if and only if one of the following exclusive conditions is met:
      \begin{itemize}[wide]
      \item $\othermatA=\anymatA$ and $(\indt,\inds)=(\indj,\indi)$.
        \item $\othermatA=\anymatAT$ and $(\indt,\inds)=(\indi,\indj)$.
      \end{itemize}      
      In particular, $\othermatAX{\indt}{\inds}\neq\anymatAX{\indj}{\indi}$ whenever $\othermat\in\{\anymatB,\anymatBT\}$.
    \end{lemma}
    \begin{proof}
      Follows from the assumption that $\thegens=\{\thewuniX{\indl}{\indk},\thebuniX{\indl}{\indk}\Ssetbuilder (\indl,\indk)\in\SYnumbers{\thedim}^{\Ssetmonoidalproduct 2}\}$ has $2\thedim^2$ elements and Lem\-ma~\ref{lemma:any_two_generators}.
    \end{proof}
}
{
        \newcommand{\indi}{i}
    \newcommand{\indj}{j}
    \newcommand{\inds}{s}
    \newcommand{\indt}{t}    
    \begin{lemma}
      \label{lemma:monomial_order_reformulation}
      For any $\{\anymat,\othermat\}\subseteq \thematrices$ and $\{(\indj,\indi),(\indt,\inds)\}\subseteq \SYnumbers{\thedim}^{\Ssetmonoidalproduct 2}$,
      \begin{IEEEeqnarray*}{rCl}
        \othermatAX{\indt}{\inds}&\theorderRstrict& \anymatAX{\indj}{\indi}
      \end{IEEEeqnarray*}
      if and only if one of the following mutually exclusive conditions is met:
      \begin{itemize}[wide]
      \item         $\othermatA\in\{\anymatB,\anymatBT\}$ and $\anymat\in\{\theuniA,\theuniAT\}$.
      \item $\othermatA=\anymatA$ and
            \begin{IEEEeqnarray*}{rrClCrCl}
              &\anymat&\in&\{\theuniA,\theuniB\}&\hspace{.25em}{}\Sand{}\hspace{.25em}& (\indt,\inds)&{}\lexorderRstrict{}&(\indj,\indi)\\
             \Sor
              \hspace{2em}\hphantom{.}& \anymat&\in&\{\theuniAT,\theuniBT\}&{}\Sand{}& (\inds,\indt)&{}\lexorderRstrict{}&(\indi,\indj).\hspace{2em}\hphantom{\text{or}}
      \end{IEEEeqnarray*}
    \item $\othermatA=\anymatAT$ and
            \begin{IEEEeqnarray*}{rrClCrCl}
              &\anymat&\in&\{\theuniA,\theuniB\}&\hspace{.25em}{}\Sand{}\hspace{.25em}& (\inds,\indt)&{}\lexorderRstrict{}&(\indj,\indi)\\
              \Sor
              \hspace{2em}\hphantom{.}& \anymat&\in&\{\theuniAT,\theuniBT\}&{}\Sand{}& (\indt,\inds)&{}\lexorderRstrict{}&(\indi,\indj).\hspace{2em}\hphantom{\text{or}}
      \end{IEEEeqnarray*}
    \end{itemize}
    In particular, $\anymatAX{\indj}{\inds}\theorderRstrict \anymatAX{\indj}{\indi}$ if and only if $\inds\theorderRstrict\indi$ and, likewise, $\anymatAX{\indt}{\indi}\theorderRstrict \anymatAX{\indj}{\indi}$ if and only if  $\indt\theorderRstrict\indj$.
\end{lemma}
\begin{proof}
  The case $\othermatA\in\{\anymatB,\anymatBT\}$ is clear, as then $(\othermatA,\anymatA)\in\{\theuniB,\theuniBT\}\times\{\theuniA,\theuniAT\}$. Moreover, the  case $\othermatA=\anymatAT$ follows from the case $\othermat=\anymat$ by exchanging the roles $\indt\leftrightarrow\inds$. But also in order to prove the assertion in the remaining case $\othermat=\anymat$ it  can be assumed that $\anymat\in\{\theuniA,\theuniB\}$.
  That is because, if $\anymat\in\{\theuniAT,\theuniBT\}$, it is possible to apply the case $\anymat\in\{\theuniA,\theuniB\}$ after exchanging both $\indt\leftrightarrow\inds$ and $\indj\leftrightarrow\indi$.
  \par 
 If $\anymat=\theuni$, then $\anymatAX{\indt}{\inds}\lexorderRstrict\anymatAX{\indj}{\indi}$ if and only if  $\thewuniX{\indt}{\inds}\theorderRstrict\thewuniX{\indj}{\indi}$, which by definition holds precisely if and only if $(\indt,\inds)\lexorderRstrict(\indj,\indi)$.
  \par 
If $\anymat=\theuni\SskewstarP$, then $\anymatAX{\indt}{\inds}\theorderRstrict\anymatAX{\indj}{\indi}$ if and only if $\thebuniX{\Sexchanged(\indt)}{\Sexchanged(\inds)}\theorderRstrict\thebuniX{\Sexchanged(\indj)}{\Sexchanged(\indi)}$, which by definition is equivalent to $(\Sexchanged(\indj),\Sexchanged(\indi))\lexorderRstrict(\Sexchanged(\indt),\Sexchanged(\inds))$. The simultaneous application of  $\Sexchanged$ in both components is an anti-automorphism of the lexicographic order of $\SYnumbers{\thedim}^{\Ssetmonoidalproduct 2}$. Indeed, $(\Sexchanged(\indj),\Sexchanged(\indi))\lexorderRstrict(\Sexchanged(\indt),\Sexchanged(\inds))$ means that either $\Sexchanged(\indj)<\Sexchanged(\indt)$ or both $\Sexchanged(\indj)=\Sexchanged(\indt)$ and $\Sexchanged(\indi)<\Sexchanged(\inds)$. That is true if and only if either $\indt<\indj$ or both $\indt=\indj$ and $\inds<\indi$, i.e., if and only if $(\indt,\inds)\lexorderRstrict(\indj,\indi)$. Hence, the assertion is true.
\end{proof}
One particular consequence of  Lem\-ma~\ref{lemma:monomial_order_reformulation} deserves emphasis.
\begin{lemma}
  \label{lemma:skewdagger_comparisons}
    For any $\anymatA\in \thematrices$, any $\{(\indt,\inds),(\indj,\indi)\}\subseteq \SYnumbers{\thedim}^{\Ssetmonoidalproduct 2}$ and any $\othermatA\in\{\anymatA,\anymatAT\}$,
    \begin{IEEEeqnarray*}{rClCrCl}
      \othermatBTX{\inds}{\indt}&\theorderRstrict&\anymatBTX{\indi}{\indj} &\hspace{1em}\Siff\hspace{1em}& \othermatAX{\indt}{\inds}&\theorderRstrict&\anymatAX{\indj}{\indi}.
    \end{IEEEeqnarray*}
  \end{lemma}
  \begin{proof}
The claim follows by comparing the conclusion of Lem\-ma~\ref{lemma:monomial_order_reformulation} to the conclusion it yields when applied $(\othermatBT,\anymatBT)$ in the role of $(\othermatA,\anymatA)$.  Because $\othermatA\in\{\anymatA,\anymatAT\}$, and thus $\othermatBT\in\{\anymatB,\anymatBT\}=\{\anymatBT,\anymat^{\Sskewdagger,\Stranspose}\}$ only the last two cases of Lem\-ma~\ref{lemma:monomial_order_reformulation} need to be considered. Obviously, $\othermat=\anymat$ is equivalent to $\othermatBT=\anymatBT$ and, likewise, $\othermatA=\anymatAT$ to $\othermatBT=\anymat^{\Sskewdagger,\Stranspose}=\anymatB$. That is why the second case for $(\othermatA,\anymatA)$ must be compared to the second case for $(\othermatBT,\anymatBT)$ and, the third case for $(\othermatA,\anymatA)$ to the third case for $(\othermatBT,\anymatBT)$. Within each of these two cases, however, the exchange $\anymat\leftrightarrow\anymatBT$ has swapped the condition $\anymat\in \{\theuniA,\theuniAT\}$ for $\anymat\in \{\theuniB,\theuniBT\}$ and vice versa. But that change has the same effect as exchanging the roles $\indt\leftrightarrow\inds$ and $\indj\leftrightarrow\indi$ would. Thus, the comparison between the two conclusions of Lem\-ma~\ref{lemma:monomial_order_reformulation} proves the claim.
  \end{proof}
}

\subsection{Relations as non-minimal Gröbner basis}
It is actually more convenient to  first prove that $\therels$ itself happens to be a (non-minimal) Gröbner basis. Recall that $\Stip$ indicates tips with respect to $\theorder$.
  {
    \newcommand{\inds}{s}
    \newcommand{\indt}{t}
    \newcommand{\anymatX}[2]{\anymat_{#1,#2}}
    \newcommand{\othermatX}[2]{\othermat_{#1,#2}}
    \newcommand{\indi}{i}
    \newcommand{\indj}{j}    
    \begin{lemma}
      \label{lemma:the_leading_monomials}
  For any $\anymat\in\thematrices$ and any $(\indj,\indi)\in \SYnumbers{\thedim}^{\Ssetmonoidalproduct 2}$,
  \begin{IEEEeqnarray*}{rCl}
    \Stip \thebv{\anymat}{\indj}{\indi}=\anymatAX{\indj}{\thedim}\anymatBTX{1}{\Sexchanged(\indi)} , \hspace{3em}    \Stip \specbv{\anymat}=\anymatAX{\thedim}{\thedim}\anymatBTX{1}{1}.
  \end{IEEEeqnarray*}
\end{lemma}
\begin{proof}
  The claim about $\thebv{\anymat}{\indj}{\indi}$ is equivalent to  $\anymatAX{\indj}{\Sexchanged(\inds)}\anymatBTX{\inds}{\Sexchanged(\indi)}\theorderRstrict\anymatAX{\indj}{\thedim}\anymatBTX{1}{\Sexchanged(\indi)}$  holding for any $\inds\in \SYnumbers{\thedim}$ with $1<\inds$. This in turn is certainly true if already $\anymatX{\indj}{\Sexchanged(\inds)}\theorderRstrict\anymatX{\indj}{\thedim}$ since $\theorder$ was obtained by (degreewise) le\-xi\-co\-gra\-phi\-c extension. By the addendum to Lem\-ma~\ref{lemma:monomial_order_reformulation} the latter holds, irrespective of the value of $\anymat$, since $1<\inds$ requires $\Sexchanged(\inds)<\thedim$. That proves the assertion about  $\thebv{\anymat}{\indj}{\indi}$.
  \par
  What is claimed about $\specbv{\anymat}$ is that $\anymatAX{\indt}{\Sexchanged(\inds)}\anymatBTX{\inds}{\Sexchanged(\indt)}\theorderRstrict\anymatAX{\thedim}{\thedim}\anymatBTX{1}{1}$ for any $ \{\inds,\indt\}\subseteq \SYnumbers{\thedim}$ with $1<\inds$ and $\indt<\thedim$. As before this will follow already from $\anymatX{\indt}{\Sexchanged(\inds)}\theorderRstrict\anymatX{\thedim}{\thedim}$. A second application of Lem\-ma~\ref{lemma:monomial_order_reformulation} shows that this last condition is met if and only if either  $\anymat\in\{\theuniA,\theuniB\}$ and $(\indt,\Sexchanged(\inds))<(\thedim,\thedim)$ or $\anymat\in\{\theuniAT,\theuniBT\}$ and $(\Sexchanged(\inds),\indt)\lexorderRstrict(\thedim,\thedim)$. Again, no matter what $\anymat$ is, the respective inequality is true since $\indt<\thedim$ and since  $\Sexchanged(\inds)<\thedim$ by $1<\inds$. That concludes the proof.
\end{proof}
}
\begin{lemma}
  \label{lemma:inclusion_ambiguities}    
  Any inclusion ambiguity of $\therels$ is of the form
  \begin{IEEEeqnarray*}{rCl}
   (\thebv{\anymatAT}{\thedim}{\thedim},\thebv{\anymat}{\thedim}{\thedim},\theone,\theone) 
  \end{IEEEeqnarray*}
 for some $\anymat\in\thematrices$.
\end{lemma}
\begin{proof}  
  \newcommand{\anymatX}[2]{\anymat_{#1,#2}}    
  \newcommand{\othermatX}[2]{\othermat_{#1,#2}}
  \newcommand{\indi}{i}
  \newcommand{\indj}{j}
  \newcommand{\indk}{k}
  \newcommand{\indl}{\ell}
  \newcommand{\smallrel}{g}
  \newcommand{\bigrel}{g'}
  \newcommand{\leftmonomial}{b_1}
  \newcommand{\rightmonomial}{b_2}
  Let $(\smallrel,\bigrel,\leftmonomial,\rightmonomial)$ be any inclusion ambiguity of $\therels$. There exist
 $\{\anymat,\othermat\}\subseteq \thematrices$ and $\{(\indj,\indi),(\indl,\indk)\}\subseteq \SYnumbers{\thedim}^{\Ssetmonoidalproduct 2}$ with $\smallrel=\thebv{\othermatA}{\indl}{\indk}$ and $\bigrel=\thebv{\anymatA}{\indj}{\indi}$. Because  $\Stip\smallrel=\Stip\thebv{\othermatA}{\indl}{\indk}$ and $\Stip\bigrel=\Stip\thebv{\anymatA}{\indj}{\indi}$  have the same degree by Lem\-ma~\ref{lemma:the_leading_monomials} the assumption $\leftmonomial(\Stip\smallrel)\rightmonomial=\Stip\bigrel$ requires $\leftmonomial=\rightmonomial=\theone$ and $\Stip\smallrel=\Stip\bigrel$. By Lem\-ma~\ref{lemma:the_leading_monomials} the latter means
  \begin{IEEEeqnarray*}{rCl}
   \othermatAX{\indl}{\thedim}\othermatBTX{1}{\Sexchanged(\indk)}&=&\anymatAX{\indj}{\thedim}\anymatBTX{1}{\Sexchanged(\indi)}.
 \end{IEEEeqnarray*}
 Since that requires in particular $\othermatAX{\indl}{\thedim}=\anymatAX{\indj}{\thedim}$ Lem\-ma~\ref{lemma:generators_reformulation} excludes $\othermat\in \{\anymat\SskewstarP,\anymat\SskewdaggerP\}$, which by Lem\-ma~\ref{lemma:any_two_generators}  then requires $\othermat\in\{\anymat,\anymat\StransposeP\}$.
 \par
Moreover,  $\othermatA\neq\anymatA$. Indeed, otherwise $\anymatAX{\indl}{\thedim}\anymatBTX{1}{\Sexchanged(\indk)}=\anymatAX{\indj}{\thedim}\anymatBTX{1}{\Sexchanged(\indi)}$ would follow, which by Lem\-ma~\ref{lemma:generators_reformulation} would demand $(\indl,\thedim)=(\indj,\thedim)$ and $(1,\Sexchanged(\indk))=(1,\Sexchanged(\indi))$. That would imply $(\indl,\indk)=(\indj,\indi)$ and thus ultimately $\thebv{\othermat}{\indl}{\indk}=\thebv{\anymat}{\indj}{\indi}$, in contradiction to the assumption $\smallrel\neq\bigrel$. 
 \par
Hence, instead, $\othermatA=\anymatAT$  and thus $\anymatAX{\thedim}{\indl}\anymatBTX{\Sexchanged(\indk)}{1}=\anymatAX{\indj}{\thedim}\anymatBTX{1}{\Sexchanged(\indi)}$. By Lem\-ma~\ref{lemma:generators_reformulation} it follows that $(\thedim,\indl)=(\indj,\thedim)$ and $(\Sexchanged(\indk),1)=(1,\Sexchanged(\indi))$, which is to say $\indi=\indj=\indk=\indl=\thedim$. Hence,  $\thebv{\othermat}{\indl}{\indk}=\thebv{\anymat\StransposeP}{\thedim}{\thedim}$ and $\thebv{\anymat}{\indj}{\indi}=\thebv{\anymat}{\thedim}{\thedim}$. And this is the only inclusion ambiguity, as claimed.
\end{proof}

\begin{lemma}
  \label{lemma:inclusion_ambiguities_resolve}
   All inclusion ambiguities of $\therels$ resolve.
 \end{lemma}
 \begin{proof}
  \newcommand{\anymatX}[2]{\anymat_{#1,#2}}
  \newcommand{\othermatX}[2]{\othermat_{#1,#2}}
  \newcommand{\inds}{s}
  \newcommand{\indt}{t}
  By Lem\-ma~\ref{lemma:inclusion_ambiguities} it suffices to show that for any $\anymatA\in\thematrices$ and $\othermatA\Seqpd \anymatAT$ the terms $\Stip\thebv{\anymatA}{\thedim}{\thedim}-\thebv{\anymatA}{\thedim}{\thedim}$ and $\Stip\thebv{\othermatA}{\thedim}{\thedim}-\thebv{\othermatA}{\thedim}{\thedim}$ reduce to the same expression.
  \par
Indeed,  by definition,
\begin{IEEEeqnarray*}{rCl}    \Stip\thebv{\anymatA}{\thedim}{\thedim}-\thebv{\anymatA}{\thedim}{\thedim}&=&-\textstyle\sum_{\inds=2}^{\thedim}\anymatAX{\thedim}{\Sexchanged(\inds)}\anymatBTX{\inds}{1}+\theone\\
  &=&-\textstyle\sum_{\inds=2}^{\thedim}\othermatAX{\Sexchanged(\inds)}{\thedim}\othermatBTX{1}{\inds}+\theone\\
  &=&-\textstyle\sum_{\inds=1}^{\thedim-1}\othermatAX{\inds}{\thedim}\othermatBTX{1}{\Sexchanged(\inds)}+\theone\\
  &=&-\textstyle\sum_{\inds=1}^{\thedim-1}\Stip\thebv{\othermat}{\inds}{\inds}+\theone
\end{IEEEeqnarray*}
where the third identity is obtained by reindexing with $\Sexchanged$.
  Once reduced by $\thebv{\othermat}{\inds}{\inds}$ for each $\inds\in\SYnumbers{\thedim}$ with $\inds<\thedim$ the above takes on the form
  \begin{IEEEeqnarray*}{rCl}
    \Stip\thebv{\anymatA}{\thedim}{\thedim}-\thebv{\anymatA}{\thedim}{\thedim}&\Sreducesto&
-\textstyle\sum_{\inds=1}^{\thedim-1}(-\textstyle\sum_{\indt=2}^{\thedim}\othermatAX{\inds}{\Sexchanged(\indt)}\othermatBTX{\indt}{\Sexchanged(\inds)}+\theone)+\theone\\
      &=&    \textstyle\sum_{\inds=1}^{\thedim-1}\sum_{\indt=2}^{\thedim}\anymatAX{\Sexchanged(\indt)}{\inds}\anymatBTX{\Sexchanged(\inds)}{\indt}- (\thedim-2)\theone.\IEEEeqnarraynumspace\IEEEyesnumber\label{eq:inclusion_ambiguities_resolve-1}
    \end{IEEEeqnarray*}
    \par 
    Exchanging the roles $\anymat\leftrightarrow\othermat$ in \eqref{eq:inclusion_ambiguities_resolve-1} yields
\begin{IEEEeqnarray*}{rCl}
  \Stip\thebv{\othermatA}{\thedim}{\thedim}-\thebv{\othermatA}{\thedim}{\thedim}&\Sreducesto&\textstyle\sum_{\inds=1}^{\thedim-1}\sum_{\indt=2}^{\thedim}\othermatAX{\Sexchanged(\indt)}{\inds}\othermatBTX{\Sexchanged(\inds)}{\indt}- (\thedim-2)\theone\\
  &=&\textstyle\sum_{\inds=1}^{\thedim-1}\sum_{\indt=2}^{\thedim}\anymatAX{\inds}{\Sexchanged(\indt)}\anymatBTX{\indt}{\Sexchanged(\inds)}- (\thedim-2)\theone.
    \end{IEEEeqnarray*}
    Switching the names of the summation indices $\inds\leftrightarrow\indt$ here gives
\begin{IEEEeqnarray*}{rCl}          \Stip\thebv{\othermatA}{\thedim}{\thedim}-\thebv{\othermatA}{\thedim}{\thedim}&\Sreducesto&\textstyle\sum_{\inds=2}^{\thedim}\sum_{\indt=1}^{\thedim-1}\anymatAX{\indt}{\Sexchanged(\inds)}\anymatBTX{\inds}{\Sexchanged(\indt)}- (\thedim-2)\theone,
    \end{IEEEeqnarray*}
    which after transforming the domains of the summations with $\Sexchanged$ each amounts to
    \begin{IEEEeqnarray*}{rCl}          \Stip\thebv{\othermatA}{\thedim}{\thedim}-\thebv{\othermatA}{\thedim}{\thedim}&\Sreducesto&\textstyle\sum_{\inds=1}^{\thedim-1}\sum_{\indt=2}^{\thedim}\anymatAX{\Sexchanged(\indt)}{\inds}\anymatBTX{\Sexchanged(\inds)}{\indt}- (\thedim-2)\theone.\IEEEeqnarraynumspace\IEEEyesnumber\label{eq:inclusion_ambiguities_resolve-2}
    \end{IEEEeqnarray*}
Since the right-hand sides of  \eqref{eq:inclusion_ambiguities_resolve-1} and  \eqref{eq:inclusion_ambiguities_resolve-2} agree, all inclusion ambiguities resolve.
 \end{proof}

{
  \newcommand{\anymatX}[2]{\anymat_{#1,#2}}
  \newcommand{\othermatX}[2]{\othermat_{#1,#2}}
  \newcommand{\indi}{i}
  \newcommand{\indj}{j}
  \newcommand{\indk}{k}
  \newcommand{\indl}{\ell}
  \begin{lemma}
  \label{lemma:overlap_ambiguities}
  Any overlap ambiguity of $\therels$ is of the form
  \begin{IEEEeqnarray*}{C}
   (\thebv{\anymatBT}{\indl}{1},\thebv{\anymatA}{1}{\indi},\anymatBTX{\indl}{\thedim},\anymatBTX{1}{\Sexchanged(\indi)}) 
  \end{IEEEeqnarray*}
  for some $\anymat\in\thematrices$ and $(\indl,\indi)\in \SYnumbers{\thedim}^{\Ssetmonoidalproduct 2}$.
\end{lemma}
\begin{proof}
  \newcommand{\leftbv}{b_1}
  \newcommand{\rightbv}{b_2}
  \newcommand{\leftrel}{g_1}
  \newcommand{\rightrel}{g_2}
  Let $(\leftrel,\rightrel,\leftbv,\rightbv)$ be any overlap ambiguity of $\therels$. Then there exist $\{\anymat,\othermat\}\subseteq \thematrices$ and $\{(\indj,\indi),(\indl,\indk)\}\subseteq \SYnumbers{\thedim}^{\Ssetmonoidalproduct 2}$ with $\leftrel=\thebv{\othermat}{\indl}{\indk}$ and $\rightrel=\thebv{\anymat}{\indj}{\indi}$.
  \par
  Moreover, necessarily, $\{\leftbv,\rightbv\}\subseteq \thegens$ for the following reasons. If, say, $\rightbv$ had degree $2$ or greater, then because $\Stip\rightrel=\Stip\thebv{\anymat}{\indj}{\indi}$ has degree $2$ by Lem\-ma~\ref{lemma:the_leading_monomials} the identity    $(\Stip\leftrel)\rightbv=\leftbv(\Stip\rightrel)$ would imply that $\Stip\rightrel$ divides $\rightbv$, in contradiction to the assumption. Likewise, $\leftbv$ has degree less than $2$ because also  $\Stip\leftrel=\Stip\thebv{\othermat}{\indl}{\indk}$ has degree $2$ by the same lemma. Furthermore, by $(\Stip\leftrel)\rightbv=\leftbv(\Stip\rightrel)$ the monomials $\leftbv$ and $\rightbv$ must have the same degree because $\Stip\leftrel$ and $\Stip\rightrel$ do. As it was assumed that $(\leftbv,\rightbv)\neq (\theone,\theone)$ that only leaves the possibility that $\{\leftbv,\rightbv\}\subseteq \thegens$.
  \par
It follows that the identity $(\Stip\leftrel)\rightbv=\leftbv(\Stip\rightrel)$, which by  Lem\-ma~\ref{lemma:the_leading_monomials} is equivalent to 
  \begin{IEEEeqnarray*}{rCl}
\othermatAX{\indl}{\thedim}\othermatBTX{1}{\Sexchanged(\indk)}\rightbv=\leftbv \anymatAX{\indj}{\thedim}\anymatBTX{1}{\Sexchanged(\indi)},
\end{IEEEeqnarray*}
requires $\leftbv=\othermatAX{\indl}{\thedim}$ and $\rightbv=\anymatBTX{1}{\Sexchanged(\indi)}$ and, most of all, $\anymatAX{\indj}{\thedim}=\othermatBTX{1}{\Sexchanged(\indk)}$. By Lemmata~\ref{lemma:any_two_generators} and \ref{lemma:generators_reformulation} the latter demands $\othermatBT\in \{\anymatA,\anymatAT\}$, which is to say  $\othermatA\in \{\anymatB,\anymatBT\}$. If $\othermat=\anymat\SskewstarP$ was true, it would follow that $\anymatAX{\indj}{\thedim}=\anymatAX{\Sexchanged(\indk)}{1}$ and thus $(\indj,\thedim)=(\Sexchanged(\indk),1)$ by Lem\-ma~\ref{lemma:generators_reformulation}. Because that would violate the assumption $2\leq \thedim$ it must instead be the case that $\othermatA=\anymatBT$ and thus  $\leftbv=\anymatBTX{\indl}{\thedim}$ and  $\anymatAX{\indj}{\thedim}=\anymatAX{1}{\Sexchanged(\indk)}$. According to Lem\-ma~\ref{lemma:generators_reformulation} then $(\indj,\thedim)=(1,\Sexchanged(\indk))$, i.e., $\indj=\indk=1$. Thus, $\leftrel=\thebv{\anymatBT}{\indl}{1}$ and $\rightrel=\thebv{\anymatA}{1}{\indi}$. That proves the claim.
\end{proof}
  }

\begin{lemma}
  \label{lemma:overlap_ambiguities_resolve}
All overlap ambiguities of $\therels$ resolve.
\end{lemma}
\begin{proof}
  \newcommand{\anymatX}[2]{\anymat_{#1,#2}}
  \newcommand{\othermatX}[2]{\othermat_{#1,#2}}
  \newcommand{\indi}{i}
  \newcommand{\indl}{\ell}
  \newcommand{\inds}{s}
  \newcommand{\indt}{t}
By  Lem\-ma~\ref{lemma:overlap_ambiguities} it is enough to prove that for any $\anymatA\in\thematrices$ and $\othermatA\Seqpd \anymatBT$ and any $(\indl,\indi)\in \SYnumbers{\thedim}^{\Ssetmonoidalproduct 2}$ the polynomials $(\Stip\thebv{\othermatA}{\indl}{1}-\thebv{\othermatA}{\indl}{1})\othermatAX{1}{\Sexchanged(\indi)}$ and $\othermatAX{\indl}{\thedim}(\Stip\thebv{\anymatA}{1}{\indi}-\thebv{\anymatA}{1}{\indi})$ have a common reduction.
  \par
  On the one hand, the definitions imply
  \begin{IEEEeqnarray*}{rCl}
    (\Stip\thebv{\othermat}{\indl}{1}-\thebv{\othermat}{\indl}{1})\othermatAX{1}{\Sexchanged(\indi)}&=& ({-}\textstyle\sum_{\inds=2}^{\thedim}\othermatAX{\indl}{\Sexchanged(s)}\othermatBTX{\inds}{\thedim}+\Skronecker{\indl}{1}\theone)\othermatAX{1}{\Sexchanged(\indi)}\\
    &=&\textstyle-\sum_{\inds=2}^{\thedim}\othermatX{\indl}{\Sexchanged(s)}\anymatAX{\inds}{\thedim}\anymatBTX{1}{\Sexchanged(\indi)}+\Skronecker{\indl}{1}\othermatAX{1}{\Sexchanged(\indi)}\\
&=&\textstyle-\sum_{\inds=2}^{\thedim}\othermatX{\indl}{\Sexchanged(s)}(\Stip\thebv{\anymatA}{\inds}{\indi})+\Skronecker{\indl}{1}\othermatAX{1}{\Sexchanged(\indi)},
  \end{IEEEeqnarray*}
  which upon reduction by $\thebv{\anymat}{\inds}{\indi}$ for each $\inds\in\SYnumbers{\thedim}$ with $1<\inds$ becomes
  \begin{IEEEeqnarray*}{rCl}
  \IEEEeqnarraymulticol{3}{l}{
    (\Stip\thebv{\othermat}{\indl}{1}-\thebv{\othermat}{\indl}{1})\othermatAX{1}{\Sexchanged(\indi)}
  }\\
  \hspace{3em}
&\Sreducesto&
\textstyle {-}\sum_{\inds=2}^{\thedim}\othermatX{\indl}{\Sexchanged(\inds)}({-}\sum_{\indt=2}^{\thedim}\anymatX{\inds}{\Sexchanged(\indt)}\anymatBTX{\indt}{\Sexchanged(\indi)}+\Skronecker{\inds}{\indi}\theone)+\Skronecker{\indl}{1}\othermatAX{1}{\Sexchanged(\indi)}\\    
&=&\textstyle\sum_{\inds,\indt=2}^{\thedim}\othermatAX{\indl}{\Sexchanged(\inds)}\anymatAX{\inds}{\Sexchanged(\indt)}\othermatAX{\indt}{\Sexchanged(\indi)}-\othermatAX{\indl}{\Sexchanged(\indi)}+\Skronecker{\indi}{1}\othermatAX{\indl}{\thedim}+\Skronecker{\indl}{1}\othermatAX{1}{\Sexchanged(\indi)},\IEEEeqnarraynumspace\IEEEyesnumber    \label{eq:overlap_ambiguities_resolve_1}
  \end{IEEEeqnarray*}
  where the last identity is due to
  \begin{IEEEeqnarray*}{rCl}
    \textstyle   \sum_{\inds=2}^{\thedim}\Skronecker{\inds}{\indi}\othermatAX{\indl}{\Sexchanged(\inds)}&=&\textstyle\sum_{\inds=1}^{\thedim}\Skronecker{\inds}{\indi}\othermatAX{\indl}{\Sexchanged(\inds)}-\Skronecker{1}{\indi}\othermatAX{\indl}{\Sexchanged(1)}=\othermatAX{\indl}{\Sexchanged(\indi)}-\Skronecker{\indi}{1} \othermatX{\indl}{\thedim}.
  \end{IEEEeqnarray*}
\par
On the other hand, by definition
\begin{IEEEeqnarray*}{rCl}
  \othermatAX{\indl}{\thedim}(\Stip\thebv{\anymatA}{1}{\indi}-\thebv{\anymatA}{1}{\indi}) &=& \textstyle\othermatAX{\indl}{\thedim}(-\sum_{\indt=2}^{\thedim}\anymatAX{1}{\Sexchanged(\indt)}\anymatBTX{\indt}{\Sexchanged(\indi)}+\Skronecker{1}{\indi}\theone)\\  &=&\textstyle-\sum_{\indt=2}^{\thedim}\othermatX{\indl}{\thedim}\othermatBTX{1}{\Sexchanged(\indt)}\othermatAX{\indt}{\Sexchanged(\indi)}+\Skronecker{\indi}{1}\othermatAX{\indl}{\thedim}\\
  &=&\textstyle-\sum_{\indt=2}^{\thedim}(\Stip\thebv{\othermatA}{\indl}{\indt})\othermatAX{\indt}{\Sexchanged(\indi)}+\Skronecker{\indi}{1}\othermatAX{\indl}{\thedim},
\end{IEEEeqnarray*}
which when reduced by $\thebv{\othermatA}{\indl}{\indt}$ for each $\indt\in\SYnumbers{\thedim}$ with $1<\indt$ takes on the form
\begin{IEEEeqnarray*}{rCl}
  \IEEEeqnarraymulticol{3}{l}{
    \othermatAX{\indl}{\thedim}(\Stip\thebv{\anymatA}{1}{\indi}-\thebv{\anymatA}{1}{\indi})
  }\\
  \hspace{3em}&\Sreducesto& \textstyle-\sum_{\indt=2}^{\thedim}(-\sum_{\inds=2}^{\thedim}\othermatAX{\indl}{\Sexchanged(\inds)}\othermatBTX{\inds}{\Sexchanged(\indt)}+\Skronecker{\indl}{\indt}\theone)\othermatAX{\indt}{\Sexchanged(\indi)}+\Skronecker{\indi}{1}\othermatAX{\indl}{\thedim}\\
 &=&\textstyle \sum_{\inds,\indt=2}^{\thedim}\othermatAX{\indl}{\Sexchanged(\inds)}\anymatAX{\inds}{\Sexchanged(\indt)}\othermatAX{\indt}{\Sexchanged(\indi)}-\othermatAX{\indl}{\Sexchanged(\indi)}+\Skronecker{\indl}{1}\othermatAX{1}{\Sexchanged(\indi)}+\Skronecker{\indi}{1}\othermatAX{\indl}{\thedim},\IEEEeqnarraynumspace\IEEEyesnumber    \label{eq:overlap_ambiguities_resolve_2}
\end{IEEEeqnarray*}
where now the last step comes from 
\begin{IEEEeqnarray*}{rCl}
  \textstyle \sum_{\indt=2}^{\thedim}\Skronecker{\indt}{\indl}\othermatAX{\indt}{\Sexchanged(\indi)}&=&\textstyle \sum_{\indt=1}^{\thedim}\Skronecker{\indt}{\indl}\othermatAX{\indt}{\Sexchanged(\indi)}-\Skronecker{1}{\indl}\othermatX{1}{\Sexchanged(\indi)}=\othermatAX{\indl}{\Sexchanged(\indi)}-\Skronecker{\indl}{1}\othermatAX{1}{\Sexchanged(\indi)}.
\end{IEEEeqnarray*}
Since   \eqref{eq:overlap_ambiguities_resolve_1} and  \eqref{eq:overlap_ambiguities_resolve_2} have the same right-hand sides, all overlap ambiguities resolve.
\end{proof}

\begin{proposition}
  \label{proposition:relations_are_groebner_basis}
  $\therels$ is a Gröbner basis of $\theideal$.
\end{proposition}
\begin{proof}
Follows from Lem\-ma\-ta~\ref{lemma:inclusion_ambiguities_resolve} and   \ref{lemma:overlap_ambiguities_resolve} by Pro\-po\-si\-tion~\ref{proposition:bergman}, Berg\-man's dia\-mond lem\-ma.  
\end{proof}

\subsection{Reduced Gröbner basis}
Wth Proposition~\ref{proposition:relations_are_groebner_basis} at hand it is now not hard to prove that $\thegroebner$ is the reduced Gröbner basis.
{
  \newcommand{\inds}{s}
  \newcommand{\indt}{t}
  \begin{lemma}
    \label{lemma:relationship_between_basis_vectors}
    For any $\anymatA\in \thematrices$,
    \begin{IEEEeqnarray*}{rCl}
\textstyle      \specbv{\anymatA}=\thebv{\anymatA}{\thedim}{\thedim}-\sum_{\inds=1}^{\thedim-1}\thebv{\anymatAT}{\inds}{\inds}.
    \end{IEEEeqnarray*}
  \end{lemma}
  \begin{proof}
    If $\othermatA\Seqpd \anymatAT$, then for any $\inds\in\SYnumbers{\thedim}$ by definition
    \begin{IEEEeqnarray*}{rCl}
      \thebv{\othermatA}{\inds}{\inds}&=&\textstyle \othermatAX{\inds}{\thedim}\othermatBTX{1}{\Sexchanged(\inds)}-(-\sum_{\indt=2}^{\thedim}\othermatAX{\inds}{\Sexchanged(\indt)}\othermatBTX{\indt}{\Sexchanged(\inds)}+\theone)\\
      &=&\textstyle\sum_{\indt=1}^{\thedim}\anymatAX{\Sexchanged(\indt)}{\inds}\anymatBTX{\Sexchanged(\inds)}{\indt}-\theone\\
      &=&\textstyle\sum_{\indt=1}^{\thedim}\anymatAX{\indt}{\inds}\anymatBTX{\Sexchanged(\inds)}{\Sexchanged(\indt)}-\theone.
\end{IEEEeqnarray*}
Summing up these  identities thus   yields
    \begin{IEEEeqnarray*}{rCl}
      \textstyle\sum_{\inds=1}^{\thedim-1}\thebv{\othermatA}{\inds}{\inds}&=&\textstyle\textstyle\sum_{\inds=1}^{\thedim-1}\sum_{\indt=1}^{\thedim}\anymatAX{\indt}{\inds}\anymatBTX{\Sexchanged(\inds)}{\Sexchanged(\indt)}-(\thedim-1)\theone\\
      &=&\textstyle\textstyle\sum_{\inds=2}^{\thedim}\sum_{\indt=1}^{\thedim}\anymatAX{\indt}{\Sexchanged(\inds)}\anymatBTX{\inds}{\Sexchanged(\indt)}-(\thedim-1)\theone.
\end{IEEEeqnarray*}
Hence,  subtracting this equality from the definition
    \begin{IEEEeqnarray*}{rCl}
          \thebv{\anymatA}{\thedim}{\thedim}&=& \anymatAX{\thedim}{\thedim}\anymatBTX{1}{1}+\textstyle\sum_{\inds=2}^{\thedim}\anymatAX{\thedim}{\Sexchanged(\inds)}\anymatBTX{\inds}{\Sexchanged(\thedim)}-\theone
        \end{IEEEeqnarray*}
 of $          \thebv{\anymatA}{\thedim}{\thedim}$  yields
    \begin{IEEEeqnarray*}{rCl}
          \thebv{\anymatA}{\thedim}{\thedim}-\textstyle\sum_{\inds=1}^{\thedim-1}\thebv{\othermat}{\inds}{\inds}&=& \anymatAX{\thedim}{\thedim}\anymatBTX{1}{1}-\textstyle\sum_{\inds=2}^\thedim\sum_{\indt=1}^{\thedim-1}\anymatAX{\indt}{\Sexchanged(\inds)}\anymatBTX{\inds}{\Sexchanged(\indt)}+(\thedim-2)\theone=\specbv{\anymatA},
        \end{IEEEeqnarray*}
        which is what was claimed.
  \end{proof}
}

  \begin{proposition}
     $\thegroebner$ is the reduced Gröbner basis of $\theideal$.
   \end{proposition}
   \begin{proof}
     \newcommand{\anymatX}[2]{\anymat_{#1,#2}}
     \newcommand{\othermatX}[2]{\othermat_{#1,#2}}
     \newcommand{\firstbv}{g}
     \newcommand{\secondbv}{g'}
     \newcommand{\indt}{t}
     \newcommand{\inds}{s}
     \newcommand{\indi}{i}
     \newcommand{\indj}{j}
     \newcommand{\indk}{k}
     \newcommand{\indl}{\ell}     
     Recall that $\thegroebner$ results from $\therels$ by replacing $\{\thebv{\anymat}{\thedim}{\thedim}\Ssetbuilder\anymat\in\thematrices\}$ by $\{\specbv{\anymat}\Ssetbuilder \anymat\in\thematrices\}$.     Lem\-ma~\ref{lemma:relationship_between_basis_vectors} therefore shows that $\thegroebner$ generates the same ideal as $\therels$. Moreover, $\StipO\thegroebner=\StipO\therels$ by Lem\-ma~\ref{lemma:the_leading_monomials}. Since $\therels$ is a Gröbner basis of $\theideal$ by Proposition~\ref{proposition:relations_are_groebner_basis}, hence so is $\thegroebner$.
     \par
     It remains to prove that $\thegroebner$ is reduced, i.e., that the tip of any element of $\thegroebner$ does not divide any monomial occurring in any other element of $\thegroebner$. In fact, for degree reasons, any such divisibility relation is already necessarily an identity. 
     Let  $\{\firstbv,\secondbv\}\subseteq \thegroebner$ be arbitrary and suppose that $\Stip\firstbv$ occurs in $\secondbv$. By definition of $\thegroebner$ there then exist $\{\anymat,\othermat\}\subseteq \thematrices$ and $\{(\indj,\indi),(\indl,\indk)\}\subseteq \SYnumbers{\thedim}^{\Ssetmonoidalproduct 2}$ such that $(\indl,\indk)\neq (\thedim,\thedim)\neq (\indj,\indi)$ and $\firstbv\in\{\specbv{\othermat},\thebv{\othermat}{\indl}{\indk}\}$ and  $\secondbv\in\{\specbv{\anymat},\thebv{\anymat}{\indj}{\indi}\}$. 
     \par
     \emph{Step~1.} First,  $(\firstbv,\secondbv)\neq (\thebv{\othermat}{\indl}{\indk},\specbv{\anymat})$ is proved indirectly. If $(\firstbv,\secondbv)$ was $(\thebv{\othermat}{\indl}{\indk},\specbv{\anymat})$, there would exist  $(\indt,\inds)\in\SYnumbers{\thedim}^{\Ssetmonoidalproduct 2}$ such that either both $1<\inds$ and $\indt<\thedim$ or $(\indt,\inds)=(\thedim,1)$ and such that $\othermatAX{\indl}{\thedim}\othermatBTX{1}{\Sexchanged(\indk)}=\anymatAX{\indt}{\Sexchanged(\inds)}\anymatBTX{\inds}{\Sexchanged(\indt)}$. By Lemmata~\ref{lemma:any_two_generators} and \ref{lemma:generators_reformulation} it would follow $\othermat\in \{\anymat,\anymat\StransposeP\}$.
     \par
     \emph{Case~1.1.} If $\othermatA=\anymatA$ was the case, then Lem\-ma~\ref{lemma:generators_reformulation} would demand both $(\indl,\thedim)=(\indt,\Sexchanged(\inds))$ and $(1,\Sexchanged(\indk))=(\inds,\Sexchanged(\indt))$. In other words, $\indl=\indt=\indk$ and $\inds=1$. By assumption on $\inds$ that would force $\indt=\thedim$ and thus $(\indl,\indk)=(\thedim,\thedim)$ in contradiction to the assumptions on $\indl$ and $\indk$.
\par
     \emph{Case~1.2.} And in the opposite case that $\othermatA=\anymatAT$ the same lemma would require both $(\indl,\thedim)=(\Sexchanged(\inds),\indt)$ and $(1,\Sexchanged(\indk))=(\Sexchanged(\indt),\inds)$, which is to say $\indl=\Sexchanged(\inds)=\indk$ and $\indt=\thedim$. Again, $\indt=\thedim$ would entail $\inds=1$, which would yield the contradiction $(\indl,\indk)=(\thedim,\thedim)$.
     \par
     \emph{Step~2.} 
     Similarly, $(\firstbv,\secondbv)\neq(\specbv{\othermat},\thebv{\anymat}{\indj}{\indi})$. Indeed, if $(\firstbv,\secondbv)$ was $(\specbv{\othermat},\thebv{\anymat}{\indj}{\indi})$, there would exist $\inds\in\SYnumbers{\thedim}$ such that $\othermatAX{\thedim}{\thedim}\othermatBTX{1}{1}=\anymatAX{\indj}{\Sexchanged(\inds)}\anymatBTX{\inds}{\Sexchanged(\indi)}$. Lem\-ma\-ta~\ref{lemma:any_two_generators} and \ref{lemma:generators_reformulation} would imply $\othermat\in\{\anymat,\anymat\StransposeP\}$ and hence $(\indj,\Sexchanged(\inds))=(\thedim,\thedim)$ and $(\inds,\Sexchanged(\indi))=(1,1)$. In particular,  $(\indj,\indi)\neq (\thedim,\thedim)$ would be violated.
     \par
     \emph{Step~3.} From Steps~1 and~2 it follows $(\firstbv,\secondbv)\in\{(\specbv{\othermat},\specbv{\anymat}),(\thebv{\othermat}{\indl}{\indk},\thebv{\anymat}{\indj}{\indi})\}$. Hence, there are two cases to distinguish. It is enough to prove $\firstbv=\secondbv$ in each of them.     
     \par 
     \emph{Case~3.1:} If  $(\firstbv,\secondbv)=(\specbv{\othermat},\specbv{\anymat})$, then there need to exist $(\indt,\inds)\in\SYnumbers{\thedim}^{\Ssetmonoidalproduct 2}$ such that either both $1<\inds$ and $\indt<\thedim$ or $(\indt,\inds)=(\thedim,1)$ and such that $\othermatAX{\thedim}{\thedim}\othermatBTX{1}{1}=\anymatAX{\indt}{\Sexchanged(\inds)}\anymatBTX{\inds}{\Sexchanged(\indt)}$. Lemmata~\ref{lemma:any_two_generators} and \ref{lemma:generators_reformulation} then force in particular $\othermat\in \{\anymatA,\anymatAT\}$. Because $\specbv{\anymatAT}=\specbv{\anymat}$ that proves $\firstbv=\secondbv$, no matter whether $\othermatA=\anymatA$ or $\othermatA=\anymatAT$.
     \par
     \emph{Case~3.2:} Alternatively, if $(\firstbv,\secondbv)=(\thebv{\othermatA}{\indl}{\indk},\thebv{\anymatA}{\indj}{\indi})$, there is $\inds\in\SYnumbers{\thedim}$ such that $\othermatAX{\indl}{\thedim}\othermatBTX{1}{\Sexchanged(\indk)}=\anymatAX{\indj}{\Sexchanged(\inds)}\anymatBTX{\inds}{\Sexchanged(\indi)}$. Hence,  $\othermat\in \{\anymat,\anymat\StransposeP\}$ by Lemmata~\ref{lemma:any_two_generators} and \ref{lemma:generators_reformulation}. If $\othermat$ was $\anymat\StransposeP$, that would require both $(\thedim,\indl)=(\indj,\Sexchanged(\inds))$ and $(\Sexchanged(\indk),1)=(\inds,\Sexchanged(\indi))$ and thus in particular $(\indj,\indi)=(\thedim,\thedim)$, which is false. Hence, instead, $\othermat=\anymat$ and thus $(\indl,\thedim)=(\indj,\Sexchanged(\inds))$ and $(1,\Sexchanged(\indk))=(\inds,\Sexchanged(\indi))$. That implies $(\indl,\indk)=(\indj,\indi)$ and, in conclusion, $\firstbv=\secondbv$, which  means that $\thegroebner$ is reduced.
   \end{proof}

\section{\texorpdfstring{Anick resolution for $U_n^+$}{Anick resolution for the free unitary quantum group}}
\label{section:anick_computation}
 {
  \newcommand{\anymatX}[2]{\anymat_{#1,#2}}  
  \newcommand{\othermatX}[2]{\othermat_{#1,#2}}  
  \newcommand{\indj}{j}
  \newcommand{\indi}{i}
  \newcommand{\inds}{s}
  \newcommand{\indt}{t}
  \newcommand{\indk}{k}
  \newcommand{\indl}{\ell}
  \newcommand{\recvar}{m}
  \newcommand{\outerrecvar}{\ell}
  \newcommand{\orderind}{\ell}
  \newcommand{\halforderind}{m}
  \newcommand{\shela}{a}
  \newcommand{\shelb}{b}
  \newcommand{\shelc}{c}
  \newcommand{\sheld}{d}
  \newcommand{\shelaX}[2]{\shela^{#2}_{#1}}
  \newcommand{\shelbX}[2]{\shelb^{#2}_{#1}}
  \newcommand{\shelcX}[2]{\shelc^{#2}_{#1}}
  \newcommand{\sheldX}[2]{\sheld^{#2}_{#1}}
  \newcommand{\indfirsts}{s_1}
  \newcommand{\indseconds}{s_2}
{
  \newcommand{\anypoly}{p}
  In addition to the assumptions from Section~\ref{section:groebner_computation} let also $\thecounit$ and $\themodule$ be as in the  \hyperref[main-result]{Main result}. 
  In this section, the Anick resolution of $\themodule$ with respect to $\theorder$ as a right $\thealgebra$-module is computed explicitly by solving the determining recursion equations.

}

\subsection{Auxiliary results}
Recall that $\Stip$ indicates tips and let $\thecore$ be the core of $\Stip\theideal$, i.e., the minimal generating set of the monoidal ideal $\Stip\theideal$ in the monoid $\themonoid$, as explained in Section~\ref{section:anick-monoidal_ideals}.
\par
The next proposition  guarantees in particular that $\thegens\cap\StipO\theideal=\emptyset$ and thus that  Anick's construction is applicable.

\begin{proposition}
  \label{proposition:core}
  $\thecore=\{\anymatAX{\indj}{\thedim}\anymatBTX{1}{\Sexchanged(\indi)}\Ssetbuilder\anymat\in\thematrices\Sand (\indj,\indi)\in\SYnumbers{\thedim}^{\Ssetmonoidalproduct 2}\}$.
\end{proposition}
\begin{proof}
  By Proposition~\ref{proposition:relations_are_groebner_basis} the set $\thegroebner$ is a reduced and thus in particular minimal Gröbner basis of $\theideal$, whence $\thecore=\Stip\thegroebner$. Thus, Lem\-ma~\ref{lemma:the_leading_monomials} proves the claim.
\end{proof}

{
  \newcommand{\outerind}{j}
  \newcommand{\outerindX}[1]{\outerind_{#1}}
  \newcommand{\innerind}{i}
  \newcommand{\innerindX}[1]{\innerind_{#1}}
  \newcommand{\firstleftind}{\outerindX{1}}
  \newcommand{\firstrightind}{\innerindX{1}}
  \newcommand{\secondleftind}{\innerindX{2}}
  \newcommand{\secondrightind}{\outerindX{2}}
The following simple fact will be used so often that it seems worthy of its own lemma.
  \begin{lemma}
    \label{characterization_of_reduced_terms_of_order_two}
    For any $\anymat\in\thematrices$ and
    $\{\outerind,\innerind\}\subseteq \SYnumbers{\thedim}^{\Ssetmonoidalproduct 2}$,
    \begin{IEEEeqnarray*}{rCl}
      \anymatAX{\firstleftind}{\firstrightind}\anymatBTX{\secondleftind}{\secondrightind} \notin\Stip\theideal 
    \end{IEEEeqnarray*}
    if and only if
    \begin{IEEEeqnarray*}{rCl}
     (\firstrightind,\secondleftind)\neq(\thedim,1)\neq (\firstleftind,\secondrightind).
    \end{IEEEeqnarray*}
  \end{lemma}
  \begin{proof}
    \newcommand{\somemonomial}{b}
    \newcommand{\leftmonomial}{b_1}
    \newcommand{\smallmonomial}{a}
    \newcommand{\rightmonomial}{b_2}
    \newcommand{\anyleftind}{t}
    \newcommand{\anyrightind}{s}
    The contraposition, $\somemonomial\Seqpd\anymatAX{\firstleftind}{\firstrightind}\anymatBTX{\secondleftind}{\secondrightind} \in\StipO\theideal$ if and only if  $(\firstrightind,\secondleftind)=(\thedim,1)$ or $(\firstleftind,\secondrightind)=(\thedim,1)$, is what is proved below.
    \par
    \emph{Sufficient condition.} First, if $(\firstrightind,\secondleftind)=(\thedim,1)$, then $\somemonomial=\anymatAX{\firstleftind}{\thedim}\anymatBTX{1}{\Sexchanged(\Sexchanged(\secondrightind))}=\Stip\thebv{\anymatA}{\firstleftind}{\Sexchanged(\secondrightind)}$ by Lem\-ma~\ref{lemma:the_leading_monomials} and thus $\somemonomial\in\Stip\theideal$ because $\thebv{\anymatA}{\firstleftind}{\Sexchanged(\secondrightind)}\in \therels$ and $\therels\subseteq \theideal$. Likewise, if $(\firstleftind,\secondrightind) =(\thedim,1)$ and if $\othermatA\Seqpd \anymatAT$, then $\somemonomial=\anymatAX{\thedim}{\firstrightind}\anymatBTX{\secondleftind}{1}=\othermatAX{\firstrightind}{\thedim}\othermatBTX{1}{\Sexchanged(\Sexchanged(\secondleftind))}=\Stip\thebv{\othermatA}{\firstrightind}{\Sexchanged(\secondleftind)}$ and hence $\somemonomial\in\Stip\theideal$.
    \par
\emph{Necessary condition.}  Conversely, if $\somemonomial\in\Stip\theideal$, then there exist $\smallmonomial\in\thecore$ and $\{\leftmonomial,\rightmonomial\}\subseteq \themonoid$ with $\leftmonomial\smallmonomial\rightmonomial=\somemonomial$. Because $\smallmonomial$ has degree $2$ by  Proposition~\ref{proposition:core}, the same degree as $\somemonomial$, that demands $\leftmonomial=\rightmonomial=1$ and thus $\smallmonomial=\somemonomial$, which then proves that actually $\somemonomial\in\thecore$. Hence, by a second application of Proposition~\ref{proposition:core} there exist  $\othermat\in\thematrices$ and $(\anyleftind,\anyrightind)\in\SYnumbers{\thedim}^{\Ssetmonoidalproduct 2}$ with 
    \begin{IEEEeqnarray*}{rCl}
      \IEEEyesnumber\label{eq:characterization_of_reduced_terms_of_order_two_1}
\anymatAX{\firstleftind}{\firstrightind}\anymatBTX{\secondleftind}{\secondrightind}=\othermatAX{\anyleftind}{\thedim}\anymatBTX{1}{\Sexchanged(\anyrightind)}.
    \end{IEEEeqnarray*}
    By Lemmata~\ref{lemma:any_two_generators} and \ref{lemma:generators_reformulation} that requires $\othermatA\in\{\anymatA,\anymatAT\}$.
    \par
    \emph{Case~1.} If $\othermat=\anymat$, then according to Lem\-ma~\ref{lemma:generators_reformulation}  the identity \eqref{eq:characterization_of_reduced_terms_of_order_two_1} demands $(\firstleftind,\firstrightind)=(\anyleftind,\thedim)$ and $(\secondleftind,\secondrightind)=(1,\anyrightind)$, in particular $(\firstrightind,\secondleftind)=(\thedim,1)$.
    \par
    \emph{Case~2.} Alternatively, if $\othermatA=\anymatAT$, then \eqref{eq:characterization_of_reduced_terms_of_order_two_1} requires $(\firstrightind,\firstleftind)=(\anyleftind,\thedim)$ and $(\secondrightind,\secondleftind)=(1,\anyrightind)$ by Lem\-ma~\ref{lemma:generators_reformulation}, which implies $(\firstleftind,\secondrightind)=(\thedim,1)$ in particular. That concludes the proof.    
\end{proof}
}
For any $\orderind\in\Sintegersp$, any $\anymat\in\thematrices$ and any $(\indj,\indi)\in \SYnumbers{\thedim}^{\Ssetmonoidalproduct 2}$ let $\Sbvfull{\anymat}{\orderind}{\indj}{\indi}$ and $\anymatPA{\orderind}$ be as in the \hyperref[main-result]{Main result}. 
It is important to note that there are non-trivial identities between the $\Sbvfull{\anymat}{\orderind}{\indj}{\indi}$ for different values of $(\anymat,\indj,\indi)$ -- however, only for $\orderind\leq 2$.
\begin{lemma}
  \label{lemma:evil_identities}
  For any $\{\anymatA,\othermat\}\subseteq\thematrices$, any $\{(\indj,\indi),(\indt,\inds)\}\subseteq \SYnumbers{\thedim}^{\Ssetmonoidalproduct 2}$ and any $\orderind\in\Sintegersp$ the identity
  \begin{IEEEeqnarray*}{rCl}
    \Sbvfull{\othermatA}{\orderind}{\indt}{\inds}=\Sbvfull{\anymatA}{\orderind}{\indj}{\indi}
  \end{IEEEeqnarray*}
  holds if and only if both $\othermatA=\anymatA$ and $(\indt,\inds)=(\indj,\indi)$ or  one of the following mutually exclusive conditions is met:
  \begin{itemize}[wide]
    \item $\orderind=1$ and $\othermatA=\anymatAT$ and $(\indt,\inds)=(\Sexchanged(\indi),\Sexchanged(\indj))$
    \item $\orderind=2$ and $\othermatA=\anymatAT$ and $(\indt,\inds)=(\thedim,\thedim)=(\indj,\indi)$.
  \end{itemize}
\end{lemma}
\begin{proof} The proof requires a distinction of cases.
  \par
  \emph{Case~1.} If $\orderind=1$,  the definitions and Lem\-ma~\ref{lemma:generators_reformulation} imply  that
  \begin{IEEEeqnarray*}{rCl}
    \Sbvfull{\othermat}{1}{\indt}{\inds}=\othermatAX{\indt}{\Sexchanged(\inds)}=\anymatAX{\indj}{\Sexchanged(\indi)}=\Sbvfull{\anymatA}{1}{\indj}{\indi}
  \end{IEEEeqnarray*}
  is satisfied if and only if either both  $\othermatA=\anymatA$ and $(\indt,\Sexchanged(\inds))=(\indj,\Sexchanged(\indi))$ or both  $\othermatA=\anymatAT$ and $(\indt,\Sexchanged(\inds))=(\Sexchanged(\indi),\indj)$. That proves the claim in this case.
  \par
  \emph{Case~2.} Similarly, if $\orderind=2$, then the identity
  \begin{IEEEeqnarray*}{rCl}
    \Sbvfull{\othermatA}{2}{\indt}{\inds}=\othermatAX{\indt}{\thedim}\othermatBTX{1}{\Sexchanged(\inds)}=\anymatAX{\indj}{\thedim}\anymatBTX{1}{\Sexchanged(\indi)}=\Sbvfull{\anymatA}{2}{\indj}{\indi}
  \end{IEEEeqnarray*}
holds if and only if both $\othermatAX{\indt}{\thedim}=\anymatAX{\indj}{\thedim}$ and $\othermatBTX{1}{\Sexchanged(\inds)}=\anymatBTX{1}{\Sexchanged(\indi)}$. By Lem\-ma~\ref{lemma:generators_reformulation} that is true if and only if either $\othermatA=\anymatA$ and $(\indt,\thedim)=(\indj,\thedim)$ and $(1,\Sexchanged(\inds))=(1,\Sexchanged(\indi))$ or $\othermatA=\anymatAT$ and $(\thedim,\indt)=(\indj,\thedim)$ and $(\Sexchanged(\inds),1)=(1,\Sexchanged(\indi))$.
The first of these possibilities, of course, means that $(\othermat,\indt,\inds)=(\anymat,\indj,\indi)$, the other that $\othermatA=\anymatAT$ and $(\indt,\inds)=(\thedim,\thedim)=(\indj,\indi)$. Hence, the claim is true for $\orderind=2$ as well.
  \par
\emph{Case~3.} If, next, $3\leq \orderind$, then what needs to be proved is that $\Sbvfull{\othermatA}{\orderind}{\indt}{\inds}$ can only equal $\Sbvfull{\anymatA}{\orderind}{\indj}{\indi}$ if $(\othermatA,\indt,\inds)=(\anymatA,\indj,\indi)$. The vectors $\Sbvfull{\othermatA}{\orderind}{\indt}{\inds}$ and $\Sbvfull{\anymatA}{\orderind}{\indj}{\indi}$ being equal is equivalent to the conditions that either $\orderind$ is odd and
\begin{IEEEeqnarray*}{l}
\othermatAX{\indt}{\thedim}(\othermatBTX{1}{\thedim}\othermatAX{1}{\thedim})^{\frac{\orderind-3}{2}}\othermatBTX{1}{\thedim}\othermatAX{1}{\Sexchanged(\inds)}=\anymatAX{\indj}{\thedim}(\anymatBTX{1}{\thedim}\anymatAX{1}{\thedim})^{\frac{\orderind-3}{2}}\anymatBTX{1}{\thedim}\anymatAX{1}{\Sexchanged(\indi)}\IEEEyesnumber\label{eq:evil_identities_1}
\end{IEEEeqnarray*}
or  $\orderind$ is even and 
  \begin{IEEEeqnarray*}{rCl}
\othermatAX{\indt}{\thedim}(\othermatBTX{1}{\thedim}\othermatAX{1}{\thedim})^{\frac{\orderind-4}{2}}\othermatBTX{1}{\thedim}\othermatAX{1}{\thedim}\othermatBTX{1}{\Sexchanged(\inds)}=\anymatAX{\indj}{\thedim}(\anymatBTX{1}{\thedim}\anymatAX{1}{\thedim})^{\frac{\orderind-4}{2}}\anymatBTX{1}{\thedim}\anymatAX{1}{\thedim}\anymatBTX{1}{\Sexchanged(\indi)}\IEEEyesnumber\label{eq:evil_identities_2}.
\end{IEEEeqnarray*}
Any of \eqref{eq:evil_identities_1} and \eqref{eq:evil_identities_2} requires in particular  $\othermatBTX{1}{\thedim}=\anymatBTX{1}{\thedim}$. By Lem\-ma~\ref{lemma:generators_reformulation} that is only possible if either both $\othermatA=\anymatA$ and $(1,\thedim)=(1,\thedim)$ or both $\othermatA=\anymatAT$ and $(\thedim,1)=(1,\thedim)$. As the second of these conclusions would violate the assumption that $2\leq \thedim$ it must be the first one which is true.
\par
Of course, \eqref{eq:evil_identities_1} and \eqref{eq:evil_identities_2} moreover demand, on the one hand, $\othermatAX{\indt}{\thedim}=\anymatAX{\indj}{\thedim}$ and, on the other hand, $\othermatAX{1}{\Sexchanged(\inds)}=\anymatAX{1}{\Sexchanged(\indi)}$ respectively $\othermatBTX{1}{\Sexchanged(\inds)}=\anymatBTX{1}{\Sexchanged(\indi)}$. According to Lem\-ma~\ref{lemma:generators_reformulation}, given that $\othermatA=\anymatA$ and thus $\othermatBT=\anymatBT$, these identities are only satisfied if $(\indt,\thedim)=(\indj,\thedim)$ and $(1,\Sexchanged(\inds))=(1,\Sexchanged(\indi))$. But that is only true if  $(\indt,\inds)=(\indj,\indi)$. Thus, the claim has been verified for all  $\orderind$.
\end{proof}

It will moreover be necessary to  make comparisons between those vectors with respect to $\theorder$.

\begin{lemma}
  \label{lemma:base_monomial_order}
  For any $\orderind\in\Sintegersp$, any  $\{\anymat,\othermat\}\subseteq \thematrices$ and $\{(\indj,\indi),(\indt,\inds)\}\subseteq \SYnumbers{\thedim}^{\Ssetmonoidalproduct 2}$, the inequality
  \begin{IEEEeqnarray*}{rCl}
    \Sbvfull{\othermat}{\orderind}{\indt}{\inds}&    \theorderRstrict&\Sbvfull{\anymatA}{\orderind}{\indj}{\indi}
  \end{IEEEeqnarray*}
  holds if and only if either both $\othermatA\in\{\anymatB,\anymatBT\}$ and $\anymatA\in\{\theuniA,\theuniAT\}$ or one of the following mutually exclusive conditions is met.
  
  \begin{itemize}[wide]
  \item $\orderind=1$ and, alternatively, 
    \begin{itemize}
    \item  $\othermatA=\anymat$ and $\anymat\in \{\theuniA,\theuniB\}$ and $(\indt,\indi)\lexorderRstrict(\indj,\inds)$
    \item $\othermatA=\anymat$ and $\anymat\in \{\theuniAT,\theuniBT\}$ and $(\indi,\indt)\lexorderRstrict(\inds,\indj)$
    \item $\othermatA=\anymatAT$ and $\anymat\in\{\theuniA,\theuniB\}$ and $(\Sexchanged(\indj),\indt)<(\inds,\Sexchanged(\indi))$
      \item  $\othermatA=\anymatAT$ and $\anymat\in\{\theuniAT,\theuniBT\}$ and $(\indt,\Sexchanged(\indj))<(\Sexchanged(\indi),\inds)$.
    \end{itemize}
  \item $\orderind=2$ and, alternatively,
    \begin{itemize}
  \item $\othermat=\anymat$ and $(\indt,\indi)<(\indj,\inds)$.    
  \item $\othermatA=\anymatAT$ and  $\anymat\in\{\theuni,\theuni\SskewstarP\}$ and $\indj=\thedim$ and either $\indt\neq\thedim$ or   $\indt=\thedim\neq\indi$.
  \item $\othermatA=\anymatAT$ and $\anymat\in\{\theuni\StransposeP,\theuni\SskewdaggerP\}$ and  either $\indt\neq\thedim$ or  $\indt=\inds=\thedim=\indj\neq \indi$.
    \end{itemize}
  \item $3\leq \orderind$ and, alternatively,
    \begin{itemize}
  \item $\othermat=\anymat$ and $(\indt,\indi)<(\indj,\inds)$.    
  \item $\othermatA=\anymatAT$ and $\anymat\in\{\theuni,\theuni\SskewstarP\}$ and $\indj= \thedim$.
  \item $\othermatA=\anymatAT$ and $\anymat\in\{\theuni\StransposeP,\theuni\SskewdaggerP\}$ and $\indt\neq \thedim$.
    \end{itemize}
  \end{itemize}
  In particular, $\Sbvfull{\anymatA}{\orderind}{\indj}{\inds}    \theorderRstrict\Sbvfull{\anymatA}{\orderind}{\indj}{\indi}$ if and only if  $\indi<\inds$ and, likewise, $\Sbvfull{\anymatA}{\orderind}{\indt}{\indi}    \theorderRstrict\Sbvfull{\anymatA}{\orderind}{\indj}{\indi}$ if and only if  $\indj<\indt$.
  \end{lemma}
  \begin{proof}
    \newcommand{\leftdummyindex}{b}
    \newcommand{\rightdummyindex}{a}
    \newcommand{\leftrest}{q}
    \newcommand{\rightrest}{p}    
    No matter the value of $\orderind$, there exist $\{\leftdummyindex,\rightdummyindex\}\subseteq \SYnumbers{\thedim}$ and $\{\leftrest,\rightrest\}\subseteq \themonoid$ such that $\Sbvfull{\othermat}{\orderind}{\indt}{\inds}=\othermatAX{\indt}{\leftdummyindex}\leftrest$ and $\Sbvfull{\anymatA}{\orderind}{\indj}{\indi} =\anymatAX{\indj}{\rightdummyindex}\rightrest$. If $\othermat\in\{\anymatB,\anymatBT\}$, then already $\othermatAX{\indt}{\leftdummyindex}\theorderRstrict\anymatAX{\indj}{\rightdummyindex}$ if and only if $\anymatA\in\{\theuniA,\theuniAT\}$ by Lem\-ma~\ref{lemma:monomial_order_reformulation} and thus, as claimed,  $\Sbvfull{\othermat}{\orderind}{\indt}{\inds}\theorderRstrict\Sbvfull{\anymatA}{\orderind}{\indj}{\indi}$ if and only if $\anymatA\in\{\theuniA,\theuniAT\}$
    since $\theorder$ was extended (degreewise) lexicographically. For that reason it can be assumed  in the following that $\othermatA\notin\{\anymatB,\anymatBT\}$, which by Lem\-ma~\ref{lemma:any_two_generators} means that $\othermatA\in\{\anymatA,\anymatAT\}$. It is necessary to distinguish both between different values of $\orderind$ and between whether $\othermatA$ is $\anymatA$ or $\anymatAT$.
    \par 
    \emph{Case~1.} If $\orderind=1$, then by Lem\-ma~\ref{lemma:monomial_order_reformulation} the inequality
    \begin{IEEEeqnarray*}{rCl}
     \Sbvfull{\othermat}{1}{\indt}{\inds}=\othermatAX{\indt}{\Sexchanged(\inds)}&\theorderRstrict&\anymatAX{\indj}{\Sexchanged(\indi)}=\Sbvfull{\anymatA}{1}{\indj}{\indi} 
    \end{IEEEeqnarray*}
    is true  if and only if either $\othermatA=\anymat$ and $\anymat\in \{\theuniA,\theuniB\}$ and $(\indt,\Sexchanged(\inds))\lexorderRstrict(\indj,\Sexchanged(\indi))$ or $\othermatA=\anymat$ and $\anymat\in \{\theuniAT,\theuniBT\}$ and $(\Sexchanged(\inds),\indt)\lexorderRstrict(\Sexchanged(\indi),\indj)$ or $\othermatA=\anymatAT$ and $\anymat\in\{\theuniA,\theuniB\}$ and $(\Sexchanged(\inds),\indt)<(\indj,\Sexchanged(\indi))$ or $\othermatA=\anymatAT$ and $\anymat\in\{\theuniAT,\theuniBT\}$ and $(\indt,\Sexchanged(\inds))<(\Sexchanged(\indi),\indj)$. Since $\Sexchanged$ reverses the order of $\SYnumbers{\thedim}$ that proves the claim in this case.
    \par
    \emph{Case~2.} For $\orderind=2$, the inequality $\Sbvfull{\othermat}{2}{\indt}{\inds}=\othermatAX{\indt}{\thedim}\othermatBTX{1}{\Sexchanged(\inds)}\theorderRstrict\anymatAX{\indj}{\thedim}\anymatBTX{1}{\Sexchanged(\indi)}=\Sbvfull{\anymatA}{2}{\indj}{\indi}$ is equivalent to 
   \begin{IEEEeqnarray*}{rCl}
\othermatAX{\indt}{\thedim}\othermatAX{\Sexchanged(\inds)}{1}\theorderRstrict\anymatAX{\indj}{\thedim}\anymatAX{\Sexchanged(\indi)}{1}\IEEEyesnumber\label{eq:base_monomial_order_1}
        \end{IEEEeqnarray*}
        by Lem\-ma~\ref{lemma:skewdagger_comparisons} since  $\othermatA\in\{\anymatA,\anymatAT\}$.
   It is necessary to distinguish further whether $\othermatA$ is $\anymatA$ and $\anymatAT$ and between different values of $\anymat$.
      \par
      \emph{Case~2.1.} If $\othermat=\anymatA$, then by Lem\-ma~\ref{lemma:generators_reformulation} and by the addendum to Lem\-ma~\ref{lemma:monomial_order_reformulation} condition \eqref{eq:base_monomial_order_1} is equivalent to the statement that either  $\indt<\indj$ or both $\indt=\indj$ and $\Sexchanged(\inds)<\Sexchanged(\indi)$, which is to say $\indi<\inds$. In other words, \eqref{eq:base_monomial_order_1} is equivalent to $(\indt,\indi)\lexorderRstrict(\indj,\inds)$ if $\othermatA=\anymatA$, as asserted.
   \par
   \emph{Case~2.2.} If, next, $\othermatA=\anymatAT$ and $\anymat\in\{\theuniA,\theuniB\}$, then Lemmata~\ref{lemma:generators_reformulation} and \ref{lemma:monomial_order_reformulation} imply that \eqref{eq:base_monomial_order_1} holds if and only if either $(\thedim,\indt)\lexorderRstrict(\indj,\thedim)$  or both $(\thedim,\indt)=(\indj,\thedim)$ and $(1,\Sexchanged(\inds))\lexorderRstrict(\Sexchanged(\indi),1)$. Those conditions are equivalent to either $\indt<\thedim=\indj$ or both $\indt=\thedim=\indj$ and  $1<\Sexchanged(\indi)$, i.e., $\indi<\thedim$, being true. In particular, $\indj=\thedim$ is required for \eqref{eq:base_monomial_order_1} to hold. And, if so, then \eqref{eq:base_monomial_order_1} is valid if and only if either $\indt\neq \thedim$ or both $\indt=\thedim$ and $\indi\neq\thedim$.
   \par 
   \emph{Case~2.3.} Finally, if $\othermatA=\anymatAT$ and $\anymat\in\{\theuniAT,\theuniBT\}$, then by  Lemmata~\ref{lemma:generators_reformulation} and \ref{lemma:monomial_order_reformulation} the inequality \eqref{eq:base_monomial_order_1} is satisfied  if and only if either $(\indt,\thedim)\lexorderRstrict(\thedim,\indj)$  or both $(\indt,\thedim)=(\thedim,\indj)$ and $(\Sexchanged(\inds),1)\lexorderRstrict(1,\Sexchanged(\indi))$. This is true if and only if  either $\indt<\thedim$ or both $\indt=\thedim=\indj$ and  $\Sexchanged(\inds)=1<\Sexchanged(\indi)$, i.e., either $\indt\neq \thedim$ or $\indt=\inds=\thedim=\indj\neq\indi$. That is what was claimed in this case.
   \par
   \emph{Case~3.} Lastly, for $3\leq \orderind$, the definitions imply that $\Sbvfull{\othermat}{\orderind}{\indt}{\inds}\theorderRstrict\Sbvfull{\anymatA}{\orderind}{\indj}{\indi}$ is true if and only if either $\orderind$ is odd and
\begin{IEEEeqnarray*}{l}
\othermatAX{\indt}{\thedim}(\othermatBTX{1}{\thedim}\othermatAX{1}{\thedim})^{\frac{\orderind-3}{2}}\othermatBTX{1}{\thedim}\othermatAX{1}{\Sexchanged(\inds)}\theorderRstrict\anymatAX{\indj}{\thedim}(\anymatBTX{1}{\thedim}\anymatAX{1}{\thedim})^{\frac{\orderind-3}{2}}\anymatBTX{1}{\thedim}\anymatAX{1}{\Sexchanged(\indi)}
\end{IEEEeqnarray*}
or  $\orderind$ is even and 
  \begin{IEEEeqnarray*}{rCl}
\othermatAX{\indt}{\thedim}(\othermatBTX{1}{\thedim}\othermatAX{1}{\thedim})^{\frac{\orderind-2}{2}}\othermatBTX{1}{\Sexchanged(\inds)}\theorderRstrict\anymatAX{\indj}{\thedim}(\anymatBTX{1}{\thedim}\anymatAX{1}{\thedim})^{\frac{\orderind-2}{2}}\anymatBTX{1}{\Sexchanged(\indi)}.
\end{IEEEeqnarray*}
Since $\othermatA\in\{\anymatA, \anymatAT\}$ these conditions can be simplified with the help of Lemma~\ref{lemma:skewdagger_comparisons} into the form that, if $\orderind$ is odd,
\begin{IEEEeqnarray*}{l}
\othermatAX{\indt}{\thedim}\othermatAX{\thedim}{1}(\othermatAX{1}{\thedim}\othermatAX{\thedim}{1})^{\frac{\orderind-3}{2}}\othermatAX{1}{\Sexchanged(\inds)}\theorderRstrict\anymatAX{\indj}{\thedim}\anymatAX{\thedim}{1}(\anymatAX{1}{\thedim}\anymatAX{\thedim}{1})^{\frac{\orderind-3}{2}}\anymatAX{1}{\Sexchanged(\indi)}\IEEEeqnarraynumspace\IEEEyesnumber\label{eq:base_monomial_order_2}
\end{IEEEeqnarray*}
and, if $\orderind$ is even,
  \begin{IEEEeqnarray*}{rCl}
\othermatAX{\indt}{\thedim}\othermatAX{\thedim}{1}(\othermatAX{1}{\thedim}\othermatAX{\thedim}{1})^{\frac{\orderind-4}{2}}\othermatAX{1}{\thedim}\othermatAX{\Sexchanged(\inds)}{1}\theorderRstrict\anymatAX{\indj}{\thedim}\anymatAX{\thedim}{1}(\anymatAX{1}{\thedim}\anymatAX{\thedim}{1})^{\frac{\orderind-4}{2}}\anymatAX{1}{\thedim}\anymatAX{\Sexchanged(\indi)}{1}.\IEEEeqnarraynumspace\IEEEyesnumber\label{eq:base_monomial_order_3}
\end{IEEEeqnarray*}
Again, the argument depends on the value of $(\othermat,\anymat)$.
\par
\emph{Case~3.1.} If $\othermatA=\anymatA$, then also $\othermatAX{\thedim}{1}(\othermatAX{\thedim}{1}\othermatAX{1}{\thedim})^{\frac{\orderind-3}{2}}=\anymatAX{\thedim}{1}(\anymatAX{\thedim}{1}\anymatAX{1}{\thedim})^{\frac{\orderind-3}{2}}$ if $\orderind$ is odd and, likewise, $\othermatAX{\thedim}{1}(\othermatAX{1}{\thedim}\othermatAX{\thedim}{1})^{\frac{\orderind-4}{2}}\othermatAX{1}{\thedim}=\anymatAX{\thedim}{1}(\anymatAX{1}{\thedim}\anymatAX{\thedim}{1})^{\frac{\orderind-4}{2}}\anymatAX{1}{\thedim}$ if $\orderind$ is even. Hence, \eqref{eq:base_monomial_order_2} is true if and only if $\anymatAX{\indt}{\thedim}\anymatAX{1}{\Sexchanged(\inds)}\theorderRstrict\anymatAX{\indj}{\thedim}\anymatAX{1}{\Sexchanged(\indi)}$ and \eqref{eq:base_monomial_order_3}  if and only if $\anymatAX{\indt}{\thedim}\anymatAX{\Sexchanged(\inds)}{1}\theorderRstrict\anymatAX{\indj}{\thedim}\anymatAX{\Sexchanged(\indi)}{1}$. By Lemma~\ref{lemma:generators_reformulation} and the addendum to Lemma~\ref{lemma:monomial_order_reformulation} that is true if and only if either $\indt<\indj$ or both $\indt=\indj$ and $\Sexchanged(\inds)<\Sexchanged(\indi)$, regardless of what parity $\orderind$ has. That means $\Sbvfull{\othermat}{\orderind}{\indt}{\inds}\theorderRstrict\Sbvfull{\anymatA}{\orderind}{\indj}{\indi}$ if and only if $(\indt,\indi)\lexorderRstrict(\indj,\inds)$, as was claimed.
\par
\emph{Case~3.2.} If $\othermatA=\anymatAT$ and $\anymat\in\{\theuniA,\theuniB\}$, then the assumption $2\leq \thedim$ requires  $(1,\thedim)\lexorderRstrict(\thedim,1)$ and thus  $\othermatAX{\thedim}{1}\theorderRstrict\anymatAX{\thedim}{1}$ by Lemma~\ref{lemma:monomial_order_reformulation}. In consequence, each of \eqref{eq:base_monomial_order_2} and \eqref{eq:base_monomial_order_3} is satisfied if and only if $\othermatAX{\indt}{\thedim}\theorderR\anymatAX{\indj}{\thedim}$. According to  Lemmata~\ref{lemma:generators_reformulation} and \ref{lemma:monomial_order_reformulation} that is equivalent to the condition that  $(\thedim,\indt)\lexorderR(\indj,\thedim)$. Because the latter is met if and only if $\indj=\thedim$ that proves the claim in this case.
\par
\emph{Case~3.3.} In the final case that $\othermatA=\anymatAT$ and $\anymat\in\{\theuniAT,\theuniBT\}$ the fact that  $(\thedim,1)\lexorderRstrict(1,\thedim)$ is false  implies by Lemma~\ref{lemma:monomial_order_reformulation} that also $\othermatAX{\thedim}{1}\theorderRstrict\anymatAX{\thedim}{1}$ is false. Hence, any one of \eqref{eq:base_monomial_order_2} and \eqref{eq:base_monomial_order_3} is equivalent to  $\othermatAX{\indt}{\thedim}\theorderRstrict\anymatAX{\indj}{\thedim}$. By Lemma~\ref{lemma:monomial_order_reformulation} that is true if and only if $(\indt,\thedim)\lexorderRstrict(\thedim,\indj)$, i.e., if $\indt\neq \thedim$. As this is what was claimed in this case that concludes the proof.
  \end{proof}

\par

The central identity for the recursion will be the following.
\begin{lemma}
  \label{lemma:base_recursion}
  For any $\orderind\in\Sintegersp$,  any $\anymat\in \thematrices$ and any $(\indj,\indi)\in \SYnumbers{\thedim}^{\Ssetmonoidalproduct 2}$, if   $2\leq \orderind$, then
  \begin{IEEEeqnarray*}{rCl}
\Sbvfull{\anymatA}{\orderind}{\indj}{\indi}&=  & \Sbvfull{\anymatA}{\orderind-1}{\indj}{1}\cdot \anymatPAX{\orderind}{1}{\Sexchanged(\indi)}    \IEEEyesnumber\label{eq:base_recursion_01}
\end{IEEEeqnarray*}
and, if $3\leq \orderind$, then
  \begin{IEEEeqnarray*}{rCl}
\Sbvfull{\anymatA}{\orderind-1}{\indj}{\indi}&=&\Sbvfull{\anymatA}{\orderind-2}{\indj}{1}\cdot\anymatPBTX{\orderind}{1}{\Sexchanged(\indi)},\IEEEyesnumber\label{eq:base_recursion_02}
\end{IEEEeqnarray*}
where $\cdot$ is the multiplication of $\thefreealg$.
\end{lemma}
\begin{proof}
  Because $\anymatPBTX{\orderind}{1}{\Sexchanged(\indi)}=\anymatPAX{\orderind-1}{1}{\Sexchanged(\indi)}$ identity \eqref{eq:base_recursion_02} follows from \eqref{eq:base_recursion_01}, which is thus the only one that needs proving.
  \par  
  \emph{Case~1.}  If $\orderind=2$, then by definition
  \begin{IEEEeqnarray*}{rClCrClCrCl}
\Sbvfull{\anymatA}{2}{\indj}{\indi}&= & \anymatAX{\indj}{\thedim}\anymatBTX{1}{\Sexchanged(\indi)}, &\hspace{2em}&\Sbvfull{\anymatA}{1}{\indj}{1}&=&\anymatAX{\indj}{\Sexchanged(1)}=\anymatAX{\indj}{\thedim},&\hspace{2em}& \anymatPAX{\orderind}{1}{\Sexchanged(\indi)}&=&\anymatBTX{1}{\Sexchanged(\indi)},
\end{IEEEeqnarray*}
where the last identity is due to $\orderind$ being even.
  That proves \eqref{eq:base_recursion_01} in this case.
  \par
  \emph{Case~3.} If $\orderind$ is odd with $3\leq \orderind$, then $\anymatPAX{\orderind}{1}{\Sexchanged(\indi)}=\anymatAX{1}{\Sexchanged(\indi)}$ as well as
  \begin{IEEEeqnarray*}{rCl}
    \Sbvfull{\anymatA}{\orderind}{\indj}{\indi}&=&\anymatX{\indj}{\thedim}(\anymatX{1}{\thedim}\SskewdaggerP\anymatX{1}{\thedim})^{\frac{\orderind-3}{2}}\anymatX{1}{\thedim}\SskewdaggerP\anymatX{1}{\Sexchanged(\indi)},\\
    \Sbvfull{\anymatA}{\orderind-1}{\indj}{1}=\anymatX{\indj}{\thedim}(\anymatX{1}{\thedim}\SskewdaggerP\anymatX{1}{\thedim})^{\frac{(\orderind-1)-2}{2}} \anymatX{1}{\Sexchanged(1)}\SskewdaggerP&=&\anymatX{\indj}{\thedim}(\anymatX{1}{\thedim}\SskewdaggerP\anymatX{1}{\thedim})^{\frac{\orderind-3}{2}} \anymatX{1}{\thedim}\SskewdaggerP.
  \end{IEEEeqnarray*}  
  Comparing the polynomials now proves \eqref{eq:base_recursion_01}.
  \par
  \emph{Case~3.} Finally, if $\orderind$ is even and $4\leq \orderind$, then it was defined that
  \begin{IEEEeqnarray*}{rCl}
\Sbvfull{\anymatA}{\orderind}{\indj}{\indi}=\anymatX{\indj}{\thedim}(\anymatX{1}{\thedim}\SskewdaggerP\anymatX{1}{\thedim})^{\frac{\orderind-2}{2}}\anymatX{1}{\Sexchanged(\indi)}\SskewdaggerP&=&\anymatX{\indj}{\thedim}(\anymatX{1}{\thedim}\SskewdaggerP\anymatX{1}{\thedim})^{\frac{\orderind-4}{2}}\anymatX{1}{\thedim}\SskewdaggerP\anymatX{1}{\thedim} \anymatX{1}{\Sexchanged(\indi)}\SskewdaggerP,\\
    \Sbvfull{\anymatA}{\orderind-1}{\indj}{1}=\anymatX{\indj}{\thedim}(\anymatX{1}{\thedim}\SskewdaggerP\anymatX{1}{\thedim})^{\frac{(\orderind-1)-3}{2}}\anymatX{1}{\thedim}\SskewdaggerP\anymatX{1}{\Sexchanged(1)}&=&\anymatX{\indj}{\thedim}(\anymatX{1}{\thedim}\SskewdaggerP\anymatX{1}{\thedim})^{\frac{\orderind-4}{2}}\anymatX{1}{\thedim}\SskewdaggerP\anymatX{1}{\thedim}
  \end{IEEEeqnarray*}
  since $\orderind-1$ is odd. Given that by definition $\anymatPAX{\orderind}{1}{\Sexchanged(\indi)}=\anymatBTX{1}{\Sexchanged(\indi)}$ the identity \eqref{eq:base_recursion_01} is satisfied in this instance as well.
\end{proof}

\subsection{Chains, combinatorial differentials, cycles, splittings} When computing Anick's resolution, the first order of business is to determine the sets of chains and the combinatorial differentials, cycles and splittings.
\subsubsection{Chains}
  For any $\orderind\in\Sintegersnn$  let $\theprechains{\orderind}$ respectively $\thechains{\orderind}$ denote the sets of $\orderind$\-/pre\-/chains and -chains with respect to $\theideal$ and $\theorder$.
{
  \newcommand{\chaingen}{e}
  \newcommand{\chaingenX}[1]{\chaingen_{#1}}
  \newcommand{\chainends}{\beta}
  \newcommand{\chainendsX}[1]{\chainends_{#1}}
  \newcommand{\chainstarts}{\alpha}
  \newcommand{\chainstartsX}[1]{\chainstarts_{#1}}
  \newcommand{\posind}{s}
  \newcommand{\anymonomial}{c}
  \begin{lemma}
    \label{lemma:basis_vectors_reformulation}
    For any $\orderind\in\Sintegersp$ with $2\leq \orderind$, any $\anymat\in\thematrices$ and any $(\indj,\indi)\in \SYnumbers{\thedim}^{\Ssetmonoidalproduct 2}$  any $\anymonomial\in\themonoid$ is equal to $\Sbvfull{\anymat}{\orderind}{\indj}{\indi}$ if and only if there exists $\chaingen\in\thegens^{\Ssetmonoidalproduct\orderind}$  such that $\anymonomial=\chaingenX{1}\chaingenX{2}\hdots\chaingenX{\orderind}$ and such that for any $\posind\in\SYnumbers{\orderind-1}$,
    \begin{IEEEeqnarray*}{rCl}
    \chaingenX{\posind}\chaingenX{\posind+1}&=&
    \begin{cases}
      \anymatPAX{\posind}{\indj}{\thedim}\anymatPBTX{\posind}{1}{\thedim}  &\Scase\posind=1\\
      \anymatPAX{\posind}{1}{\thedim}\anymatPBTX{\posind}{1}{\thedim} & \Scase 2\leq\posind\leq\orderind-2\\
      \anymatPAX{\posind}{1}{\thedim}\anymatPBTX{\posind}{1}{\Sexchanged(\indi)} & \Scase \posind=\orderind-1.
    \end{cases}      
\IEEEyesnumber\label{eq:basis_vectors_reformulation}      
    \end{IEEEeqnarray*}
  \end{lemma}
  \begin{proof}
    Follows from the definition by distinguishing between even and odd $\orderind$.
  \end{proof}

  \begin{proposition}
    \label{proposition:chains}
  For any $\orderind\in\Sintegersp$,
  \begin{IEEEeqnarray*}{C}    \theprechains{\orderind}=\thechains{\orderind}=\{\Sbvfull{\anymat}{\orderind}{\indj}{\indi}\Ssetbuilder\anymat\in\thematrices\Sand(\indj,\indi)\in\SYnumbers{\thedim}^{\Ssetmonoidalproduct 2}\}.
  \end{IEEEeqnarray*}
  Moreover, for any $\anymat\in\thematrices$ and any $(\indj,\indi)\in\SYnumbers{\thedim}^{\Ssetmonoidalproduct 2}$, if $(\chainstarts,\chainends)$ are the chain\-/indices of $\Sbvfull{\anymat}{\orderind}{\indj}{\indi}$, then $\chainstartsX{\posind}=\posind$ and $\chainendsX{\posind}=\posind+1$ for any $\posind\in\SYnumbers{\orderind-1}$.
\end{proposition}
\begin{proof}
  \newcommand{\anychain}{c}
  \newcommand{\chaindegree}{\orderind-1}
  \newcommand{\chainlen}{m}
  \newcommand{\smallerorder}{\orderind'}
  \newcommand{\bigposind}{t}
  \newcommand{\anymatI}[1]{v^{(#1)}}
  \newcommand{\anymatIA}[1]{\anymat^{(#1)\theniv}}
  \newcommand{\anymatIAT}[1]{\anymat^{(#1)\hspace{1pt}\theniv\Stranspose}}
  \newcommand{\anymatIB}[1]{\anymat^{(#1)\hspace{1pt}\theniv\Sskewstar}}
  \newcommand{\anymatIBT}[1]{\anymat^{(#1)\hspace{1pt}\theniv\Sskewdagger}}
  \newcommand{\anymatIAX}[3]{\anymatIA{#1}_{#2,#3}}
  \newcommand{\anymatIATX}[3]{\anymatIAT{#1}_{#2,#3}}
  \newcommand{\anymatIBX}[3]{\anymatIB{#1}_{#2,#3}}
  \newcommand{\anymatIBTX}[3]{\anymatIBT{#1}_{#2,#3}}  
  \newcommand{\anymatIPA}[2]{\anymatIA{#1}{#2}}
  \newcommand{\anymatIPAT}[2]{\anymatIAT{#1}{#2}}
  \newcommand{\anymatIPB}[2]{\anymatIB{#1}{#2}}
  \newcommand{\anymatIPBT}[2]{\anymatIBT{#1}{#2}}
  \newcommand{\anymatIPAX}[4]{\anymatIPA{#1}{#2}_{#3,#4}}
  \newcommand{\anymatIPATX}[4]{\anymatIPAT{#1}{#2}_{#3,#4}}
  \newcommand{\anymatIPBX}[4]{\anymatIPB{#1}{#2}_{#3,#4}}
  \newcommand{\anymatIPBTX}[4]{\anymatIPBT{#1}{#2}_{#3,#4}}
  \newcommand{\leftindX}[1]{\indj_{#1}}
  \newcommand{\rightindX}[1]{\indi_{#1}}
  \newcommand{\inda}{a}
  \newcommand{\indb}{b}
  The proof varies with $\orderind$. Hence, a case distinction is necessary.
  \par
\emph{Case~1.}  If $\orderind=1$, then 
 $\theprechains{1}=\thechains{1}=\thegens$ already by definition and the addendum is vacuously true. The only actual claim is that $\thegens=\{\Sbvfull{\anymat}{1}{\indj}{\indi}\Ssetbuilder\anymat\in\thematrices\Sand(\indj,\indi)\in\SYnumbers{\thedim}^{\Ssetmonoidalproduct 2}\}$. But that is clear because $\Sbvfull{\anymat}{1}{\indj}{\indi}=\anymatAX{\indj}{\Sexchanged(\indi)}$ for any $\anymat\in\thematrices$ and $(\indj,\indi)\in\SYnumbers{\thedim}^{\Ssetmonoidalproduct 2}$.
  \par
\emph{Case~2.} Also in the case that $\orderind=2$ most of the claims are clear a priori.  Indeed, irrespective of the particular choices of $\thegens$,  $\theideal$ and $\theorder$, it can be seen that  $\theprechains{2}=\thechains{2}=\thecore$ and that the chain indices of any $2$-chain $\anychain$ are necessarily $(1,\chainlen)$, where $\chainlen$ is the degree of $\anychain$. Hence, the only non-trivial claim in this case is that $\thecore=\{\Sbvfull{\anymat}{2}{\indj}{\indi}\Ssetbuilder\anymat\in\thematrices\Sand(\indj,\indi)\in\SYnumbers{\thedim}^{\Ssetmonoidalproduct 2}\}$. But again that is obvious by Proposition~\ref{proposition:core} because  $\Sbvfull{\anymat}{2}{\indj}{\indi}=\anymatAX{\indj}{\thedim}\anymatBTX{1}{\Sexchanged(\indi)}$ for any $\anymat\in\thematrices$ and $(\indj,\indi)\in\SYnumbers{\thedim}^{\Ssetmonoidalproduct 2}$.
  \par
  \emph{Case~3.} The proof for the case $3\leq \orderind$ is divided into two steps. The first is
to prove $\theprechains{\orderind}\subseteq \{\Sbvfull{\anymat}{\orderind}{\indj}{\indi}\Ssetbuilder\anymat\in\thematrices\Sand(\indj,\indi)\in\SYnumbers{\thedim}^{\Ssetmonoidalproduct 2}\}$. Because  $\thechains{\orderind}\subseteq\theprechains{\orderind}$ by definition, it is then enough 
  to check, second, that $\{\Sbvfull{\anymat}{\orderind}{\indj}{\indi}\Ssetbuilder\anymat\in\thematrices\Sand(\indj,\indi)\in\SYnumbers{\thedim}^{\Ssetmonoidalproduct 2}\}\subseteq \thechains{\orderind}$, in the process of which the addendum will  be proved as well.
  \par
  \emph{Step~3.1.} For any $\chainlen\in\Sintegersnn$ and any $\chaingen\in\thegens^{\Ssetmonoidalproduct 2}$, if $\anychain=\chaingenX{1}\chaingenX{2}\ldots\chaingenX{\chainlen}\in\theprechains{\orderind}$, then there exist  $\chainstarts\in\SYnumbers{\chainlen}^{\Ssetmonoidalproduct\chaindegree}$ and  $\chainends\in\SYnumbers{\chainlen}^{\Ssetmonoidalproduct\chaindegree}$ such that  $1=\chainstartsX{1}<\chainstartsX{2}\leq \chainendsX{1}$ and $\chainstartsX{\chaindegree}\leq\chainendsX{\orderind-2}<\chainendsX{\chaindegree}= \chainlen$ and $\chainstartsX{\posind+1}\leq \chainendsX{\posind}<\chainstartsX{\posind+2}$ for any $\posind\in\SYnumbers{\orderind-3}$ and such that $\chaingenX{\chainstartsX{\posind}}\chaingenX{\chainstartsX{\posind}+1}\ldots\chaingenX{\chainendsX{\posind}}\in \thecore$ for any $\posind\in\SYnumbers{\chaindegree}$. Since all elements of $\thecore$ have degree $2$ by Proposition~\ref{proposition:core} this last condition implies $\chainendsX{\posind}=\chainstartsX{\posind}+1$ for each $\posind\in\SYnumbers{\chaindegree}$. That implies $\chainstartsX{\posind}=\posind$ and $\chainendsX{\posind}=\posind+1$ for each $\posind\in\SYnumbers{\chaindegree}$. In particular, $\chainlen=\chainendsX{\chaindegree}=\orderind$. In fact, by Proposition~\ref{proposition:core} for any $\posind\in\SYnumbers{\orderind-1}$ there exist $\anymatIA{\posind}\in \thematrices$ and $(\leftindX{\posind},\rightindX{\posind})\in \SYnumbers{\thedim}^{\Ssetmonoidalproduct 2}$ with $\chaingenX{\chainstartsX{\posind}}\chaingenX{\chainstartsX{\posind}+1}\hdots\chaingenX{\chainendsX{\posind}}=\chaingenX{\posind}\chaingenX{\posind+1}=\anymatIAX{\posind}{\leftindX{\posind}}{\thedim}\anymatIBTX{\posind}{1}{\Sexchanged(\rightindX{\posind})}$.
  \par
  If $\anymatA\Seqpd \anymatIA{1}$, if $\indj\Seqpd \leftindX{1}$ and if $\indi\Seqpd\rightindX{\orderind-1}$, then $\anychain=\Sbvfull{\anymat}{\orderind}{\indj}{\indi}$ by the  following argument. For each $\posind\in \SYnumbers{\orderind-2}$ the equality $\chainendsX{\posind}=\chainstartsX{\posind+1}$ demands  $\anymatIBTX{\posind}{1}{\Sexchanged(\rightindX{\posind})}=\chaingenX{\posind+1}=\anymatIAX{\posind+1}{\leftindX{\posind+1}}{\thedim}$. According to Lem\-ma~\ref{lemma:generators_reformulation} that requires  $\anymatIA{\posind+1}\in\{\anymatIB{\posind},\anymatIBT{\posind}\}$ for each $\posind\in \SYnumbers{\orderind-2}$. If $\anymatIA{\posind+1}$ was $\anymatIB{\posind}$ for such $\posind$, it would follow that $(1,\Sexchanged(\rightindX{\posind}))=(\thedim,\leftindX{\posind+1})$, which would violate the assumption $2\leq \thedim$. Hence, instead, $\anymatIA{\posind+1}=\anymatIBT{\posind}$ and $(1,\Sexchanged(\rightindX{\posind}))=(\leftindX{\posind+1},\thedim)$, i.e., $\leftindX{\posind+1}=1=\rightindX{\posind}$, for each $\posind\in \SYnumbers{\orderind-2}$. In particular, $\anymatIA{\posind}=\anymatPA{\posind}$ for any $\posind\in\SYnumbers{\orderind-1}$. Actually, altogether that proves that $\anychain$ satisfies the identities \eqref{eq:basis_vectors_reformulation} in Lem\-ma~\ref{lemma:basis_vectors_reformulation} and must thus be equal to $\Sbvfull{\anymat}{\orderind}{\indj}{\indi}$, as claimed.  
  \par
  \emph{Step~3.2.} For any $\anymat\in\thematrices$ and $(\indj,\indi)\in\SYnumbers{\thedim}^{\Ssetmonoidalproduct 2}$, if $\chainstartsX{\posind}=\posind$ and $\chainendsX{\posind}=\posind+1$ for any $\posind\in\SYnumbers{\orderind-1}$, then  $\Sbvfull{\anymat}{\orderind}{\indj}{\indi}$ is an $\orderind$\-/chain with chain indices $(\chainstarts,\chainends)$, as explained hereafter. 
  \par
  Of course, $1=\chainstartsX{1}<\chainstartsX{2}\leq \chainendsX{1}$ by $\chainstartsX{2}=2=\chainendsX{1}$. The degree of $\Sbvfull{\anymat}{\orderind}{\indj}{\indi}$ is $\orderind$  and $\chainstartsX{\orderind-1}\leq\chainendsX{\orderind-2}<\chainendsX{\orderind-1}= \orderind$ holds by $\chainstartsX{\orderind-1}=\orderind-1$ and $\chainendsX{\orderind-2}=\orderind-1$. Likewise, for any $\posind\in\SYnumbers{\orderind-3}$ the definitions $\chainstartsX{\posind+1}=\posind+1$ and $\chainendsX{\posind}=\posind+1$ and $\chainstartsX{\posind+2}=\posind+2$ imply  $\chainstartsX{\posind+1}\leq \chainendsX{\posind}<\chainstartsX{\posind+2}$.
  \par
  Let $\chaingen\in\thegens^{\Ssetmonoidalproduct \orderind}$ be the unique sequence of generators with $\Sbvfull{\anymat}{\orderind}{\indj}{\indi}=\chaingenX{1}\chaingenX{2}\hdots\chaingenX{\orderind}$. By Lem\-ma~\ref{lemma:basis_vectors_reformulation} then the identities \eqref{eq:basis_vectors_reformulation} hold. According to Proposition~\ref{proposition:core} those then prove  $\chaingenX{\chainstartsX{\posind}}\chaingenX{\chainstartsX{\posind}+1}\hdots\chaingenX{\chainendsX{\posind}}=\chaingenX{\inds}\chaingenX{\inds+1}\in\thecore$ for each $\posind\in\SYnumbers{\orderind-1}$.  Hence,  $\Sbvfull{\anymat}{\orderind}{\indj}{\indi}\in\theprechains{\orderind}$.
  \par
Moreover, for any $\smallerorder\in \SYnumbers{\orderind}$ with $2\leq \smallerorder$ and any   $\bigposind\in\SYnumbers{\chainendsX{\smallerorder-1}-1}$ the monomial $\chaingenX{1}\chaingenX{2}\hdots\chaingenX{\bigposind}$ is no element of $\theprechains{\smallerorder}$ because its degree is less than $\chainendsX{\smallerorder-1}=(\smallerorder-1)+1=\smallerorder$, whereas by Step~3.1 any element of $\theprechains{\smallerorder}$ is of the form $\Sbvfull{\othermat}{\smallerorder}{\indb}{\inda}$ for some $\othermat\in\thematrices$ and $(\indb,\inda)\in \SYnumbers{\thedim}^{\Ssetmonoidalproduct 2}$ and thus of degree $\smallerorder$. 
  \par
In conclusion, $\Sbvfull{\anymat}{\orderind}{\indj}{\indi}$ is an $\orderind$\-/chain and $(\chainstarts,\chainends)$ are its chain indices. That verifies the claim in the case $3\leq \orderind$ and thus completes the proof overall.
\end{proof}
}

\subsubsection{Combinatorial differential}
For any $\orderind\in \Sintegersp$ let $\thedeconcatenator{\orderind}$ be the combinatorial differential associated with $\theideal$ and $\theorder$ from Section~\ref{section:anick-combinatorial_differentials_and_splitting}.
\par
\begin{proposition}
  \label{proposition:combinatorial_differential}
  For any $\orderind\in\Sintegersp$ with $2\leq \orderind$, any $\anymat\in\thematrices$ and any $(\indj,\indi)\in\SYnumbers{\thedim}^{\Ssetmonoidalproduct 2}$,
\begin{IEEEeqnarray*}{rCl}
  \thedeconcatenator{\orderind}(\Sbvfull{\anymatA}{\orderind}{\indj}{\indi})&=  & (\Sbvfull{\anymatA}{\orderind-1}{\indj}{1}, \anymatPAX{\orderind}{1}{\Sexchanged(\indi)}).
\end{IEEEeqnarray*}  
\end{proposition}
\begin{proof}
  \newcommand{\chaingen}{e}
  \newcommand{\chaingenX}[1]{\chaingen_{#1}}
  \newcommand{\chainstarts}{\alpha}
  \newcommand{\chainstartsX}[1]{\chainstarts_{#1}}
  \newcommand{\chainends}{\beta}
  \newcommand{\chainendsX}[1]{\chainends_{#1}}
  \newcommand{\chaingenalt}{e'}
  \newcommand{\chaingenaltX}[1]{\chaingenalt_{#1}}
  \newcommand{\chainstartsalt}{\alpha'}
  \newcommand{\chainstartsaltX}[1]{\chainstartsalt_{#1}}
  \newcommand{\chainendsalt}{\beta'}
  \newcommand{\chainendsaltX}[1]{\chainendsalt_{#1}}  
  \newcommand{\posind}{s}
If  $(\chainstarts,\chainends)$ are the chain indices of $\Sbvfull{\anymatA}{\orderind}{\indj}{\indi}$, then
  $\chainstartsX{\posind}=\posind$ and $\chainendsX{\posind}=\posind+1$  for any $\posind\in\SYnumbers{\orderind-1}$ by the addendum to Proposition~\ref{proposition:chains}. Accordingly, if
  $\chaingen\in\thegens^{\Ssetmonoidalproduct\orderind}$ is such that $\Sbvfull{\anymatA}{\orderind}{\indj}{\indi}=\chaingenX{1}\chaingenX{2}\ldots\chaingenX{\orderind}$, then 
by definition  
$\thedeconcatenator{\orderind}(\Sbvfull{\anymatA}{\orderind}{\indj}{\indi})\Seqpd (\chaingenX{1}\chaingenX{2}\hdots\chaingenX{\chainendsX{\orderind-2}},\chaingenX{\chainendsX{\orderind-2}+1}\chaingenX{\chainendsX{\orderind-2}+2}\hdots\chaingenX{\orderind})= (\chaingenX{1}\chaingenX{2}\hdots\chaingenX{\orderind-1},\chaingenX{\orderind})$. Since $\chaingenX{1}\chaingenX{2}\hdots\chaingenX{\orderind-1}=\Sbvfull{\anymatA}{\orderind-1}{\indj}{1}$  and $\chaingenX{\orderind}=\anymatPAX{\orderind}{1}{\Sexchanged(\indi)}$ 
by \eqref{eq:base_recursion_01} in Lemma~\ref{lemma:base_recursion} that proves the claim.
\end{proof}

\subsubsection{Combinatorial cycles}
For any $\orderind\in \Sintegersnn$ let  $\thechaincycles{\orderind}$ be the set of combinatorial cycles associated with $\theideal$ and $\theorder$ from Section~\ref{section:anick-combinatorial_differentials_and_splitting}.
\par
{
  \newcommand{\anyrest}{h}
\begin{proposition}
  For any $\orderind\in\Sintegersp$ with $3\leq \orderind$, 
  \begin{IEEEeqnarray*}{rCl}
    \label{proposition:combinatorial_cycles}
  \thechaincycles{\orderind-2}&=&
   \{(\Sbvfull{\anymatA}{\orderind-2}{\indj}{1},\anymatPBTX{\orderind}{1}{\Sexchanged(\indi)}\anyrest)\Ssetbuilder\anymat\in\thematrices\Sand (\indj,\indi)\in\SYnumbers{\thedim}^{\Ssetmonoidalproduct 2},\anyrest\in\themonoid\backslash\StipO\theideal\}.
\end{IEEEeqnarray*}  
 \end{proposition}
 \begin{proof}
   \newcommand{\anychain}{c}
   \newcommand{\anyextra}{d}
   \newcommand{\smallchain}{a}
   \newcommand{\chainrest}{b}
   \newcommand{\anygen}{e}
   \newcommand{\anygenX}[1]{\anygen_{#1}}
   \newcommand{\dummyindex}{s}
   \newcommand{\extralen}{t}
   \newcommand{\anylen}{m}
   \newcommand{\dummypos}{z}
   \newcommand{\leftind}{q}
  Each inclusion is proved separately.
  \par
  \emph{Step~1.} For any $\anymat\in\thematrices$ and any $(\indj,\indi)\in \SYnumbers{\thedim}^{\Ssetmonoidalproduct 2}$  the pair   $\thedeconcatenator{\orderind-2}(\Sbvfull{\anymatA}{\orderind-2}{\indj}{1})$ is given by $(\theone,\Sbvfull{\anymatA}{1}{\indj}{1})$ if $\orderind=3$ and otherwise by  $(\Sbvfull{\anymatA}{\orderind-3}{\indj}{1},\anymatPAX{\orderind-2}{1}{\Sexchanged(1)})$ according to Pro\-po\-si\-tion~\ref{proposition:combinatorial_differential}. By definition, $\Sbvfull{\anymatA}{1}{\indj}{1}=\anymatAX{\indj}{\Sexchanged(1)}=\anymatPAX{3}{\indj}{\thedim}$ and $\anymatPAX{\orderind-2}{1}{\Sexchanged(1)}=\anymatPAX{\orderind}{1}{\thedim}$. Both $\anymatPAX{3}{\indj}{\thedim}\anymatPBTX{3}{1}{\Sexchanged(\indi)}$ and $\anymatPAX{\orderind}{1}{\thedim}\anymatPBTX{\orderind}{1}{\Sexchanged(\indi)}$ are elements of $\thecore\subseteq\StipO\theideal$ by Pro\-po\-si\-tion~\ref{proposition:core}. Hence, for any $\anyrest\in\themonoid\backslash \StipO\theideal$ neither $\anymatPAX{3}{\indj}{\thedim}\anymatPBTX{3}{1}{\Sexchanged(\indi)}\anyrest$ nor $\anymatPAX{\orderind}{1}{\thedim}\anymatPBTX{\orderind}{1}{\Sexchanged(\indi)}\anyrest$ is normal, which makes  $(\Sbvfull{\anymatA}{\orderind-2}{\indj}{1},\anymatPBTX{\orderind}{1}{\Sexchanged(\indi)}\anyrest)$ an element of $\thechaincycles{\orderind-2}$.
  \par
  \emph{Step~2.} Let $\{\anylen,\extralen\}\subseteq \Sintegersnn$ and  $\anygen\in\thegens^{\Ssetmonoidalproduct (\anylen+\extralen)}$ be arbitrary with the property, that, if  $\anychain=\anygenX{1}\anygenX{2}\hdots\anygenX{\anylen}$ and $\anyextra=\anygenX{\anylen+1}\anygenX{\anylen+2}\hdots\anygenX{\anylen+\extralen}$, then   $(\anychain, \anyextra)\in \thechaincycles{\orderind-2}$. By Proposition~\ref{proposition:chains} there then exist  $\anymat\in\thematrices$ and $(\indj,\dummyindex)\in\SYnumbers{\thematrices}^{\Ssetmonoidalproduct 2}$ with $\anychain=\Sbvfull{\anymat}{\orderind-2}{\indj}{\dummyindex}$ and, in particular, $\anylen=\orderind-2$. Consequently,   $(\smallchain,\chainrest)\Seqpd \thedeconcatenator{\orderind-2}(\Sbvfull{\anymat}{\orderind-2}{\indj}{\dummyindex})$ is given by $(1,\Sbvfull{\anymat}{1}{\indj}{\dummyindex})=(1,\anymatPAX{1}{\indj}{\Sexchanged(\inds)})$ if $\orderind=3$ and otherwise by $(\Sbvfull{\anymat}{\orderind-3}{\indj}{1},\anymatPAX{\orderind-2}{1}{\Sexchanged(\dummyindex)})$ according to Proposition~\ref{proposition:combinatorial_differential}. 
  \par
  Since $(\anychain, \anyextra)\in \thechaincycles{\orderind-2}$ then $\chainrest\anyextra\in\StipO\theideal$. Because $\chainrest\notin\StipO\theideal$ that requires $\anyextra\neq \theone$, which is to say $1\leq \extralen$. Since $\anyextra\notin \StipO\theideal$ the assumption that  $\chainrest\anyextra\in\StipO\theideal$ then demands the existence of $\dummypos\in\SYnumbers{\extralen}$ with  $\chainrest\anygenX{\anylen+1}\anygenX{\anylen+2}\hdots\anygenX{\anylen+\dummypos}\in \StipO\theideal$. In fact there must exist $\dummypos\in\SYnumbers{\extralen}$ with  $\chainrest\anygenX{\anylen+1}\anygenX{\anylen+2}\hdots\anygenX{\anylen+\dummypos}\in \thecore$ because otherwise the assumption $\anyextra\notin\StipO\theideal$ would be violated. In fact, by nature of $\thecore$ there can exist only such $\dummypos$. Even more precisely, because any element of $\thecore$ has degree $2$ by Proposition~\ref{proposition:core} and since $\chainrest$ has degree $1$ the index $\dummypos$ must be $1$. In conclusion, if $\anyrest\Seqpd \anygenX{\anylen+2}\anygenX{\anylen+3}\hdots\anygenX{\anylen+\extralen}$, then, first, $\anyrest\notin\StipO\theideal$ because already $\anyextra\notin\StipO\theideal$ and, second,  $\anyextra=\anygenX{\anylen+1}\anyrest$ and $\chainrest\anygenX{\anylen+1}\in \thecore$.  Another implication of Proposition~\ref{proposition:core} is therefore the existence of $\othermat\in\thematrices$ and $(\leftind,\indi)\in \SYnumbers{\thedim}^{\Ssetmonoidalproduct 2}$ with $\chainrest\anygenX{\anylen+1}=\othermatAX{\leftind}{\thedim}\othermatBTX{1}{\Sexchanged(\indi)}$. By Lem\-ma~\ref{lemma:generators_reformulation} the ensuing equalities $\chainrest=\anymatPAX{3}{\indj}{\Sexchanged(\dummyindex)}=\othermatAX{\leftind}{\thedim}$ for $\orderind=3$  and $\chainrest=\anymatPAX{\orderind}{1}{\Sexchanged(\dummyindex)}=\othermatAX{\leftind}{\thedim}$ for $4\leq\orderind$ demand  $\othermatA\in \{\anymatPA{\orderind},\anymatPAT{\orderind}\}$ for any $3\leq\orderind$.
  \par
  \emph{Case~2.1.} If $\orderind=3$ and $\othermatA=\anymatPA{\orderind}$, then by the same lemma, $(\indj,\Sexchanged(\dummyindex))=(\leftind,\thedim)$. Thus, $\anychain=\Sbvfull{\anymat}{1}{\indj}{1}$ and $\anygenX{\anylen+1}=\othermatBTX{1}{\Sexchanged(\indi)}=\anymatBTX{1}{\Sexchanged(\indi)}$ and hence $\anyextra=\anygenX{\anylen+1}\anyrest=\anymatPBTX{3}{1}{\Sexchanged(\indi)}\anyrest$. Ultimately,  $(\anychain,\anyextra)=(\Sbvfull{\anymatA}{1}{\indj}{1},\anymatPBTX{3}{1}{\Sexchanged(\indi)}\anyrest)$. That proves the claim in this case.  
  \par
  \emph{Case~2.2.} In case $\orderind=3$ and $\othermatA=\anymatPAT{\orderind}$, the equality  $\anymatAX{\indj}{\Sexchanged(\dummyindex)}=\othermatAX{\leftind}{\thedim}$ requires  $(\indj,\Sexchanged(\dummyindex))=(\thedim,\leftind)$ by Lem\-ma~\ref{lemma:generators_reformulation}. With the help of Lemma~\ref{lemma:evil_identities} it follows $\anychain=\Sbvfull{\anymat}{1}{\indj}{\inds}=\Sbvfull{\anymatA}{1}{\thedim}{\Sexchanged(\leftind)}=\Sbvfull{\anymatAT}{1}{\leftind}{1}$. Because, moreover, $\anygenX{\anylen+1}=\othermatBTX{1}{\Sexchanged(\indi)}=\anymatBX{1}{\Sexchanged(\indi)}$ and thus $\anyextra=\anygenX{\anylen+1}\anyrest=\anymatPBX{3}{1}{\Sexchanged(\indi)}\anyrest$ the conclusion $(\anychain,\anyextra)=(\Sbvfull{\anymatAT}{1}{\leftind}{1},\anymatPBX{3}{1}{\Sexchanged(\indi)}\anyrest)$ proves the claim in this instance.
  \par
  \emph{Case~2.3.} Finally, for $4\leq\orderind$ Lem\-ma~\ref{lemma:generators_reformulation} and the equality $\anymatPAX{\orderind}{1}{\Sexchanged(\dummyindex)}=\othermatAX{\leftind}{\thedim}$ actually require  $\othermatA=\anymatPA{\orderind}$ and $(1,\Sexchanged(\dummyindex))=(\leftind,\thedim)$ because if $\othermatA$ was $\anymatPAT{\orderind}$ the then inevitable identity between   $(1,\Sexchanged(\dummyindex))$ and $(\thedim,\leftind)$ would contradict the assumption $2\leq \thedim$. In conclusion, $\dummyindex=1$ and thus $\anychain=\Sbvfull{\anymatA}{\orderind-2}{\indj}{1}$  and  $\anygenX{\anylen+1}=\othermatBTX{1}{\Sexchanged(\indi)}=\anymatPBTX{\orderind}{1}{\Sexchanged(\indi)}$, i.e., $(\anychain,\anyextra)=(\Sbvfull{\anymatA}{\orderind-2}{\indj}{1},\anymatPBTX{\orderind}{1}{\Sexchanged(\indi)}\anyrest)$, which concludes the proof.
\end{proof}
}

\subsubsection{Combinatorial splitting}
For any $\orderind\in \Sintegersnn$ let $\thereconcatenator{\orderind}$ be the combinatorial splitting associated with $\theideal$ and $\theorder$ from Section~\ref{section:anick-combinatorial_differentials_and_splitting}.
\par
{
  \newcommand{\anyrest}{h}
  \begin{proposition}
        \label{proposition:combinatorial_splitting}
  For any $\orderind\in\Sintegersp$ with $3\leq \orderind$, any $\anymat\in\thematrices$, any $(\indj,\indi)\in\SYnumbers{\thedim}^{\Ssetmonoidalproduct 2}$ and any $\anyrest\in\themonoid\backslash\StipO\theideal$,
\begin{IEEEeqnarray*}{rCl}
\textstyle  \thereconcatenator{\orderind-2}(\Sbvfull{\anymatA}{\orderind-2}{\indj}{1},\anymatPBTX{\orderind}{1}{\Sexchanged(\indi)}\anyrest)=(\Sbvfull{\anymatA}{\orderind-1}{\indj}{\indi},\anyrest).
\end{IEEEeqnarray*}  
\end{proposition}
\begin{proof}
  \newcommand{\chainstarts}{\alpha}
  \newcommand{\chainstartsX}[1]{\chainstarts_{#1}}
  \newcommand{\chainends}{\beta}
  \newcommand{\chainendsX}[1]{\chainends_{#1}}
  \newcommand{\posind}{s}
   \newcommand{\anygen}{e}
   \newcommand{\anygenX}[1]{\anygen_{#1}}
   \newcommand{\extralen}{t}
   \newcommand{\anylen}{m}
   \newcommand{\somepos}{m}   
With the abbreviation $\anylen\Seqpd \orderind-2$, let $\extralen\in\Sintegersp$ and $\anygen\in \SYnumbers{\thedim}^{\Ssetmonoidalproduct(\anylen+\extralen)}$ be such that $\Sbvfull{\anymatA}{\orderind-2}{\indj}{1}=\anygenX{1}\anygenX{2}\hdots\anygenX{\anylen}$ and $\anymatPBTX{\orderind}{1}{\Sexchanged(\indi)}\anyrest=\anygenX{\anylen+1}\anygenX{\anylen+2}\hdots\anygenX{\anylen+\extralen}$. 
Since by Proposition~\ref{proposition:chains} the chain indices $(\chainstarts,\chainends)$ of $\Sbvfull{\anymatA}{\orderind-2}{\indj}{1}$ satisfy $\chainstartsX{\posind}=\posind$ and $\chainendsX{\posind}=\posind+1$ for any $\posind\in \SYnumbers{\orderind-3}$ there is exactly  one $\chainstartsX{\orderind-2}\in\Sintegersp$ such that $\chainstartsX{\orderind-2}=1$ if $\orderind=3$, such that $1<\chainstartsX{\orderind-2}\leq \anylen$ if $\orderind=4$ and  such that $\chainendsX{\orderind-4}<\chainstartsX{\orderind-2}\leq \anylen$ if $5\leq\orderind$, namely, $\chainstartsX{\orderind-2}=\anylen$. Because any element of $\thecore$ has degree $2$ by Proposition~\ref{proposition:core} the unique $\chainendsX{\orderind-2}\in\Sintegersp$ with $\anylen<\chainendsX{\orderind-2}\leq \anylen+\extralen$ and $\anygenX{\chainstartsX{\orderind-2}}\anygenX{\chainstartsX{\orderind-2}+1}\hdots \anygenX{\chainendsX{\orderind-2}}\in \thecore$ is $\chainendsX{\orderind-2}=\chainstartsX{\orderind-2}+1=\anylen+1$. In consequence, $\thereconcatenator{\orderind-2}(\Sbvfull{\anymatA}{\orderind-2}{\indj}{1},\anymatPBTX{\orderind}{1}{\Sexchanged(\indi)}\anyrest)=(\anygenX{1}\anygenX{2}\hdots\anygenX{\chainendsX{\orderind-2}},\anygenX{\chainendsX{\orderind-2}+1}\anygenX{\chainendsX{\orderind-2}+2}\hdots\anygenX{\anylen+\extralen})$ is given by $(\anygenX{1}\anygenX{2}\hdots\anygenX{\anylen+1},\anygenX{\anylen+2}\anygenX{\anylen+3}\hdots\anygenX{\anylen+\extralen})$. Evidently, $\anygenX{\anylen+2}\anygenX{\anylen+3}\hdots\anygenX{\anylen+\extralen}$ is equal to $\anyrest$. And \eqref{eq:base_recursion_02} in Lemma~\ref{lemma:base_recursion} implies $\anygenX{1}\anygenX{2}\hdots\anygenX{\anylen+1}=\Sbvfull{\anymatA}{\orderind-1}{\indj}{\indi}$. Hence, the claim is true.
\end{proof}
}

\subsection{Zeroth and first order}
In addition to the assumptions from the previous two sections, for any $\orderind\in\Sintegersnn$ let  $\thechainsmodule{\orderind}=\thefield\thechains{\orderind}\Smonoidalproduct\thealgebra$ be the free right $\thealgebra$-module over the set $\thechains{\orderind}$ of $\orderind$\-/chains and $\thechainsorder{\orderind}$ its order, tips with respect to which are indicated by $\thechainstip{\orderind}$,  from Section~\ref{section:anick-chain_modules} and let $\thedifferential{\orderind}$ and $\thesplitting{\orderind}$ be the order-$\orderind$ differential and splitting, respectively, of the Anick resolution from Section~\ref{section:anick-anick_resolution}.
\par 
Since the set $\thechains{0}$ of $0$\-/chains always has only one element, the unit $\theone$ of $\thefreealg$, it is common to identify the right $\thealgebra$-modules $\thefield\thechains{0}\Smonoidalproduct\thealgebra$ and $\thealgebra=\thechainsmodule{0}$. This practice is adopted in the following. The differential of order $0$  does not depend on $\thegens$, $\theideal$ or $\theorder$ either. It is always given by the map $\thedifferential{0}$ from the \hyperref[main-result]{Main theorem}.
\begin{proposition}
  \label{proposition:zeroth_order}
  $\thedifferential{0}(\theone\Smonoidalproduct\theone)=1$.
\end{proposition}
\par
The first-order differential of any Anick resolution is completely determined by the set $\thegens$ of generators and the augmentation $\thecounit$. Taking into account the aforementioned identification it is precisely the mapping $\thedifferential{1}$  from \hyperref[main-result]{Main theorem}.
\begin{proposition}
  \label{proposition:first_order}
For any $\anymat\in\thematrices$ and $(\indj,\indi)\in\SYnumbers{\thedim}^{\Ssetmonoidalproduct 2}$,
  \begin{IEEEeqnarray*}{rCl}
    \thedifferential{1}(\Sbv{\anymatA}{\indj}{\indi}\Smonoidalproduct\theone)=\anymatAX{\indj}{\Sexchanged(\indi)}-\Skronecker{\indj}{\Sexchanged(\indi)}\theone+\theideal.
  \end{IEEEeqnarray*}
\end{proposition}

Hence, it is only the orders two and up which need computing. That is carried out over the next three sections.
\par
  {
    \newcommand{\anypoly}{p}
  Beware that, in order to increase legibility, for any $\anypoly\in\thefreealg$ the symbol $\anypoly$ will from now on also be used for the element $\anypoly+\theideal$ of $\thealgebra$. 
}  
\subsection{Second order}
{
  \newcommand{\anylen}{m}  
  \newcommand{\anygen}{e}
  \newcommand{\anygenX}[1]{\anygen_{#1}}
  Computing $\thedifferential{2}$ requires determining (some of) the values of the splitting $\thesplitting{0}$. Note that, taking into acccount the identification $\thefield\thechains{0}\Smonoidalproduct\thealgebra\cong\thealgebra$, the order $\thechainsorder{0}$ is simply $\theorder$ and the combinatorial splitting $\thereconcatenator{0}$ becomes the mapping $\Sfromto{\themonoid\backslash(\{\theone\}\cup \StipO\theideal)}{\thechains{1}\Ssetmonoidalproduct \themonoid}$ which  for any $\anylen\in\Sintegersp$ and any $\anygen\in\thegens^{\Ssetmonoidalproduct \anylen}$ with $\anygenX{1}\anygenX{2}\hdots\anygenX{\anylen}\in\themonoid\backslash \StipO\theideal$  satisfies $\thereconcatenator{0}(\anygenX{1}\anygenX{2}\hdots\anygenX{\anylen})=(\anygenX{1},\anygenX{2}\hdots\anygenX{\anylen})$.
}
  \par

  For computing $\thedifferential{2}$ it is convenient to distinguish two cases. 
  \subsubsection{Generic case}
  First,  $\thedifferential{2}(\Sbv{\anymat}{\indj}{\indi}\Smonoidalproduct\theone)$ is computed for  any $\anymat\in\thematrices$ and any $(\indj,\indi)\in \SYnumbers{\thedim}^{\Ssetmonoidalproduct 2}$ with $(\indj, \indi)\neq (\thedim,\thedim)$.
 That takes $\thedim+1$ steps.
  \par
{
  \newcommand{\firstnumber}{s}
  \newcommand{\secondnumber}{t}
  For any $\{\firstnumber,\secondnumber\}\subseteq \Sintegers$ let $\Szetafunction{\firstnumber}{\secondnumber}\in\thefield$ be $1$ if $\firstnumber\leq \secondnumber$ and $0$ otherwise.
  }
  \begin{lemma}
    \label{lemma:second_order_generic_case_helper}
    For any $\recvar\in\SYnumbers{\thedim}$ and any $(\indj,\indi)\in \SYnumbers{\thedim}^{\Ssetmonoidalproduct 2}\backslash \{(\thedim,\thedim)\}$,
    \begin{IEEEeqnarray*}{rCl}
\thedifferential{2}(\Sbv{\anymatA}{\indj}{\indi}\Smonoidalproduct\theone)
&=&\textstyle\sum_{\inds=1}^{\recvar}\Sbv{\anymatA}{\indj}{\inds}\Smonoidalproduct\anymatBTX{\inds}{\Sexchanged(\indi)}\IEEEyesnumber
      \label{eq:second_order_generic_case_helper_0}\\
      &&\textstyle{}+\thesplitting{0}\big(\sum_{\inds=\recvar+1}^{\thedim}\anymatAX{\indj}{\Sexchanged(\inds)}\anymatBTX{\inds}{\Sexchanged(\indi)}+\Szetafunction{\Sexchanged(\indj)}{\recvar}\Saction\anymatBTX{\Sexchanged(\indj)}{\Sexchanged(\indi)}-\Skronecker{\indj}{\indi}\Saction\theone\big).
    \end{IEEEeqnarray*}
  \end{lemma}
  \begin{proof}
    \newcommand{\dummyindex}{k}
    The claim is proved by induction over $\recvar$. It is convenient, though, to first note that for any $\dummyindex\in\SYnumbers{\thedim}$ by Proposition~\ref{proposition:first_order}
    \begin{IEEEeqnarray*}{rCl}
      -\thedifferential{1}(\Sbv{\anymat}{\indj}{\dummyindex}\Smonoidalproduct \anymatBTX{\dummyindex}{\Sexchanged(\indi)})&=&     - \thedifferential{1}(\Sbv{\anymat}{\indj}{\dummyindex}\Smonoidalproduct \theone)\anymatBTX{\dummyindex}{\Sexchanged(\indi)}\\
      &=&-(\anymatAX{\indj}{\Sexchanged(\dummyindex)}-\Skronecker{\indj}{\Sexchanged(\dummyindex)}\theone)\Saction\anymatBTX{\dummyindex}{\Sexchanged(\indi)}\\
      &=&-\anymatAX{\indj}{\Sexchanged(\dummyindex)}\anymatBTX{\dummyindex}{\Sexchanged(\indi)}+\Skronecker{\Sexchanged(\indj)}{\dummyindex}\Saction\anymatBTX{\Sexchanged(\indj)}{\Sexchanged(\indi)}.      \IEEEyesnumber\label{eq:second_order_generic_case_helper_1}
    \end{IEEEeqnarray*}
    \par
    \emph{Induction base.} For $\recvar=1$ the claim is that
    \begin{IEEEeqnarray*}{rCl}
      \thedifferential{2}(\Sbv{\anymatA}{\indj}{\indi}\Smonoidalproduct\theone)
      &=&\textstyle\Sbv{\anymat}{\indj}{1}\Smonoidalproduct\anymatBTX{1}{\Sexchanged(\indi)}\IEEEyesnumber      \label{eq:second_order_generic_case_helper_2}\\      &&\textstyle{}+\thesplitting{0}\big(\sum_{\inds=2}^{\thedim}\anymatAX{\indj}{\Sexchanged(\inds)}\anymatBTX{\inds}{\Sexchanged(\indi)}+\Skronecker{\Sexchanged(\indj)}{1}\Saction\anymatBTX{\Sexchanged(\indj)}{\Sexchanged(\indi)}-\Skronecker{\indj}{\indi}\Saction\theone)
    \end{IEEEeqnarray*}
    because $\Szetafunction{\Sexchanged(\indj)}{1}=\Skronecker{\Sexchanged(\indj)}{1}$ by $1\leq \Sexchanged(\indj)$.\par
    \par
        Because    by Proposition~\ref{proposition:combinatorial_differential}
    \begin{IEEEeqnarray*}{rCl}
      \thedeconcatenator{2}(\Sbvfull{\anymat}{2}{\indj}{\indi})&=&(\Sbvfull{\anymatA}{1}{\indj}{1},\anymatBTX{1}{\Sexchanged(\indi)})
    \end{IEEEeqnarray*}
    the definition of $\thedifferential{2}$ implies
    \begin{IEEEeqnarray*}{rCl}
      \thedifferential{2}(\Sbv{\anymat}{\indj}{\indi}\Smonoidalproduct\theone)
&=&\textstyle\Sbv{\anymatA}{\indj}{1}\Smonoidalproduct\anymatBTX{1}{\Sexchanged(\indi)}+\thesplitting{0}(-\thedifferential{1}(\Sbv{\anymatA}{\indj}{1}\Smonoidalproduct\anymatBTX{1}{\Sexchanged(\indi)}))\IEEEyesnumber\label{eq:second_order_generic_case_helper_2a}. 
\end{IEEEeqnarray*}
Hence, it is enough to show that the vector \eqref{eq:second_order_generic_case_helper_2a} is the same as  the  argument of $\thesplitting{0}$ on the right-hand side of \eqref{eq:second_order_generic_case_helper_2}.
\par
And, indeed, specializing $\dummyindex$ to $1$  in \eqref{eq:second_order_generic_case_helper_1} gives
\begin{IEEEeqnarray*}{rCl}  
      -\thedifferential{1}(\Sbv{\anymat}{\indj}{1}\Smonoidalproduct \anymatBTX{1}{\Sexchanged(\indi)}) &=&-\anymatAX{\indj}{\thedim}\anymatBTX{1}{\Sexchanged(\indi)}+\Skronecker{\Sexchanged(\indj)}{1}\Saction\anymatBTX{\Sexchanged(\indj)}{\Sexchanged(\indi)}.\IEEEyesnumber\label{eq:second_order_generic_case_helper_3}
    \end{IEEEeqnarray*}
    Given that $-\anymatAX{\indj}{\thedim}\anymatBTX{1}{\Sexchanged(\indi)}=\sum_{\inds=2}^{\thedim}\anymatAX{\indj}{\Sexchanged(\inds)}\anymatBTX{\inds}{\Sexchanged(\indi)}-\Skronecker{\indj}{\indi}\theone$, this thus proves the claim for $\recvar=1$.
\par
\emph{Induction step.} Let \eqref{eq:second_order_generic_case_helper_0} be true for some $\recvar$ with $ \recvar\leq \thedim-1$. Then it also holds for $\recvar+1$ in place of $\recvar$ for the following reasons. Because $(\indj,\indi)\neq (\thedim,\thedim)$ and $1\neq \recvar+1$ the vector $\anymatAX{\indj}{\Sexchanged(\recvar+1)}\anymatBTX{\recvar+1}{\Sexchanged(\indi)}$ is normal by Lem\-ma~\ref{characterization_of_reduced_terms_of_order_two}. It is the tip of the argument of $\thesplitting{0}$ in \eqref{eq:second_order_generic_case_helper_0}. Indeed, any other term in the same grouped sum is less because already $\anymatAX{\indj}{\Sexchanged(\inds)}<\anymatAX{\indj}{\Sexchanged(\recvar+1)}$ for any $\inds\in\SYnumbers{\thedim}$ with $\recvar+1<\inds$, i.e., $\Sexchanged(\inds)<\Sexchanged(\recvar+1)$, by the addendum to Lem\-ma~\ref{lemma:monomial_order_reformulation}. And all remaining terms have lower degree.
\par
Because by definition
\begin{IEEEeqnarray*}{rCl}
\textstyle\thereconcatenator{0}(\anymatAX{\indj}{\Sexchanged(\recvar+1)}\anymatBTX{\recvar+1}{\Sexchanged(\indi)})=  (\Sbvfull{\anymat}{1}{\indj}{\recvar+1},\anymatBTX{\recvar+1}{\Sexchanged(\indi)})
\end{IEEEeqnarray*}
the equality in \eqref{eq:second_order_generic_case_helper_0} hence remains true if  $\Sbv{\anymat}{\indj}{\recvar+1}\Smonoidalproduct\anymatBTX{\recvar+1}{\Sexchanged(\indi)}$ is added to the right-hand side of \eqref{eq:second_order_generic_case_helper_0} outside of $\thesplitting{0}$ and 
$-\thedifferential{1}(\Sbv{\anymat}{\indj}{\recvar+1}\Smonoidalproduct\anymatBTX{\recvar+1}{\Sexchanged(\indi)})$ to the inside.
\par
That evidently changes the outside in the same way as replacing $\recvar$ by $\recvar+1$ would.
At the same time, by \eqref{eq:second_order_generic_case_helper_1} for $\dummyindex=\recvar+1$, 
    \begin{IEEEeqnarray*}{rCl}
      -\thedifferential{1}(\Sbv{\anymat}{\indj}{\recvar+1}\Smonoidalproduct \anymatBTX{\recvar+1}{\Sexchanged(\indi)}) &=&-\anymatAX{\indj}{\Sexchanged(\recvar+1)}\anymatBTX{\recvar+1}{\Sexchanged(\indi)}+\Skronecker{\Sexchanged(\indj)}{\recvar+1}\Saction\anymatBTX{\Sexchanged(\indj)}{\Sexchanged(\indi)},      \IEEEeqnarraynumspace\IEEEyesnumber\label{eq:second_order_generic_case_helper_4}
    \end{IEEEeqnarray*}
    Keeping in mind that $\Szetafunction{\Sexchanged(\indj)}{\recvar}+\Skronecker{\Sexchanged(\indj)}{\recvar+1}=\Szetafunction{\Sexchanged(\indj)}{\recvar+1}$ it is now clear that adding the vector in \eqref{eq:second_order_generic_case_helper_4} to the argument of $\thesplitting{0}$ in \eqref{eq:second_order_generic_case_helper_0} has the same effect as replacing $\recvar$ by $\recvar+1$ there. Thus, the claim is true for all $\recvar\in\SYnumbers{\thedim}$.
  \end{proof}

  \begin{proposition}
    \label{proposition:second_order_generic_case}
    For any $\anymat\in\thematrices$ and any $(\indj,\indi)\in \SYnumbers{\thedim}^{\Ssetmonoidalproduct 2}\backslash \{(\thedim,\thedim)\}$,
      \begin{IEEEeqnarray*}{rCl}    \thedifferential{2}(\Sbv{\anymat}{\indj}{\indi}\Smonoidalproduct\theone)=\textstyle\sum_{\inds=1}^\thedim\Sbv{\anymat}{\indj}{\inds}\Smonoidalproduct\anymatBTX{\inds}{\Sexchanged(\indi)}+\Sbv{\anymatBT}{\Sexchanged(\indj)}{\indi}\Smonoidalproduct\theone.\IEEEyesnumber\label{eq:second_order_generic_case_0}
  \end{IEEEeqnarray*}
\end{proposition}
\begin{proof}
In the special  $\recvar=\thedim$ the identity \eqref{eq:second_order_generic_case_helper_0} in Lem\-ma~\ref{lemma:second_order_generic_case_helper} reads
  \begin{IEEEeqnarray*}{rCl}
\thedifferential{2}(\Sbv{\anymat}{\indj}{\indi}\Smonoidalproduct\theone)
&=&\textstyle\sum_{\inds=1}^\thedim\Sbv{\anymat}{\indj}{\inds}\Smonoidalproduct\anymatBTX{\inds}{\Sexchanged(\indi)}+\thesplitting{0}(\anymatBTX{\Sexchanged(\indj)}{\Sexchanged(\indi)}-\Skronecker{\indj}{\indi}\Saction\theone).\IEEEyesnumber\label{eq:second_order_generic_case_1}
\end{IEEEeqnarray*}
because $\Szetafunction{\Sexchanged(\indj)}{\thedim}=1$.
\par
Since $\anymatBTX{\Sexchanged(\indj)}{\Sexchanged(\indi)}$ is normal and greater than $\theone$ it is the tip of the argument of $\thesplitting{0}$.
Because it was defined that
\begin{IEEEeqnarray*}{rCl}
\textstyle\thereconcatenator{0}(\anymatBTX{\Sexchanged(\indj)}{\Sexchanged(\indi)})=(\Sbvfull{\anymatBT}{1}{\Sexchanged
  (\indj)}{\indi} ,\theone)
\end{IEEEeqnarray*}
 the veracity of \eqref{eq:second_order_generic_case_1} is hence not affected by adding on the right-hand side $\Sbv{\anymatBT}{\Sexchanged
  (\indj)}{\indi}\Smonoidalproduct\theone$ to the outside of $\thesplitting{0}$ and $-\thedifferential{1}(\Sbv{\anymatBT}{\Sexchanged
  (\indj)}{\indi}\Smonoidalproduct\theone)$ to the inside.
\par
Since by Proposition~\ref{proposition:first_order}
\begin{IEEEeqnarray*}{rCl}
 -\thedifferential{1}(\Sbv{\anymatBT}{\Sexchanged
  (\indj)}{\indi}\Smonoidalproduct\theone)=-\anymatBTX{\Sexchanged(\indj)}{\Sexchanged(\indi)}+\Skronecker{\indj}{\indi}\Saction\theone 
\end{IEEEeqnarray*}
 the outcome of that addition is precisely \eqref{eq:second_order_generic_case_0}.
\end{proof}
  
\subsubsection{Special case}
It remains to compute $\thedifferential{2}(\Sbv{\anymat}{\thedim}{\thedim}\Smonoidalproduct\theone)$ for any $\anymat\in\thematrices$. Actually, since $\Sbv{\anymat}{\thedim}{\thedim}=\anymatAX{\thedim}{\thedim}\anymatBTX{1}{1}$ is invariant under replacing $\anymat$ with $\anymat\StransposeP$ it is enough to consider $\anymat\in \{\theuni,\theuni\SskewstarP\}$, which  greatly simplifies the proof. This time exactly $\thedim^2$ steps are necessary to solve the recursion.
\begin{lemma}
  \label{lemma:second_order_special_case_first_step}
  For any $\anymat\in \{\theuni,\theuni\SskewstarP\}$ and any $\{\outerrecvar,\recvar\}\subseteq\SYnumbers{\thedim}$ with $2\leq \outerrecvar$,
  \begin{IEEEeqnarray*}{rCl}
    \IEEEeqnarraymulticol{3}{l}{
      \thedifferential{2}(\Sbv{\anymat}{\thedim}{\thedim}\Smonoidalproduct \theone)
    }\\    \hspace{1em}&=&\textstyle\Sbv{\anymat}{\thedim}{1}\Smonoidalproduct\anymatBTX{1}{1}-\sum_{\inds=2}^{\thedim}\sum_{\indt=\outerrecvar}^{\thedim-1}\Sbv{\anymat}{\indt}{\inds}\Smonoidalproduct\anymatBTX{\inds}{\Sexchanged(\indt)}    \IEEEyesnumber
    \label{eq:second_order_special_case_first_step_0}\\    &&\textstyle\hspace{13.25em}{}-\sum_{\inds=2}^{\recvar}\Sbv{\anymat}{\outerrecvar-1}{\inds}\Smonoidalproduct\anymatBTX{\inds}{\Sexchanged(\outerrecvar-1)}\\    &&\textstyle{}+\thesplitting{0}\big({-}\sum_{\inds=\recvar+1}^{\thedim}\anymatX{\outerrecvar-1}{\Sexchanged(\inds)}\anymatBTX{\inds}{\Sexchanged(\outerrecvar-1)}-\sum_{\inds=2}^{\thedim}\sum_{\indt=1}^{\outerrecvar-2}\anymatAX{\indt}{\Sexchanged(\inds)}\anymatBTX{\inds}{\Sexchanged(\indt)}\\    &&\textstyle\hspace{2.1875em}{}-\sum_{\inds=\outerrecvar}^{\thedim-1}\anymatBTX{\Sexchanged(\inds)}{\Sexchanged(\inds)}-\Szetafunction{\Sexchanged(\outerrecvar-1)}{\recvar}\Saction\anymatBTX{\Sexchanged(\outerrecvar-1)}{\Sexchanged(\outerrecvar-1)}+\anymatBTX{1}{1}\\
    &&\textstyle\hspace{18em}{}+(\thedim-2)\Saction\theone\big).
  \end{IEEEeqnarray*}  
\end{lemma}
\begin{proof}
  The claim is proved by nested induction, an outer one over $\Sexchanged(\outerrecvar)$ and an inner one over $\recvar$.
  \par
  \emph{Part~1: Outer induction base.} As base case $\outerrecvar=\thedim$ what needs to be shown is that for any $\recvar\in\SYnumbers{\thedim}$,
  \begin{IEEEeqnarray*}{rCl}
    \IEEEeqnarraymulticol{3}{l}{
      \thedifferential{2}(\Sbv{\anymat}{\thedim}{\thedim}\Smonoidalproduct \theone)
    }\\
    \hspace{1em}&=&\textstyle\Sbv{\anymat}{\thedim}{1}\Smonoidalproduct\anymatBTX{1}{1}-\sum_{\inds=2}^{\recvar}\Sbv{\anymat}{\thedim-1}{\inds}\Smonoidalproduct\anymatBTX{\inds}{\Sexchanged(\thedim-1)} \IEEEeqnarraynumspace\IEEEyesnumber
    \label{eq:second_order_special_case_first_step_1}\\
&&\textstyle{}+\thesplitting{0}\big({-}\sum_{\inds=\recvar+1}^{\thedim}\anymatX{\thedim-1}{\Sexchanged(\inds)}\anymatBTX{\inds}{\Sexchanged(\thedim-1)}-\sum_{\inds=2}^{\thedim}\sum_{\indt=1}^{\thedim-2}\anymatAX{\indt}{\Sexchanged(\inds)}\anymatBTX{\inds}{\Sexchanged(\indt)}\\    &&\textstyle\hspace{2.1875em}{}-\Szetafunction{\Sexchanged(\thedim-1)}{\recvar}\Saction\anymatBTX{\Sexchanged(\thedim-1)}{\Sexchanged(\thedim-1)}+\anymatBTX{1}{1}+(\thedim-2)\Saction\theone\big).
  \end{IEEEeqnarray*}    
  \par
  \emph{Part~1.1: Inner induction base.} In particular, as the start $\recvar=1$ it is claimed that
  \begin{IEEEeqnarray*}{rCl}
    \IEEEeqnarraymulticol{3}{l}{
      \thedifferential{2}(\Sbv{\anymat}{\thedim}{\thedim}\Smonoidalproduct \theone)
    }\\    \hspace{.5em}&=&\textstyle\Sbv{\anymat}{\thedim}{1}\Smonoidalproduct\anymatBTX{1}{1}+\thesplitting{0}\big({-}\sum_{\inds=2}^{\thedim}\sum_{\indt=1}^{\thedim-1}\anymatAX{\indt}{\Sexchanged(\inds)}\anymatBTX{\inds}{\Sexchanged(\indt)}+\anymatBTX{1}{1}+(\thedim-2)\Saction\theone \big).
  \end{IEEEeqnarray*}      
  where the two grouped sums in the argument of $\thesplitting{0}$ appearing in \eqref{eq:second_order_special_case_first_step_1} have been comined into one and  $\Szetafunction{\Sexchanged(\thedim-1)}{1}=0$ has been used. And, indeed, because by Proposition~\ref{proposition:combinatorial_differential}
  \begin{IEEEeqnarray*}{rCl}
\thedeconcatenator{2}( \Sbvfull{\anymat}{2}{\thedim}{\thedim})=(\Sbvfull{\anymat}{1}{\thedim}{1}, \anymatBTX{1}{1})
  \end{IEEEeqnarray*}
  and because by Proposition~\ref{proposition:first_order}
  \begin{IEEEeqnarray*}{rCl}
    -\thedifferential{1}(\Sbv{\anymat}{\thedim}{1}\Smonoidalproduct \theone)\anymatBTX{1}{1}&=&-(\anymatAX{\thedim}{\thedim}-\theone)\anymatBTX{1}{1}\\
    &=&-\anymatAX{\thedim}{\thedim}\anymatBTX{1}{1}+\anymatBTX{1}{1}\\
    &=&\textstyle-(\sum_{\inds=2}^{\thedim}\sum_{\indt=1}^{\thedim-1}\anymatAX{\indt}{\Sexchanged(\inds)}\anymatBTX{\inds}{\Sexchanged(\indt)} -(\thedim-2)\Saction\theone)+\anymatBTX{1}{1}
  \end{IEEEeqnarray*}
  that is what the definition
  \begin{IEEEeqnarray*}{rCl}
\thedifferential{2}(\Sbv{\anymat}{\thedim}{\thedim}\Smonoidalproduct \theone)=\Sbv{\anymat}{\thedim}{1}\Smonoidalproduct \anymatBTX{1}{1}+\thesplitting{0}(    -\thedifferential{1}(\Sbv{\anymat}{\thedim}{1}\Smonoidalproduct \anymatBTX{1}{}))
  \end{IEEEeqnarray*}
  implies.
  \par
  \emph{Part~1.2: Inner induction step.}    Now, suppose \eqref{eq:second_order_special_case_first_step_1} holds for some $\recvar\in\SYnumbers{\thedim}$ with $ \recvar\leq \thedim-1$. By the ensuing argument \eqref{eq:second_order_special_case_first_step_1} remains true if $\recvar$ is replaced by $\recvar+1$.
  \par
  By Lem\-ma~\ref{characterization_of_reduced_terms_of_order_two} the vector $\anymatAX{\thedim-1}{\Sexchanged(\recvar+1)}\anymatBTX{\recvar+1}{\Sexchanged(\thedim-1)}$ is normal. It is moreover the tip of the argument of $\thesplitting{0}$ \eqref{eq:second_order_special_case_first_step_1} as explained hereafter. For degree reasons the only monomials which could possibly be greater or equal appear in the third line of \eqref{eq:second_order_special_case_first_step_1}. More precisely, the other summands in this line are precisely the $\anymatAX{\indt}{\Sexchanged(\inds)}\anymatBTX{\inds}{\Sexchanged(\indt)}$ for $\{\indt,\inds\}\subseteq \SYnumbers{\thedim}$ with  $\inds\neq 1$ such that either $\indt<\thedim-1$ or both $\indt=\thedim-1$ and $\recvar+1<\inds$, which is to say with $(\indt,\Sexchanged(\inds))\lexorderRstrict(\thedim-1,\Sexchanged(\recvar+1))$. Because $\anymat\in\{\theuni,\theuni\SskewstarP\}$ was assumed, Lem\-ma~\ref{lemma:monomial_order_reformulation} proves that already $\anymatAX{\indt}{\Sexchanged(\inds)}\theorderRstrict\anymatAX{\thedim-1}{\Sexchanged(\recvar+1)}$ for those $\inds$ and $\indt$, which then  ensures $\anymatAX{\indt}{\Sexchanged(\inds)}\anymatBTX{\inds}{\Sexchanged(\indt)}\theorderRstrict\anymatAX{\thedim-1}{\Sexchanged(\recvar+1)}\anymatBTX{\recvar+1}{\Sexchanged(\thedim-1)}$.  
  \par
  Because by definition
  \begin{IEEEeqnarray*}{rCl}
\textstyle\thereconcatenator{0}(\anymatAX{\thedim-1}{\Sexchanged(\recvar+1)}\anymatBTX{\recvar+1}{\Sexchanged(\thedim-1)})=(\Sbvfull{\anymat}{1}{\thedim-1}{\recvar+1},\anymatBTX{\recvar+1}{\Sexchanged(\thedim-1)})
  \end{IEEEeqnarray*}
  the identity in \eqref{eq:second_order_special_case_first_step_1} is hence preserved if  $-\Sbv{\anymat}{\thedim-1}{\recvar+1}\Smonoidalproduct\anymatBTX{\recvar+1}{\Sexchanged(\thedim-1)}$ is added outside of $\thesplitting{0}$ and $\thedifferential{1}(\Sbv{\anymat}{\thedim-1}{\recvar+1}\Smonoidalproduct \anymatBTX{\recvar+1}{\Sexchanged(\thedim-1)})$ inside on the right-hand side.
  \par
  This addition evidently transforms  the exterior of $\thesplitting{0}$ in \eqref{eq:second_order_special_case_first_step_1} in the same way as replacing $\recvar$ by $\recvar+1$ would. At the same time, adding
  \begin{IEEEeqnarray*}{rCl}
    \IEEEeqnarraymulticol{3}{l}{ \thedifferential{1}(\Sbv{\anymat}{\thedim-1}{\recvar+1}\Smonoidalproduct\anymatBTX{\recvar+1}{\Sexchanged(\thedim-1)})
    }\\
    \hspace{3em}&=&(\anymatAX{\thedim-1}{\Sexchanged(\recvar+1)}-\Skronecker{\thedim-1}{\Sexchanged(\recvar+1)}\Saction \theone)\anymatBTX{\recvar+1}{\Sexchanged(\thedim-1)}\\    &=&\anymatAX{\thedim-1}{\Sexchanged(\recvar+1)}\anymatBTX{\recvar+1}{\Sexchanged(\thedim-1)}-\Skronecker{\Sexchanged(\thedim-1)}{\recvar+1} \Saction\anymatBTX{\Sexchanged(\thedim-1)}{\Sexchanged(\thedim-1)}.
  \end{IEEEeqnarray*}
to the argument of $\thesplitting{0}$ in \eqref{eq:second_order_special_case_first_step_1} has the same effect there because $\Szetafunction{\Sexchanged(\thedim-1)}{\recvar}+\Skronecker{\Sexchanged(\thedim-1)}{\recvar+1}=\Szetafunction{\Sexchanged(\outerrecvar-1)}{\recvar+1}$. In conclusion, the base case for the outer induction has been established.
  \par
  \emph{Part~2: Outer induction step.} Next, suppose \eqref{eq:second_order_special_case_first_step_0} holds for some $\outerrecvar\in\SYnumbers{\thedim}$ with $3\leq \outerrecvar$ and all $\recvar\in \SYnumbers{\thedim}$. It remains to  show that \eqref{eq:second_order_special_case_first_step_0} is true with $\outerrecvar-1$ in place of $\outerrecvar$, i.e., that for any $\recvar\in\SYnumbers{\thedim}$,
  \begin{IEEEeqnarray*}{rCl}
    \IEEEeqnarraymulticol{3}{l}{
      \thedifferential{2}(\Sbv{\anymat}{\thedim}{\thedim}\Smonoidalproduct \theone)
    }\\    \hspace{1em}&=&\textstyle\Sbv{\anymat}{\thedim}{1}\Smonoidalproduct\anymatBTX{1}{1}-\sum_{\inds=2}^{\thedim}\sum_{\indt=\outerrecvar-1}^{\thedim-1}\Sbv{\anymat}{\indt}{\inds}\Smonoidalproduct\anymatBTX{\inds}{\Sexchanged(\indt)}    \IEEEyesnumber
    \label{eq:second_order_special_case_first_step_2}\\
    &&\textstyle\hspace{13em}{}-\sum_{\inds=2}^{\recvar}\Sbv{\anymat}{\outerrecvar-2}{\inds}\Smonoidalproduct\anymatBTX{\inds}{\Sexchanged(\outerrecvar-2)}\\    &&\textstyle{}+\thesplitting{0}\big({-}\sum_{\inds=\recvar+1}^{\thedim}\anymatX{\outerrecvar-2}{\Sexchanged(\inds)}\anymatBTX{\inds}{\Sexchanged(\outerrecvar-2)}-\sum_{\inds=2}^{\thedim}\sum_{\indt=1}^{\outerrecvar-3}\anymatAX{\indt}{\Sexchanged(\inds)}\anymatBTX{\inds}{\Sexchanged(\indt)}\\    &&\textstyle\hspace{2.1875em}{}-\sum_{\inds=\outerrecvar-1}^{\thedim-1}\anymatBTX{\Sexchanged(\inds)}{\Sexchanged(\inds)}-\Szetafunction{\Sexchanged(\outerrecvar-2)}{\recvar}\Saction\anymatBTX{\Sexchanged(\outerrecvar-2)}{\Sexchanged(\outerrecvar-2)}+\anymatBTX{1}{1}\\
    &&\textstyle\hspace{18em}{}+(\thedim-2)\theone\big).
  \end{IEEEeqnarray*}    
  \par
  \emph{Part~2.1: Inner induction base.} As the base case $\recvar=1$ this means proving that
  \begin{IEEEeqnarray*}{rCl}
\thedifferential{2}(\Sbv{\anymat}{\thedim}{\thedim}\Smonoidalproduct \theone)&=&\textstyle\Sbv{\anymat}{\thedim}{1}\Smonoidalproduct\anymatBTX{1}{1}-\sum_{\inds=2}^{\thedim}\sum_{\indt=\outerrecvar-1}^{\thedim-1}\Sbv{\anymat}{\indt}{\inds}\Smonoidalproduct\anymatBTX{\inds}{\Sexchanged(\indt)}    \IEEEyesnumber     \label{eq:second_order_special_case_first_step_3}\\
&&\textstyle{}+\thesplitting{0}\big({-}\sum_{\inds=2}^{\thedim}\sum_{\indt=1}^{\outerrecvar-2}\anymatAX{\indt}{\Sexchanged(\inds)}\anymatBTX{\inds}{\Sexchanged(\indt)}\\
&&\textstyle\hspace{2.25em}{}-\sum_{\inds=\outerrecvar-1}^{\thedim-1}\anymatBTX{\Sexchanged(\inds)}{\Sexchanged(\inds)}+\anymatBTX{1}{1}+(\thedim-2)\theone\big).
  \end{IEEEeqnarray*}
  where, again, two grouped sums in the argument of $\thesplitting{0}$ in \eqref{eq:second_order_special_case_first_step_2} have been combined into one and where it has been used that  $\outerrecvar\leq\thedim$ implies $\outerrecvar-2\leq \thedim-2$ and thus $3\leq \Sexchanged(\outerrecvar-2)$, which is why  $\Szetafunction{\Sexchanged(\outerrecvar-2)}{1}=0$.
  \par
  And that is precisely what the special case $\recvar=\thedim$ of the induction hypothesis \eqref{eq:second_order_special_case_first_step_0} says: 
The two grouped sums on the exterior of $\thesplitting{0}$ can be combined into one. The first grouped sum in the fourth line vanishes. And because, of course,  $\Szetafunction{\Sexchanged(\outerrecvar-1)}{\thedim}=1$ the terms $-\sum_{\inds=\outerrecvar}^{\thedim-1}\anymatBTX{\Sexchanged(\inds)}{\Sexchanged(\inds)}-\Szetafunction{\Sexchanged(\outerrecvar-1)}{\thedim}\Saction\anymatBTX{\Sexchanged(\outerrecvar-1)}{\Sexchanged(\outerrecvar-1)}$ equal $-\sum_{\inds=\outerrecvar-1}^{\thedim-1}\anymatBTX{\Sexchanged(\inds)}{\Sexchanged(\inds)}$.
  \par
  \emph{Part~2.2: Inner induction step.} Finally, let \eqref{eq:second_order_special_case_first_step_2} be true for some $\recvar\in\SYnumbers{\thedim}$ with $ \recvar\leq\thedim-1$. It follows the proof that \eqref{eq:second_order_special_case_first_step_2} is still true if  $\recvar$ is replaced by $\recvar+1$.
  \par
  The vector   $\anymatAX{\outerrecvar-2}{\Sexchanged(\recvar+1)}\anymatBTX{\recvar+1}{\Sexchanged(\outerrecvar-2)}$ is normal by Lem\-ma~\ref{characterization_of_reduced_terms_of_order_two}. It is  the tip of the argument of $\thesplitting{0}$ in \eqref{eq:second_order_special_case_first_step_2}. Indeed, already by degree it is greater than any monomial appearing in the last two lines of \eqref{eq:second_order_special_case_first_step_2}. The remaining summands inside $\thesplitting{0}$ are exactly all $\anymatAX{\indt}{\Sexchanged(\inds)}\anymatBTX{\inds}{\Sexchanged(\indt)}$ for $\{\inds,\indt\}\subseteq \SYnumbers{\thedim}$ with $\inds\neq 1$ such that either $\indt<\outerrecvar-2$ or both $\indt=\outerrecvar-2$ and $\recvar+1<\inds$. Since $(\indt,\Sexchanged(\inds))\lexorderRstrict(\outerrecvar-2,\Sexchanged(\recvar+1))$ and thus $\anymatAX{\indt}{\Sexchanged(\inds)}\theorderRstrict\anymatAX{\outerrecvar-2}{\Sexchanged(\recvar+1)}$ by Lem\-ma~\ref{lemma:monomial_order_reformulation} also $\anymatAX{\indt}{\Sexchanged(\inds)}\anymatBTX{\inds}{\Sexchanged(\indt)}<\anymatAX{\outerrecvar-2}{\Sexchanged(\recvar+1)}\anymatBTX{\recvar+1}{\Sexchanged(\outerrecvar-2)}$. 
  \par
  Consequently, given that by definition
  \begin{IEEEeqnarray*}{rCl}
\textstyle\thereconcatenator{0}(\anymatAX{\outerrecvar-2}{\Sexchanged(\recvar+1)}\anymatBTX{\recvar+1}{\Sexchanged(\outerrecvar-2)})=(\Sbvfull{\anymat}{1}{\outerrecvar-2}{\recvar+1},\anymatBTX{\recvar+1}{\Sexchanged(\outerrecvar-2)}),
  \end{IEEEeqnarray*}
the equation resulting from \eqref{eq:second_order_special_case_first_step_2} by adding on the right-hand side  $-\Sbv{\anymat}{\outerrecvar-2}{\recvar+1}\Smonoidalproduct\anymatBTX{\recvar+1}{\Sexchanged(\outerrecvar-2)}$ outside $\thesplitting{0}$ and $\thedifferential{1}(-\Sbv{\anymat}{\outerrecvar-2}{\recvar+1}\Smonoidalproduct\anymatBTX{\recvar+1}{\Sexchanged(\outerrecvar-2)})$ inside is true as well.
  \par 
Said addition effectively increases $\recvar$ by $1$ in \eqref{eq:second_order_special_case_first_step_2} outside of $\thesplitting{0}$. Simultaneously, by Proposition~\ref{proposition:first_order}
  \begin{IEEEeqnarray*}{rCl}
    \IEEEeqnarraymulticol{3}{l}{ \thedifferential{1}(\Sbv{\anymat}{\outerrecvar-2}{\recvar+1}\Smonoidalproduct\anymatBTX{\recvar+1}{\Sexchanged(\outerrecvar-2)})
}\\      \hspace{3em}&=&(\anymatAX{\outerrecvar-2}{\Sexchanged(\recvar+1)}-\Skronecker{\outerrecvar-2}{\Sexchanged(\recvar+1)}\Saction\theone )\anymatBTX{\recvar+1}{\Sexchanged(\outerrecvar-2)}\\
    &=&\anymatAX{\outerrecvar-2}{\Sexchanged(\recvar+1)}\anymatBTX{\recvar+1}{\Sexchanged(\outerrecvar-2)}-\Skronecker{\Sexchanged(\outerrecvar-2)}{\recvar+1}\Saction\anymatBTX{\Sexchanged(\outerrecvar-2)}{\Sexchanged(\outerrecvar-2)}\IEEEyesnumber\label{eq:second_order_special_case_first_step_4}
  \end{IEEEeqnarray*}
  and $\Szetafunction{\Sexchanged(\outerrecvar-2)}{\recvar}+\Skronecker{\Sexchanged(\recvar+1)}{\outerrecvar-2}=\Szetafunction{\Sexchanged(\outerrecvar-2)}{\recvar+1}$. An inspection of the argument of $\thesplitting{0}$ in \eqref{eq:second_order_special_case_first_step_2} therefore proves that the addition of the term  \eqref{eq:second_order_special_case_first_step_4} has the same effect as replacing $\recvar$ by $\recvar+1$ there. Thus, \eqref{eq:second_order_special_case_first_step_2} is still true for $\recvar+1$ instead of $\recvar$, which concludes the proof of \eqref{eq:second_order_special_case_first_step_0} for arbitrary $\outerrecvar$ and $\recvar$.
\end{proof}

  \begin{lemma}
    \label{lemma:second_order_special_case_second_step}
  For any $\anymat\in\{\theuni,\theuni\SskewstarP\}$ and any $\recvar\in\SYnumbers{\thedim}$ with $\recvar\leq \thedim-1$,
  \begin{IEEEeqnarray*}{rCl}
    \IEEEeqnarraymulticol{3}{l}{
      \thedifferential{2}(\Sbv{\anymat}{\thedim}{\thedim}\Smonoidalproduct\theone)
      }\\
\hspace{1em}&=&\textstyle\Sbv{\anymatA}{\thedim}{1}\Smonoidalproduct\anymatBTX{1}{1}-\sum_{\inds=2}^\thedim\sum_{\indt=1}^{\thedim-1}\Sbv{\anymat}{\indt}{\inds}\Smonoidalproduct\anymatBTX{\inds}{\Sexchanged(\indt)}-\sum_{\inds=1}^{\recvar}\Sbv{\anymatBT}{\Sexchanged(\inds)}{\inds}\Smonoidalproduct\theone\IEEEeqnarraynumspace
    \IEEEyesnumber
    \label{eq:second_order_special_case_second_step_0}\\ &&\textstyle{}+\thesplitting{0}\big({-}\sum_{\inds=\recvar+1}^{\thedim-1}\anymatBTX{\Sexchanged(\inds)}{\Sexchanged(\inds)}+\anymatBTX{1}{1}+\Sexchanged(\recvar+3)\Saction\theone\big).
  \end{IEEEeqnarray*}
\end{lemma}
\begin{proof}
  Once more the proof goes by induction over $\recvar$.
  \par
  \emph{Induction base.} For $\recvar=1$ what needs to be shown is that
  \begin{IEEEeqnarray*}{rCl}
    \IEEEeqnarraymulticol{3}{l}{
      \thedifferential{2}(\Sbv{\anymat}{\thedim}{\thedim}\Smonoidalproduct\theone)
      }\\
\hspace{1em}&=&\textstyle\Sbv{\anymatA}{\thedim}{1}\Smonoidalproduct\anymatBTX{1}{1}-\sum_{\inds=2}^\thedim\sum_{\indt=1}^{\thedim-1}\Sbv{\anymat}{\indt}{\inds}\Smonoidalproduct\anymatBTX{\inds}{\Sexchanged(\indt)}-\Sbv{\anymatBT}{\thedim}{1}\Smonoidalproduct\theone\IEEEeqnarraynumspace
    \IEEEyesnumber
    \label{eq:second_order_special_case_second_step_1}\\ &&\textstyle{}+\thesplitting{0}\big({-}\sum_{\inds=2}^{\thedim-1}\anymatBTX{\Sexchanged(\inds)}{\Sexchanged(\inds)}+\anymatBTX{1}{1}+(\thedim-3)\Saction\theone\big).
  \end{IEEEeqnarray*}  
\par
In the special case $\outerrecvar=2$ and $\recvar=\thedim$  the grouped sum in the third line of \eqref{eq:second_order_special_case_first_step_0} in  Lem\-ma~\ref{lemma:second_order_special_case_first_step} can be merged with that in the second line. Moreover, all the arguments of $\thesplitting{0}$ in the fourth line vanish. Since, moreover, $\Szetafunction{\Sexchanged(1)}{\thedim}=1$ the resulting identity then reads
  \begin{IEEEeqnarray*}{rCl}
      \thedifferential{2}(\Sbv{\anymat}{\thedim}{\thedim}\Smonoidalproduct \theone)
&=&\textstyle\Sbv{\anymat}{\thedim}{1}\Smonoidalproduct\anymatBTX{1}{1}-\sum_{\inds=2}^{\thedim}\sum_{\indt=1}^{\thedim-1}\Sbv{\anymat}{\indt}{\inds}\Smonoidalproduct\anymatBTX{\inds}{\Sexchanged(\indt)} \IEEEeqnarraynumspace   \IEEEyesnumber
    \label{eq:second_order_special_case_second_step_2}\\ &&\textstyle{}+\thesplitting{0}\big({-}\sum_{\inds=1}^{\thedim-1}\anymatBTX{\Sexchanged(\inds)}{\Sexchanged(\inds)}+\anymatBTX{1}{1}+(\thedim-2)\Saction\theone\big).
  \end{IEEEeqnarray*}
  \par
  Obviously,   $\anymatBTX{\thedim}{\thedim}$ is normal and the tip in the argument of $\thesplitting{0}$ in \eqref{eq:second_order_special_case_second_step_2} by Lem\-ma~\ref{lemma:monomial_order_reformulation}.   Because by definition 
  \begin{IEEEeqnarray*}{rCl}
\textstyle\thereconcatenator{0}(\anymatBTX{\thedim}{\thedim})=    (\Sbvfull{\anymatBT}{1}{\thedim}{1},\theone)
\end{IEEEeqnarray*}
the veracity of \eqref{eq:second_order_special_case_second_step_2} is thus not affected by adding $-\Sbv{\anymatBT}{\thedim}{1}\Smonoidalproduct \theone$ outside of $\thesplitting{0}$ and $\thedifferential{1}(\Sbv{\anymatBT}{\thedim}{1}\Smonoidalproduct \theone)$ inside.
\par
Since, according to Proposition~\ref{proposition:first_order},
  \begin{IEEEeqnarray*}{rCl}
\thedifferential{1}(\Sbv{\anymatBT}{\thedim}{1}\Smonoidalproduct \theone)=\anymatBTX{\thedim}{\thedim}-\theone
\end{IEEEeqnarray*}
  this change turns \eqref{eq:second_order_special_case_second_step_2} into \eqref{eq:second_order_special_case_second_step_1}, providing the induction base.
  \par
  \emph{Induction step.} Let $\recvar\in\SYnumbers{\thedim}$ be such that $ \recvar\leq \thedim-2$ and such that \eqref{eq:second_order_special_case_second_step_0} is true. The entire argument of  $\thesplitting{0}$ in \eqref{eq:second_order_special_case_second_step_0} is evidently normal and    by Lem\-ma~\ref{lemma:monomial_order_reformulation} has tip $\anymatBTX{\Sexchanged(\recvar+1)}{\Sexchanged(\recvar+1)}$. Because
\begin{IEEEeqnarray*}{rCl}
    \textstyle\thereconcatenator{0}(\anymatBTX{\Sexchanged(\recvar+1)}{\Sexchanged(\recvar+1)})&=&(\Sbvfull{\anymatBT}{1}{\Sexchanged(\recvar+1)}{\recvar+1},\theone)
  \end{IEEEeqnarray*}
a further valid identity is produced by adding to \eqref{eq:second_order_special_case_second_step_0} on the right-hand side $-\Sbv{\anymatBT}{\Sexchanged(\recvar+1)}{\recvar+1}\Smonoidalproduct \theone $ outside of $\thesplitting{0}$ and $\thedifferential{1}(\Sbv{\anymatBT}{\Sexchanged(\recvar+1)}{\recvar+1}\Smonoidalproduct \theone)$ inside.
  \par
  Since Proposition~\ref{proposition:first_order} implies
  \begin{IEEEeqnarray*}{rCl}
\thedifferential{1}(\Sbv{\anymatBT}{\Sexchanged(\recvar+1)}{\recvar+1}\Smonoidalproduct\theone)&=&\anymatBTX{\Sexchanged(\recvar+1)}{\Sexchanged(\recvar+1)}-\theone
  \end{IEEEeqnarray*}  
 the  identity resulting from the additions is just \eqref{eq:second_order_special_case_second_step_0} with $\recvar$  replaced by $\recvar+1$.
\end{proof}

  \begin{proposition}
    \label{proposition:second_order_special_case}
    For any $\anymat\in\thematrices$,
  \begin{IEEEeqnarray*}{rCl}
\thedifferential{2}(\Sbv{\anymat}{\thedim}{\thedim}\Smonoidalproduct\theone)&=&\textstyle\Sbv{\anymatA}{\thedim}{1}\Smonoidalproduct\anymatBTX{1}{1}-\sum_{\inds=2}^\thedim\sum_{\indt=1}^{\thedim-1}\Sbv{\anymat}{\indt}{\inds}\Smonoidalproduct\anymatBTX{\inds}{\Sexchanged(\indt)}
    \IEEEyesnumber
    \label{eq:second_order_special_case_0}\\
      &&\textstyle\hspace{8em}{}-\sum_{\inds=2}^{\thedim}\Sbv{\anymatBT}{\inds}{\Sexchanged(\inds)}\Smonoidalproduct\theone+\Sbv{\anymatBT}{1}{\thedim}\Smonoidalproduct\theone.
    \end{IEEEeqnarray*}
    \end{proposition}
  \begin{proof}
    Since Lemma~\ref{lemma:evil_identities} shows that \eqref{eq:second_order_special_case_0} is invariant under exchanging the roles $\anymatA\leftrightarrow\anymatAT$ it is sufficient to prove it for any $\anymat\in \{\theuniA,\theuniB\}$.
    \par
    Instantiating \eqref{eq:second_order_special_case_second_step_0} from Lem\-ma~\ref{lemma:second_order_special_case_second_step} at $\recvar=\thedim-1$ yields the identity
      \begin{IEEEeqnarray*}{rCl}
    \IEEEeqnarraymulticol{3}{l}{
      \thedifferential{2}(\Sbv{\anymat}{\thedim}{\thedim}\Smonoidalproduct\theone)
      }\\
\hspace{1em}&=&\textstyle\Sbv{\anymatA}{\thedim}{1}\Smonoidalproduct\anymatBTX{1}{1}-\sum_{\inds=2}^\thedim\sum_{\indt=1}^{\thedim-1}\Sbv{\anymat}{\indt}{\inds}\Smonoidalproduct\anymatBTX{\inds}{\Sexchanged(\indt)}-\sum_{\inds=1}^{\thedim-1}\Sbv{\anymatBT}{\Sexchanged(\inds)}{\inds}\Smonoidalproduct\theone\IEEEeqnarraynumspace
    \IEEEyesnumber
    \label{eq:second_order_special_case_1}\\ &&\textstyle{}+\thesplitting{0}(\anymatBTX{1}{1}-\Saction\theone)
  \end{IEEEeqnarray*}
  since $\Sexchanged(\thedim+2)=\thedim-(\thedim+2)+1=-1$. Because the tip of the argument of $\thesplitting{0}$ in \eqref{eq:second_order_special_case_1}  is $\anymatBTX{1}{1}$ and because
  \begin{IEEEeqnarray*}{rCl}
\textstyle\thereconcatenator{0}(\anymatBTX{1}{1})&=&(\Sbv{\anymatBT}{1}{\thedim},\theone)
  \end{IEEEeqnarray*}
  adding on the right-hand side of \eqref{eq:second_order_special_case_1} the vector $\Sbv{\anymatBT}{1}{\thedim}\Smonoidalproduct \theone$ to the exterior of $\thesplitting{0}$ while at the same time subtracting
  \begin{IEEEeqnarray*}{rCl}
    \thedifferential{2}(\Sbv{\anymatBT}{1}{\thedim}\Smonoidalproduct \theone)&=&\anymatBTX{1}{1}-\theone
  \end{IEEEeqnarray*}
  from the interior gives another true identity. As that procedure transforms \eqref{eq:second_order_special_case_1} into \eqref{eq:second_order_special_case_0} this proves the claim. 
  \end{proof}

  \subsubsection{Synthesis} Finally, it will be useful at certain points to have a formula for $\thedifferential{2}$ that covers both the generic and the special case simultaneously.
  \begin{theorem}
      \label{theorem:second_order}
    For any $\anymat\in\thematrices$ and any $(\indj,\indi)\in \SYnumbers{\thedim}^{\Ssetmonoidalproduct 2}$,
    \begin{IEEEeqnarray*}{rCl}
\thedifferential{2}(\Sbv{\anymatA}{\indj}{\indi}\Smonoidalproduct \theone)  &=&  \textstyle\sum_{\inds=1}^\thedim\Sbv{\anymatA}{\indj}{\inds}\Smonoidalproduct\anymatBTX{\inds}{\Sexchanged(\indi)}+\Sbv{\anymatBT}{\Sexchanged(\indj)}{\indi}\Smonoidalproduct\theone\IEEEyesnumber\label{eq:theorem:second_order_0}\\
&&\textstyle{}-\Skronecker{\indj}{\thedim}\Skronecker{\indi}{\thedim}\sum_{\inds=2}^{\thedim}(\sum_{\indt=1}^{\thedim}\Sbv{\anymatA}{\indt}{\inds}\Smonoidalproduct\anymatBTX{\inds}{\Sexchanged(\indt)}+\Sbv{\anymatBT}{\inds}{\Sexchanged(\inds)}\Smonoidalproduct\theone).
\end{IEEEeqnarray*}    
\end{theorem}
\begin{proof}
  It is clear by Proposition~\ref{proposition:second_order_generic_case} that \eqref{eq:theorem:second_order_0} is valid if $(\indj,\indi)\neq (\thedim,\thedim)$. In the case $(\indj,\indi)= (\thedim,\thedim)$ the  sum of all degree-two terms in \eqref{eq:theorem:second_order_0} amounts to 
  \begin{IEEEeqnarray*}{rCl}
\textstyle\Sbv{\anymatA}{\thedim}{1}\Smonoidalproduct\anymatBTX{1}{\Sexchanged(\thedim)}+\sum_{\inds=2}^\thedim\Sbv{\anymatA}{\thedim}{\inds}\Smonoidalproduct\anymatBTX{\inds}{\Sexchanged(\thedim)}-\sum_{\inds=2}^{\thedim}\sum_{\indt=1}^{\thedim}\Sbv{\anymatA}{\indt}{\inds}\Smonoidalproduct\anymatBTX{\inds}{\Sexchanged(\indt)}.
\end{IEEEeqnarray*}
As the grouped sum in the middle is canceled by the  one on the right it is now evident  that \eqref{eq:theorem:second_order_0} also agrees with \eqref{eq:second_order_special_case_0} in Proposition~\ref{proposition:second_order_special_case}.
\end{proof}

\subsection{Third order}
While the formula for $\thedifferential{3}$ is very similar to that for $\thedifferential{\orderind}$ for odd $\orderind\in \Sintegersp$ with $3<\orderind$, the fact that $\thedifferential{2}$ is so different from $\thedifferential{\orderind}$ for even $\orderind\in\Sintegersp$ with $2<\orderind$, makes it necessary to treat $\thedifferential{3}$ separately.

\subsubsection{Case-independent identities}
It is convenient to treat all basis vectors simultaneously up to a certain point and then distinguish cases for the remainder of the proof. There will be $\thedim$ recursion steps before the case distinction.

\begin{lemma}
  \label{lemma:third_order_common}
  For any $\anymat\in \thematrices$, any $(\indj,\indi)\in\SYnumbers{\thedim}^{\Ssetmonoidalproduct 2}$ and any $\recvar\in \SYnumbers{\thedim}$, 
  \begin{IEEEeqnarray*}{rCl}
    \IEEEeqnarraymulticol{3}{l}{
      \thedifferential{3}(\Sbv{\anymatA}{\indj}{\indi}\Smonoidalproduct\theone)
    }\\
    \hspace{1.em}   
&=&\textstyle\sum_{\inds=1}^{\recvar}\Sbv{\anymatA}{\indj}{\inds}\Smonoidalproduct\anymatAX{\inds}{\Sexchanged(\indi)}\IEEEyesnumber
    \label{eq:third_order_common_0}\\
    &&\textstyle{}+\thesplitting{1}\big(\sum_{\indt=\recvar+1}^\thedim\Sbv{\anymatA}{\indj}{1}\Smonoidalproduct\anymatBTX{1}{\Sexchanged(\indt)}\anymatAX{\indt}{\Sexchanged(\indi)}-\sum_{\inds=2}^\thedim\sum_{\indt=1}^{\recvar}\Sbv{\anymatA}{\indj}{\inds}\Smonoidalproduct\anymatBTX{\inds}{\Sexchanged(\indt)}\anymatAX{\indt}{\Sexchanged(\indi)}\\    &&\textstyle\hfill\hspace{.em}{}-\sum_{\inds=1}^\recvar\Sbv{\anymatBT}{\Sexchanged(\indj)}{\inds}\Smonoidalproduct\anymatAX{\inds}{\Sexchanged(\indi)}-\Skronecker{\indi}{1}\Saction\Sbv{\anymatA}{\indj}{\indi}\Smonoidalproduct\theone\\
    &&\textstyle\hspace{2.25em}{}+\Szetafunction{\thedim}{\recvar}\Skronecker{\indj}{\thedim}(\sum_{\inds=2}^\thedim\sum_{\indt=1}^\thedim\Sbv{\anymatA}{\indt}{\inds}\Smonoidalproduct\anymatBTX{\inds}{\Sexchanged(\indt)}\anymatAX{\thedim}{\Sexchanged(\indi)}\\
    &&\textstyle\hfill\hspace{.em}{}+\sum_{\inds=2}^\thedim\Sbv{\anymatBT}{\inds}{\Sexchanged(\inds)}\Smonoidalproduct\anymatX{\thedim}{\Sexchanged(\indi)})\big).
  \end{IEEEeqnarray*}  
\end{lemma}
\begin{proof}
  \newcommand{\dummyindex}{k}
  \newcommand{\abbrterm}{q_{\indj}}
  The assertion will be proved  by induction over $\recvar$. It is economic, though, to first note that, if using the shorthand
  \begin{IEEEeqnarray*}{rCl}
    \abbrterm&\Seqpd&\textstyle \Skronecker{\indj}{\thedim}(\sum_{\inds=2}^\thedim\sum_{\indt=1}^\thedim\Sbv{\anymatA}{\indt}{\inds}\Smonoidalproduct\anymatBTX{\inds}{\Sexchanged(\indt)}+\sum_{\inds=2}^\thedim\Sbv{\anymatBT}{\inds}{\Sexchanged(\inds)}\Smonoidalproduct\theone),
  \end{IEEEeqnarray*}
 Theorem~\ref{theorem:second_order} implies for any $\dummyindex\in \SYnumbers{\thedim}$,
  \begin{IEEEeqnarray*}{rCl}    -\thedifferential{2}(\Sbv{\anymat}{\indj}{\dummyindex}\Smonoidalproduct\anymatAX{\dummyindex}{\Sexchanged(\indi)})&=&\textstyle-(\sum_{\inds=1}^\thedim\Sbv{\anymat}{\indj}{\inds}\Smonoidalproduct\anymatBTX{\inds}{\Sexchanged(\dummyindex)}+\Sbv{\anymatBT}{\Sexchanged(\indj)}{\dummyindex}\Smonoidalproduct\theone-\Skronecker{\thedim}{\dummyindex}\Saction\abbrterm)\anymatAX{\dummyindex}{\Sexchanged(\indi)}\\
    &=&\textstyle-\Saction\Sbv{\anymat}{\indj}{1}\Smonoidalproduct\anymatBTX{1}{\Sexchanged(\dummyindex)}\anymatAX{\dummyindex}{\Sexchanged(\indi)}-\sum_{\inds=2}^{\thedim}\Sbv{\anymat}{\indj}{\inds}\Smonoidalproduct\anymatBTX{\inds}{\Sexchanged(\dummyindex)}\anymatAX{\dummyindex}{\Sexchanged(\indi)}\IEEEeqnarraynumspace\IEEEyesnumber
    \label{eq:third_order_common_1}\\    &&\textstyle{}-\Sbv{\anymatBT}{\Sexchanged(\indj)}{\dummyindex}\Smonoidalproduct\anymatAX{\dummyindex}{\Sexchanged(\indi)}+\Skronecker{\thedim}{\dummyindex}\Saction\abbrterm\anymatAX{\dummyindex}{\Sexchanged(\indi)}.
  \end{IEEEeqnarray*}
  \par
  \emph{Induction base.} Because $2\leq \thedim$ and thus $\Szetafunction{\thedim}{1}=0$ what is claimed in the special case $\recvar=1$  is that
  \begin{IEEEeqnarray*}{rCl}
\thedifferential{3}(\Sbv{\anymatA}{\indj}{\indi}\Smonoidalproduct\theone)&=&\textstyle\Sbv{\anymatA}{\indj}{1}\Smonoidalproduct\anymatAX{1}{\Sexchanged(\indi)}\IEEEyesnumber
    \label{eq:third_order_common_2}\\
    &&\textstyle{}+\thesplitting{1}\big(\sum_{\indt=2}^\thedim\Sbv{\anymatA}{\indj}{1}\Smonoidalproduct\anymatBTX{1}{\Sexchanged(\indt)}\anymatAX{\indt}{\Sexchanged(\indi)}-\sum_{\inds=2}^\thedim\Sbv{\anymatA}{\indj}{\inds}\Smonoidalproduct\anymatBTX{\inds}{\Sexchanged(1)}\anymatAX{1}{\Sexchanged(\indi)}\\    &&\textstyle\hfill\hspace{.em}{}-\Sbv{\anymatBT}{\Sexchanged(\indj)}{1}\Smonoidalproduct\anymatAX{1}{\Sexchanged(\indi)}-\Skronecker{\indi}{1}\Saction\Sbv{\anymatA}{\indj}{\indi}\Smonoidalproduct\theone\big).    
  \end{IEEEeqnarray*}  
  \par 
  Since by Proposition~\ref{proposition:combinatorial_differential}
  \begin{IEEEeqnarray*}{rCl}
\thedeconcatenator{3}(\Sbvfull{\anymat}{3}{\indj}{\indi})=(\Sbvfull{\anymat}{2}{\indj}{1}, \anymatAX{1}{\Sexchanged(\indi)})
  \end{IEEEeqnarray*}
the definition of $\thedifferential{3}$ implies
\begin{IEEEeqnarray*}{rCl}
\thedifferential{3}(\Sbv{\anymat}{\indj}{\indi}\Smonoidalproduct\theone)&=&  \Sbv{\anymat}{\indj}{1}\Smonoidalproduct \anymatAX{1}{\Sexchanged(\indi)}+\thesplitting{1}(-\thedifferential{2}(\Sbv{\anymat}{\indj}{1}\Smonoidalproduct \anymatAX{1}{\Sexchanged(\indi)})).\IEEEyesnumber\label{eq:third_order_common_2c}
\end{IEEEeqnarray*}
Thus it is enough to prove that the arguments of $\thesplitting{1}$ on the respective right-hand sides of \eqref{eq:third_order_common_2} and \eqref{eq:third_order_common_2c} are equal. 
  \par
  Reducing  $\anymatBTX{1}{\Sexchanged(1)}\anymatAX{1}{\Sexchanged(\indi)}$ in  \eqref{eq:third_order_common_1} for $\dummyindex=1$ yields
  \begin{IEEEeqnarray*}{rCl}    -\thedifferential{2}(\Sbv{\anymat}{\indj}{1}\Smonoidalproduct\anymatAX{1}{\Sexchanged(\indi)})&=&\textstyle-\Saction\Sbv{\anymat}{\indj}{1}\Smonoidalproduct(-\sum_{\indt=2}^\thedim\anymatBTX{1}{\Sexchanged(\indt)}\anymatAX{\indt}{\Sexchanged(\indi)}+\Skronecker{\indi}{1}\Saction\theone)\IEEEeqnarraynumspace\IEEEyesnumber
    \label{eq:third_order_common_2a}\\    &&\textstyle{}-\sum_{\inds=2}^{\thedim}\Sbv{\anymat}{\indj}{\inds}\Smonoidalproduct\anymatBTX{\inds}{\Sexchanged(1)}\anymatAX{1}{\Sexchanged(\indi)}-\Sbv{\anymatBT}{\Sexchanged(\indj)}{1}\Smonoidalproduct\anymatAX{1}{\Sexchanged(\indi)},
  \end{IEEEeqnarray*}
  which proves the claimed agreement of the two arguments of $\thesplitting{0}$.
  Hence, \eqref{eq:third_order_common_0} is true for $\recvar=1$.
\par
\emph{Induction step.} Let \eqref{eq:third_order_common_0} hold for some $\recvar\in \SYnumbers{\thedim}$ with  $\recvar\leq \thedim-1$. As shown hereafter,  \eqref{eq:third_order_common_0} remains true if $\recvar$ is replaced by $\recvar+1$.
\par
The values of $\anymat$, $\indj$ and $\indi$ be as they may,  $\Sbv{\anymat}{\indj}{1}\Smonoidalproduct\anymatBTX{1}{\Sexchanged(\recvar+1)}\anymatAX{\recvar+1}{\Sexchanged(\indi)}$ is normal by Lem\-ma~\ref{characterization_of_reduced_terms_of_order_two}. It is the tip of the argument of  $\thesplitting{1}$ in the induction hypothesis \eqref{eq:third_order_common_0} for the following reasons. First of all, by $\Szetafunction{\thedim}{\recvar}=0$ the  last two lines of \eqref{eq:third_order_common_0} are zero.  Furthermore, any vector in the fourth line is less already by degree. Hence, the only vectors against which to compare are in the third line of \eqref{eq:third_order_common_0}.
For any $\inds\in \SYnumbers{\thedim}$ with $1<\inds$ the fact that $(\indj,1)\lexorderRstrict(\indj,\inds)$ implies $\Sbv{\anymat}{\indj}{\inds}\theorderRstrict\Sbv{\anymat}{\indj}{1}$ by Lem\-ma~\ref{lemma:base_monomial_order} and thus $\Sbv{\anymatA}{\indj}{\inds}\Smonoidalproduct\anymatBTX{\inds}{\Sexchanged(\indt)}\anymatAX{\indt}{\Sexchanged(\indi)}\thechainsorderRstrict{1}\Sbv{\anymat}{\indj}{1}\Smonoidalproduct\anymatBTX{1}{\Sexchanged(\recvar+1)}\anymatAX{\recvar+1}{\Sexchanged(\indi)}$ for arbitrary $\indt\in\SYnumbers{\thedim}$. Similarly, for any $\indt\in\SYnumbers{\thedim}$ with $\recvar+1<\indt$ the addendum to Lem\-ma~\ref{lemma:monomial_order_reformulation} proves  $\anymatBTX{1}{\Sexchanged(\indt)}\theorderRstrict\anymatBTX{1}{\Sexchanged(\recvar+1)}$ and thus $\Sbv{\anymat}{\indj}{1}\Smonoidalproduct\anymatBTX{1}{\Sexchanged(\indt)}\anymatAX{\indt}{\Sexchanged(\indi)}\thechainsorderRstrict{1}\Sbv{\anymat}{\indj}{1}\Smonoidalproduct\anymatBTX{1}{\Sexchanged(\recvar+1)}\anymatAX{\recvar+1}{\Sexchanged(\indi)}$.
\par 
In consequence, since by Proposition~\ref{proposition:combinatorial_splitting}
\begin{IEEEeqnarray*}{rCl}
\textstyle\thereconcatenator{1}(  \Sbvfull{\anymat}{1}{\indj}{1},\anymatBTX{1}{\Sexchanged(\recvar+1)}\anymatAX{\recvar+1}{\Sexchanged(\indi)})=(\Sbvfull{2}{\anymat}{\indj}{\recvar+1},\anymatAX{\recvar+1}{\Sexchanged(\indi)})
\end{IEEEeqnarray*}
the veracity of \eqref{eq:third_order_common_0} is not affected by adding on the right-hand side $\Sbv{\anymat}{\indj}{\recvar+1}\Smonoidalproduct\anymatAX{\recvar+1}{\Sexchanged(\indi)}$ outside of $\thesplitting{1}$ and $-\thedifferential{2}(\Sbv{\anymat}{\indj}{\recvar+1}\Smonoidalproduct\anymatAX{\recvar+1}{\Sexchanged(\indi)})$ inside.
\par
Clearly, on the exterior of $\thesplitting{1}$ this addition has the same effect as replacing $\recvar$ by $\recvar+1$.
Simultaneously, specializing \eqref{eq:third_order_common_1} to $\dummyindex=\recvar+1$ yields
\begin{IEEEeqnarray*}{rCl}
  \IEEEeqnarraymulticol{3}{l}{
    -\thedifferential{2}(\Sbv{\anymat}{\indj}{\recvar+1}\Smonoidalproduct\anymatAX{\recvar+1}{\Sexchanged(\indi)})
  }\\
  \hspace{.5em}    &=&\textstyle-\Saction\Sbv{\anymat}{\indj}{1}\Smonoidalproduct\anymatBTX{1}{\Sexchanged(\recvar+1)}\anymatAX{\recvar+1}{\Sexchanged(\indi)}-\sum_{\inds=2}^{\thedim}\Sbv{\anymat}{\indj}{\inds}\Smonoidalproduct\anymatBTX{\inds}{\Sexchanged(\recvar+1)}\anymatAX{\recvar+1}{\Sexchanged(\indi)}\IEEEeqnarraynumspace\IEEEyesnumber
    \label{eq:third_order_common_3}\\    &&\textstyle{}-\Sbv{\anymatBT}{\Sexchanged(\indj)}{\recvar+1}\Smonoidalproduct\anymatAX{\recvar+1}{\Sexchanged(\indi)}+\Skronecker{\thedim}{\recvar+1}\Saction\abbrterm\anymatAX{\recvar+1}{\Sexchanged(\indi)}.
  \end{IEEEeqnarray*}
Because $\Szetafunction{\thedim}{\recvar}+\Skronecker{\thedim}{\recvar+1}=\Skronecker{\thedim}{\recvar+1}$  it is now evident that the addition of the vector \eqref{eq:third_order_common_3} changes the argument of $\thesplitting{1}$ in \eqref{eq:third_order_common_0} in the same way as increasing $\recvar$ by $1$ would. Hence, the assertion is true for all $\recvar$. 
\end{proof}

\begin{lemma}
  \label{lemma:third_order_common_final}
  For any $\anymat\in\thematrices$ and any $(\indj,\indi)\in\SYnumbers{\thedim}^{\Ssetmonoidalproduct 2}$,
  \begin{IEEEeqnarray*}{rCl}
    \IEEEeqnarraymulticol{3}{l}{
      \thedifferential{3}(\Sbv{\anymatA}{\indj}{\indi}\Smonoidalproduct\theone)
    }\\
    \hspace{.5em}   
&=&\textstyle\sum_{\inds=1}^{\thedim}\Sbv{\anymatA}{\indj}{\inds}\Smonoidalproduct\anymatAX{\inds}{\Sexchanged(\indi)}\IEEEyesnumber
    \label{eq:third_order_common_final_0}\\
    &&\textstyle{}+\thesplitting{1}\big({-}\sum_{\inds=1}^\thedim\Sbv{\anymatBT}{\Sexchanged(\indj)}{\inds}\Smonoidalproduct\anymatAX{\inds}{\Sexchanged(\indi)}-\Sbv{\anymatA}{\indj}{\indi}\Smonoidalproduct\theone\\
    &&\textstyle\hspace{2.25em}{}+\Skronecker{\indj}{\thedim}(\sum_{\inds=2}^\thedim\sum_{\indt=1}^\thedim\Sbv{\anymatA}{\indt}{\inds}\Smonoidalproduct\anymatBTX{\inds}{\Sexchanged(\indt)}\anymatAX{\thedim}{\Sexchanged(\indi)}+\sum_{\inds=2}^\thedim\Sbv{\anymatBT}{\inds}{\Sexchanged(\inds)}\Smonoidalproduct\anymatX{\thedim}{\Sexchanged(\indi)})\big).
  \end{IEEEeqnarray*}    
\end{lemma}
\begin{proof}
  In the special case $\recvar=\thedim$ the first grouped sum in the third line of \eqref{eq:third_order_common_0} in Lemma~\ref{lemma:third_order_common} vanishes and what remains is the identity
    \begin{IEEEeqnarray*}{rCl}
    \IEEEeqnarraymulticol{3}{l}{
      \thedifferential{3}(\Sbv{\anymatA}{\indj}{\indi}\Smonoidalproduct\theone)
    }\\
    \hspace{.em}   
&=&\textstyle\sum_{\inds=1}^{\thedim}\Sbv{\anymatA}{\indj}{\inds}\Smonoidalproduct\anymatAX{\inds}{\Sexchanged(\indi)}\IEEEyesnumber
    \label{eq:third_order_common_final_1}\\    &&\textstyle{}+\thesplitting{1}\big({-}\sum_{\inds=2}^\thedim\Sbv{\anymatA}{\indj}{\inds}\Smonoidalproduct\sum_{\indt=1}^{\thedim}\anymatBTX{\inds}{\Sexchanged(\indt)}\anymatAX{\indt}{\Sexchanged(\indi)}\\    &&\textstyle\hfill\hspace{.em}{}-\sum_{\inds=1}^\thedim\Sbv{\anymatBT}{\Sexchanged(\indj)}{\inds}\Smonoidalproduct\anymatAX{\inds}{\Sexchanged(\indi)}-\Skronecker{\indi}{1}\Saction\Sbv{\anymatA}{\indj}{\indi}\Smonoidalproduct\theone\\
    &&\textstyle\hspace{2.25em}{}+\Skronecker{\indj}{\thedim}(\sum_{\inds=2}^\thedim\sum_{\indt=1}^\thedim\Sbv{\anymatA}{\indt}{\inds}\Smonoidalproduct\anymatBTX{\inds}{\Sexchanged(\indt)}\anymatAX{\thedim}{\Sexchanged(\indi)}+\sum_{\inds=2}^\thedim\Sbv{\anymatBT}{\inds}{\Sexchanged(\inds)}\Smonoidalproduct\anymatX{\thedim}{\Sexchanged(\indi)})\big).
  \end{IEEEeqnarray*}
  And because
  \begin{IEEEeqnarray*}{rCl}
    \IEEEeqnarraymulticol{3}{l}{
      \textstyle -\sum_{\inds=2}^\thedim\Sbv{\anymatA}{\indj}{\inds}\Smonoidalproduct\sum_{\indt=1}^{\thedim}\anymatBTX{\inds}{\Sexchanged(\indt)}\anymatAX{\indt}{\Sexchanged(\indi)}-\Skronecker{\indi}{1}\Saction\Sbv{\anymatA}{\indj}{\indi}\Smonoidalproduct\theone
    }\\    \hspace{0em}&=&\textstyle-\sum_{\inds=2}^\thedim\Skronecker{\inds}{\indi}\Saction\Sbv{\anymatA}{\indj}{\inds}\Smonoidalproduct\theone-\Skronecker{\indi}{1}\Saction\Sbv{\anymatA}{\indj}{1}\Smonoidalproduct\theone=-\sum_{\inds=1}^\thedim\Skronecker{\inds}{\indi}\Saction\Sbv{\anymatA}{\indj}{\inds}\Smonoidalproduct\theone=-\Sbv{\anymatA}{\indj}{\indi}\Smonoidalproduct\theone
  \end{IEEEeqnarray*}
  that already proves  \eqref{eq:third_order_common_final_0}.
\end{proof}
From now on it makes more sense to distinguish cases, ultimately four of them.
\subsubsection{Towards the first two cases}
For the first two out of the four cases that need to be considered there is a common starting point, obtained via one recursion step.
\begin{lemma}
  \label{lemma:third_order_first_two_cases_common}
  For any $\anymat\in\thematrices$ and any $(\indj,\indi)\in\SYnumbers{\thedim}^{\Ssetmonoidalproduct 2}$, if $\indj\neq \thedim$, then
    \begin{IEEEeqnarray*}{rCl}
\thedifferential{3}(\Sbv{\anymatA}{\indj}{\indi}\Smonoidalproduct\theone)&=&\textstyle\sum_{\inds=1}^{\thedim}\Sbv{\anymatA}{\indj}{\inds}\Smonoidalproduct\anymatAX{\inds}{\Sexchanged(\indi)}-\Sbv{\anymatBT}{\Sexchanged(\indj)}{\indi}\Smonoidalproduct\theone\IEEEyesnumber
    \label{eq:third_order_first_two_cases_common_0}\\
    &&\textstyle{}+\thesplitting{1}\big({-}\sum_{\inds=1}^{\thedim}\Sbv{\anymatBT}{\Sexchanged(\indj)}{\inds}\Smonoidalproduct\anymatAX{\inds}{\Sexchanged(\indi)}-\Sbv{\anymatA}{\indj}{\indi}\Smonoidalproduct \theone+\thedifferential{2}(\Sbv{\anymatBT}{\Sexchanged(\indj)}{\indi}\Smonoidalproduct\theone)).
  \end{IEEEeqnarray*}
\end{lemma}
\begin{proof}
What the identity \eqref{eq:third_order_common_final_0} in Lem\-ma~\ref{lemma:third_order_common_final} says in the special case $\indj\neq \thedim$  is that
    \begin{IEEEeqnarray*}{rCl}
\thedifferential{3}(\Sbv{\anymatA}{\indj}{\indi}\Smonoidalproduct\theone)&=&\textstyle\sum_{\inds=1}^{\thedim}\Sbv{\anymatA}{\indj}{\inds}\Smonoidalproduct\anymatAX{\inds}{\Sexchanged(\indi)}\IEEEyesnumber
    \label{eq:third_order_first_two_cases_common_1}\\
    &&\textstyle{}+\thesplitting{1}({-}\sum_{\inds=1}^{\thedim}\Sbv{\anymatBT}{\Sexchanged(\indj)}{\inds}\Smonoidalproduct\anymatAX{\inds}{\Sexchanged(\indi)}-\Sbv{\anymatA}{\indj}{\indi}\Smonoidalproduct \theone).
  \end{IEEEeqnarray*}  
  The argument of $\thesplitting{1}$ in this identity is obviously normal and its tip is $\Sbv{\anymatBT}{\Sexchanged(\indj)}{1}\Smonoidalproduct\anymatAX{1}{\Sexchanged(\indi)}$ because already  $\Sbv{\anymatBT}{\Sexchanged(\indj)}{\inds}<\Sbv{\anymatBT}{\Sexchanged(\indj)}{1}$ for any $\inds\in\SYnumbers{\thedim}$ with $1<\inds$ by Lem\-ma~\ref{lemma:base_monomial_order}.
  \par
  Because by Proposition~\ref{proposition:combinatorial_splitting}
  \begin{IEEEeqnarray*}{rCl}
   \textstyle\thereconcatenator{1}( \Sbvfull{\anymatBT}{1}{\Sexchanged(\indj)}{1},\anymatAX{1}{\Sexchanged(\indi)})=(\Sbvfull{\anymatBT}{2}{\Sexchanged(\indj)}{\indi},\theone)
  \end{IEEEeqnarray*}
  another valid identity is produced by adding to \eqref{eq:third_order_first_two_cases_common_1} on the right-hand side the term $-\Sbv{\anymatBT}{\Sexchanged(\indj)}{\indi}\Smonoidalproduct \theone$ on the exterior of $\thesplitting{1}$ and $\thedifferential{2}(\Sbv{\anymatBT}{\Sexchanged(\indj)}{\indi}\Smonoidalproduct \theone)$. The resulting equation is precisely \eqref{eq:third_order_first_two_cases_common_0}.
\end{proof}
\subsubsection{First case} The first case is very simple, requiring no further recursions.
\begin{proposition}
  \label{proposition:third_order_first_case}
  For any $\anymat\in\thematrices$ and any $(\indj,\indi)\in\SYnumbers{\thedim}^{\Ssetmonoidalproduct 2}$, if $\indj\neq \thedim$ and $(\indj,\indi)\neq (1,\thedim)$, then
  \begin{IEEEeqnarray*}{rCl}
    \thedifferential{3}(\Sbv{\anymatA}{\indj}{\indi}\Smonoidalproduct\theone)&=&\textstyle\sum_{\inds=1}^{\thedim}\Sbv{\anymatA}{\indj}{\inds}\Smonoidalproduct\anymatAX{\inds}{\Sexchanged(\indi)}-\Sbv{\anymatBT}{\Sexchanged(\indj)}{\indi}\Smonoidalproduct\theone.\IEEEyesnumber
    \label{eq:third_order_first_case_0}
  \end{IEEEeqnarray*}
\end{proposition}
\begin{proof}
  By Proposition~\ref{proposition:second_order_generic_case} the assumption that  $(\indj,\indi)\neq (1,\thedim)$ implies
  \begin{IEEEeqnarray*}{rCl}
    \thedifferential{2}(\Sbv{\anymatBT}{\Sexchanged(\indj)}{\indi}\Smonoidalproduct\theone)&=&\textstyle\sum_{\inds=1}^{\thedim}\Sbv{\anymatBT}{\Sexchanged(\indj)}{\inds}\Smonoidalproduct\anymatAX{\inds}{\Sexchanged(\indi)}+\Sbv{\anymatA}{\indj}{\indi}\Smonoidalproduct \theone,
  \end{IEEEeqnarray*}
  which is precisely the negative of the remainder of the argument of $\thesplitting{1}$ in \eqref{eq:third_order_first_two_cases_common_0}. Hence, Lem\-ma~\ref{lemma:third_order_first_two_cases_common} proves the assertion.
\end{proof}
\subsubsection{Second case} The second case is still  handled quickly, in $\thedim$ steps.
\begin{lemma}
  \label{lemma:third_order_second_case_helper}  
  For any $\anymat\in\thematrices$ and any $\recvar\in\SYnumbers{\thedim}$,
    \begin{IEEEeqnarray*}{rCl}
      \IEEEeqnarraymulticol{3}{l}{ \thedifferential{3}(\Sbv{\anymatA}{1}{\thedim}\Smonoidalproduct\theone)
      }\\
      \hspace{1em}&=&\textstyle\sum_{\inds=1}^\thedim\Sbv{\anymatA}{1}{\inds}\Smonoidalproduct\anymatAX{\inds}{1}-\Sbv{\anymatBT}{\thedim}{\thedim}\Smonoidalproduct\theone-\sum_{\inds=2}^\recvar\Sbv{\anymatB}{\Sexchanged(\inds)}{\Sexchanged(\inds)}\Smonoidalproduct\theone\IEEEyesnumber\label{eq:third_order_second_case_helper_0}\\
      &&\textstyle{}+\thesplitting{1}\big({-}\sum_{\inds=1}^\thedim\sum_{\indt=\recvar+1}^\thedim\Sbv{\anymatB}{\Sexchanged(\indt)}{\inds}\Smonoidalproduct\anymatATX{\inds}{\indt}-\sum_{\inds=\recvar+1}^\thedim\Sbv{\anymatAT}{\inds}{\Sexchanged(\inds)}\Smonoidalproduct\theone).
  \end{IEEEeqnarray*}
\end{lemma}
\begin{proof}
  The proof goes by induction over $\recvar$.
  \par
  \emph{Induction base.} For the base  case $\recvar=1$ what needs to be shown is that
    \begin{IEEEeqnarray*}{rCl}
\thedifferential{3}(\Sbv{\anymatA}{1}{\thedim}\Smonoidalproduct\theone)&=&\textstyle\sum_{\indt=1}^\thedim\Sbv{\anymatA}{1}{\inds}\Smonoidalproduct\anymatAX{\inds}{1}-\Sbv{\anymatBT}{\thedim}{\thedim}\Smonoidalproduct\theone\IEEEyesnumber\label{eq:third_order_second_case_helper_1}\\
      &&\textstyle{}+\thesplitting{1}\big({-}\sum_{\inds=1}^\thedim\sum_{\indt=2}^\thedim\Sbv{\anymatB}{\Sexchanged(\indt)}{\inds}\Smonoidalproduct\anymatATX{\inds}{\indt}-\sum_{\inds=2}^\thedim\Sbv{\anymatAT}{\inds}{\Sexchanged(\inds)}\Smonoidalproduct\theone).
    \end{IEEEeqnarray*}
    \par
    Since by Proposition~\ref{proposition:second_order_special_case}
      \begin{IEEEeqnarray*}{rCl}
\thedifferential{2}(\Sbv{\anymatBT}{\thedim}{\thedim}\Smonoidalproduct\theone)&=&\textstyle\Sbv{\anymatBT}{\thedim}{1}\Smonoidalproduct\anymatAX{1}{1}-\sum_{\inds=2}^\thedim\sum_{\indt=1}^{\thedim-1}\Sbv{\anymatBT}{\indt}{\inds}\Smonoidalproduct\anymatAX{\inds}{\Sexchanged(\indt)}\\
&&\textstyle\hspace{8em}{}-\sum_{\inds=2}^{\thedim}\Sbv{\anymatA}{\inds}{\Sexchanged(\inds)}\Smonoidalproduct\theone+\Sbv{\anymatA}{1}{\thedim}\Smonoidalproduct\theone
    \end{IEEEeqnarray*}
the conclusion \eqref{eq:third_order_first_two_cases_common_0} of Lem\-ma~\ref{lemma:third_order_first_two_cases_common} in the special case $(\indj,\indi)=(1,\thedim)$ reads
        \begin{IEEEeqnarray*}{rCl}
\thedifferential{3}(\Sbv{\anymatA}{1}{\thedim}\Smonoidalproduct\theone)&=&\textstyle\sum_{\inds=1}^{\thedim}\Sbv{\anymatA}{1}{\inds}\Smonoidalproduct\anymatAX{\inds}{1}-\Sbv{\anymatBT}{\thedim}{\thedim}\Smonoidalproduct\theone\IEEEyesnumber
    \label{eq:third_order_second_case_helper_1a}\\
    &&\textstyle{}+\thesplitting{1}\big({-}\sum_{\inds=2}^\thedim\sum_{\indt=1}^{\thedim}\Sbv{\anymatBT}{\indt}{\inds}\Smonoidalproduct\anymatAX{\inds}{\Sexchanged(\indt)}-\sum_{\inds=2}^{\thedim}\Sbv{\anymatA}{\inds}{\Sexchanged(\inds)}\Smonoidalproduct\theone\big).
  \end{IEEEeqnarray*}
  where two instances each of $\Sbv{\anymatA}{1}{\thedim}\Smonoidalproduct\theone$ and $\Sbv{\anymatBT}{\thedim}{1}\Smonoidalproduct\anymatAX{1}{1}$ with different signs have canceled each other out and where two grouped sums have been combined into one. Hence, it is enough to rewrite the argument of $\thesplitting{1}$ in  \eqref{eq:third_order_second_case_helper_1a} until it agrees with that in \eqref{eq:third_order_second_case_helper_1}.
  \par
  Reindexing the  summation over $\indt$  with $\Sexchanged$ and afterwards exchanging the names $\inds\leftrightarrow\indt$ of the summation indices makes the argument of $\thesplitting{1}$ in  \eqref{eq:third_order_second_case_helper_1a}  take on the form
  \begin{IEEEeqnarray*}{rCl}
\textstyle    -\sum_{\inds=1}^\thedim\sum_{\indt=2}^{\thedim}\Sbv{\anymatBT}{\Sexchanged(\inds)}{\indt}\Smonoidalproduct\anymatAX{\indt}{\inds}-\sum_{\inds=2}^{\thedim}\Sbv{\anymatA}{\inds}{\Sexchanged(\inds)}\Smonoidalproduct\theone.
  \end{IEEEeqnarray*}
  Recognizing that for any $\{\inds,\indt\}\subseteq \SYnumbers{\thedim}$ not only  $\anymatAX{\indt}{\inds}=\anymatATX{\inds}{\indt}$  but by Lemma~\ref{lemma:evil_identities} also
  \begin{IEEEeqnarray*}{rCl}
    \Sbvfull{\anymatBT}{1}{\Sexchanged(\inds)}{\indt}&=&\Sbvfull{\anymatB}{1}{\Sexchanged(\indt)}{\inds}
  \end{IEEEeqnarray*}
  as well as
  \begin{IEEEeqnarray*}{rCl}
    \Sbvfull{\anymatA}{1}{\inds}{\Sexchanged(\inds)}&=&\Sbvfull{\anymatAT}{1}{\inds}{\Sexchanged(\inds)}
  \end{IEEEeqnarray*}
  now shows the claimed equality between the arguments in \eqref{eq:third_order_second_case_helper_1a} and \eqref{eq:third_order_second_case_helper_1}, thus proving \eqref{eq:third_order_second_case_helper_0} for $\recvar=1$.
  \par
  \emph{Induction step.} Let \eqref{eq:third_order_second_case_helper_0} hold for some $\recvar\in\SYnumbers{\thedim}$ with $\recvar\leq \thedim-1$. By the following argument \eqref{eq:third_order_second_case_helper_0} is also true for $\recvar+1$ in place of $\recvar$.
  \par
  The normal vector $\Sbv{\anymatB}{\Sexchanged(\recvar+1)}{1}\Smonoidalproduct\anymatATX{1}{\recvar+1}$ is the tip of the argument of $\thesplitting{1}$ in \eqref{eq:third_order_second_case_helper_0} as explained hereafter. For reasons of degree the only terms which could possibly be greater or equal are of the form $\Sbv{\anymatB}{\Sexchanged(\indt)}{\inds}\Smonoidalproduct\anymatATX{\inds}{\indt}$ for $\{\inds,\indt\}\subseteq \SYnumbers{\thedim}$ such that either $\recvar+1<\indt$ or $\recvar+1=\indt$ and $1<\inds$, which is to say such that $(\recvar+1,1)\lexorderRstrict(\indt,\inds)$. But those conditions are actually equivalent to $(1,\recvar+1)\lexorderRstrict(\inds,\indt)$. Since thus both $(\Sexchanged(\indt),1)\lexorderRstrict(\Sexchanged(\recvar+1),\inds)$ and  $(1,\Sexchanged(\indt))\lexorderRstrict(\inds,\Sexchanged(\recvar+1))$ for such $\indt$ and $\inds$ Lem\-ma~\ref{lemma:base_monomial_order} proves that then $\Sbv{\anymatB}{\Sexchanged(\indt)}{\inds}\theorderRstrict\Sbv{\anymatB}{\Sexchanged(\recvar+1)}{1}$, irrespective of the value of $\anymat$.
  \par
  Because by Proposition~\ref{proposition:combinatorial_splitting}
  \begin{IEEEeqnarray*}{rCl}
\textstyle   \thereconcatenator{1}( \Sbvfull{\anymatB}{1}{\Sexchanged(\recvar+1)}{1},\anymatATX{1}{\recvar+1})=(\Sbvfull{\anymatB}{2}{\Sexchanged(\recvar+1)}{\Sexchanged(\recvar+1)},1)
  \end{IEEEeqnarray*}
  the definition of $\thesplitting{1}$ implies that a valid identity is derived by adding to the right-hand side of \eqref{eq:third_order_second_case_helper_0} the term $-\Sbv{\anymatB}{\Sexchanged(\recvar+1)}{\Sexchanged(\recvar+1)}\Smonoidalproduct\theone$ outside of $\thesplitting{1}$ and $\thedifferential{2}(\Sbv{\anymatB}{\Sexchanged(\recvar+1)}{\Sexchanged(\recvar+1)}\Smonoidalproduct\theone)$ inside of it.
  \par
  Evidently, the outside of this new identity is the same as that of \eqref{eq:third_order_second_case_helper_0} for $\recvar+1$ instead of $\recvar$. Likewise, the addition of, by Proposition~\ref{proposition:second_order_generic_case},
  \begin{IEEEeqnarray*}{rCl}
    \thedifferential{2}(\Sbv{\anymatB}{\Sexchanged(\recvar+1)}{\Sexchanged(\recvar+1)}\Smonoidalproduct\theone)=\textstyle\sum_{\inds=1}^\thedim\Sbv{\anymatB}{\Sexchanged(\recvar+1)}{\inds}\Smonoidalproduct\anymatATX{\inds}{\recvar+1}+\Sbv{\anymatAT}{\recvar+1}{\Sexchanged(\recvar+1)}\Smonoidalproduct\theone.
  \end{IEEEeqnarray*}
  modifies the argument of $\thesplitting{1}$ in \eqref{eq:third_order_second_case_helper_0} in the same way as replacing $\recvar$  by $\recvar+1$ would. 
That proves \eqref{eq:third_order_second_case_helper_0} to hold for arbitrary $\recvar$.
\end{proof}

\begin{proposition}
  \label{proposition:third_order_second_case}
  For any $\anymat\in\thematrices$,
    \begin{IEEEeqnarray*}{rCl}
\thedifferential{3}(\Sbv{\anymatA}{1}{\thedim}\Smonoidalproduct\theone)&=&\textstyle\sum_{\inds=1}^\thedim\Sbv{\anymatA}{1}{\inds}\Smonoidalproduct\anymatAX{\inds}{1}-\Sbv{\anymatBT}{\thedim}{\thedim}\Smonoidalproduct\theone-\sum_{\inds=1}^{\thedim-1}\Sbv{\anymatB}{\inds}{\inds}\Smonoidalproduct\theone.\IEEEeqnarraynumspace\IEEEyesnumber\label{eq:third_order_second_case_0}
  \end{IEEEeqnarray*}  
\end{proposition}
\begin{proof}
  The claim \eqref{eq:third_order_second_case_0} is precisely \eqref{eq:third_order_second_case_helper_0} in Lem\-ma~\ref{lemma:third_order_second_case_helper} for $\recvar=\thedim$ with $\sum_{\inds=2}^{\thedim}\Sbv{\anymatB}{\Sexchanged(\inds)}{\Sexchanged(\inds)}\Smonoidalproduct\theone$ reindexed.
\end{proof}
\subsubsection{Towards the last two cases} Just like for the first two cases there is a common starting point for the last two, arrived at by $\thedim-1$ steps.
\begin{lemma}
  \label{lemma:third_order_last_two_cases_common}
  For any $\anymat\in \thematrices$, any  $\indi\in\SYnumbers{\thedim}$ and any $\recvar\in\SYnumbers{\thedim}$ with $\recvar\leq \thedim-1$,
  \begin{IEEEeqnarray*}{rCl}
    \IEEEeqnarraymulticol{3}{l}{
      \thedifferential{3}(\Sbv{\anymatA}{\thedim}{\indi}\Smonoidalproduct\theone)
    }\\
    \hspace{.em}   
&=&\textstyle\sum_{\inds=1}^{\thedim}\Sbv{\anymatA}{\thedim}{\inds}\Smonoidalproduct\anymatAX{\inds}{\Sexchanged(\indi)}+\sum_{\inds=2}^\recvar\Sbv{\anymatAT}{\Sexchanged(\inds)}{\Sexchanged(\inds)}\Smonoidalproduct\anymatATX{\Sexchanged(\indi)}{\thedim}\IEEEyesnumber
    \label{eq:third_order_last_two_cases_common_0}\\
    &&\textstyle{}+\thesplitting{1}\big(\sum_{\inds=\recvar+1}^\thedim\Sbv{\anymatAT}{\Sexchanged(\inds)}{1}\Smonoidalproduct\anymatBX{1}{\inds}\anymatATX{\Sexchanged(\indi)}{\thedim}+\sum_{\inds=\recvar+1}^\thedim\sum_{\indt=2}^\thedim\Sbv{\anymatAT}{\Sexchanged(\inds)}{\indt}\Smonoidalproduct\anymatBX{\indt}{\inds}\anymatATX{\Sexchanged(\indi)}{\thedim}\\
    &&\textstyle\hfill{}+\sum_{\inds=\recvar+1}^\thedim\Sbv{\anymatB}{\inds}{\Sexchanged(\inds)}\Smonoidalproduct\anymatATX{\Sexchanged(\indi)}{\thedim}-\sum_{\inds=1}^\thedim\Sbv{\anymatB}{\Sexchanged(\inds)}{\thedim}\Smonoidalproduct\anymatATX{\Sexchanged(\indi)}{\inds}-\Sbv{\anymatAT}{\Sexchanged(\indi)}{1}\Smonoidalproduct\theone\big).
  \end{IEEEeqnarray*} 
\end{lemma}
\begin{proof}
\newcommand{\dummyterm}[2]{k_{#1,#2}^\indi}
  The proof is by  induction over $\recvar$.
  \par
  \emph{Induction base.} For the base case $\recvar=1$ what needs to be proved is that
  \begin{IEEEeqnarray*}{rCl}
    \IEEEeqnarraymulticol{3}{l}{
      \thedifferential{3}(\Sbv{\anymatA}{\thedim}{\indi}\Smonoidalproduct\theone)
    }\\
    \hspace{.em}   
&=&\textstyle\sum_{\inds=1}^{\thedim}\Sbv{\anymatA}{\thedim}{\inds}\Smonoidalproduct\anymatAX{\inds}{\Sexchanged(\indi)}\IEEEyesnumber
    \label{eq:third_order_last_two_cases_common_1}\\
    &&\textstyle{}+\thesplitting{1}\big(\sum_{\inds=2}^\thedim\sum_{\indt=1}^\thedim\Sbv{\anymatAT}{\Sexchanged(\inds)}{\indt}\Smonoidalproduct\anymatBX{\indt}{\inds}\anymatATX{\Sexchanged(\indi)}{\thedim}\\
    &&\textstyle\hspace{2.5em}{}+\sum_{\inds=2}^\thedim\Sbv{\anymatB}{\inds}{\Sexchanged(\inds)}\Smonoidalproduct\anymatATX{\Sexchanged(\indi)}{\thedim}-\sum_{\inds=1}^\thedim\Sbv{\anymatB}{\Sexchanged(\inds)}{\thedim}\Smonoidalproduct\anymatATX{\Sexchanged(\indi)}{\inds}-\Sbv{\anymatAT}{\Sexchanged(\indi)}{1}\Smonoidalproduct\theone\big).
  \end{IEEEeqnarray*}   
  \par
  But \eqref{eq:third_order_last_two_cases_common_1}
  is precisely what  \eqref{eq:third_order_common_final_0} in Lem\-ma~\ref{lemma:third_order_common_final} says in the special case $\indj=\thedim$ after the summation over $\indt$ is reindexed with $\Sexchanged$ and for any $\{\inds,\indt\}\subseteq \SYnumbers{\thedim}$ with $1<\inds$ the identities
  \begin{IEEEeqnarray*}{rCl}
    \Sbvfull{\anymat}{1}{\Sexchanged(\indt)}{\inds}\Smonoidalproduct\anymatBTX{\inds}{\indt}\anymatAX{\thedim}{\Sexchanged(\indi)} =\Sbvfull{\anymatAT}{1}{\Sexchanged(\inds)}{\indt}\Smonoidalproduct
    \anymatBX{\indt}{\inds}\anymatATX{\Sexchanged(\indi)}{\thedim}
  \end{IEEEeqnarray*}
  and 
  \begin{IEEEeqnarray*}{rCl}
    \Sbvfull{\anymatBT}{1}{\inds}{\Sexchanged(\inds)}\Smonoidalproduct
    \anymatAX{\thedim}{\Sexchanged(\indi)}=\Sbvfull{\anymatB}{1}{\inds}{\Sexchanged(\inds)}\Smonoidalproduct\anymatATX{\Sexchanged(\indi)}{\thedim},
 \end{IEEEeqnarray*}  
 and
  \begin{IEEEeqnarray*}{rCl}
   \Sbvfull{\anymatBT}{1}{\Sexchanged(\thedim)}{\inds}\Smonoidalproduct   \anymatAX{\inds}{\Sexchanged(\indi)}
=\Sbvfull{\anymatB}{1}{\Sexchanged(\inds)}{\thedim}\Smonoidalproduct\anymatATX{\Sexchanged(\indi)}{\inds}
 \end{IEEEeqnarray*}
 and
  \begin{IEEEeqnarray*}{rCl}
   \Sbvfull{\anymat}{1}{\thedim}{\indi}=\Sbvfull{\anymatAT}{1}{\Sexchanged(\indi)}{1},
 \end{IEEEeqnarray*}  
each  partly implied by Lemma~\ref{lemma:evil_identities}, are taken into account.
  \par
  \emph{Induction step.} Let \eqref{eq:third_order_last_two_cases_common_0} be valid for some  $\recvar\in\SYnumbers{\thedim}$ with $ \recvar\leq \thedim-2$. By the ensuing argument it holds for $\recvar+1$ in place of $\recvar$.
  \par
  Regardless of the value of $\indi$, the vector $\Sbv{\anymatAT}{\Sexchanged(\recvar+1)}{1}\Smonoidalproduct\anymatBX{1}{\recvar+1}\anymatATX{\Sexchanged(\indi)}{\thedim}$ is normal by Lem\-ma~\ref{characterization_of_reduced_terms_of_order_two} since $\recvar+1\neq \thedim$. It is moreover the tip of the argument of  $\thesplitting{1}$ in \eqref{eq:third_order_last_two_cases_common_0} for the following reasons.
  Any vector in the fourth line of \eqref{eq:third_order_last_two_cases_common_0} is dominated already by degree. All the remaining other vectors in the argument are of the form $\Sbv{\anymatAT}{\Sexchanged(\inds)}{\indt}\Smonoidalproduct\anymatBX{\indt}{\inds}\anymatATX{\Sexchanged(\indi)}{\thedim}$ for $\{\inds,\indt\}\subseteq \SYnumbers{\thedim}$ with $\recvar+1\leq \inds$ such that either $2\leq \indt$ or both $\indt=1$ and $\recvar+1<\inds$. The latter condition can also be stated as $(1,\recvar+1)\lexorderRstrict(\indt,\inds)$, which is actually  equivalent to $(\recvar+1,1)\lexorderRstrict(\inds,\indt)$. Because thus both $(1,\Sexchanged(\inds))\lexorderRstrict(\indt,\Sexchanged(\recvar+1))$ and $(\Sexchanged(\inds),1)\lexorderRstrict(\Sexchanged(\recvar+1),\indt)$ Lem\-ma~\ref{lemma:base_monomial_order} proves $\Sbv{\anymatAT}{\Sexchanged(\inds)}{\indt}\theorderRstrict\Sbv{\anymatAT}{\Sexchanged(\recvar+1)}{1}$, which makes $\Sbv{\anymatAT}{\Sexchanged(\recvar+1)}{1}\Smonoidalproduct\anymatBX{1}{\recvar+1}\anymatATX{\Sexchanged(\indi)}{\thedim}$ the tip.
  \par
  Because moreover by Proposition~\ref{proposition:combinatorial_splitting} 
  \begin{IEEEeqnarray*}{rCl}
   \textstyle\thereconcatenator{1}( \Sbvfull{\anymatAT}{1}{\Sexchanged(\recvar+1)}{1},\anymatBX{1}{\recvar+1}\anymatATX{\Sexchanged(\indi)}{\thedim})&=&(\Sbvfull{\anymatAT}{2}{\Sexchanged(\recvar+1)}{\Sexchanged(\recvar+1)},\anymatATX{\Sexchanged(\indi)}{\thedim})
  \end{IEEEeqnarray*}
  the equality in \eqref{eq:third_order_last_two_cases_common_0} is preserved if $\Sbv{\anymatAT}{\Sexchanged(\recvar+1)}{\Sexchanged(\recvar+1)}\Smonoidalproduct\anymatATX{\Sexchanged(\indi)}{\thedim}$ is added on the right-hand side to the exterior of $\thesplitting{1}$ and $-\thedifferential{2}(\Sbv{\anymatAT}{\Sexchanged(\recvar+1)}{\Sexchanged(\recvar+1)}\Smonoidalproduct\anymatATX{\Sexchanged(\indi)}{\thedim})$ to the interior.
  \par
  Evidently, outside of $\thesplitting{1}$ this has the same effect on \eqref{eq:third_order_last_two_cases_common_0} as replacing $\recvar$ by $\recvar+1$ would. Because $(\Sexchanged(\recvar+1),\Sexchanged(\recvar+1))\neq (\thedim,\thedim)$ by $1\leq \recvar$ Proposition~\ref{proposition:second_order_generic_case} implies
  \begin{IEEEeqnarray*}{rCl}
    \IEEEeqnarraymulticol{3}{l}{ -\thedifferential{2}(\Sbv{\anymatAT}{\Sexchanged(\recvar+1)}{\Sexchanged(\recvar+1)}\Smonoidalproduct\anymatATX{\Sexchanged(\indi)}{\thedim})      }\\
\hspace{1em}&=&\textstyle-(\sum_{\indt=1}^\thedim\Sbv{\anymatAT}{\Sexchanged(\recvar+1)}{\indt}\Smonoidalproduct\anymatBX{\indt}{\recvar+1}+\Sbv{\anymatB}{\recvar+1}{\Sexchanged(\recvar+1)}\Smonoidalproduct\theone)\anymatATX{\Sexchanged(\indi)}{\thedim}\\
&=&\textstyle-\Sbv{\anymatAT}{\Sexchanged(\recvar+1)}{1}\Smonoidalproduct\anymatBX{1}{\recvar+1}\anymatATX{\Sexchanged(\indi)}{\thedim}-\sum_{\indt=2}^\thedim\Sbv{\anymatAT}{\Sexchanged(\recvar+1)}{\indt}\Smonoidalproduct\anymatBX{\indt}{\recvar+1}\anymatATX{\Sexchanged(\indi)}{\thedim}\IEEEeqnarraynumspace\IEEEyesnumber\label{eq:third_order_last_two_cases_common_2}\\
&&\hspace{16em}\textstyle{}-\Sbv{\anymatB}{\recvar+1}{\Sexchanged(\recvar+1)}\Smonoidalproduct\anymatATX{\Sexchanged(\indi)}{\thedim}.
\end{IEEEeqnarray*}
Hence, a comparison with \eqref{eq:third_order_last_two_cases_common_0} shows that adding the term \eqref{eq:third_order_last_two_cases_common_2} to the argument of $\thesplitting{1}$ in \eqref{eq:third_order_last_two_cases_common_0} is the same as increasing $\recvar$ by $1$ there. That verifies the claim for any $\recvar$.
\end{proof}

\begin{lemma}
  \label{lemma:third_order_last_two_cases_common_final}
  For any $\anymat\in \thematrices$ and any  $\indi\in\SYnumbers{\thedim}$,
    \begin{IEEEeqnarray*}{rCl}
    \IEEEeqnarraymulticol{3}{l}{
      \thedifferential{3}(\Sbv{\anymatA}{\thedim}{\indi}\Smonoidalproduct\theone)
    }\\
    \hspace{1.em}   
&=&\textstyle\sum_{\inds=1}^{\thedim}\Sbv{\anymatA}{\thedim}{\inds}\Smonoidalproduct\anymatAX{\inds}{\Sexchanged(\indi)}+\sum_{\inds=2}^{\thedim-1}\Sbv{\anymatAT}{\inds}{\inds}\Smonoidalproduct\anymatATX{\Sexchanged(\indi)}{\thedim}\IEEEyesnumber
    \label{eq:third_order_last_two_cases_common_final_0}\\
    &&\textstyle{}+\thesplitting{1}\big(\Sbv{\anymatAT}{1}{1}\Smonoidalproduct\anymatBX{1}{\thedim}\anymatATX{\Sexchanged(\indi)}{\thedim}+\sum_{\indt=2}^\thedim\Sbv{\anymatAT}{1}{\indt}\Smonoidalproduct\anymatBX{\indt}{\thedim}\anymatATX{\Sexchanged(\indi)}{\thedim}\\
    &&\textstyle\hspace{2.25em}{}+\Sbv{\anymatB}{\thedim}{1}\Smonoidalproduct\anymatATX{\Sexchanged(\indi)}{\thedim}-\sum_{\inds=1}^\thedim \Sbv{\anymatBT}{1}{\inds}\Smonoidalproduct\anymatAX{\inds}{\Sexchanged(\indi)}-\Sbv{\anymat}{\thedim}{\indi}\Smonoidalproduct\theone\big).
  \end{IEEEeqnarray*} 
\end{lemma}
\begin{proof}
Because $\sum_{\inds=2}^{\thedim-1}\Sbv{\anymatAT}{\Sexchanged(\inds)}{\Sexchanged(\inds)}\Smonoidalproduct\anymatATX{\Sexchanged(\indi)}{\thedim}=\sum_{\inds=2}^{\thedim-1}\Sbv{\anymatAT}{\inds}{\inds}\Smonoidalproduct\anymatATX{\Sexchanged(\indi)}{\thedim}$ and because by Lem\-ma~\ref{lemma:evil_identities}
  \begin{IEEEeqnarray*}{rCl}
     \textstyle-\sum_{\inds=1}^\thedim\Sbv{\anymatB}{\Sexchanged(\inds)}{\thedim}\Smonoidalproduct\anymatATX{\Sexchanged(\indi)}{\inds}-\Sbv{\anymatAT}{\Sexchanged(\indi)}{1}\Smonoidalproduct\theone&=&    \textstyle-\sum_{\inds=1}^\thedim \Sbv{\anymatBT}{1}{\inds}\Smonoidalproduct\anymatAX{\inds}{\Sexchanged(\indi)}-\Sbv{\anymat}{\thedim}{\indi}\Smonoidalproduct\theone
  \end{IEEEeqnarray*}
the claim is the conclusion \eqref{eq:third_order_last_two_cases_common_0} of  Lem\-ma~\ref{lemma:third_order_last_two_cases_common} in the special case $\recvar=\thedim-1$. 
\end{proof}

\subsubsection{Third case}
Again, one of the two remaining cases is quicker to deal with than the other, requiring only a single step.
\begin{lemma}
  \label{lemma:third_order_third_case_helper}
  For any $\anymat\in\thematrices$ and any $\indi\in\SYnumbers{\thedim}$, if $\indi\neq \thedim$, then
    \begin{IEEEeqnarray*}{rCl}
      \thedifferential{3}(\Sbv{\anymatA}{\thedim}{\indi}\Smonoidalproduct\theone)
&=&\textstyle\sum_{\inds=1}^{\thedim}\Sbv{\anymatA}{\thedim}{\inds}\Smonoidalproduct\anymatAX{\inds}{\Sexchanged(\indi)}+\sum_{\inds=1}^{\thedim-1}\Sbv{\anymatAT}{\inds}{\inds}\Smonoidalproduct\anymatATX{\Sexchanged(\indi)}{\thedim}\IEEEeqnarraynumspace\IEEEyesnumber
    \label{eq:third_order_third_case_helper_0}\\
    &&\textstyle{}+\thesplitting{1}\big({-}\sum_{\inds=1}^\thedim\Sbv{\anymatBT}{1}{\inds}\Smonoidalproduct\anymatAX{\inds}{\Sexchanged(\indi)}-\Sbv{\anymatA}{\thedim}{\indi}\Smonoidalproduct\theone\big).
  \end{IEEEeqnarray*} 
\end{lemma}
\begin{proof}
  Because $\indi\neq \thedim$ 
  the vector $\Sbv{\anymatAT}{1}{1}\Smonoidalproduct\anymatBX{1}{\thedim}\anymatATX{\Sexchanged(\indi)}{\thedim}$ is normal by Lem\-ma~\ref{characterization_of_reduced_terms_of_order_two}. It is the tip of  the argument of $\thesplitting{1}$ in \eqref{eq:third_order_last_two_cases_common_final_0} in Lem\-ma~\ref{lemma:third_order_last_two_cases_common_final} for the following reasons. Any term in the fourth line of \eqref{eq:third_order_last_two_cases_common_final_0} is dominated already by degree. And for any $\indt\in \SYnumbers{\thedim}$ with $1<\indt$, i.e., $\Sexchanged(\indt)<\thedim$, the fact that both
  $(1,1)\lexorderRstrict(1,\indt)$ and $(1,1)\lexorderRstrict(\indt,1)$
  demands $\Sbv{\anymatAT}{1}{\indt}\theorderRstrict\Sbv{\anymatAT}{1}{1}$ by Lemma~\ref{lemma:base_monomial_order}. Hence, $\Sbv{\anymatAT}{1}{1}\Smonoidalproduct\anymatBX{1}{\thedim}\anymatATX{\Sexchanged(\indi)}{\thedim}$ subjugates also any other term in the third line of \eqref{eq:third_order_last_two_cases_common_final_0}.
  \par
  Given that by Proposition~\ref{proposition:combinatorial_splitting}
  \begin{IEEEeqnarray*}{rCl}
\textstyle\thereconcatenator{1}(    \Sbvfull{\anymatAT}{1}{1}{1},\anymatBX{1}{\thedim}\anymatATX{\Sexchanged(\indi)}{\thedim})=(\Sbvfull{\anymatAT}{2}{1}{1},\anymatATX{\Sexchanged(\indi)}{\thedim}),
  \end{IEEEeqnarray*}
  adding $\Sbv{\anymatAT}{1}{1}\Smonoidalproduct\anymatATX{\Sexchanged(\indi)}{\thedim}$ to the outside of $\thesplitting{1}$ in \eqref{eq:third_order_last_two_cases_common_final_0} and $-\thedifferential{2}(\Sbv{\anymatAT}{1}{1}\Smonoidalproduct\anymatATX{\Sexchanged(\indi)}{\thedim})$ to the inside does not affect the equality.
  \par
This addition turns the outside of $\thesplitting{1}$ in  \eqref{eq:third_order_last_two_cases_common_final_0}  into that of \eqref{eq:third_order_third_case_helper_0}. Since Proposition~\ref{proposition:second_order_generic_case} implies
  \begin{IEEEeqnarray*}{rCl}
-\thedifferential{2}(\Sbv{\anymatAT}{1}{1}\Smonoidalproduct\anymatATX{\Sexchanged(\indi)}{\thedim})      
&=&\textstyle-\sum_{\indt=1}^\thedim\Sbv{\anymatAT}{1}{\indt}\Smonoidalproduct\anymatBX{\indt}{\thedim}\anymatATX{\Sexchanged(\indi)}{\thedim}-\Sbv{\anymatB}{\thedim}{1}\Smonoidalproduct\anymatATX{\Sexchanged(\indi)}{\thedim}\IEEEeqnarraynumspace\IEEEyesnumber\label{eq:third_order_third_case_helper_1}
\end{IEEEeqnarray*}
 adding \eqref{eq:third_order_third_case_helper_1} the inside of $\thesplitting{1}$ likewise makes the arguments of $\thesplitting{1}$ in  \eqref{eq:third_order_last_two_cases_common_final_0}  and \eqref{eq:third_order_third_case_helper_0} agree. That proves the claim. 
 \end{proof}

\begin{proposition}
  \label{proposition:third_order_third_case}
  For any $\anymat\in\thematrices$ and any $\indi\in\SYnumbers{\thedim}$, if $\indi\neq \thedim$, then
    \begin{IEEEeqnarray*}{rCl}
    \IEEEeqnarraymulticol{3}{l}{
      \thedifferential{3}(\Sbv{\anymatA}{\thedim}{\indi}\Smonoidalproduct\theone)
    }\\
    \hspace{2.em}   
&=&\textstyle\sum_{\inds=1}^{\thedim}\Sbv{\anymatA}{\thedim}{\inds}\Smonoidalproduct\anymatAX{\inds}{\Sexchanged(\indi)}+\sum_{\inds=1}^{\thedim-1}\Sbv{\anymatAT}{\inds}{\inds}\Smonoidalproduct\anymatATX{\Sexchanged(\indi)}{\thedim}-\Sbv{\anymatBT}{1}{\indi}\Smonoidalproduct\theone.\IEEEeqnarraynumspace\IEEEyesnumber
    \label{eq:third_order_third_case_0}
  \end{IEEEeqnarray*}
\end{proposition}
\begin{proof}
  The argument of $\thesplitting{1}$ on the right-hand side of \eqref{eq:third_order_third_case_helper_0} in Lem\-ma~\ref{lemma:third_order_third_case_helper} is obviously entirely normal. Its tip is $\Sbv{\anymatBT}{1}{1}\Smonoidalproduct\anymatAX{1}{\Sexchanged(\indi)}$, which dominates $\Sbv{\anymat}{\thedim}{\indi}\Smonoidalproduct\theone$ by degree and which, according to Lem\-ma~\ref{lemma:base_monomial_order}, is greater than $\Sbv{\anymatBT}{1}{\inds}\Smonoidalproduct\anymatAX{\inds}{\Sexchanged(\indi)}$ for any $\inds\in\SYnumbers{\thedim}$ with $1<\inds$
  because $(1,1)\lexorderRstrict(1,\inds)$ and $(1,1)\lexorderRstrict(\inds,1)$,
  and thus already $\Sbv{\anymatBT}{1}{\inds}\theorderRstrict\Sbv{\anymatBT}{1}{1}$.
  \par
  In consequence, since by Proposition~\ref{proposition:combinatorial_splitting}
  \begin{IEEEeqnarray*}{rCl} 
\textstyle\thereconcatenator{1}(    \Sbvfull{\anymatBT}{1}{1}{1},\anymatAX{1}{\Sexchanged(\indi)})=(\Sbvfull{\anymatBT}{2}{1}{\indi},\theone)
  \end{IEEEeqnarray*}
  the veracity of \eqref{eq:third_order_third_case_helper_0} is not affected by adding $-\Sbv{\anymatBT}{1}{\indi}\Smonoidalproduct\theone$ to the outside of $\thesplitting{1}$ and $\thedifferential{2}(\Sbv{\anymatBT}{1}{\indi}\Smonoidalproduct\theone)$ to the inside. Because  $\thedifferential{2}(\Sbv{\anymatBT}{1}{\indi}\Smonoidalproduct\theone)$ by Proposition~\ref{proposition:second_order_generic_case} is precisely the negative of the argument of $\thesplitting{1}$ in \eqref{eq:third_order_third_case_helper_0} that proves \eqref{eq:third_order_third_case_0}.
\end{proof}

\subsubsection{Fourth case}
It only remains to treat a single case, in $\thedim+2$ steps.
\begin{lemma}
  \label{lemma:third_order_fourth_case_helper}
  For any $\anymat\in\thematrices$ and any $\recvar\in\SYnumbers{\thedim}$,
      \begin{IEEEeqnarray*}{rCl}
    \IEEEeqnarraymulticol{3}{l}{
      \thedifferential{3}(\Sbv{\anymatA}{\thedim}{\thedim}\Smonoidalproduct\theone)
    }\\
    \hspace{.em}   
&=&\textstyle\sum_{\inds=1}^{\thedim}\Sbv{\anymatA}{\thedim}{\inds}\Smonoidalproduct\anymatAX{\inds}{1}+\sum_{\inds=2}^{\thedim-1}\Sbv{\anymatAT}{\inds}{\inds}\Smonoidalproduct\anymatATX{1}{\thedim}-\sum_{\inds=2}^{\recvar}\Sbv{\anymatAT}{1}{\inds}\Smonoidalproduct\anymatATX{\inds}{\thedim}\IEEEeqnarraynumspace\IEEEyesnumber
    \label{eq:third_order_fourth_case_helper_0}\\
    &&\textstyle{}+\thesplitting{1}\big({-}\sum_{\inds=\recvar+1}^\thedim\Sbv{\anymatAT}{1}{1}\Smonoidalproduct\anymatBX{1}{\Sexchanged(\inds)}\anymatATX{\inds}{\thedim}+\sum_{\inds=1}^{\recvar}\sum_{\indt=2}^\thedim\Sbv{\anymatAT}{1}{\indt}\Smonoidalproduct\anymatBX{\indt}{\Sexchanged(\inds)}\anymatATX{\inds}{\thedim}\\
    &&\textstyle\hfill{}+\sum_{\inds=1}^{\recvar}\Sbv{\anymatB}{\thedim}{\inds}\Smonoidalproduct\anymatATX{\inds}{\thedim}-\sum_{\inds=1}^\thedim\Sbv{\anymatBT}{1}{\inds}\Smonoidalproduct\anymatAX{\inds}{1}\big).
  \end{IEEEeqnarray*} 
\end{lemma}
\begin{proof}
  The claim is proved by induction over $\recvar$.
  \par
  \emph{Induction base.} In the base case $\recvar=1$ the asserted identity \eqref{eq:third_order_fourth_case_helper_0} reads
      \begin{IEEEeqnarray*}{rCl}
\thedifferential{3}(\Sbv{\anymatA}{\thedim}{\thedim}\Smonoidalproduct\theone)&=&\textstyle\sum_{\inds=1}^{\thedim}\Sbv{\anymatA}{\thedim}{\inds}\Smonoidalproduct\anymatAX{\inds}{1}+\sum_{\inds=2}^{\thedim-1}\Sbv{\anymatAT}{\inds}{\inds}\Smonoidalproduct\anymatATX{1}{\thedim}\IEEEeqnarraynumspace\IEEEyesnumber
    \label{eq:third_order_fourth_case_helper_1}\\    &&\textstyle{}+\thesplitting{1}\big({-}\sum_{\inds=2}^\thedim\Sbv{\anymatAT}{1}{1}\Smonoidalproduct\anymatBX{1}{\Sexchanged(\inds)}\anymatATX{\inds}{\thedim}+\sum_{\indt=2}^\thedim\Sbv{\anymatAT}{1}{\indt}\Smonoidalproduct\anymatBX{\indt}{\thedim}\anymatATX{1}{\thedim}\\
    &&\textstyle\hfill{}+\Sbv{\anymatB}{\thedim}{1}\Smonoidalproduct\anymatATX{1}{\thedim}-\sum_{\inds=1}^\thedim\Sbv{\anymatBT}{1}{\inds}\Smonoidalproduct\anymatAX{\inds}{1}\big).
  \end{IEEEeqnarray*} 
What  \eqref{eq:third_order_last_two_cases_common_final_0} in Lem\-ma~\ref{lemma:third_order_last_two_cases_common_final} says in the special case $\indi=\thedim$ is that
      \begin{IEEEeqnarray*}{rCl}
\thedifferential{3}(\Sbv{\anymatA}{\thedim}{\thedim}\Smonoidalproduct\theone)&=&\textstyle\sum_{\inds=1}^{\thedim}\Sbv{\anymatA}{\thedim}{\inds}\Smonoidalproduct\anymatAX{\inds}{1}+\sum_{\inds=2}^{\thedim-1}\Sbv{\anymatAT}{\inds}{\inds}\Smonoidalproduct\anymatATX{1}{\thedim}\IEEEyesnumber    \label{eq:third_order_fourth_case_helper_2}\\
    &&\textstyle{}+\thesplitting{1}\big(\Sbv{\anymatAT}{1}{1}\Smonoidalproduct\anymatBX{1}{\thedim}\anymatATX{1}{\thedim}+\sum_{\indt=2}^\thedim\Sbv{\anymatAT}{1}{\indt}\Smonoidalproduct\anymatBX{\indt}{\thedim}\anymatATX{1}{\thedim}\\
    &&\textstyle\hspace{2.25em}{}+\Sbv{\anymatB}{\thedim}{1}\Smonoidalproduct\anymatATX{1}{\thedim}-\sum_{\inds=1}^\thedim\Sbv{\anymatBT}{1}{\inds}\Smonoidalproduct\anymatAX{\inds}{1}-\Sbv{\anymatA}{\thedim}{\thedim}\Smonoidalproduct\theone\big).
  \end{IEEEeqnarray*}
Therefore, it suffices to prove that the arguments of $\thesplitting{1}$ in \eqref{eq:third_order_fourth_case_helper_1} and \eqref{eq:third_order_fourth_case_helper_2} are equal.
  \par 
  Because $\Sbv{\anymatA}{\thedim}{\thedim}=\Sbv{\anymatAT}{1}{1}$ by Lem\-ma~\ref{lemma:evil_identities} and thus
  \begin{IEEEeqnarray*}{rCl}
    \Sbv{\anymatAT}{1}{1}\Smonoidalproduct\anymatBX{1}{\thedim}\anymatATX{1}{\thedim}-\Sbv{\anymatA}{\thedim}{\thedim}\Smonoidalproduct\theone&=&\textstyle\Sbv{\anymatAT}{1}{1}\Smonoidalproduct(-\sum_{\inds=2}^\thedim\anymatBX{1}{\Sexchanged(\inds)}\anymatATX{\inds}{\thedim}+\theone)-\Sbv{\anymatAT}{1}{1}\Smonoidalproduct\theone\\
    &=&\textstyle-\sum_{\inds=2}^\thedim \Sbv{\anymatAT}{1}{1}\Smonoidalproduct\anymatBX{1}{\Sexchanged(\inds)}\anymatATX{\inds}{\thedim}
  \end{IEEEeqnarray*}
   that is indeed the case. Hence, \eqref{eq:third_order_fourth_case_helper_1} is true.
  \par
  \emph{Induction step.}  Let $\recvar\in \SYnumbers{\thedim}$ be arbitrary with $\recvar\leq \thedim-1$ and with the property that \eqref{eq:third_order_fourth_case_helper_0} is satisfied. Then \eqref{eq:third_order_fourth_case_helper_0} still holds if $\recvar$ is replaced by $\recvar+1$ by the ensuing argument.
  \par
  Since $1\leq \recvar$ the monomial $\anymatBX{1}{\Sexchanged(\recvar+1)}\anymatATX{\recvar+1}{\thedim}$ is normal by Lemma~\ref{characterization_of_reduced_terms_of_order_two} and the vector $\Sbv{\anymatAT}{1}{1}\Smonoidalproduct\anymatBX{1}{\Sexchanged(\recvar+1)}\anymatATX{\recvar+1}{\thedim}$ is   the tip of the argument of $\thesplitting{1}$ in \eqref{eq:third_order_fourth_case_helper_0} for the following reasons. Any vector in the fourth line of \eqref{eq:third_order_fourth_case_helper_0} is less already by degree. The only terms against which to compare therefore live in the third line and are of the form $\Sbv{\anymatAT}{1}{\indt}\Smonoidalproduct\anymatBX{\indt}{\Sexchanged(\inds)}\anymatATX{\inds}{\thedim}$ for $\{\inds,\indt\}\subseteq \SYnumbers{\thedim}$ with either $1<\indt$ or both $\indt=1$ and $\recvar+1<\inds$.
  If $1<\indt$, then the fact that both $(1,1)\lexorderRstrict(1,\indt)$ and $(1,1)\lexorderRstrict(\indt,1)$ implies $\Sbv{\anymatAT}{1}{\indt}\theorderRstrict\Sbv{\anymatAT}{1}{1}$ by Lem\-ma~\ref{lemma:base_monomial_order}  and thus $\Sbv{\anymatAT}{1}{\indt}\Smonoidalproduct\anymatBX{\indt}{\Sexchanged(\inds)}\anymatATX{\inds}{\thedim}\thechainsorderRstrict{1}\Sbv{\anymatAT}{1}{1}\Smonoidalproduct\anymatBX{1}{\Sexchanged(\recvar+1)}\anymatATX{\recvar+1}{\thedim}$. Alternatively, if $1=\indt$ and $\recvar+1<\inds$, which is to say, $\Sexchanged(\inds)<\Sexchanged(\recvar+1)$, then it is the addendum to Lemma~\ref{lemma:monomial_order_reformulation} that guarantees $\anymatBX{1}{\Sexchanged(\inds)}\theorderRstrict\anymatBX{1}{\Sexchanged(\recvar+1)}$  and thus $\Sbv{\anymatAT}{1}{1}\Smonoidalproduct\anymatBX{1}{\Sexchanged(\inds)}\anymatATX{\inds}{\thedim}\thechainsorderRstrict{1}\Sbv{\anymatAT}{1}{1}\Smonoidalproduct\anymatBX{1}{\Sexchanged(\recvar+1)}\anymatATX{\recvar+1}{\thedim}$ as well. 
\par
Hence, and because by Proposition~\ref{proposition:combinatorial_splitting}
\begin{IEEEeqnarray*}{rCl}
  \textstyle\thereconcatenator{1}(\Sbvfull{\anymatAT}{1}{1}{1},\anymatBX{1}{\Sexchanged(\recvar+1)}\anymatATX{\recvar+1}{\thedim})&=&(\Sbvfull{\anymatAT}{2}{1}{\recvar+1},\anymatATX{\recvar+1}{\thedim})
\end{IEEEeqnarray*}
adding to the right-hand side of \eqref{eq:third_order_fourth_case_helper_0} outside of $\thesplitting{1}$ the term $-\Sbv{\anymatAT}{1}{\recvar+1}\Smonoidalproduct\anymatATX{\recvar+1}{\thedim}$ and inside the term $\thedifferential{2}(\Sbv{\anymatAT}{1}{\recvar+1}\Smonoidalproduct\anymatATX{\recvar+1}{\thedim})$ gives another true identity.
  \par
  Evidently, this has the same effect on  the outside of $\thesplitting{1}$  in \eqref{eq:third_order_fourth_case_helper_0} as replacing $\recvar$ by $\recvar+1$ would have. Because by Proposition~\ref{proposition:second_order_generic_case}
  \begin{IEEEeqnarray*}{rCl}
    \IEEEeqnarraymulticol{3}{l}{ \thedifferential{2}(\Sbv{\anymatAT}{1}{\recvar+1}\Smonoidalproduct\anymatATX{\recvar+1}{\thedim})
    }\\
    \hspace{5em}&=&\textstyle\sum_{\indt=1}^\thedim\Sbv{\anymatAT}{1}{\indt}\Smonoidalproduct\anymatBX{\indt}{\Sexchanged(\recvar+1)}\anymatATX{\recvar+1}{\thedim}+\Sbv{\anymatB}{\thedim}{\recvar+1}\Smonoidalproduct\anymatATX{\recvar+1}{\thedim}\IEEEyesnumber\label{eq:third_order_fourth_case_helper_3}
  \end{IEEEeqnarray*}
the addition of \eqref{eq:third_order_fourth_case_helper_3} likewise changes  the argument of $\thesplitting{1}$ in \eqref{eq:third_order_fourth_case_helper_0} in the same way as increasing $\recvar$ by $1$ would. Hence, the claim is true for all $\recvar$.
\end{proof}

\begin{lemma}
  \label{lemma:third_order_fourth_case_second_helper}
  For any $\anymat\in\thematrices$,
  \begin{IEEEeqnarray*}{rCl}
\thedifferential{3}(\Sbv{\anymat}{\thedim}{\thedim}\Smonoidalproduct\theone)&=&\textstyle\sum_{\inds=1}^\thedim\Sbv{\anymat}{\thedim}{\inds}\Smonoidalproduct\anymatAX{\inds}{1}+\sum_{\inds=2}^{\thedim-1}\Sbv{\anymatAT}{\inds}{\inds}\Smonoidalproduct\anymatATX{1}{\thedim}\IEEEyesnumber\label{eq:third_order_fourth_case_second_helper_0}\\    &&\textstyle\hspace{8em}{}-\sum_{\inds=2}^\thedim\Sbv{\anymatAT}{1}{\inds}\Smonoidalproduct\anymatATX{\inds}{\thedim}+\Sbv{\anymatB}{\thedim}{1}\Smonoidalproduct\theone\\    &&\textstyle{}+\thesplitting{1}\big({-}\sum_{\inds=1}^\thedim\Sbv{\anymatBT}{1}{\inds}\Smonoidalproduct\anymatAX{\inds}{1}-\Sbv{\anymatA}{\thedim}{\thedim}\Smonoidalproduct\theone\big).
  \end{IEEEeqnarray*}
\end{lemma}
\begin{proof}
  In the special case $\recvar=\thedim$ the first grouped sum in the third line of \eqref{eq:third_order_fourth_case_helper_0} in  Lem\-ma~\ref{lemma:third_order_fourth_case_helper} vanishes and, taking into  account the fact that $\sum_{\inds=1}^{\thedim}\anymatBX{\indt}{\Sexchanged(\inds)}\anymatATX{\inds}{\thedim}=\Skronecker{\indt}{1}\Saction\theone$ implies
  \begin{IEEEeqnarray*}{rCl}
   \textstyle \sum_{\inds=1}^{\thedim}\sum_{\indt=2}^\thedim\Sbv{\anymatAT}{1}{\indt}\Smonoidalproduct\anymatBX{\indt}{\Sexchanged(\inds)}\anymatATX{\inds}{\thedim}=0,
  \end{IEEEeqnarray*}
  what remains of \eqref{eq:third_order_fourth_case_helper_0} is the identity
      \begin{IEEEeqnarray*}{rCl}
    \IEEEeqnarraymulticol{3}{l}{
      \thedifferential{3}(\Sbv{\anymatA}{\thedim}{\thedim}\Smonoidalproduct\theone)
    }\\
    \hspace{.em}   
&=&\textstyle\sum_{\inds=1}^{\thedim}\Sbv{\anymatA}{\thedim}{\inds}\Smonoidalproduct\anymatAX{\inds}{1}+\sum_{\inds=2}^{\thedim-1}\Sbv{\anymatAT}{\inds}{\inds}\Smonoidalproduct\anymatATX{1}{\thedim}-\sum_{\inds=2}^{\thedim}\Sbv{\anymatAT}{1}{\inds}\Smonoidalproduct\anymatATX{\inds}{\thedim}\IEEEeqnarraynumspace\IEEEyesnumber
    \label{eq:third_order_fourth_case_second_helper_1}\\
    &&\textstyle{}+\thesplitting{1}\big(\sum_{\inds=1}^{\thedim}\Sbv{\anymatB}{\thedim}{\inds}\Smonoidalproduct\anymatATX{\inds}{\thedim}-\sum_{\inds=1}^\thedim\Sbv{\anymatBT}{1}{\inds}\Smonoidalproduct\anymatAX{\inds}{1}\big).
  \end{IEEEeqnarray*}
  \par
  The tip of the argument of $\thesplitting{1}$ in \eqref{eq:third_order_fourth_case_second_helper_1} is $\Sbv{\anymatB}{\thedim}{1}\Smonoidalproduct\anymatATX{1}{\thedim}$. Indeed, on the one hand, for any $\inds\in \SYnumbers{\thedim}$ with $1<\inds$ already $\Sbv{\anymatB}{\thedim}{\inds}\theorderRstrict\Sbv{\anymatB}{\thedim}{1}$ by the addendum to Lem\-ma~\ref{lemma:base_monomial_order}, and thus $\Sbv{\anymatB}{\thedim}{\inds}\Smonoidalproduct\anymatATX{\inds}{\thedim}\thechainsorderRstrict{1 }\Sbv{\anymatB}{\thedim}{1}\Smonoidalproduct\anymatATX{1}{\thedim}$. On the other hand, for any $\inds\in\SYnumbers{\thedim}$, even $\inds=1$ by $2\leq \thedim$, the facts that both $(1,1)=(\Sexchanged(\thedim),1)\lexorderRstrict(\inds,\Sexchanged(1))=(\inds,\thedim)$ and $(1,1)=(1,\Sexchanged(\thedim))\lexorderRstrict(\Sexchanged(1),\inds)=(\thedim,\inds)$ imply $\Sbv{\anymatBT}{1}{\inds}\theorderRstrict\Sbv{\anymatB}{\thedim}{1}$  by Lem\-ma~\ref{lemma:base_monomial_order} and consequently $\Sbv{\anymatBT}{1}{\inds}\Smonoidalproduct\anymatAX{\inds}{1}\thechainsorderRstrict{1 }\Sbv{\anymatB}{\thedim}{1}\Smonoidalproduct\anymatATX{1}{\thedim}$.
  \par
  Therefore and since by Proposition~\ref{proposition:combinatorial_splitting}
  \begin{IEEEeqnarray*}{rCl}
\textstyle    \thereconcatenator{1}(\Sbvfull{\anymatB}{1}{\thedim}{1},\anymatATX{1}{\thedim})=(\Sbvfull{\anymatB}{2}{\thedim}{1},\theone)
  \end{IEEEeqnarray*}
  adding on the right-hand side of \eqref{eq:third_order_fourth_case_second_helper_1} the term $\Sbv{\anymatB}{\thedim}{1}\Smonoidalproduct\theone$ to the exterior of $\thesplitting{1}$ and $-\thedifferential{2}(\Sbv{\anymatB}{\thedim}{1}\Smonoidalproduct\theone)$ to the interior preserves the equality.
  \par
  Of course, this change is precisely what is necessary to turn the outside of $\thesplitting{1}$ in \eqref{eq:third_order_fourth_case_second_helper_1} into that of \eqref{eq:third_order_fourth_case_second_helper_0}.
  At the same time, since by Proposition~\ref{proposition:second_order_generic_case}
  \begin{IEEEeqnarray*}{rCl}
    -\thedifferential{2}(\Sbv{\anymatB}{\thedim}{1}\Smonoidalproduct\theone)&=&\textstyle-\sum_{\inds=1}^\thedim \Sbv{\anymatB}{\thedim}{\inds}\Smonoidalproduct\anymatATX{\inds}{\thedim}-\Sbv{\anymatAT}{1}{1}\Smonoidalproduct\theone\IEEEyesnumber\label{eq:third_order_fourth_case_second_helper_2}
  \end{IEEEeqnarray*}
and since $-\Sbv{\anymatAT}{1}{1}\Smonoidalproduct\theone=-\Sbv{\anymatA}{\thedim}{\thedim}\Smonoidalproduct\theone$ by Lemma~\ref{lemma:evil_identities} the addition of the vector \eqref{eq:third_order_fourth_case_second_helper_2}  likewise transforms the argument of $\thesplitting{1}$ in \eqref{eq:third_order_fourth_case_second_helper_1} into that of \eqref{eq:third_order_fourth_case_second_helper_0}, which proves the latter identity.
\end{proof}

\begin{proposition}
  \label{proposition:third_order_fourth_case}
  For any $\anymat\in\thematrices$,
  \begin{IEEEeqnarray*}{rCl}
    \IEEEeqnarraymulticol{3}{l}{
      \thedifferential{3}(\Sbv{\anymat}{\thedim}{\thedim}\Smonoidalproduct\theone)
    }\\    \hspace{.5em}&=&\textstyle\sum_{\inds=1}^\thedim\Sbv{\anymat}{\thedim}{\inds}\Smonoidalproduct\anymatAX{\inds}{1}+\sum_{\inds=2}^{\thedim-1}\Sbv{\anymatAT}{\inds}{\inds}\Smonoidalproduct\anymatATX{1}{\thedim}-\sum_{\inds=2}^\thedim\Sbv{\anymatAT}{1}{\inds}\Smonoidalproduct\anymatATX{\inds}{\thedim}\IEEEyesnumber\label{eq:third_order_fourth_case_0}\\    &&\hspace{17em}{}+\Sbv{\anymatB}{\thedim}{1}\Smonoidalproduct\theone-\Sbv{\anymatBT}{1}{\thedim}\Smonoidalproduct\theone.
  \end{IEEEeqnarray*}  
\end{proposition}
\begin{proof}
  The tip of the argument of $\thesplitting{1}$ in \eqref{eq:third_order_fourth_case_second_helper_0} in Lem\-ma~\ref{lemma:third_order_fourth_case_second_helper} is given by $\Sbv{\anymatBT}{1}{1}\Smonoidalproduct\anymatAX{1}{1}$ because already $\Sbv{\anymatBT}{1}{\inds}\theorderRstrict\Sbv{\anymatBT}{1}{1}$ for any $\inds\in\SYnumbers{\thedim}$ with $1<\inds$ by the addendum to Lem\-ma~\ref{lemma:base_monomial_order} and thus $\Sbv{\anymatBT}{1}{\inds}\Smonoidalproduct\anymatAX{\inds}{1}\thechainsorderRstrict{1}\Sbv{\anymatBT}{1}{1}\Smonoidalproduct\anymatAX{1}{1}$ and because $\Sbv{\anymat}{\thedim}{\thedim}\Smonoidalproduct\theone\thechainsorderRstrict{1}\Sbv{\anymatBT}{1}{1}\Smonoidalproduct\anymatAX{1}{1}$ for degree reasons.
  \par
  In consequence, since by Proposition~\ref{proposition:combinatorial_splitting}
  \begin{IEEEeqnarray*}{rCl}
\textstyle\thereconcatenator{1}(    \Sbvfull{\anymatBT}{1}{1}{1},\anymatAX{1}{1})=(\Sbvfull{\anymatBT}{2}{1}{\thedim},\theone)
  \end{IEEEeqnarray*}
  the validity of \eqref{eq:third_order_fourth_case_second_helper_0} is not affected by adding
  $-\Sbv{\anymatBT}{1}{\thedim}\Smonoidalproduct\theone$ to the outside of $\thesplitting{1}$ on the right-hand side and $\thedifferential{2}(\Sbv{\anymatBT}{1}{\thedim}\Smonoidalproduct\theone)$ to the inside.
  \par
  Obviously that makes the new outside of $\thesplitting{1}$ in \eqref{eq:third_order_fourth_case_second_helper_0} agree with that in \eqref{eq:third_order_fourth_case_0}. And since by Proposition~\ref{proposition:second_order_generic_case}  the vector
  \begin{IEEEeqnarray*}{rCl}
    \thedifferential{2}(\Sbv{\anymatBT}{1}{\thedim}\Smonoidalproduct\theone)&=&\textstyle\sum_{\inds=1}^\thedim \Sbv{\anymatBT}{1}{\inds}\Smonoidalproduct\anymatAX{\inds}{1}+\Sbv{\anymat}{\thedim}{\thedim}\Smonoidalproduct\theone
  \end{IEEEeqnarray*}
 is precisely the negative of the argument of $\thesplitting{1}$  in \eqref{eq:third_order_fourth_case_second_helper_0} that verifies \eqref{eq:third_order_fourth_case_0} and thus concludes the proof.
\end{proof}

\subsubsection{Synthesis}
The individual results obtained in the four cases distinguished combine into the aggregated statement from the \hyperref[main-result]{Main result}.
\begin{theorem}
  \label{theorem:third_order}
  For any $\anymat\in\thematrices$ and any $(\indj,\indi)\in\SYnumbers{\thedim}^{\Ssetmonoidalproduct 2}$,
\begin{IEEEeqnarray*}{rCl}
  \IEEEeqnarraymulticol{3}{l}{
    \thedifferential{3}(\Sbv{\anymatA}{\indj}{\indi}\Smonoidalproduct \theone)
  }\\
\hspace{1em}  &= &\textstyle  \sum_{\inds=1}^\thedim\Sbv{\anymatA}{\indj}{\inds}\Smonoidalproduct\anymatAX{\inds}{\Sexchanged(\indi)}-\Saction\Sbv{\anymatBT}{\Sexchanged(\indj)}{\indi}\Smonoidalproduct\theone\IEEEyesnumber\label{eq:third_order_0}\\
&&\textstyle{}+\Skronecker{\indj}{\thedim}(\sum_{\inds=1}^{\thedim-1}\Sbv{\anymatAT}{\inds}{\inds}\Smonoidalproduct\anymatATX{\Sexchanged(\indi)}{\thedim}-\Skronecker{\indi}{\thedim}(\sum_{\inds=1}^\thedim\Sbv{\anymatAT}{1}{\inds}\Smonoidalproduct\anymatATX{\inds}{\thedim}-\Saction\Sbv{\anymatB}{\thedim}{1}\Smonoidalproduct\theone))\\
&&\textstyle{}-\Skronecker{\indj}{1}\Skronecker{\indi}{\thedim}\sum_{\inds=1}^{\thedim-1}\Sbv{\anymatB}{\inds}{\inds}\Smonoidalproduct\theone.
\end{IEEEeqnarray*}  
\end{theorem}
\begin{proof}
By the vanishing of the last two lines, in the case $\indj\neq \thedim$ and $(\indj,\indi)\neq (1,\thedim)$  the formula \eqref{eq:third_order_0} collapses into 
\eqref{eq:third_order_first_case_0} from Proposition~\ref{proposition:third_order_first_case}.
\par
If, still $\indj\neq \thedim$ but $(\indj,\indi)=(1,\thedim)$, while the third line of \eqref{eq:third_order_0} is zero, the fourth becomes 
$-\sum_{\inds=1}^{\thedim-1}\Sbv{\anymatB}{\inds}{\inds}\Smonoidalproduct\theone$, which is why in this case  \eqref{eq:third_order_0} is precisely
\eqref{eq:third_order_second_case_0} from Proposition~\ref{proposition:third_order_second_case}.
\par
In the instance that $\indj=\thedim\neq \indi$ the two last lines of \eqref{eq:third_order_0} amount to nothing more than $\sum_{\inds=1}^{\thedim-1}\Sbv{\anymatAT}{\inds}{\inds}\Smonoidalproduct\anymatATX{\Sexchanged(\indi)}{\thedim}$, which proves \eqref{eq:third_order_0} identical to 
\eqref{eq:third_order_third_case_0} from Proposition~\ref{proposition:third_order_third_case} in this case.
\par 
Finally, for $(\indj,\indi)=(\thedim,\thedim)$ the sum of all degree-$2$ terms in  \eqref{eq:third_order_0} is
\begin{IEEEeqnarray*}{rCl}
\textstyle  \sum_{\inds=1}^\thedim\Sbv{\anymatA}{\thedim}{\inds}\Smonoidalproduct\anymatAX{\inds}{1}
+\sum_{\inds=1}^{\thedim-1}\Sbv{\anymatAT}{\inds}{\inds}\Smonoidalproduct\anymatATX{1}{\thedim}-\sum_{\inds=1}^\thedim\Sbv{\anymatAT}{1}{\inds}\Smonoidalproduct\anymatATX{\inds}{\thedim}.
\IEEEyesnumber\label{eq:third_order_1}
\end{IEEEeqnarray*}
Because the term $\Sbv{\anymatAT}{1}{1}\Smonoidalproduct\anymatATX{1}{\thedim}$ appears in both the middle and right grouped sum in \eqref{eq:third_order_1}, but with different signs, \eqref{eq:third_order_1} actually agrees with the second line of 
\eqref{eq:third_order_fourth_case_0} in Proposition~\ref{proposition:third_order_fourth_case}. Now it is evident that \eqref{eq:third_order_fourth_case_0} is true in all cases. 
\end{proof}


\subsection{Fourth and higher orders}
The fourth and higher orders of the resolution can be computed in a uniform manner.
\begin{assumption}
  \label{assumption:higher_orders}
  Let $\orderind\in\Sintegersnn$ be such that  $4\leq \orderind$ and for any $\anymat\in\thematrices$ and  $(\indj,\indi)\in\SYnumbers{\thedim}^{\Ssetmonoidalproduct 2}$,
\begin{IEEEeqnarray*}{rCl}
    \thedifferential{\orderind-1}(\Sbv{\anymatA}{\indj}{\indi}\Smonoidalproduct \theone)
&= &\textstyle  \sum_{\inds=1}^\thedim\Sbv{\anymatA}{\indj}{\inds}\Smonoidalproduct\anymatPAX{\orderind-1}{\inds}{\Sexchanged(\indi)}+(-1)^{\orderind-1}\Saction\Sbv{\anymatBT}{\Sexchanged(\indj)}{\indi}\Smonoidalproduct\theone\\
&&\textstyle{}+\Skronecker{\indj}{\thedim}((\Sbv{\anymatAT}{1}{1}+\Skronecker{\orderind-1}{3}\sum_{\inds=2}^{\thedim-1}\Sbv{\anymatAT}{\inds}{\inds})\Smonoidalproduct\anymatPATX{\orderind-1}{\Sexchanged(\indi)}{\thedim}\\
&&\textstyle\hspace{3em}\hfill{}-\Skronecker{\indi}{\thedim}(\sum_{\inds=1}^\thedim\Sbv{\anymatAT}{1}{\inds}\Smonoidalproduct\anymatPATX{\orderind-1}{\inds}{\thedim}+(-1)^{\orderind-1}\Saction\Sbv{\anymatB}{\thedim}{1}\Smonoidalproduct\theone))\\
&&\textstyle{}+\Skronecker{\indj}{1}\Skronecker{\indi}{\thedim}(-1)^{\orderind-1}(\Sbv{\anymatB}{1}{1}+\Skronecker{\orderind-1}{3}\sum_{\inds=2}^{\thedim-1}\Sbv{\anymatB}{\inds}{\inds})\Smonoidalproduct\theone.
\end{IEEEeqnarray*}
\end{assumption}
Note that, if $\orderind=4$, this assumption holds by Theorem~\ref{theorem:third_order}.

\subsubsection{Case-independent identities} It is convenient to carry out $\thedim$ steps of the recursion for all basis vectors simultaneously and distinguish cases afterwards.
\begin{lemma}
  \label{lemma:higher_orders_common_helper}
  For any $\anymat\in \thematrices$, any $(\indj,\indi)\in\SYnumbers{\thedim}^{\Ssetmonoidalproduct 2}$ and any $\recvar\in\SYnumbers{\thedim}$,
  \begin{IEEEeqnarray*}{rCl}
      \IEEEeqnarraymulticol{3}{l}{\thedifferential{\orderind}(\Sbv{\anymatA}{\indj}{\indi}\Smonoidalproduct \theone)
  }\\
  \hspace{.5em}  &=&\textstyle\sum_{\inds=1}^\recvar\Sbv{\anymatA}{\indj}{\inds}\Smonoidalproduct\anymatPAX{\orderind}{\inds}{\Sexchanged(\indi)}\IEEEyesnumber\label{eq:higher_orders_common_helper_0}\\  &&\textstyle{}+\thesplitting{\orderind-2}\big(\sum_{\inds=\recvar+1}^{\thedim}\Sbv{\anymatA}{\indj}{1}\Smonoidalproduct\anymatPBTX{\orderind}{1}{\Sexchanged(\inds)}\anymatPAX{\orderind}{ \inds}{\Sexchanged(\indi)}\\
  &&\textstyle\hspace{6em}{}-\sum_{\inds=2}^\thedim\sum_{\indt=1}^{\recvar}\Sbv{\anymatA}{\indj}{\inds}\Smonoidalproduct\anymatPBTX{\orderind}{\inds}{\Sexchanged(\indt)}\anymatPAX{\orderind}{\indt}{\Sexchanged(\indi)}\\
  &&\textstyle\hspace{8em}{}+(-1)^{\orderind}\sum_{\inds=1}^\recvar\Sbv{\anymatBT}{\Sexchanged(\indj)}{\inds}\Smonoidalproduct\anymatPAX{\orderind}{\inds}{\Sexchanged(\indi)}-\Skronecker{\indi}{1}\Saction\Sbv{\anymatA}{\indj}{1}\Smonoidalproduct\theone\\
  &&\textstyle\hspace{3em}{}-\Skronecker{\indj}{\thedim}(\Sbv{\anymatAT}{1}{1}+\Skronecker{\orderind}{4}\sum_{\inds=2}^{\thedim-1}\Sbv{\anymatAT}{\inds}{\inds})\Smonoidalproduct\sum_{\inds=1}^\recvar\anymatPBX{\orderind}{\Sexchanged(\inds)}{\thedim}\anymatPATX{\orderind}{\Sexchanged(\indi)}{\inds}\\  &&\textstyle\hspace{3em}{}+\Szetafunction{\thedim}{\recvar}(\Skronecker{\indj}{\thedim}(\sum_{\inds=1}^\thedim\Sbv{\anymatAT}{1}{\inds}\Smonoidalproduct\anymatPBX{\orderind}{\inds}{\thedim}\anymatPATX{\orderind}{\Sexchanged(\indi)}{\thedim}-(-1)^{\orderind}\Saction\Sbv{\anymatB}{\thedim}{1}\Smonoidalproduct\anymatPATX{\orderind}{\Sexchanged(\indi)}{\thedim})\\
&&\textstyle\hspace{6em}{}+\Skronecker{\indj}{1}(-1)^\orderind(\Sbv{\anymatB}{1}{1}+\Skronecker{\orderind}{4}\sum_{\inds=2}^{\thedim-1}\Sbv{\anymatB}{\inds}{\inds})\Smonoidalproduct\anymatPATX{\orderind}{\Sexchanged(\indi)}{\thedim})\big).
  \end{IEEEeqnarray*}
\end{lemma}
\begin{proof}
  \newcommand{\dummyindex}{k}
  The proof goes by induction over $\recvar$. It is economic though to first note  that for any $\dummyindex\in\SYnumbers{\thedim}$ renaming the $\indi$ in Assumption~\ref{assumption:higher_orders} to $\dummyindex$ and multiplying the resulting identity with $-\anymatPAX{\orderind}{\dummyindex}{\Sexchanged(i)}$ from the right yields
\begin{IEEEeqnarray*}{rCl}
  \IEEEeqnarraymulticol{3}{l}{
    -    \thedifferential{\orderind-1}(\Sbv{\anymatA}{\indj}{\dummyindex}\Smonoidalproduct \anymatPAX{\orderind}{\dummyindex}{\Sexchanged(\indi)})
    }\\
\hspace{1em}&= &\textstyle
-\Sbv{\anymatA}{\indj}{1}\Smonoidalproduct\anymatPBTX{\orderind}{1}{\Sexchanged(\dummyindex)}\anymatPAX{\orderind}{\dummyindex}{\Sexchanged(\indi)}-\sum_{\inds=2}^\thedim\Sbv{\anymatA}{\indj}{\inds}\Smonoidalproduct\anymatPBTX{\orderind}{\inds}{\Sexchanged(\dummyindex)}\anymatPAX{\orderind}{\dummyindex}{\Sexchanged(\indi)}
\IEEEyesnumber\label{eq:higher_orders_common_helper_1}\\
&&\textstyle\hspace{15em}{}+(-1)^{\orderind}\Saction\Sbv{\anymatBT}{\Sexchanged(\indj)}{\dummyindex}\Smonoidalproduct\anymatPAX{\orderind}{\dummyindex}{\Sexchanged(\indi)}\IEEEeqnarraynumspace\\
&&\textstyle{}+\Skronecker{\indj}{\thedim}(-(\Sbv{\anymatAT}{1}{1}+\Skronecker{\orderind}{4}\sum_{\inds=2}^{\thedim-1}\Sbv{\anymatAT}{\inds}{\inds})\Smonoidalproduct\anymatPBX{\orderind}{\Sexchanged(\dummyindex)}{\thedim}\anymatPATX{\orderind}{\Sexchanged(\indi)}{\dummyindex}\\
&&\textstyle\hspace{3em}{}+\Skronecker{\dummyindex}{\thedim}(\sum_{\inds=1}^\thedim\Sbv{\anymatAT}{1}{\inds}\Smonoidalproduct\anymatPBX{\orderind}{\inds}{\thedim}\anymatPATX{\orderind}{\Sexchanged(\indi)}{\thedim}-(-1)^{\orderind}\Saction\Sbv{\anymatB}{\thedim}{1}\Smonoidalproduct\anymatPATX{\orderind}{\Sexchanged(\indi)}{\thedim}))\\
&&\textstyle{}+\Skronecker{\indj}{1}\Skronecker{\dummyindex}{\thedim}(-1)^{\orderind}(\Sbv{\anymatB}{1}{1}+\Skronecker{\orderind}{4}\sum_{\inds=2}^{\thedim-1}\Sbv{\anymatB}{\inds}{\inds})\Smonoidalproduct\anymatPATX{\orderind}{\Sexchanged(\indi)}{\thedim},
\end{IEEEeqnarray*}
 where,  also, the $\inds=1$ term has been splitt off from the first grouped sum on the right-hand side.
  \par
  \emph{Induction base.} Since $2\leq \thedim$ and thus $\Szetafunction{\thedim}{1}=0$, in the base case $\recvar=1$ what needs to be proved is that
  \begin{IEEEeqnarray*}{rCl}
    \IEEEeqnarraymulticol{3}{l}{
      \thedifferential{\orderind}(\Sbv{\anymatA}{\indj}{\indi}\Smonoidalproduct \theone)
    }\\    \hspace{1em}&=&\textstyle\Sbv{\anymatA}{\indj}{1}\Smonoidalproduct\anymatPAX{\orderind}{1}{\Sexchanged(\indi)}\IEEEyesnumber\label{eq:higher_orders_common_helper_2}\\  &&\textstyle{}+\thesplitting{\orderind-2}\big(\sum_{\inds=2}^{\thedim}\Sbv{\anymatA}{\indj}{1}\Smonoidalproduct\anymatPBTX{\orderind}{1}{\Sexchanged(\inds)}\anymatPAX{\orderind}{ \inds}{\Sexchanged(\indi)}-\sum_{\inds=2}^{\thedim}\Sbv{\anymatA}{\indj}{\inds}\Smonoidalproduct\anymatPBTX{\orderind}{\inds}{\Sexchanged(1)}\anymatPAX{\orderind}{1}{\Sexchanged(\indi)}\\
  &&\textstyle\hspace{3em}\hfill{}+(-1)^{\orderind}\Saction\Sbv{\anymatBT}{\Sexchanged(\indj)}{1}\Smonoidalproduct\anymatPAX{\orderind}{1}{\Sexchanged(\indi)}-\Skronecker{\indi}{1}\Saction\Sbv{\anymatA}{\indj}{1}\Smonoidalproduct\theone\\  &&\textstyle\hspace{3em}{}-\Skronecker{\indj}{\thedim}(\Sbv{\anymatAT}{1}{1}+\Skronecker{\orderind}{4}\sum_{\inds=2}^{\thedim-1}\Sbv{\anymatAT}{\inds}{\inds})\Smonoidalproduct\anymatPBX{\orderind}{\Sexchanged(1)}{\thedim}\anymatPATX{\orderind}{\Sexchanged(\indi)}{1}\big).
\end{IEEEeqnarray*}
\par 
Given that by Proposition~\ref{proposition:combinatorial_differential}
\begin{IEEEeqnarray*}{rClCrCl}
\textstyle  \thedeconcatenator{\orderind}(\Sbvfull{\anymatA}{\orderind}{\indj}{\indi})&=&(\Sbvfull{\anymatA}{\orderind-1}{\indj}{1}, \anymatPAX{\orderind}{1}{\Sexchanged(\indi)})
\end{IEEEeqnarray*}
the definition of $\thedifferential{\orderind}$ means in this case
  \begin{IEEEeqnarray*}{rCl}
\thedifferential{\orderind}(\Sbv{\anymatA}{\indj}{\indi}\Smonoidalproduct \theone)  &=&\Sbv{\anymatA}{\indj}{1}\Smonoidalproduct \anymatPAX{\orderind}{1}{\Sexchanged(\indi)}+\thesplitting{\orderind-2}(-\thedifferential{\orderind-1}(\Sbv{\anymatA}{\indj}{1}\Smonoidalproduct \anymatPAX{\orderind}{1}{\Sexchanged(\indi)})).
\end{IEEEeqnarray*}
Hence, it suffices to prove that $-\thedifferential{\orderind-1}(\Sbv{\anymatA}{\indj}{1}\Smonoidalproduct \anymatPAX{\orderind}{1}{\Sexchanged(\indi)})$ agrees with the argument of $\thesplitting{\orderind-2}$ in \eqref{eq:higher_orders_common_helper_2}.
By $2\leq \thedim$ the last two lines of \eqref{eq:higher_orders_common_helper_1} do not survive instantiation at $\dummyindex=1$, which thus yields
\begin{IEEEeqnarray*}{rCl}
-\thedifferential{\orderind-1}(\Sbv{\anymatA}{\indj}{1}\Smonoidalproduct \anymatPAX{\orderind}{1}{\Sexchanged(\indi)})&= &\textstyle -\Sbv{\anymatA}{\indj}{1}\Smonoidalproduct\anymatPBTX{\orderind}{1}{\Sexchanged(1)}\anymatPAX{\orderind}{1}{\Sexchanged(\indi)} -\sum_{\inds=2}^\thedim\Sbv{\anymatA}{\indj}{\inds}\Smonoidalproduct\anymatPBTX{\orderind}{\inds}{\Sexchanged(1)}\anymatPAX{\orderind}{1}{\Sexchanged(\indi)}\\
    &&\textstyle\hfill{}+(-1)^\orderind\Saction\Sbv{\anymatBT}{\Sexchanged(\indj)}{1}\Smonoidalproduct\anymatPAX{\orderind}{1}{\Sexchanged(\indi)}\\
&&\textstyle{}-\Skronecker{\indj}{\thedim}(\Sbv{\anymatAT}{1}{1}+\Skronecker{\orderind}{4}\sum_{\inds=2}^{\thedim-1}\Sbv{\anymatAT}{\inds}{\inds})\Smonoidalproduct\anymatPBX{\orderind}{\Sexchanged(1)}{\thedim}\anymatPATX{\orderind}{\Sexchanged(\indi)}{1}.
\end{IEEEeqnarray*}
Since $\anymatPBTX{\orderind}{1}{\Sexchanged(1)}\anymatPAX{\orderind}{1}{\Sexchanged(\indi)}=-\sum_{\inds=2}^{\thedim}\anymatPBTX{\orderind}{1}{\Sexchanged(\inds)}\anymatPAX{\orderind}{\inds}{\Sexchanged(\indi)}+\Skronecker{\indi}{1}\Saction\theone$ it is now visible that $    -\thedifferential{\orderind-1}(\Sbv{\anymatA}{\indj}{1}\Smonoidalproduct \anymatBTX{1}{\Sexchanged(\indi)})$ is precisely the argument of $\thesplitting{\orderind-2}$ in \eqref{eq:higher_orders_common_helper_2}. Hence, the claim holds for $\recvar=1$.
\par
\emph{Induction step.} Now, suppose $\recvar\in\SYnumbers{\thedim}$ is such that $\recvar\leq \thedim-1$ and such that \eqref{eq:higher_orders_common_helper_0} is true. It follows the proof that \eqref{eq:higher_orders_common_helper_0} still holds if $\recvar$ is replaced by $\recvar+1$.
\par
Since $1\leq \recvar$ the monomial $\anymatPBTX{\orderind}{1}{\Sexchanged(\recvar+1)}\anymatPAX{\orderind}{\recvar+1}{\Sexchanged(\indi)}$ is normal by Lemma~\ref{characterization_of_reduced_terms_of_order_two} and the vector $\Sbv{\anymatA}{\indj}{1}\Smonoidalproduct\anymatPBTX{\orderind}{1}{\Sexchanged(\recvar+1)}\anymatPAX{\orderind}{\recvar+1}{\Sexchanged(\indi)}$ is the tip of the argument of $\thesplitting{\orderind-2}$ in  \eqref{eq:higher_orders_common_helper_0} for the reasons given hereafter. It dominates the other terms in the grouped sum where it originates because $\anymatPBTX{\orderind}{1}{\Sexchanged(\inds)}\theorderRstrict\anymatPBTX{\orderind}{1}{\Sexchanged(\recvar+1)}$ for any $\inds\in\SYnumbers{\thedim}$ with $\recvar+1<\inds$  by the addendum to Lem\-ma~\ref{lemma:monomial_order_reformulation}. It is also greater than any of the terms in the grouped sum in the fourth line of \eqref{eq:higher_orders_common_helper_0} because $\Sbv{\anymatA}{\indj}{\indt}\theorderRstrict\Sbv{\anymatA}{\indj}{1}$ for any $\indt\in\SYnumbers{\thedim}$ with $1<\indt$ by the  addendum to Lem\-ma~\ref{lemma:base_monomial_order}. Any term in the fifth line of \eqref{eq:higher_orders_common_helper_0} is less already for degree reasons. Also, any term in the sixth line is dominated because, if it is non-zero, then $\indj=\thedim$ and hence $\Sbv{\anymatAT}{\inds}{\inds}<\Sbv{\anymatA}{\indj}{\thedim}$ for any $\inds\in\SYnumbers{\thedim}$ with $\inds\neq \thedim$ by  Lem\-ma~\ref{lemma:base_monomial_order}. 
\par
Because by Proposition~\ref{proposition:combinatorial_splitting}
\begin{IEEEeqnarray*}{rCl}
  \textstyle\thereconcatenator{\orderind-2}(\Sbvfull{\anymatA}{\orderind-2}{\indj}{1},\anymatPBTX{\orderind}{1}{\Sexchanged(\recvar+1)}\anymatPAX{\orderind}{\recvar+1}{\Sexchanged(\indi)})=(\Sbvfull{\anymatA}{\orderind-1}{\indj}{\recvar+1},\anymatPAX{\orderind}{\recvar+1}{\Sexchanged(\indi)})
\end{IEEEeqnarray*}
the veracity of \eqref{eq:higher_orders_common_helper_0} is not affected by adding on the right-hand side $\Sbv{\anymatA}{\indj}{\recvar+1}\Smonoidalproduct\anymatPAX{\orderind}{\recvar+1}{\Sexchanged(\indi)}$ outside of  $\thesplitting{\orderind-2}$ and $-\thedifferential{\orderind-1}(\Sbv{\anymatA}{\indj}{\recvar+1}\Smonoidalproduct\anymatPAX{\orderind}{\recvar+1}{\Sexchanged(\indi)})$ inside.
\par 
It is clear that this has the same effect on the outside of $\thesplitting{\orderind-2}$ in \eqref{eq:higher_orders_common_helper_0} as replacing $\recvar$ by $\recvar+1$ would have. Simultaneously, specializing the $\dummyindex$ in \eqref{eq:higher_orders_common_helper_1} to $\recvar+1$ and separating the fifth line from the fourth and combining the fifth and sixth lines shows that
\begin{IEEEeqnarray*}{rCl}
  \IEEEeqnarraymulticol{3}{l}{
    -    \thedifferential{\orderind-1}(\Sbv{\anymatA}{\indj}{\recvar+1}\Smonoidalproduct \anymatPAX{\orderind}{\recvar+1}{\Sexchanged(\indi)})
    }\\
\hspace{1em}&= &\textstyle
-\Sbv{\anymatA}{\indj}{1}\Smonoidalproduct\anymatPBTX{\orderind}{1}{\Sexchanged(\recvar+1)}\anymatPAX{\orderind}{\recvar+1}{\Sexchanged(\indi)}-\sum_{\inds=2}^\thedim\Sbv{\anymatA}{\indj}{\inds}\Smonoidalproduct\anymatPBTX{\orderind}{\inds}{\Sexchanged(\recvar+1)}\anymatPAX{\orderind}{\recvar+1}{\Sexchanged(\indi)}\IEEEeqnarraynumspace\IEEEyesnumber\label{eq:higher_orders_common_helper_3}\\
&&\textstyle\hspace{13em}{}
+(-1)^\orderind\Saction\Sbv{\anymatBT}{\Sexchanged(\indj)}{\recvar+1}\Smonoidalproduct\anymatPAX{\orderind}{\recvar+1}{\Sexchanged(\indi)}\\
&&\textstyle{}-\Skronecker{\indj}{\thedim}(\Sbv{\anymatAT}{1}{1}+\Skronecker{\orderind}{4}\sum_{\inds=2}^{\thedim-1}\Sbv{\anymatAT}{\inds}{\inds})\Smonoidalproduct\anymatPBX{\orderind}{\Sexchanged(\recvar+1)}{\thedim}\anymatPATX{\orderind}{\Sexchanged(\indi)}{\recvar+1}\\
&&\textstyle{}+\Skronecker{\recvar+1}{\thedim}(\Skronecker{\indj}{\thedim}(\sum_{\inds=1}^\thedim\Sbv{\anymatAT}{1}{\inds}\Smonoidalproduct\anymatPBX{\orderind}{\inds}{\thedim}\anymatPATX{\orderind}{\Sexchanged(\indi)}{\thedim}-(-1)^\orderind\Saction\Sbv{\anymatB}{\thedim}{1}\Smonoidalproduct\anymatPATX{\orderind}{\Sexchanged(\indi)}{\thedim})\\
&&\textstyle\hspace{4em}{}+\Skronecker{\indj}{1}(-1)^\orderind(\Sbv{\anymatB}{1}{1}+\Skronecker{\orderind}{4}\sum_{\inds=2}^{\thedim-1}\Sbv{\anymatB}{\inds}{\inds})\Smonoidalproduct\anymatPATX{\orderind}{\Sexchanged(\indi)}{\thedim}).
\end{IEEEeqnarray*}
Keeping in mind that $\Szetafunction{\thedim}{\recvar}+\Skronecker{\thedim}{\recvar+1}=\Szetafunction{\thedim}{\recvar+1}$, a comparison  now proves that adding the term  \eqref{eq:higher_orders_common_helper_3} to the argument  of $\thesplitting{\orderind-2}$ in  \eqref{eq:higher_orders_common_helper_0} has the same effect as replacing $\recvar$ by $\recvar+1$ there.
\end{proof}

\begin{lemma}
  \label{lemma:higher_orders_common_final}
  For any $\anymat\in\thematrices$ and any $(\indj,\indi)\in\SYnumbers{\thedim}^{\Ssetmonoidalproduct 2}$,
  \begin{IEEEeqnarray*}{rCl}
    \IEEEeqnarraymulticol{3}{l}{
      \thedifferential{\orderind}(\Sbv{\anymatA}{\indj}{\indi}\Smonoidalproduct \theone)
    }\\    \hspace{1em}&=&\textstyle\sum_{\inds=1}^\thedim\Sbv{\anymatA}{\indj}{\inds}\Smonoidalproduct\anymatPAX{\orderind}{\inds}{\Sexchanged(\indi)}\IEEEyesnumber\label{eq:higher_orders_common_final_0}\\  &&\textstyle{}+\thesplitting{\orderind-2}\big((-1)^\orderind\sum_{\inds=1}^\thedim\Sbv{\anymatBT}{\Sexchanged(\indj)}{\inds}\Smonoidalproduct\anymatPAX{\orderind}{\inds}{\Sexchanged(\indi)}-\Sbv{\anymatA}{\indj}{\indi}\Smonoidalproduct\theone\\
  &&\textstyle\hspace{3em}{}+\Skronecker{\indj}{\thedim}(\sum_{\inds=1}^\thedim\Sbv{\anymatAT}{1}{\inds}\Smonoidalproduct\anymatPBX{\orderind}{\inds}{\thedim}\anymatPATX{\orderind}{\Sexchanged(\indi)}{\thedim}-(-1)^\orderind\Saction\Sbv{\anymatB}{\thedim}{1}\Smonoidalproduct\anymatPATX{\orderind}{\Sexchanged(\indi)}{\thedim}\\  &&\textstyle\hspace{11em}{}-\Skronecker{\indi}{\thedim}(\Sbv{\anymatAT}{1}{1}+\Skronecker{\orderind}{4}\sum_{\inds=2}^{\thedim-1}\Sbv{\anymatAT}{\inds}{\inds})\Smonoidalproduct\theone)\\
&&\textstyle\hspace{3em}{}+\Skronecker{\indj}{1}(-1)^\orderind(\Sbv{\anymatB}{1}{1}+\Skronecker{\orderind}{4}\sum_{\inds=2}^{\thedim-1}\Sbv{\anymatB}{\inds}{\inds})\Smonoidalproduct\anymatPATX{\orderind}{\Sexchanged(\indi)}{\thedim}\big).
  \end{IEEEeqnarray*}
\end{lemma}
\begin{proof}
  The special case $\recvar=\thedim$ of \eqref{eq:higher_orders_common_helper_0} in Lem\-ma~\ref{lemma:higher_orders_common_helper} amounts to the identity
  \begin{IEEEeqnarray*}{rCl}
    \IEEEeqnarraymulticol{3}{l}{
      \thedifferential{\orderind}(\Sbv{\anymatA}{\indj}{\indi}\Smonoidalproduct \theone)
  }\\
  \hspace{.5em}  &=&\textstyle\sum_{\inds=1}^\thedim\Sbv{\anymatA}{\indj}{\inds}\Smonoidalproduct\anymatPAX{\orderind}{\inds}{\Sexchanged(\indi)}\IEEEyesnumber\label{eq:higher_orders_common_final_1}\\  &&\textstyle{}+\thesplitting{\orderind-2}\big({-}\sum_{\inds=2}^{\thedim}\Sbv{\anymatA}{\indj}{\inds}\Smonoidalproduct\sum_{\indt=1}^\thedim\anymatPBTX{\orderind}{\inds}{\Sexchanged(\indt)}\anymatPAX{\orderind}{\indt}{\Sexchanged(\indi)}\\
  &&\textstyle\hspace{8em}{}+(-1)^\orderind\sum_{\inds=1}^\thedim\Sbv{\anymatBT}{\Sexchanged(\indj)}{\inds}\Smonoidalproduct\anymatPAX{\orderind}{\inds}{\Sexchanged(\indi)}-\Skronecker{\indi}{1}\Saction\Sbv{\anymatA}{\indj}{1}\Smonoidalproduct\theone\\
  &&\textstyle\hspace{3em}{}+\Skronecker{\indj}{\thedim}(-(\Sbv{\anymatAT}{1}{1}+\Skronecker{\orderind}{4}\sum_{\inds=2}^{\thedim-1}\Sbv{\anymatAT}{\inds}{\inds})\Smonoidalproduct\sum_{\inds=1}^\thedim\anymatPBX{\orderind}{\Sexchanged(\inds)}{\thedim}\anymatPATX{\orderind}{\Sexchanged(\indi)}{\inds}\\  &&\textstyle\hspace{5.875em}{}+\sum_{\inds=1}^\thedim\Sbv{\anymatAT}{1}{\inds}\Smonoidalproduct\anymatPBX{\orderind}{\inds}{\thedim}\anymatPATX{\orderind}{\Sexchanged(\indi)}{\thedim}-(-1)^\orderind\Saction\Sbv{\anymatB}{\thedim}{1}\Smonoidalproduct\anymatPATX{\orderind}{\Sexchanged(\indi)}{\thedim})\\
&&\textstyle\hspace{3em}{}+\Skronecker{\indj}{1}(-1)^\orderind(\Sbv{\anymatB}{1}{1}+\Skronecker{\orderind}{4}\sum_{\inds=2}^{\thedim-1}\Sbv{\anymatB}{\inds}{\inds})\Smonoidalproduct\anymatPATX{\orderind}{\Sexchanged(\indi)}{\thedim}\big).
\end{IEEEeqnarray*}
\par 
From the fact that $\sum_{\indt=1}^\thedim\anymatPBTX{\orderind}{\inds}{\Sexchanged(\indt)}\anymatPAX{\orderind}{\indt}{\Sexchanged(\indi)}=\Skronecker{\inds}{\indi}\Saction\theone$ for any $\inds\in\SYnumbers{\thedim}$ with $2\leq \inds$ it follows that
\begin{IEEEeqnarray*}{rCl}
  \IEEEeqnarraymulticol{3}{l}{
    \textstyle -\sum_{\inds=2}^{\thedim}\Sbv{\anymatA}{\indj}{\inds}\Smonoidalproduct\sum_{\indt=1}^\thedim\anymatPBTX{\orderind}{\inds}{\Sexchanged(\indt)}\anymatPAX{\orderind}{\indt}{\Sexchanged(\indi)}-\Skronecker{\indi}{1}\Saction\Sbv{\anymatA}{\indj}{1}\Smonoidalproduct\theone
  }\\  \hspace{.em}&=&\textstyle-\sum_{\inds=2}^{\thedim}\Skronecker{\inds}{\indi}\Saction \Sbv{\anymatA}{\indj}{\inds}\Smonoidalproduct\theone-\Skronecker{1}{\indi}\Saction\Sbv{\anymatA}{\indj}{1}\Smonoidalproduct\theone=-\sum_{\inds=1}^{\thedim}\Skronecker{\inds}{\indi}\Saction\Sbv{\anymatA}{\indj}{\inds}\Smonoidalproduct\theone=-\Sbv{\anymatA}{\indj}{\indi}\Smonoidalproduct\theone.
\end{IEEEeqnarray*}
That proves that the third and fourth lines of \eqref{eq:higher_orders_common_final_1} actually agree with the third line of \eqref{eq:higher_orders_common_final_0}.
\par
Similarly, because  $\sum_{\inds=1}^\thedim\anymatPBX{\orderind}{\Sexchanged(\inds)}{\thedim}\anymatPATX{\orderind}{\Sexchanged(\indi)}{\inds}=\Skronecker{\indi}{\thedim}\Saction\theone$ the fifth line of \eqref{eq:higher_orders_common_final_1} is equal to the fifth line of \eqref{eq:higher_orders_common_final_0}. And that means the identities \eqref{eq:higher_orders_common_final_1} and \eqref{eq:higher_orders_common_final_0} are equivalent. 
\end{proof}
At this point it is most efficient to distinguish four cases and treat them individually.
\subsubsection{First case} The first case is  handled in only one further step.

\begin{proposition}
  \label{proposition:higher_orders_first_case}
  For any $\anymat\in\thematrices$ and any $(\indj,\indi)\in\SYnumbers{\thedim}^{\Ssetmonoidalproduct 2}$, if $\indj\neq \thedim\neq \indi$, then 
  \begin{IEEEeqnarray*}{rCl}
      \thedifferential{\orderind}(\Sbv{\anymatA}{\indj}{\indi}\Smonoidalproduct \theone)&=&   \textstyle \sum_{\inds=1}^\thedim\Sbv{\anymatA}{\indj}{\inds}\Smonoidalproduct\anymatPAX{\orderind}{\inds}{\Sexchanged(\indi)}+(-1)^\orderind\Saction\Sbv{\anymatBT}{\Sexchanged(\indj)}{\indi}\Smonoidalproduct\theone\IEEEyesnumber\label{eq:higher_orders_first_case_0}.
  \end{IEEEeqnarray*}
\end{proposition}
\begin{proof}
  Under the assumptions that both $\indj\neq \thedim$ and $\indi\neq \thedim$ the fourth and fifth lines of \eqref{eq:higher_orders_common_final_0} in Lem\-ma~\ref{lemma:higher_orders_common_final} are zero, which hence specializes to the statement that
  \begin{IEEEeqnarray*}{rCl}
  \thedifferential{\orderind}(\Sbv{\anymatA}{\indj}{\indi}\Smonoidalproduct \theone)&=&\textstyle\sum_{\inds=1}^\thedim\Sbv{\anymatA}{\indj}{\inds}\Smonoidalproduct\anymatPAX{\orderind}{\inds}{\Sexchanged(\indi)}\IEEEyesnumber\label{eq:higher_orders_first_case_1}\\  &&\textstyle{}+\thesplitting{\orderind-2}\big((-1)^\orderind\sum_{\inds=1}^\thedim\Sbv{\anymatBT}{\Sexchanged(\indj)}{\inds}\Smonoidalproduct\anymatPAX{\orderind}{\inds}{\Sexchanged(\indi)}-\Sbv{\anymatA}{\indj}{\indi}\Smonoidalproduct\theone\\  
&&\textstyle\hspace{3em}{}+\Skronecker{\indj}{1}(-1)^\orderind(\Sbv{\anymatB}{1}{1}+\Skronecker{\orderind}{4}\sum_{\inds=2}^{\thedim-1}\Sbv{\anymatB}{\inds}{\inds})\Smonoidalproduct\anymatPATX{\orderind}{\Sexchanged(\indi)}{\thedim}).
\end{IEEEeqnarray*}
\par
The vector $\Sbv{\anymatBT}{\Sexchanged(\indj)}{1}\Smonoidalproduct\anymatPAX{\orderind}{1}{\Sexchanged(\indi)}$ is the tip of the argument  of $\thesplitting{\orderind-2}$ in \eqref{eq:higher_orders_first_case_1} for the following reasons. The term $\Sbv{\anymatA}{\indj}{\indi}\Smonoidalproduct\theone$ is less already by degree. All the other summands in that line of \eqref{eq:higher_orders_first_case_1} are dominated because $\Sbv{\anymatBT}{\Sexchanged(\indj)}{\inds}<\Sbv{\anymatBT}{\Sexchanged(\indj)}{1}$ for any $\inds\in\SYnumbers{\thedim}$ with $1<\inds$ by the addendum to Lem\-ma~\ref{lemma:base_monomial_order}. Finally, the entire third line of \eqref{eq:higher_orders_first_case_1} is less because, if it is non-zero, then $\indj=1$, i.e., $\Sexchanged(\indj)=\thedim$ and hence $\Sbv{\anymatB}{\inds}{\inds}<\Sbv{\anymatBT}{\Sexchanged(\indj)}{1}$ for any $\inds\in\SYnumbers{\thedim}$ with $\inds<\thedim$ by Lem\-ma~\ref{lemma:base_monomial_order}.
\par
Because moreover by Proposition~\ref{proposition:combinatorial_splitting}
\begin{IEEEeqnarray*}{rCl}
  \textstyle\thereconcatenator{\orderind-2}(\Sbvfull{\anymatBT}{\orderind-2}{\Sexchanged(\indj)}{1},\anymatPAX{\orderind}{1}{\Sexchanged(\indi)})=(\Sbvfull{\anymatBT}{\orderind-1}{\Sexchanged(\indj)}{\indi},\theone)
\end{IEEEeqnarray*}
another valid identity is produced by adding to the right-hand side of \eqref{eq:higher_orders_first_case_1} the term $(-1)^\orderind\Saction\Sbv{\anymatBT}{\Sexchanged(\indj)}{\indi}\Smonoidalproduct\theone$ outside of $\thesplitting{\orderind-2}$ and $-(-1)^\orderind\thedifferential{\orderind-1}(\Sbvfull{\anymatBT}{\orderind-1}{\Sexchanged(\indj)}{\indi}\Smonoidalproduct\theone)$ inside.
\par
That addtion transforms the outside of $\thesplitting{\orderind-1}$ in \eqref{eq:higher_orders_first_case_1} into that in \eqref{eq:higher_orders_first_case_0}. And since by $\Sexchanged(\indj)\neq 1$ and $ \indi\neq \thedim$ Assumption~\ref{assumption:higher_orders} gives
\begin{IEEEeqnarray*}{rCl}  -(-1)^\orderind\thedifferential{\orderind-1}(\Sbvfull{\anymatBT}{\orderind-1}{\Sexchanged(\indj)}{\indi}\Smonoidalproduct\theone)&=&\textstyle-(-1)^\orderind\sum_{\inds=1}^\thedim\Sbv{\anymatBT}{\Sexchanged(\indj)}{\inds}\Smonoidalproduct\anymatPAX{\orderind}{\inds}{\Sexchanged(\indi)}+\Sbv{\anymatA}{\indj}{\indi}\Smonoidalproduct\theone\\  
&&\textstyle{}-\Skronecker{\indj}{1}(-1)^\orderind(\Sbv{\anymatB}{1}{1}+\Skronecker{\orderind}{4}\sum_{\inds=2}^{\thedim-1}\Sbv{\anymatB}{\inds}{\inds})\Smonoidalproduct\anymatPATX{\orderind}{\Sexchanged(\indi)}{\thedim}
\end{IEEEeqnarray*}
the same is true of the inside of $\thesplitting{\orderind-2}$, which proves the claim. 
\end{proof}
\subsubsection{Second case}
For the second case, $\thedifferential{\orderind}(\Sbv{\anymatA}{\indj}{\indi}\Smonoidalproduct\theone)$ for $\indj=\thedim$ and $\indi\in\SYnumbers{\thedim}$ with $\indi\neq \thedim$, it takes two more recursion steps to eliminate $\thesplitting{\orderind-2}$ entirely.
\begin{lemma}
  \label{lemma:higher_orders_second_case_helper}
  For any $\anymat\in\thematrices$ and any $\indi\in\SYnumbers{\thedim}$, if $\indi\neq \thedim$, then
  \begin{IEEEeqnarray*}{rCl}
    \thedifferential{\orderind}(\Sbv{\anymatA}{\thedim}{\indi}\Smonoidalproduct\theone) &=&\textstyle\sum_{\inds=1}^\thedim\Sbv{\anymatA}{\thedim}{\inds}\Smonoidalproduct\anymatPAX{\orderind}{\inds}{\Sexchanged(\indi)}+\Sbv{\anymatAT}{1}{1}\Smonoidalproduct\anymatPATX{\orderind}{\Sexchanged(\indi)}{\thedim}\IEEEyesnumber\label{eq:higher_orders_second_case_helper_0}\\
    &&\textstyle{}+\thesplitting{\orderind-2}\big((-1)^\orderind\sum_{\inds=1}^\thedim \Sbv{\anymatBT}{1}{\inds}\Smonoidalproduct\anymatPAX{\orderind}{\inds}{\Sexchanged(\indi)}-\Sbv{\anymatA}{\thedim}{\indi}\Smonoidalproduct\theone\big).
  \end{IEEEeqnarray*}
\end{lemma}
\begin{proof}
Under the assumption $\indi\neq \thedim$  instantiating Lem\-ma~\ref{lemma:higher_orders_common_final} at $\indj=\thedim$ yields
  \begin{IEEEeqnarray*}{rCl}
  \thedifferential{\orderind}(\Sbv{\anymatA}{\thedim}{\indi}\Smonoidalproduct \theone)&=&\textstyle\sum_{\inds=1}^\thedim\Sbv{\anymatA}{\thedim}{\inds}\Smonoidalproduct\anymatPAX{\orderind}{\inds}{\Sexchanged(\indi)}\IEEEyesnumber\label{eq:higher_orders_second_case_helper_1}\\  &&\textstyle{}+\thesplitting{\orderind-2}\big(\sum_{\inds=1}^\thedim\Sbv{\anymatAT}{1}{\inds}\Smonoidalproduct\anymatPBX{\orderind}{\inds}{\thedim}\anymatPATX{\orderind}{\Sexchanged(\indi)}{\thedim}-(-1)^\orderind\Saction\Sbv{\anymatB}{\thedim}{1}\Smonoidalproduct\anymatPATX{\orderind}{\Sexchanged(\indi)}{\thedim}\\  &&\textstyle\hspace{3em}{}+(-1)^\orderind\sum_{\inds=1}^\thedim\Sbv{\anymatBT}{1}{\inds}\Smonoidalproduct\anymatPAX{\orderind}{\inds}{\Sexchanged(\indi)}-\Sbv{\anymatA}{\thedim}{\indi}\Smonoidalproduct\theone\big).
\end{IEEEeqnarray*}
\par
Since $\indi\neq \thedim$ the monomial $\anymatPBX{\orderind}{1}{\thedim}\anymatPATX{\orderind}{\Sexchanged(\indi)}{\thedim}$ is reduced by Lemma~\ref{characterization_of_reduced_terms_of_order_two}. The vector  $\Sbv{\anymatAT}{1}{1}\Smonoidalproduct\anymatPBX{\orderind}{1}{\thedim}\anymatPATX{\orderind}{\Sexchanged(\indi)}{\thedim}$ is the tip of the argument of $\thesplitting{\orderind-2}$ on the right-hand side of \eqref{eq:higher_orders_second_case_helper_1}. Indeed, for degree reasons only the terms in the same grouped sum where $\Sbv{\anymatAT}{1}{1}\Smonoidalproduct\anymatPBX{\orderind}{1}{\thedim}\anymatPAX{\orderind}{\Sexchanged(\indi)}{\thedim}$ comes from could possibly be greater or equal. However, for any $\inds\in\SYnumbers{\thedim}$ with $1<\inds$ the addendum to Lem\-ma~\ref{lemma:base_monomial_order} ensures that already  $\Sbv{\anymatAT}{1}{\inds}\theorderRstrict\Sbv{\anymatAT}{1}{1}$ and hence $\Sbv{\anymatAT}{1}{\inds}\Smonoidalproduct\anymatPBX{\orderind}{\inds}{\thedim}\anymatPATX{\orderind}{\Sexchanged(\indi)}{\thedim}\thechainsorderRstrict{\orderind-2}\Sbv{\anymatAT}{1}{1}\Smonoidalproduct\anymatPBX{\orderind}{1}{\thedim}\anymatPATX{\orderind}{\Sexchanged(\indi)}{\thedim}$.
\par
Given that by Proposition~\ref{proposition:combinatorial_splitting}
\begin{IEEEeqnarray*}{rCl}
\textstyle\thereconcatenator{\orderind-2}(    \Sbvfull{\anymatAT}{\orderind-2}{1}{1},\anymatPBX{\orderind}{1}{\thedim}\anymatPATX{\orderind}{\Sexchanged(\indi)}{\thedim})=(  \Sbvfull{\anymatAT}{\orderind-1}{1}{1},\anymatPATX{\orderind}{\Sexchanged(\indi)}{\thedim}),
\end{IEEEeqnarray*}
adding $\Sbv{\anymatAT}{1}{1}\Smonoidalproduct\anymatPATX{\orderind}{\Sexchanged(\indi)}{\thedim}$ outside of $\thesplitting{\orderind-2}$ and $-\thedifferential{\orderind-1}(\Sbv{\anymatAT}{1}{1}\Smonoidalproduct\anymatPATX{\orderind}{\Sexchanged(\indi)}{\thedim})$ inside turns \eqref{eq:higher_orders_second_case_helper_1}  into another true identity.
\par
This change brings the outside of $\thesplitting{\orderind-2}$ in \eqref{eq:higher_orders_second_case_helper_1} exactly into the form of the outside in \eqref{eq:higher_orders_second_case_helper_0}. Moreover, by Assumption~\ref{assumption:higher_orders}
\begin{IEEEeqnarray*}{rCl}
    \thedifferential{\orderind-1}(\Sbv{\anymatAT}{1}{1}\Smonoidalproduct \theone)
&= &\textstyle  \sum_{\inds=1}^\thedim\Sbv{\anymatAT}{1}{\inds}\Smonoidalproduct\anymatPBX{\orderind}{\inds}{\thedim}-(-1)^\orderind\Saction\Sbv{\anymatB}{\thedim}{1}\Smonoidalproduct\theone,
\end{IEEEeqnarray*}
which is why
\begin{IEEEeqnarray*}{rCl}
  \IEEEeqnarraymulticol{3}{l}{
    -\thedifferential{\orderind-1}(\Sbv{\anymatAT}{1}{1}\Smonoidalproduct \anymatPATX{\orderind}{\Sexchanged(\indi)}{\thedim})
    }\\\hspace{3em}
&= &\textstyle  -\sum_{\inds=1}^\thedim\Sbv{\anymatAT}{1}{\inds}\Smonoidalproduct\anymatPBX{\orderind}{\inds}{\thedim}\anymatPATX{\orderind}{\Sexchanged(\indi)}{\thedim}+(-1)^\orderind\Sbv{\anymatB}{\thedim}{1}\Smonoidalproduct\anymatPATX{\orderind}{\Sexchanged(\indi)}{\thedim}.\IEEEeqnarraynumspace\IEEEyesnumber\label{eq:higher_orders_second_case_helper_2}
\end{IEEEeqnarray*}
Evidently adding the vector  \eqref{eq:higher_orders_second_case_helper_2} to the argument of $\thesplitting{\orderind-2}$ in \eqref{eq:higher_orders_second_case_helper_1} produces precisely the argument of $\thesplitting{\orderind-2}$ in \eqref{eq:higher_orders_second_case_helper_0}. Thus, the claim is true. 
\end{proof}
\begin{proposition}
  \label{proposition:higher_orders_second_case}
  For any $\anymat\in\thematrices$ and any $\indi\in\SYnumbers{\thedim}$, if $\indi\neq \thedim$, then
  \begin{IEEEeqnarray*}{rCl}
    \thedifferential{\orderind}(\Sbv{\anymatA}{\thedim}{\indi}\Smonoidalproduct\theone) &=&\textstyle\sum_{\inds=1}^\thedim\Sbv{\anymatA}{\thedim}{\inds}\Smonoidalproduct\anymatPAX{\orderind}{\inds}{\Sexchanged(\indi)}+\Sbv{\anymatAT}{1}{1}\Smonoidalproduct\anymatPATX{\orderind}{\Sexchanged(\indi)}{\thedim}+(-1)^\orderind\Saction\Sbv{\anymatBT}{1}{\indi}\Smonoidalproduct\theone.\IEEEeqnarraynumspace\IEEEyesnumber\label{eq:higher_orders_second_case_0}
  \end{IEEEeqnarray*}  
\end{proposition}
\begin{proof}
  The vector $\Sbv{\anymatBT}{1}{1}\Smonoidalproduct\anymatPAX{\orderind}{1}{\Sexchanged(\indi)}$ is the tip of the argument of $\thesplitting{\orderind-2}$ in \eqref{eq:higher_orders_second_case_helper_0}  because $\Sbv{\anymatBT}{1}{\inds}\theorderRstrict\Sbv{\anymatBT}{1}{1}$ for any $\inds\in\SYnumbers{\thedim}$ with $1<\inds$ by the addendum to  Lem\-ma~\ref{lemma:base_monomial_order}. And, of course, $\Sbv{\anymatA}{\thedim}{\indi}\Smonoidalproduct\theone$ is dominated for degree reasons.
  \par
  Since by Proposition~\ref{proposition:combinatorial_splitting}
  \begin{IEEEeqnarray*}{rCl}
    \textstyle\thereconcatenator{\orderind-2}(\Sbvfull{\anymatBT}{\orderind-2}{1}{1},\anymatPAX{\orderind}{1}{\Sexchanged(\indi)})&=&( \Sbvfull{\anymatBT}{\orderind-1}{1}{\indi},\theone)
  \end{IEEEeqnarray*}
  it follows that the equality in \eqref{eq:higher_orders_second_case_helper_0} is preserved by adding to its right-hand side $(-1)^\orderind\Saction\Sbv{\anymatBT}{1}{\indi}\Smonoidalproduct\theone$ outside of $\thesplitting{\orderind-2}$ and $-(-1)^\orderind\thedifferential{\orderind-1}(\Sbv{\anymatBT}{1}{\indi}\Smonoidalproduct\theone)$ inside.
  \par
  That immediately makes  the outsides of $\thesplitting{\orderind-2}$ in \eqref{eq:higher_orders_second_case_helper_0} and \eqref{eq:higher_orders_second_case_0} agree. At the same time   $\indi\neq \thedim$ implies  according to  Assumption~\ref{assumption:higher_orders}
  \begin{IEEEeqnarray*}{rCl}
    -(-1)^\orderind\thedifferential{\orderind-1}(\Sbv{\anymatBT}{1}{\indi}\Smonoidalproduct\theone)&=&\textstyle-(-1)^\orderind\sum_{\inds=1}^\thedim \Sbv{\anymatBT}{1}{\inds}\Smonoidalproduct\anymatPAX{\orderind}{\inds}{\Sexchanged(\indi)}+\Sbv{\anymatA}{\thedim}{\indi}\Smonoidalproduct\theone.
  \end{IEEEeqnarray*}
Hence,   the argument of $\thesplitting{\orderind-2}$ in \eqref{eq:higher_orders_second_case_helper_0} is canceled by the addition of $-(-1)^\orderind\thedifferential{\orderind-1}(\Sbv{\anymatBT}{1}{\indi}\Smonoidalproduct\theone)$, which then proves \eqref{eq:higher_orders_second_case_0}.
\end{proof}
\subsubsection{Third case} In the third case, computing $\thedifferential{\orderind}(\Sbv{\anymatA}{\indj}{\thedim}\Smonoidalproduct\theone)$ for any $\indj\in\SYnumbers{\thedim}$ with $\indj\neq \thedim$, two recursion steps  are necessary to solve.
\begin{lemma}
  \label{lemma:higher_orders_third_case_helper}
  For any $\anymat\in\thematrices$ and any $\indj\in\SYnumbers{\thedim}$, if $\indj\neq \thedim$, then
  \begin{IEEEeqnarray*}{rCl}
  \thedifferential{\orderind}(\Sbv{\anymatA}{\indj}{\thedim}\Smonoidalproduct \theone)&=&\textstyle\sum_{\inds=1}^\thedim\Sbv{\anymatA}{\indj}{\inds}\Smonoidalproduct\anymatPAX{\orderind}{\inds}{1}+(-1)^\orderind\Saction\Sbv{\anymatBT}{\Sexchanged(\indj)}{\thedim}\Smonoidalproduct\theone\IEEEyesnumber\label{eq:higher_orders_third_case_helper_0}\\  &&\textstyle{}+\Skronecker{\indj}{1}\Saction\thesplitting{\orderind-2}\big((-1)^\orderind\sum_{\inds=1}^\thedim \Sbv{\anymatB}{1}{\inds}\Smonoidalproduct\anymatPATX{\orderind}{\inds}{\thedim}-\Sbv{\anymatAT}{\thedim}{1}\Smonoidalproduct\theone\big).
  \end{IEEEeqnarray*}  
\end{lemma}
\begin{proof}
  Specializing $\indi$ to $\thedim$ in Lem\-ma~\ref{lemma:higher_orders_common_final} and recognizing that the assumption $\indj\neq \thedim$ removes the fourth and fifth lines of \eqref{eq:higher_orders_common_final_0} gives
    \begin{IEEEeqnarray*}{rCl}
  \thedifferential{\orderind}(\Sbv{\anymatA}{\indj}{\thedim}\Smonoidalproduct \theone)&=&\textstyle\sum_{\inds=1}^\thedim\Sbv{\anymatA}{\indj}{\inds}\Smonoidalproduct\anymatPAX{\orderind}{\inds}{1}\IEEEyesnumber\label{eq:higher_orders_third_case_helper_1}\\  &&\textstyle{}+\thesplitting{\orderind-2}\big((-1)^\orderind\sum_{\inds=1}^\thedim\Sbv{\anymatBT}{\Sexchanged(\indj)}{\inds}\Smonoidalproduct\anymatPAX{\orderind}{\inds}{1}-\Sbv{\anymatA}{\indj}{\thedim}\Smonoidalproduct\theone\\  
&&\textstyle\hspace{3em}{}+\Skronecker{\indj}{1}(-1)^\orderind(\Sbv{\anymatB}{1}{1}+\Skronecker{\orderind}{4}\sum_{\inds=2}^{\thedim-1}\Sbv{\anymatB}{\inds}{\inds})\Smonoidalproduct\anymatPATX{\orderind}{1}{\thedim}\big).
\end{IEEEeqnarray*}
\par
The vector $\Sbv{\anymatBT}{\Sexchanged(\indj)}{1}\Smonoidalproduct\anymatPAX{\orderind}{1}{1}$ is the tip of  the argument of $\thesplitting{\orderind-2 }$ in \eqref{eq:higher_orders_third_case_helper_1}. It dominates the other terms in the grouped sum from where it stems because for any $\inds\in\SYnumbers{\thedim}$ with $1<\inds$ necessarily $\Sbv{\anymatBT}{\Sexchanged(\indj)}{\inds}\theorderRstrict\Sbv{\anymatBT}{\Sexchanged(\indj)}{1}$ by the addendum to Lem\-ma~\ref{lemma:base_monomial_order}. And the terms in the third line of \eqref{eq:higher_orders_third_case_helper_1} are less because, if they are non-zero, then $\indj=1$, i.e., $\Sexchanged(\indj)=\thedim$ and hence $\Sbv{\anymatB}{\inds}{\inds}<\Sbv{\anymatBT}{\Sexchanged(\indj)}{1}$ for any $\inds\in\SYnumbers{\thedim}$ with $\inds\neq \thedim$ according to Lem\-ma~\ref{lemma:base_monomial_order}.
\par
Because by Proposition~\ref{proposition:combinatorial_splitting}
\begin{IEEEeqnarray*}{rCl}
  \textstyle\thereconcatenator{\orderind-2}(  \Sbvfull{\anymatBT}{\orderind-2}{\Sexchanged(\indj)}{1},\anymatPAX{\orderind}{1}{1})=(  \Sbvfull{\anymatBT}{\orderind-1}{\Sexchanged(\indj)}{\thedim},\theone)
\end{IEEEeqnarray*}
a true  identity is thus derived by adding $(-1)^\orderind\Saction\Sbv{\anymatBT}{\Sexchanged(\indj)}{\thedim}\Smonoidalproduct\theone$ on the right-hand side to the outside of $\thesplitting{\orderind-2}$ in \eqref{eq:higher_orders_third_case_helper_1} and $-(-1)^\orderind\thedifferential{\orderind-2}(\Sbv{\anymatBT}{\Sexchanged(\indj)}{\thedim}\Smonoidalproduct\theone)$ to the inside.
\par
That addition turns the outside of \eqref{eq:higher_orders_third_case_helper_1}  into that of \eqref{eq:higher_orders_third_case_helper_0}. Simultaneously, since $\Sexchanged(\indj)\neq 1$, by Assumption~\ref{assumption:higher_orders}
\begin{IEEEeqnarray*}{rCl}
  \IEEEeqnarraymulticol{3}{l}{
    -(-1)^\orderind\thedifferential{\orderind-1}(\Sbv{\anymatBT}{\Sexchanged(\indj)}{\thedim}\Smonoidalproduct \theone)
  }\\
  \hspace{3em}
&= &\textstyle  -(-1)^\orderind\sum_{\inds=1}^\thedim\Sbv{\anymatBT}{\Sexchanged(\indj)}{\inds}\Smonoidalproduct\anymatPAX{\orderind}{\inds}{1}+\Saction\Sbv{\anymatA}{\indj}{\thedim}\Smonoidalproduct\theone\IEEEyesnumber\label{eq:higher_orders_third_case_helper_2}\\
&&\textstyle{}+\Skronecker{\indj}{1}(-(-1)^\orderind(\Sbv{\anymatB}{1}{1}+\Skronecker{\orderind}{4}\sum_{\inds=2}^{\thedim-1}\Sbv{\anymatB}{\inds}{\inds})\Smonoidalproduct\anymatPATX{\orderind}{1}{\thedim}\\
&&\textstyle\hspace{5.25em}{}+(-1)^\orderind\sum_{\inds=1}^\thedim\Sbv{\anymatB}{1}{\inds}\Smonoidalproduct\anymatPATX{\orderind}{\inds}{\thedim}-\Saction\Sbv{\anymatAT}{\thedim}{1}\Smonoidalproduct\theone).
\end{IEEEeqnarray*}
It is now evident that when the vector \eqref{eq:higher_orders_third_case_helper_2} and the argument of $\thesplitting{\orderind-2}$ in \eqref{eq:higher_orders_third_case_helper_1}  are added together the result is precisely the argument of  $\thesplitting{\orderind-2}$ in  \eqref{eq:higher_orders_third_case_helper_0}. And by what was said before that proves the assertion.
\end{proof}
\begin{proposition}
  \label{proposition:higher_orders_third_case}
  For any $\anymat\in\thematrices$ and any $\indj\in\SYnumbers{\thedim}$, if $\indj \neq \thedim$, then
  \begin{IEEEeqnarray*}{rCl}
    \IEEEeqnarraymulticol{3}{l}{
      \thedifferential{\orderind}(\Sbv{\anymatA}{\indj}{\thedim}\Smonoidalproduct \theone)
}\\      \hspace{1em}&=&\textstyle\sum_{\inds=1}^\thedim\Sbv{\anymatA}{\indj}{\inds}\Smonoidalproduct\anymatPAX{\orderind}{\inds}{1}+(-1)^\orderind\Saction\Sbv{\anymatBT}{\Sexchanged(\indj)}{\thedim}\Smonoidalproduct\theone+\Skronecker{\indj}{1}(-1)^\orderind\Saction\Sbv{\anymatB}{1}{1}\Smonoidalproduct\theone.\IEEEeqnarraynumspace\IEEEyesnumber\label{eq:higher_orders_third_case_0}
  \end{IEEEeqnarray*}    
\end{proposition}
\begin{proof}
  For any $\inds\in\SYnumbers{\thedim}$ with $1<\inds$ the inequality $\Sbv{\anymatB}{1}{\inds}\theorderRstrict\Sbv{\anymatB}{1}{1}$ holds by the addendum to Lem\-ma~\ref{lemma:base_monomial_order}. That makes $\Sbv{\anymatB}{1}{1}\Smonoidalproduct\anymatPATX{\orderind}{1}{\thedim}$ the tip of the argument of $\thesplitting{\orderind-2}$ on the right\-/hand side of \eqref{eq:higher_orders_third_case_helper_0} in Lem\-ma~\ref{lemma:higher_orders_third_case_helper}.
  \par
  By Proposition~\ref{proposition:combinatorial_splitting}
  \begin{IEEEeqnarray*}{rCl}
    \textstyle\thereconcatenator{\orderind-2}(\Sbvfull{\anymatB}{\orderind-2}{1}{1},\anymatPATX{\orderind}{1}{\thedim})&=&(\Sbvfull{\anymatB}{\orderind-1}{1}{1},\theone).
  \end{IEEEeqnarray*}
  Note that $\thesplitting{\orderind-2}$ has a scalar factor of $\Skronecker{\indj}{1}$ in \eqref{eq:higher_orders_third_case_helper_0}. For those reasons another valid identity is obtained  by adding to the right-hand side of \eqref{eq:higher_orders_third_case_helper_0} the terms $\Skronecker{\indj}{1}(-1)^\orderind\Saction\Sbv{\anymatB}{1}{1}\Smonoidalproduct\theone$ outside of $\thesplitting{\orderind-2}$ and $-(-1)^\orderind\thedifferential{\orderind-1}(\Sbv{\anymatB}{1}{1}\Smonoidalproduct\theone)$ inside of it.
  \par
  It is clear that this change turns the outside of $\thesplitting{\orderind-2}$ in \eqref{eq:higher_orders_third_case_helper_0} into that in \eqref{eq:higher_orders_third_case_0}. And because according to Assumption~\ref{assumption:higher_orders}
  \begin{IEEEeqnarray*}{rCl}
    -(-1)^\orderind\thedifferential{\orderind-1}(\Sbv{\anymatB}{1}{1}\Smonoidalproduct \theone)
&= &\textstyle  -(-1)^\orderind\sum_{\inds=1}^\thedim\Sbv{\anymatB}{1}{\inds}\Smonoidalproduct\anymatPATX{\orderind}{\inds}{\thedim}+\Sbv{\anymatAT}{\thedim}{1}\Smonoidalproduct\theone
\end{IEEEeqnarray*}
the addition of $-\thedifferential{\orderind-1}(\Sbv{\anymatB}{1}{1}\Smonoidalproduct \theone)$ exactly cancels out the argument of $\thesplitting{\orderind-2}$ in \eqref{eq:higher_orders_third_case_helper_0}. That verifies \eqref{eq:higher_orders_third_case_0}.
\end{proof}
\subsubsection{Fourth case}
It remains to compute $\thedifferential{\orderind}(\Sbv{\anymatA}{\thedim}{\thedim}\Smonoidalproduct\theone)$, which is the one requiring the most effort, another $\thedim+2$ recursion steps.
\begin{lemma}
  \label{lemma:higher_orders_fourth_case_first_helper}
  For any $\anymat\in\thematrices$ and any $\recvar\in \SYnumbers{\thedim}$,
  \begin{IEEEeqnarray*}{rCl}
    \IEEEeqnarraymulticol{3}{l}{
      \thedifferential{\orderind}(\Sbv{\anymatA}{\thedim}{\thedim}\Smonoidalproduct\theone)
    }\\
    \hspace{1.5em}    &=&\textstyle\sum_{\inds=1}^\thedim\Sbv{\anymatA}{\thedim}{\inds}\Smonoidalproduct\anymatPAX{\orderind}{\inds}{1}-\sum_{\inds=2}^{\recvar}\Sbv{\anymatAT}{1}{\inds}\Smonoidalproduct\anymatPATX{\orderind}{\inds}{\thedim}\IEEEyesnumber\label{eq:higher_orders_fourth_case_first_helper_0}\\
    &&\textstyle{}+\thesplitting{\orderind-2}\big({-}\sum_{\inds=\recvar+1}^{\thedim}\Sbv{\anymatAT}{1}{1}\Smonoidalproduct\anymatPBX{\orderind}{1}{\Sexchanged(\inds)}\anymatPATX{\orderind}{\inds}{\thedim}\\
    &&\textstyle\hspace{3em}{}+\sum_{\inds=2}^\thedim\sum_{\indt=1}^{\recvar}\Sbv{\anymatAT}{1}{\inds}\Smonoidalproduct\anymatPBX{\orderind}{\inds}{\Sexchanged(\indt)}\anymatPATX{\orderind}{\indt}{\thedim}\\    &&\textstyle{}\hspace{3em}-(-1)^\orderind\sum_{\inds=1}^\recvar\Sbv{\anymatB}{\thedim}{\inds}\Smonoidalproduct\anymatPATX{\orderind}{\inds}{\thedim}+(-1)^\orderind\sum_{\inds=1}^\thedim\Sbv{\anymatBT}{1}{\inds}\Smonoidalproduct\anymatPAX{\orderind}{\inds}{1}\\
    &&\textstyle{}\hspace{3em}-\Sbv{\anymatA}{\thedim}{\thedim}\Smonoidalproduct\theone-
    \Skronecker{\orderind}{4}\sum_{\inds=2}^{\thedim-1}\Sbv{\anymatAT}{\inds}{\inds}
    \Smonoidalproduct\theone\\    &&\textstyle{}\hspace{3em}-\Szetafunction{\thedim}{\recvar}(-1)^\orderind(\Sbv{\anymatBT}{1}{1}+\Skronecker{\orderind}{4}\sum_{\inds=2}^{\thedim-1}\Sbv{\anymatBT}{\inds}{\inds})\Smonoidalproduct\anymatPAX{\orderind}{\thedim}{\thedim}).
  \end{IEEEeqnarray*}
\end{lemma}
\begin{proof}
  \newcommand{\dummyindex}{k}
  The proof goes by induction over $\recvar$. 
  \par
  \emph{Induction base.} By $\Szetafunction{\thedim}{1}=0$ what  needs to be proved in the base case $\recvar=1$ is that
  \begin{IEEEeqnarray*}{rCl}
    \IEEEeqnarraymulticol{3}{l}{
      \thedifferential{\orderind}(\Sbv{\anymatA}{\thedim}{\thedim}\Smonoidalproduct\theone)
}\\      \hspace{1em}&=&\textstyle\sum_{\inds=1}^\thedim\Sbv{\anymatA}{\thedim}{\inds}\Smonoidalproduct\anymatPAX{\orderind}{\inds}{1}\IEEEyesnumber\label{eq:higher_orders_fourth_case_first_helper_1}\\
&&\textstyle{}+\thesplitting{\orderind-2}({-}\sum_{\inds=2}^{\thedim}\Sbv{\anymatAT}{1}{1}\Smonoidalproduct\anymatPBX{\orderind}{1}{\Sexchanged(\inds)}\anymatPAX{\orderind}{\inds}{\thedim}+\sum_{\inds=2}^\thedim\Sbv{\anymatAT}{1}{\inds}\Smonoidalproduct\anymatPBX{\orderind}{\inds}{\thedim}\anymatBX{1}{\thedim}\\
    &&\textstyle{}\hspace{3em}-(-1)^\orderind\Saction\Sbv{\anymatB}{\thedim}{1}\Smonoidalproduct\anymatPATX{\orderind}{1}{\thedim}+(-1)^\orderind\sum_{\inds=1}^\thedim\Sbv{\anymatBT}{1}{\inds}\Smonoidalproduct\anymatPAX{\orderind}{\inds}{1}\\
    &&\textstyle{}\hspace{3em}-\Sbv{\anymatA}{\thedim}{\thedim}\Smonoidalproduct\theone- \Skronecker{\orderind}{4}\sum_{\inds=2}^{\thedim-1}\Sbv{\anymatAT}{\inds}{\inds}\Smonoidalproduct\theone).
  \end{IEEEeqnarray*}
  \par
  Specializing both $\indi$ and $\indj$ to $\thedim$ in \eqref{eq:higher_orders_common_final_0} in Lem\-ma~\ref{lemma:higher_orders_common_final} gives, after some rearrangement,
  \begin{IEEEeqnarray*}{rCl}
  \thedifferential{\orderind}(\Sbv{\anymatA}{\thedim}{\thedim}\Smonoidalproduct \theone)&=&\textstyle\sum_{\inds=1}^\thedim\Sbv{\anymatA}{\thedim}{\inds}\Smonoidalproduct\anymatPAX{\orderind}{\inds}{1}\IEEEyesnumber\label{eq:higher_orders_fourth_case_first_helper_2}\\  &&\textstyle{}+\thesplitting{\orderind-2}\big(\sum_{\inds=1}^\thedim\Sbv{\anymatAT}{1}{\inds}\Smonoidalproduct\anymatPBX{\orderind}{\inds}{\thedim}\anymatPATX{\orderind}{1}{\thedim}-\Sbv{\anymatAT}{1}{1}\Smonoidalproduct\theone\\
  &&\textstyle\hspace{3em}{}-(-1)^\orderind\Saction\Sbv{\anymatB}{\thedim}{1}\Smonoidalproduct\anymatPATX{\orderind}{1}{\thedim}+(-1)^\orderind\Saction\sum_{\inds=1}^\thedim\Sbv{\anymatBT}{1}{\inds}\Smonoidalproduct\anymatPAX{\orderind}{\inds}{1}\\  &&\textstyle\hspace{3em}{}-\Sbv{\anymatA}{\thedim}{\thedim}\Smonoidalproduct\theone-\Skronecker{\orderind}{4}\sum_{\inds=2}^{\thedim-1}\Sbv{\anymatAT}{\inds}{\inds}\Smonoidalproduct\theone\big).
\end{IEEEeqnarray*}
Since $\anymatPBX{\orderind}{1}{\thedim}\anymatPATX{\orderind}{1}{\thedim}=-\sum_{\indt=2}^{\thedim}\anymatPBX{\orderind}{1}{\Sexchanged(\indt)}\anymatPATX{\orderind}{\indt}{\thedim}+1$ the terms in the argument of $\thesplitting{\orderind-2}$ in the second line of \eqref{eq:higher_orders_fourth_case_first_helper_2} can also be written as 
\begin{IEEEeqnarray*}{rCl}
  \IEEEeqnarraymulticol{3}{l}{
    \textstyle  \sum_{\inds=1}^\thedim\Sbv{\anymatAT}{1}{\inds}\Smonoidalproduct\anymatPBX{\orderind}{\inds}{\thedim}\anymatPATX{\orderind}{1}{\thedim}-\Sbv{\anymatAT}{1}{1}\Smonoidalproduct\theone}\\    \hspace{1em}&=&\textstyle-\sum_{\indt=2}^{\thedim}\Sbv{\anymatAT}{1}{1}\Smonoidalproduct\anymatPBX{\orderind}{1}{\Sexchanged(\indt)}\anymatPATX{\orderind}{\indt}{\thedim}+\sum_{\inds=2}^\thedim\Sbv{\anymatAT}{1}{\inds}\Smonoidalproduct\anymatPBX{\orderind}{\inds}{\thedim}\anymatBX{1}{\thedim},
\end{IEEEeqnarray*}
i.e., as the terms  in the argument of $\thesplitting{\orderind-2}$ in the second line of \eqref{eq:higher_orders_fourth_case_first_helper_1}. Hence, the claim is true for $\recvar=1$. 
\par
\emph{Induction step.} Let $\recvar\in\SYnumbers{\thedim}$ be such that $\recvar\leq \thedim-1$ and such that \eqref{eq:higher_orders_fourth_case_first_helper_0} holds. By the argument below, \eqref{eq:higher_orders_fourth_case_first_helper_0} is then still true if $\recvar$ is replaced by $\recvar+1$.
\par
Since $1\leq \recvar$ the monomial $\anymatPBX{\orderind}{1}{\Sexchanged(\recvar+1)}\anymatPATX{\orderind}{\recvar+1}{\thedim}$ is normal by Lemma~\ref{characterization_of_reduced_terms_of_order_two}. The tip of the argument of $\thesplitting{\orderind-2}$ on the right-hand side of \eqref{eq:higher_orders_fourth_case_first_helper_0} is the vector $\Sbv{\anymatAT}{1}{1}\Smonoidalproduct\anymatPBX{\orderind}{1}{\Sexchanged(\recvar+1)}\anymatPATX{\orderind}{\recvar+1}{\thedim}$ for the following reasons. It is greater than any other summand in the grouped sum where it originates because for any $\inds\in\SYnumbers{\thedim}$ with $\recvar+1<\inds$ necessarily $\anymatPBX{\orderind}{1}{\Sexchanged(\inds)}\theorderRstrict\anymatPBX{\orderind}{1}{\Sexchanged(\recvar+1)}$ by Lem\-ma~\ref{lemma:monomial_order_reformulation} and thus $\Sbv{\anymatAT}{1}{1}\Smonoidalproduct\anymatPBX{\orderind}{1}{\Sexchanged(\inds)}\anymatPATX{\orderind}{\inds}{\thedim}\thechainsorderRstrict{\orderind-2}\Sbv{\anymatAT}{1}{1}\Smonoidalproduct\anymatPBX{\orderind}{1}{\Sexchanged(\recvar+1)}\anymatPATX{\orderind}{\recvar+1}{\thedim}$. Any term in the fourth line of \eqref{eq:higher_orders_fourth_case_first_helper_0} is dominated because according to the addendum to  Lem\-ma~\ref{lemma:base_monomial_order}  already $\Sbv{\anymatAT}{1}{\inds}<\Sbv{\anymatAT}{1}{1}$ for any $\inds\in\SYnumbers{\thedim}$ with $1<\inds$. And all the remaining terms are less for degree reasons.
\par
Because by Proposition~\ref{proposition:combinatorial_splitting}
\begin{IEEEeqnarray*}{rCl}
  \textstyle\thereconcatenator{\orderind-2}(  \Sbvfull{\anymatAT}{\orderind-2}{1}{1},\anymatPBX{\orderind}{1}{\Sexchanged(\recvar+1)}\anymatPATX{\orderind}{\recvar+1}{\thedim})&=&(  \Sbvfull{\anymatAT}{\orderind-1}{1}{\recvar+1},\anymatPATX{\orderind}{\recvar+1}{\thedim}),
\end{IEEEeqnarray*}
if on the right-hand side of \eqref{eq:higher_orders_fourth_case_first_helper_0}  the vector $-\Sbv{\anymatAT}{1}{\recvar+1}\Smonoidalproduct\anymatPATX{\orderind}{\recvar+1}{\thedim}$ is added outside of $\thesplitting{\orderind-2}$ and the vector $\thedifferential{\orderind-1}(\Sbv{\anymatAT}{1}{\recvar+1}\Smonoidalproduct\anymatPATX{\orderind}{\recvar+1}{\thedim})$ inside, the equality is preserved.
\par
Obviously, on the exterior of $\thesplitting{\orderind-2}$ in \eqref{eq:higher_orders_fourth_case_first_helper_0} this has the same effect as replacing $\recvar$ by $\recvar+1$ would. At the same time by $2\leq \thedim$ Assumption~\ref{assumption:higher_orders} gives
\begin{IEEEeqnarray*}{rCl}
    \thedifferential{\orderind-1}(\Sbv{\anymatAT}{1}{\recvar+1}\Smonoidalproduct \theone)
&= &\textstyle  \sum_{\inds=1}^\thedim\Sbv{\anymatAT}{1}{\inds}\Smonoidalproduct\anymatPBX{\orderind}{\inds}{\Sexchanged(\recvar+1)}-(-1)^\orderind\Saction\Sbv{\anymatB}{\thedim}{\recvar+1}\Smonoidalproduct\theone\\
&&\textstyle{}-\Skronecker{\recvar+1}{\thedim}(-1)^\orderind(\Sbv{\anymatBT}{1}{1}+\Skronecker{\orderind}{4}\sum_{\inds=2}^{\thedim-1}\Sbv{\anymatBT}{\inds}{\inds})\Smonoidalproduct\theone.
\end{IEEEeqnarray*}
Multiplying  with $\anymatPATX{\orderind}{\recvar+1}{\thedim}$ from the right and splitting off the $\inds=1$ term in the first line yields
\begin{IEEEeqnarray*}{rCl}
  \IEEEeqnarraymulticol{3}{l}{
    \thedifferential{\orderind-1}(\Sbv{\anymatAT}{1}{\recvar+1}\Smonoidalproduct \anymatPATX{\orderind}{\recvar+1}{\thedim})
    }\\
\hspace{1em}&= &\textstyle \Sbv{\anymatAT}{1}{1}\Smonoidalproduct\anymatPBX{\orderind}{1}{\Sexchanged(\recvar+1)}\anymatPATX{\orderind}{\recvar+1}{\thedim}+ \sum_{\inds=2}^\thedim\Sbv{\anymatAT}{1}{\inds}\Smonoidalproduct\anymatPBX{\orderind}{\inds}{\Sexchanged(\recvar+1)}\anymatPATX{\orderind}{\recvar+1}{\thedim}\IEEEyesnumber\label{eq:higher_orders_fourth_case_first_helper_3}\\
&&\textstyle{}\hspace{3em}-(-1)^\orderind\Saction\Sbv{\anymatB}{\thedim}{\recvar+1}\Smonoidalproduct\anymatPATX{\orderind}{\recvar+1}{\thedim}\\
&&\textstyle{}\hspace{6em}-\Skronecker{\thedim}{\recvar+1}(-1)^\orderind(\Sbv{\anymatBT}{1}{1}+\Skronecker{\orderind}{4}\sum_{\inds=2}^{\thedim-1}\Sbv{\anymatBT}{\inds}{\inds})\Smonoidalproduct\anymatPAX{\orderind}{\thedim}{\thedim}, 
\end{IEEEeqnarray*}
where it has been used that $\Skronecker{\recvar+1}{\thedim}\Saction\anymatPATX{\orderind}{\recvar+1}{\thedim}=\Skronecker{\thedim}{\recvar+1}\Saction\anymatPAX{\orderind}{\thedim}{\thedim}$ at the end. Because $\Szetafunction{\thedim}{\recvar}+\Skronecker{\thedim}{\recvar+1}=\Szetafunction{\thedim}{\recvar+1}$  it is now evident that adding the vector \eqref{eq:higher_orders_fourth_case_first_helper_3} to the argument of $\thesplitting{\orderind-2}$  changes \eqref{eq:higher_orders_fourth_case_first_helper_0} in the same way as increasing $\recvar$ by $1$. That completes the induction.
\end{proof}

\begin{lemma}
    \label{lemma:higher_orders_fourth_case_second_helper}
    For any $\anymat\in\thematrices$,
    \begin{IEEEeqnarray*}{rCl}
      \IEEEeqnarraymulticol{3}{l}{
        \thedifferential{\orderind}(\Sbv{\anymatA}{\thedim}{\thedim}\Smonoidalproduct\theone)        }\\      \hspace{2em}&=&\textstyle\sum_{\inds=1}^\thedim\Sbv{\anymatA}{\thedim}{\inds}\Smonoidalproduct\anymatPAX{\orderind}{\inds}{1}-\sum_{\inds=2}^{\thedim}\Sbv{\anymatAT}{1}{\inds}\Smonoidalproduct\anymatPATX{\orderind}{\inds}{\thedim}-(-1)^\orderind\Saction\Sbv{\anymatB}{\thedim}{1}\Smonoidalproduct\theone\IEEEeqnarraynumspace\IEEEyesnumber\label{eq:higher_orders_fourth_case_second_helper_0}\\
      &&\textstyle{}+\thesplitting{\orderind-2}\big((-1)^\orderind\sum_{\inds=1}^\thedim\Sbv{\anymatBT}{1}{\inds}\Smonoidalproduct\anymatPAX{\orderind}{\inds}{1}-\Sbv{\anymatA}{\thedim}{\thedim}\Smonoidalproduct\theone\\
      &&\textstyle\hspace{12em}{}-(\Sbv{\anymatAT}{1}{1}+
    \Skronecker{\orderind}{4}\sum_{\inds=2}^{\thedim-1}\Sbv{\anymatAT}{\inds}{\inds})
    \Smonoidalproduct\theone\big).
  \end{IEEEeqnarray*}    
\end{lemma}
\begin{proof}
  Instatiating \eqref{eq:higher_orders_fourth_case_first_helper_0} in Lem\-ma~\ref{lemma:higher_orders_fourth_case_first_helper} at $\recvar=\thedim$  renders the  third and fourth lines of \eqref{eq:higher_orders_fourth_case_first_helper_0} zero. In the case of the third line that is obvious. The fourth line vanishes because
  \begin{IEEEeqnarray*}{rCl}      \textstyle\sum_{\inds=2}^\thedim\Sbv{\anymatAT}{1}{\inds}\Smonoidalproduct\sum_{\indt=1}^{\thedim}\anymatPBX{\orderind}{\inds}{\Sexchanged(\indt)}\anymatPATX{\orderind}{\indt}{\thedim}&=&\textstyle\sum_{\inds=2}^\thedim\Skronecker{\inds}{1}\Saction\Sbv{\anymatAT}{1}{\inds}\Smonoidalproduct\theone=0.
  \end{IEEEeqnarray*}
  Hence, what survives the specialization to $\recvar=\thedim$ is the identity
  \begin{IEEEeqnarray*}{rCl}
    \IEEEeqnarraymulticol{3}{l}{
      \thedifferential{\orderind}(\Sbv{\anymatA}{\thedim}{\thedim}\Smonoidalproduct\theone)
    }\\
    \hspace{2em}&=&\textstyle\sum_{\inds=1}^\thedim\Sbv{\anymatA}{\thedim}{\inds}\Smonoidalproduct\anymatPAX{\orderind}{\inds}{1}-\sum_{\inds=2}^{\thedim}\Sbv{\anymatAT}{1}{\inds}\Smonoidalproduct\anymatPATX{\orderind}{\inds}{\thedim}\IEEEyesnumber\label{eq:higher_orders_fourth_case_second_helper_1}\\
    &&\textstyle{}+\thesplitting{\orderind-2}(-(-1)^\orderind\sum_{\inds=1}^\thedim\Sbv{\anymatB}{\thedim}{\inds}\Smonoidalproduct\anymatPATX{\orderind}{\inds}{\thedim}+(-1)^\orderind\sum_{\inds=1}^\thedim\Sbv{\anymatBT}{1}{\inds}\Smonoidalproduct\anymatPAX{\orderind}{\inds}{1}\\
    &&\textstyle{}\hspace{5em}-\Sbv{\anymatA}{\thedim}{\thedim}\Smonoidalproduct\theone-
    \Skronecker{\orderind}{4}\sum_{\inds=2}^{\thedim-1}\Sbv{\anymatAT}{\inds}{\inds}
    \Smonoidalproduct\theone\\    &&\textstyle{}\hspace{8em}-(-1)^\orderind(\Sbv{\anymatBT}{1}{1}+\Skronecker{\orderind}{4}\sum_{\inds=2}^{\thedim-1}\Sbv{\anymatB}{\inds}{\inds})\Smonoidalproduct\anymatPAX{\orderind}{\thedim}{\thedim}).
  \end{IEEEeqnarray*}
  \par 
The vector $\Sbv{\anymatB}{\thedim}{1}\Smonoidalproduct\anymatPATX{\orderind}{1}{\thedim}$ is the tip of the argument of $\thesplitting{\orderind-2}$ in  \eqref{eq:higher_orders_fourth_case_second_helper_1}. Indeed, any other term in the same grouped sum is less by the addendum to Lem\-ma~\ref{lemma:base_monomial_order} because $\Sbv{\anymatB}{\thedim}{\inds}\theorderRstrict\Sbv{\anymatB}{\thedim}{1}$ for any $\inds\in\SYnumbers{\thedim}$ with $1<\inds$. Any term of the form $\Sbv{\anymatBT}{1}{\inds}\Smonoidalproduct\anymatPAX{\orderind}{\inds}{1}$ for some $\inds\in\SYnumbers{\thedim}$ is dominated because already $\Sbv{\anymatBT}{1}{\inds}\theorderRstrict\Sbv{\anymatB}{\thedim}{1}$ by  Lem\-ma~\ref{lemma:base_monomial_order}. For the same reason $\Sbv{\anymatBT}{\inds}{\inds}<\Sbv{\anymatB}{\thedim}{1}$ for any $\inds\in\SYnumbers{\thedim}$ with $\inds<\thedim$. Hence, $\Sbv{\anymatB}{\thedim}{1}\Smonoidalproduct\anymatPATX{\orderind}{1}{\thedim}$ is also greater than any term in the fifth line of \eqref{eq:higher_orders_fourth_case_second_helper_1}. Of course all other terms are less already for reasons of degree.
  \par
  In consequence and because, according to Proposition~\ref{proposition:combinatorial_splitting},
  \begin{IEEEeqnarray*}{rCl}
    \textstyle\thereconcatenator{\orderind-2}(\Sbvfull{\anymatB}{\orderind-2}{\thedim}{1},\anymatPATX{\orderind}{1}{\thedim})=(\Sbvfull{\anymatB}{\orderind-1}{\thedim}{1},\theone),
  \end{IEEEeqnarray*}
  if $-(-1)^\orderind\Saction\Sbv{\anymatB}{\thedim}{1}\Smonoidalproduct\theone$ is added to the outisde of $\thesplitting{\orderind-2}$ and $(-1)^\orderind\thedifferential{\orderind-1}(\Sbv{\anymatB}{\thedim}{1}\Smonoidalproduct\theone)$ to the inside on the right-hand side of \eqref{eq:higher_orders_fourth_case_second_helper_1} the identity remains true.
  \par
  With this change, the outside of $\thesplitting{\orderind-2}$ in \eqref{eq:higher_orders_fourth_case_second_helper_1} becomes that in \eqref{eq:higher_orders_fourth_case_second_helper_0}. And since $2\leq \thedim$ Assumption~\ref{assumption:higher_orders} implies
  \begin{IEEEeqnarray*}{rCl}
(-1)^\orderind    \thedifferential{\orderind-1}(\Sbv{\anymatB}{\thedim}{1}\Smonoidalproduct \theone)
&= &\textstyle  (-1)^\orderind\sum_{\inds=1}^\thedim\Sbv{\anymatB}{\thedim}{\inds}\Smonoidalproduct\anymatPATX{\orderind}{\inds}{\thedim}-\Saction\Sbv{\anymatAT}{1}{1}\Smonoidalproduct\theone\IEEEyesnumber\label{eq:higher_orders_fourth_case_second_helper_2}\\
&&\textstyle{}+(-1)^\orderind(\Sbv{\anymatBT}{1}{1}+\Skronecker{\orderind}{4}\sum_{\inds=2}^{\thedim-1}\Sbv{\anymatBT}{\inds}{\inds})\Smonoidalproduct\anymatPAX{\orderind}{\thedim}{\thedim}.
\end{IEEEeqnarray*}
This shows that adding the term \eqref{eq:higher_orders_fourth_case_second_helper_2} to the argument of $\thesplitting{\orderind-2}$ in \eqref{eq:higher_orders_fourth_case_second_helper_1} transforms the latter into the argument in \eqref{eq:higher_orders_fourth_case_second_helper_0}. That concludes the proof.
\end{proof}
\begin{proposition}
  \label{proposition:higher_orders_fourth_case}
  For any $\anymat\in\thematrices$,
  \begin{IEEEeqnarray*}{rCl}
\thedifferential{\orderind}(\Sbv{\anymatA}{\thedim}{\thedim}\Smonoidalproduct\theone)&=&\textstyle\sum_{\inds=1}^\thedim\Sbv{\anymatA}{\thedim}{\inds}\Smonoidalproduct\anymatPAX{\orderind}{\inds}{1}-\sum_{\inds=2}^{\thedim}\Sbv{\anymatAT}{1}{\inds}\Smonoidalproduct\anymatPATX{\orderind}{\inds}{\thedim}\IEEEeqnarraynumspace\IEEEyesnumber\label{eq:higher_orders_fourth_case_0}\\
      &&\textstyle\hspace{7.5em}-(-1)^\orderind\Saction\Sbv{\anymatB}{\thedim}{1}\Smonoidalproduct\theone+(-1)^\orderind\Saction\Sbv{\anymatBT}{1}{\thedim}\Smonoidalproduct\theone.
  \end{IEEEeqnarray*}    
\end{proposition}
\begin{proof}
  The tip of the argument of $\thesplitting{\orderind-2}$ in \eqref{eq:higher_orders_fourth_case_second_helper_0} in Lem\-ma~\ref{lemma:higher_orders_fourth_case_second_helper} is  $\Sbv{\anymatBT}{1}{1}\Smonoidalproduct\anymatPAX{\orderind}{1}{1}$. Any other term in the grouped sum in the third line of \eqref{eq:higher_orders_fourth_case_second_helper_0} is less because  $\Sbv{\anymatBT}{1}{\inds}\theorderRstrict\Sbv{\anymatBT}{1}{1}$ for any $\inds\in\SYnumbers{\thedim}$ with $1<\inds$ by the addendum to  Lem\-ma~\ref{lemma:base_monomial_order}. And all remaining terms are dominated for degree reasons.
  \par
  Since by Proposition~\ref{proposition:combinatorial_splitting}
  \begin{IEEEeqnarray*}{rCl}
    \textstyle\thereconcatenator{\orderind-2}(\Sbvfull{\anymatBT}{\orderind-2}{1}{1},\anymatPAX{\orderind}{1}{1})&=&(\Sbvfull{\anymatBT}{\orderind-1}{1}{\thedim},\theone)
  \end{IEEEeqnarray*}
  it follows that the veracity of \eqref{eq:higher_orders_fourth_case_second_helper_0} is not affected by adding the vector $(-1)^\orderind\Saction\Sbv{\anymatBT}{1}{\thedim}\Smonoidalproduct\theone$ to the outside of  $\thesplitting{\orderind-2}$ and $-(-1)^\orderind\thedifferential{\orderind-1}(\Sbv{\anymatBT}{1}{\thedim}\Smonoidalproduct\theone)$ to the inside.
  \par
  This change makes the outsides of $\thesplitting{\orderind-2}$ in \eqref{eq:higher_orders_fourth_case_second_helper_0} and in \eqref{eq:higher_orders_fourth_case_0} agree. On the other hand, since by Assumption~\ref{assumption:higher_orders}
\begin{IEEEeqnarray*}{rCl}
-(-1)^\orderind    \thedifferential{\orderind-1}(\Sbv{\anymatBT}{1}{\thedim}\Smonoidalproduct \theone)
&= &\textstyle  -(-1)^\orderind\sum_{\inds=1}^\thedim\Sbv{\anymatBT}{1}{\inds}\Smonoidalproduct\anymatPAX{\orderind}{\inds}{1}+\Sbv{\anymatA}{\thedim}{\thedim}\Smonoidalproduct\theone\IEEEyesnumber\label{eq:higher_orders_fourth_case_1}\\
&&\textstyle\hspace{6em}{}+(\Sbv{\anymatAT}{1}{1}+\Skronecker{\orderind}{4}\sum_{\inds=2}^{\thedim-1}\Sbv{\anymatAT}{\inds}{\inds})\Smonoidalproduct\theone
\end{IEEEeqnarray*}
the argument of $\thesplitting{\orderind-2}$ in \eqref{eq:higher_orders_fourth_case_second_helper_0} is actually the negative of the term \eqref{eq:higher_orders_fourth_case_1}. That proves \eqref{eq:higher_orders_fourth_case_0}.
\end{proof}

\subsubsection{Synthesis}
The four particular results obtained in the different cases assemble into the claim from the \hyperref[main-result]{Main result}. Recall that $4\leq \orderind$.
\begin{theorem}
  \label{theorem:higher_orders}
  For any $\anymat\in\thematrices$ and $(\indj,\indi)\in \SYnumbers{\thedim}^{\Ssetmonoidalproduct 2}$,
  \begin{IEEEeqnarray*}{rCl}
    \IEEEeqnarraymulticol{3}{l}{
      \thedifferential{\orderind}(\Sbv{\anymatA}{\indj}{\indi}\Smonoidalproduct \theone)
    }\\
\hspace{1em}    &= &\textstyle  \sum_{\inds=1}^\thedim\Sbv{\anymatA}{\indj}{\inds}\Smonoidalproduct\anymatPAX{\orderind}{\inds}{\Sexchanged(\indi)}+(-1)^{\orderind}\Saction\Sbv{\anymatBT}{\Sexchanged(\indj)}{\indi}\Smonoidalproduct\theone\IEEEyesnumber\label{eq:higher_orders_0}\\
&&\textstyle{}+\Skronecker{\indj}{\thedim}(\Sbv{\anymatAT}{1}{1}\Smonoidalproduct\anymatPATX{\orderind}{\Sexchanged(\indi)}{\thedim}-\Skronecker{\indi}{\thedim}(\sum_{\inds=1}^\thedim\Sbv{\anymatAT}{1}{\inds}\Smonoidalproduct\anymatPATX{\orderind}{\inds}{\thedim}+(-1)^{\orderind}\Saction\Sbv{\anymatB}{\thedim}{1}\Smonoidalproduct\theone))\\
&&\textstyle{}+\Skronecker{\indj}{1}\Skronecker{\indi}{\thedim}(-1)^{\orderind}\Saction \Sbv{\anymatB}{1}{1}\Smonoidalproduct\theone.    
  \end{IEEEeqnarray*}
\end{theorem}
\begin{proof}
The first two lines of \eqref{eq:higher_orders_0} agree with \eqref{eq:higher_orders_first_case_0} from Pro\-po\-si\-tion~\ref{proposition:higher_orders_first_case}. Hence, in the case that $j\neq \thedim$ and $\indi\neq \thedim$ the claim is true.
  \par 
If $\indj=\thedim$ but $\indi\neq \thedim$, then the only non-zero term in the last two lines of     \eqref{eq:higher_orders_0} is $\Sbv{\anymatAT}{1}{1}\Smonoidalproduct\anymatPATX{\orderind}{\Sexchanged(\indi)}{\thedim}$. That is why  \eqref{eq:higher_orders_0} is identical to  \eqref{eq:higher_orders_second_case_0} from Proposition~\ref{proposition:higher_orders_second_case} in this case.
  \par 
In the opposite instance that $\indj\neq \thedim$ but $\indi=\thedim$ the third line of  \eqref{eq:higher_orders_0} vanishes and the fourth reads $\Skronecker{\indj}{1}(-1)^{\orderind}\Saction \Sbv{\anymatB}{1}{1}\Smonoidalproduct\theone$.  That is precisely the term which needs to be added on the right-hand side to turn the first two lines of \eqref{eq:higher_orders_0} into \eqref{eq:higher_orders_third_case_0} from Proposition~\ref{proposition:higher_orders_third_case}.
  \par
  In the final case where $\indj=\indi=\thedim$ the last line of \eqref{eq:higher_orders_0} is zero and what remains of the right-hand side is
  \begin{IEEEeqnarray*}{l}
\textstyle  \sum_{\inds=1}^\thedim\Sbv{\anymatA}{\thedim}{\inds}\Smonoidalproduct\anymatPAX{\orderind}{\inds}{1}+(-1)^{\orderind}\Saction\Sbv{\anymatBT}{1}{\thedim}\Smonoidalproduct\theone\IEEEyesnumber\label{eq:higher_orders_1}\\
\textstyle{}+\Sbv{\anymatAT}{1}{1}\Smonoidalproduct\anymatPATX{\orderind}{1}{\thedim}-\sum_{\inds=1}^\thedim\Sbv{\anymatAT}{1}{\inds}\Smonoidalproduct\anymatPATX{\orderind}{\inds}{\thedim}-(-1)^{\orderind}\Saction\Sbv{\anymatB}{\thedim}{1}\Smonoidalproduct\theone
\end{IEEEeqnarray*}
Because the term $\Sbv{\anymatAT}{1}{1}\Smonoidalproduct\anymatPATX{\orderind}{1}{\thedim}$ also occurs in the grouped sum in the middle of the second line of  \eqref{eq:higher_orders_1}, but with opposite sign, this proves that
\eqref{eq:higher_orders_0} is actually the same as 
  \eqref{eq:higher_orders_fourth_case_0} from Proposition~\ref{proposition:higher_orders_fourth_case} in this case. Hence, the claim holds for all $(\indj,\indi)$.
\end{proof}
}

\section{\texorpdfstring{Recomputing the cohomology of $U_n^+$}{Recomputing the cohomology of the free unitary quantum group}}
\label{section:cohomology}
{
  \newcommand{\orderind}{\ell}
  \newcommand{\indi}{i}
  \newcommand{\indj}{j}
  \newcommand{\anypoly}{p}
  \newcommand{\inds}{s}
  \newcommand{\indt}{t}
  \newcommand{\thehom}{g}
  \newcommand{\themat}[3]{g^{#1}_{#2,#3}} 
  In this section it is demonstrated that the Anick resolution of the right $\thealgebra$\-/module $\themodule$ obtained in Section~\ref{section:anick_computation} is good enough to compute
\begin{IEEEeqnarray*}{rCl}
\SextX{\thealgebra}{\themodule}{\themodule}{\Sargph},
\end{IEEEeqnarray*}
which in the case $\thefield=\Scomps$ gives the qunatum group co\-ho\-mo\-lo\-gy $H^{\Sargph}(\widehat{U_n^+})$ of the discrete dual of the free unitary quantum group $U_n^+$. Note that, as special cases of much further-reaching results, the orders $1$ and $2$ of  $\SextX{\thealgebra}{\themodule}{\themodule}{\Sargph}$ had already been computed by Das, Franz, Kula and Skalski in \cite{DasFranzKulaSkalski2018}  and the higher orders by Baraquin, Franz, Gerhold, Kula and Tobolski in \cite{BaraquinFranzGerholdKulaTobolski2023}. 
\par
Let $\ShomO{\thealgebra}$ be the external hom functor from the category of right $\thealgebra$\-/modules to the category of $\thefield$\-/vector spaces. 
Since $\thechainsmodule{\orderind}$ is finite\-/dimensional as an $\thealgebra$\-/module for each $\orderind\in\Sintegersnn$ it is enough to compute the defect of the $\thefield$\-/linear map $\ShomOX{\thealgebra}{\thedifferential{\orderind}}{\themodule}$ for each $\orderind\in\Sintegersp$.
\par
As in Section~\ref{section:anick_computation} the difference between $\anypoly$ and  $\anypoly+\theideal$ will be suppressed in the notation for any $\anypoly\in\thefreealg$.
\begin{notation}
  \label{notation:homs_on_basis}
  For any $\orderind\in \Sintegersp$, any $\thehom\in \ShomOX{\thealgebra}{\thechainsmodule{\orderind}}{\themodule}$, any $\anymat\in\thematrices$ and any $(\indj,\indi)\in\SYnumbers{\thedim}^{\Ssetmonoidalproduct 2}$  let
  \begin{IEEEeqnarray*}{rCl}
    \themat{\anymat}{\indj}{\indi}&\Seqpd &\thehom(\Sbvfull{\anymat}{\orderind}{\indj}{\indi}\Smonoidalproduct\theone).
  \end{IEEEeqnarray*}
\end{notation}

\subsection{First order}
\label{section:cohomology_first_order}
Computing the defect of $\ShomOX{\thealgebra}{\thedifferential{1}}{\themodule}$ is easy.
\begin{lemma}
  \label{lemma:homkernel_first_order}
For any $\thehom\in\ShomOX{\thealgebra}{\thechainsmodule{0}}{\themodule}$, any  $\anymat\in\thematrices$ and  $ (\indj,\indi)\in\SYnumbers{\thedim}^{\Ssetmonoidalproduct 2}$,
  \begin{IEEEeqnarray*}{rCl}
  (\thehom\Scomposition\thedifferential{1})(\Sbv{\anymat}{\indj}{\indi}\Smonoidalproduct\theone)&=&0.
  \end{IEEEeqnarray*}
\end{lemma}
\begin{proof}
Since $\thechainsmodule{0}$ is the regular $\thealgebra$-module and since by Proposition~\ref{proposition:first_order} 
  \begin{IEEEeqnarray*}{rCl}
      (\thehom\Scomposition\thedifferential{1})(\Sbv{\anymat}{\indj}{\indi}\Smonoidalproduct\theone)&=&\thehom(\anymatAX{\indj}{\indi}-\Skronecker{\indj}{\indi}\Saction\theone)=\thehom(\anymatAX{\indj}{\indi})-\Skronecker{\indj}{\indi}\Saction\thehom(\theone)
    \end{IEEEeqnarray*}
the fact     that $\thehom(\anymatAX{\indj}{\indi})=\Skronecker{\indj}{\indi}\Saction\thehom(\theone)$ by  definition of $\themodule$ proves the claim.
\end{proof}

\begin{lemma}
  \label{lemma:cohomology_first_order}
The defect and rank of  $\ShomOX{\thealgebra}{\thedifferential{1}}{\themodule}$ are $1$ and $0$ respectively.
\end{lemma}
\begin{proof} 
The domain $\thefield$\-/vector space $\ShomOX{\thealgebra}{\thechainsmodule{0}}{\themodule}=\ShomOX{\thealgebra}{\thealgebra}{\themodule}\cong \thefield$ of  $\ShomOX{\thealgebra}{\thedifferential{1}}{\themodule}$ is one\-/dimensional. Hence, the claim follows from Lemma~\ref{lemma:homkernel_first_order} by the dimension formula for $\thefield$\-/linear maps. 
\end{proof}

\subsection{Second order}
\label{section:cohomology_second_order}
The defect and rank of $\ShomOX{\thealgebra}{\thedifferential{2}}{\themodule}$ can still be computed without much effort. Recall Notation~\ref{notation:homs_on_basis}.
\begin{lemma}
  \label{lemma:homkernel_second_order_helper}
For any $\thehom\in\ShomOX{\thealgebra}{\thechainsmodule{1}}{\themodule}$, any  $\anymat\in\thematrices$ and  $ (\indj,\indi)\in\SYnumbers{\thedim}^{\Ssetmonoidalproduct 2}$,
\begin{IEEEeqnarray*}{rCl}
  (\thehom\Scomposition\thedifferential{2})(\Sbv{\anymat}{\indj}{\indi}\Smonoidalproduct\theone)&=&\textstyle\themat{\anymat}{\indj}{\Sexchanged(\indi)}+\themat{\anymatB}{\Sexchanged(\indi)}{\indj}-\Skronecker{\indj}{\thedim}\Skronecker{\indi}{\thedim}\sum_{\inds=2}^\thedim (\themat{\anymatA}{\Sexchanged(\inds)}{\inds}+\themat{\anymatB}{\inds}{\Sexchanged(\inds)}).
\IEEEeqnarraynumspace \IEEEyesnumber\label{eq:homkernel_second_order_helper_0}
  \end{IEEEeqnarray*}
\end{lemma}
\begin{proof}
 Theorem~\ref{theorem:second_order} shows that $(\thehom\Scomposition\thedifferential{2})(\Sbv{\anymat}{\indj}{\indi}\Smonoidalproduct\theone)$ is given by
  \begin{IEEEeqnarray*}{l}
\textstyle\sum_{\inds=1}^\thedim\thehom(\Sbv{\anymatA}{\indj}{\inds}\Smonoidalproduct\anymatBTX{\inds}{\Sexchanged(\indi)})+\thehom(\Sbv{\anymatBT}{\Sexchanged(\indj)}{\indi}\Smonoidalproduct\theone)\\
\hspace{4em}\textstyle{}-\Skronecker{\indj}{\thedim}\Skronecker{\indi}{\thedim}\sum_{\inds=2}^{\thedim}(\sum_{\indt=1}^{\thedim}\thehom(\Sbv{\anymatA}{\indt}{\inds}\Smonoidalproduct\anymatBTX{\inds}{\Sexchanged(\indt)})+\thehom(\Sbv{\anymatBT}{\inds}{\Sexchanged(\inds)}\Smonoidalproduct\theone)).
\end{IEEEeqnarray*}
By $\thehom(\Sbv{\anymatA}{\indj}{\inds}\Smonoidalproduct\anymatBTX{\inds}{\Sexchanged(\indi)})=\Skronecker{\inds}{\Sexchanged(\indi)}\Saction \themat{\anymatA}{\indj}{\inds}$ and $\thehom(\Sbv{\anymatA}{\indt}{\inds}\Smonoidalproduct\anymatBTX{\inds}{\Sexchanged(\indt)})=\Skronecker{\inds}{\Sexchanged(\indt)}\Saction\themat{\anymatA}{\indt}{\inds}$ this is equal to
\begin{IEEEeqnarray*}{l}
\textstyle\themat{\anymatA}{\indj}{\Sexchanged(\indi)}+\themat{\anymatBT}{\Sexchanged(\indj)}{\indi}-\Skronecker{\indj}{\thedim}\Skronecker{\indi}{\thedim}\sum_{\inds=2}^{\thedim}(\themat{\anymatA}{\Sexchanged(\inds)}{\inds}+\themat{\anymatBT}{\inds}{\Sexchanged(\inds)}).\IEEEyesnumber\label{eq:homkernel_second_order_helper_1}
  \end{IEEEeqnarray*}
Since for any   $\inds\in\SYnumbers{\thedim}$ Lemma~\ref{lemma:evil_identities} implies $\Sbvfull{\anymatBT}{1}{\Sexchanged(\indj)}{\indi}=\Sbvfull{\anymatB}{1}{\Sexchanged(\indi)}{\indj}$ and $\Sbvfull{\anymatBT}{1}{\inds}{\Sexchanged(\inds)}=\Sbvfull{\anymatB}{1}{\inds}{\Sexchanged(\inds)}$ and, thus,
  \begin{IEEEeqnarray*}{rCl}
   \themat{\anymatBT}{\Sexchanged(\indj)}{\indi}=\themat{\anymatB}{\Sexchanged(\indi)}{\indj},&\hspace{1em}\hspace{1em}& \themat{\anymatBT}{\inds}{\Sexchanged(\inds)}=\themat{\anymatB}{\inds}{\Sexchanged(\inds)} 
  \end{IEEEeqnarray*}
the vector \eqref{eq:homkernel_second_order_helper_1} is the right-hand side of \eqref{eq:homkernel_second_order_helper_0}. 
\end{proof}

\begin{lemma}
  \label{lemma:homkernel_second_order}
  Any $\thehom\in\ShomOX{\thealgebra}{\thechainsmodule{1}}{\themodule}$ satisfies  $\thehom\Scomposition\thedifferential{2}=0$ if and only if for any $\anymat\in\thematrices$ and  any $(\indj,\indi)\in\SYnumbers{\thedim}^{\Ssetmonoidalproduct 2}$,
  \begin{IEEEeqnarray*}{rCl}
  \themat{\anymatB}{\indj}{\indi}=-\themat{\anymat}{\indi}{\indj}.  \IEEEyesnumber\label{eq:homkernel_second_order_0}
\end{IEEEeqnarray*}
In fact, \eqref{eq:homkernel_second_order_0} for any $(\indj,\indi)\in\SYnumbers{\thedim}^{\Ssetmonoidalproduct 2}$ is satisfied by any $\anymat\in\thematrices$ as soon as it is by an arbitrary one $\anymat\in\thematrices$.
\end{lemma}
\begin{proof}
\emph{Step~1: Equivalence.}  By Lemma~\ref{lemma:homkernel_second_order_helper} it is enough to prove that the conditions from the claim are satisfied if and only if  for any $\anymat\in\thematrices$ and any $ (\indj,\indi)\in\SYnumbers{\thedim}^{\Ssetmonoidalproduct 2}$, 
\begin{IEEEeqnarray*}{rCl}
\textstyle \themat{\anymat}{\indj}{\Sexchanged(\indi)}+\themat{\anymatB}{\Sexchanged(\indi)}{\indj}-\Skronecker{\indj}{\thedim}\Skronecker{\indi}{\thedim}\sum_{\inds=2}^\thedim (\themat{\anymatA}{\Sexchanged(\inds)}{\inds}+\themat{\anymatB}{\inds}{\Sexchanged(\inds)})&=&0.\IEEEyesnumber\label{eq:homkernel_second_order_2}
\end{IEEEeqnarray*}
Obviously, if \eqref{eq:homkernel_second_order_0} holds for any  $\anymat\in\thematrices$ and $(\indj,\indi)\in\SYnumbers{\thedim}^{\Ssetmonoidalproduct 2}$, then so does \eqref{eq:homkernel_second_order_2}.
\par
Conversely, suppose \eqref{eq:homkernel_second_order_2} is true for any $\anymat\in\thematrices$ and $(\indj,\indi)\in\SYnumbers{\thedim}^{\Ssetmonoidalproduct 2}$. Let $\anymat\in\thematrices$ and $(\indj,\indi)\in\SYnumbers{\thedim}^{\Ssetmonoidalproduct 2}$ be arbitrary. If $(\indj,\indi)\neq (\thedim,1)$, which is to say $(\indj,\Sexchanged(\indi))\neq (\thedim,\thedim)$, condition  \eqref{eq:homkernel_second_order_2} with $(\anymatB,\indj,\Sexchanged(\indi))$ in the role of $(\anymat,\indj,\indi)$ collapes to the demand that $  \themat{\anymatB}{\indj}{\indi}+\themat{\anymatA}{\indi}{\indj}=0$. Hence, all that is left to prove is that also $\themat{\anymatB}{\thedim}{1}=-\themat{\anymatA}{1}{\thedim}$.
But for $(\anymat,1,1)$ in the role of $(\anymat,\indj,\indi)$ condition  \eqref{eq:homkernel_second_order_2} reads $\themat{\anymatA}{1}{\thedim}+\themat{\anymatB}{\thedim}{1}=0$. Hence, the claimed equivalence is true.
\par
  \emph{Step~2: Addendum.} It remains to prove the addendum. Hence, suppose there exists some $\othermat\in\thematrices$ such that, for  $\anymat=\othermat$ specifically, \eqref{eq:homkernel_second_order_0} holds for any $(\indj,\indi)\in\SYnumbers{\thedim}^{\Ssetmonoidalproduct 2}$. By symmetry it is then true that \eqref{eq:homkernel_second_order_0} is satisfied for any $\anymat\in\{\othermatA,\othermatB\}$ and $(\indj,\indi)\in\SYnumbers{\thedim}^{\Ssetmonoidalproduct 2}$.
Because Lemma~\ref{lemma:evil_identities} and the  assumption imply
\begin{IEEEeqnarray*}{rCl}
 \themat{\othermatBT}{\indj}{\indi}=\themat{\othermatB}{\Sexchanged(\indi)}{\Sexchanged(\indj)}=-\themat{\othermatA}{\Sexchanged(\indj)}{\Sexchanged(\indi)}=-\themat{\othermatAT}{\indi}{\indj} 
\end{IEEEeqnarray*}
equation  \eqref{eq:homkernel_second_order_0} is actually satisfied for any $\anymat\in\thematrices$ and $(\indj,\indi)\in\SYnumbers{\thedim}^{\Ssetmonoidalproduct 2}$. That concludes the proof.
\end{proof}

\begin{lemma}
  \label{lemma:cohomology_second_order}
Both defect and rank of $\ShomOX{\thealgebra}{\thedifferential{2}}{\themodule}$ are $\thedim^2$.
\end{lemma}
\begin{proof}
  The domain  space $\ShomOX{\thealgebra}{\thechainsmodule{1}}{\themodule}=\ShomOX{\thealgebra}{\thefield\thechains{1}\SmonoidalproductC{\thefield}\thealgebra}{\themodule}\cong \thefield\thechains{1}$ of $\ShomOX{\thealgebra}{\thedifferential{2}}{\themodule}$ has dimension $|\thechains{1}|=|\{\Sbvfull{\theuniA}{1}{\indj}{\indi},\Sbvfull{\theuniB}{1}{\indj}{\indi}\}_{\indj,\indi=1}^\thedim|=2\thedim^2$. Lemma~\ref{lemma:homkernel_second_order} shows that the kernel of $\ShomOX{\thealgebra}{\thedifferential{2}}{\themodule}$ is in particular free over the set $\{\Sbvfull{\theuniA}{1}{\indj}{\indi}\}_{\indj,\indi=1}^\thedim$, which has $\thedim^2$ many elements. Hence, the claim follows by the dimension formula.
\end{proof}
\subsection{Third and higher orders}
\label{section:cohomology_third_order}
It is most economic to first treat all the maps $\ShomOX{\thealgebra}{\thedifferential{\orderind}}{\themodule}$  for $3\leq \orderind$ simultaneously up to a certain point and then distinguish whether $\orderind=3$ or not.
\subsubsection{Common results for orders three and higher} The strategy remains the same as for orders $1$ and $2$. Only now, more effort is required.
\begin{lemma}
  \label{lemma:homkernel_third_and_higher_orders_helper}
  For any $\orderind\in \Sintegersp$ with $3\leq \orderind$, any $\thehom\in\ShomOX{\thealgebra}{\thechainsmodule{\orderind-1}}{\themodule}$, any $\anymat\in\thematrices$ and any $(\indj,\indi)\in\SYnumbers{\thedim}^{\Ssetmonoidalproduct 2}$,
  \begin{IEEEeqnarray*}{rCl}
    (\thehom\Scomposition\thedifferential{\orderind})(\Sbv{\anymatA}{\indj}{\indi}\Smonoidalproduct\theone)&=&
    \themat{\anymatA}{\indj}{\Sexchanged(\indi)}+(-1)^\orderind\Saction\themat{\anymatBT}{\Sexchanged(\indj)}{\indi}\IEEEyesnumber\label{eq:homkernel_third_and_higher_orders_helper_0}\\
    &&\textstyle{}-\Skronecker{\indj}{\thedim}\Skronecker{\indi}{\thedim}\Saction(\themat{\anymatAT}{1}{\thedim}+(-1)^\orderind\Saction\themat{\anymatB}{\thedim}{1})\\
    &&\textstyle{}+\Skronecker{\indj}{1}\Skronecker{\indi}{\thedim}(-1)^\orderind\Saction(\themat{\anymatB}{1}{1}+\Skronecker{\orderind}{3}\sum_{\inds=2}^{\thedim-1}\themat{\anymatB}{\inds}{\inds})\\
    &&\textstyle{}+\Skronecker{\indj}{\thedim}\Skronecker{\indi}{1}\Saction(\themat{\anymatAT}{1}{1}+\Skronecker{\orderind}{3}\Saction\sum_{\inds=2}^{\thedim-1}\themat{\anymatAT}{\inds}{\inds}).
  \end{IEEEeqnarray*}  
\end{lemma}
\begin{proof}
  According to Theorems~\ref{theorem:third_order} and \ref{theorem:higher_orders} the vector     $(\thehom\Scomposition\thedifferential{\orderind})(\Sbv{\anymatA}{\indj}{\indi}\Smonoidalproduct\theone)$ is given by
  \begin{IEEEeqnarray*}{rCl}
& &\textstyle  \sum_{\inds=1}^\thedim\thehom(\Sbv{\anymatA}{\indj}{\inds}\Smonoidalproduct\anymatPAX{\orderind}{\inds}{\Sexchanged(\indi)})+(-1)^{\orderind}\Saction\thehom(\Sbv{\anymatBT}{\Sexchanged(\indj)}{\indi}\Smonoidalproduct\theone)\\
&&\textstyle{}+\Skronecker{\indj}{\thedim}\Saction(\thehom(\Sbv{\anymatAT}{1}{1}\Smonoidalproduct\anymatPATX{\orderind}{\Sexchanged(\indi)}{\thedim})+\Skronecker{\orderind}{3}\sum_{\inds=2}^{\thedim-1}\thehom(\Sbv{\anymatAT}{\inds}{\inds}\Smonoidalproduct\anymatPATX{\orderind}{\Sexchanged(\indi)}{\thedim})\\
&&\textstyle\hspace{3em}\hfill{}-\Skronecker{\indi}{\thedim}\Saction(\sum_{\inds=1}^\thedim\thehom(\Sbv{\anymatAT}{1}{\inds}\Smonoidalproduct\anymatPATX{\orderind}{\inds}{\thedim})+(-1)^{\orderind}\Saction\thehom(\Sbv{\anymatB}{\thedim}{1}\Smonoidalproduct\theone)))\\
&&\textstyle{}+\Skronecker{\indj}{1}\Skronecker{\indi}{\thedim}(-1)^{\orderind}\Saction (\thehom(\Sbv{\anymatB}{1}{1}\Smonoidalproduct\theone)+\Skronecker{\orderind}{3}\sum_{\inds=2}^{\thedim-1}\Saction\thehom(\Sbv{\anymatB}{\inds}{\inds}\Smonoidalproduct\theone)).
\end{IEEEeqnarray*} 
Recognizing $\thehom(\Sbv{\anymatA}{\indj}{\inds}\Smonoidalproduct\anymatPAX{\orderind}{\inds}{\Sexchanged(\indi)})=\Skronecker{\inds}{\Sexchanged(\indi)}\Saction\themat{\anymatA}{\indj}{\inds}$ and $\thehom(\Sbv{\anymatAT}{\inds}{\inds}\Smonoidalproduct\anymatPATX{\orderind}{\Sexchanged(\indi)}{\thedim})=\Skronecker{\Sexchanged(\indi)}{\thedim}\Saction\themat{\anymatAT}{\inds}{\inds}=\Skronecker{\indi}{1}\Saction\themat{\anymatAT}{\inds}{\inds}$ and $\thehom(\Sbv{\anymatAT}{1}{\inds}\Smonoidalproduct\anymatPATX{\orderind}{\inds}{\thedim})=\Skronecker{\inds}{\thedim}\Saction\themat{\anymatAT}{1}{\inds}$ for any $\inds\in\SYnumbers{\thedim}$
and reordering the lines shows that the above is equal to the right-hand side of \eqref{eq:homkernel_third_and_higher_orders_helper_0}.
\end{proof}

\begin{lemma}
  \label{lemma:homkernel_third_and_higher_orders}
For any $\orderind\in\Sintegersp$ with $3\leq \orderind$, any $\thehom\in\ShomOX{\thealgebra}{\thechainsmodule{\orderind-1}}{\themodule}$ satisfies $\thehom\Scomposition\thedifferential{\orderind}=0$ if and only if for any $\anymat\in\thematrices$, all in conjunction, for any $(\indj,\indi)\in\SYnumbers{\thedim}^{\Ssetmonoidalproduct 2}$ with $(\indj,\indi)\neq (1,1)$,
  \begin{IEEEeqnarray*}{rCl}
    \themat{\anymatAT}{\indj}{\indi}&=&\textstyle-(-1)^\orderind \themat{\anymatB}{\Sexchanged(\indj)}{\Sexchanged(\indi)}-\Skronecker{\indj}{\thedim}\Skronecker{\indi}{\thedim}\Saction(\themat{\anymatA}{1}{1}+\Skronecker{\orderind}{3}\sum_{\inds=2}^{\thedim-1}\themat{\anymatA}{\inds}{\inds})\IEEEyesnumber\label{eq:homkernel_third_and_higher_orders_01},    \\ \themat{\anymatBT}{\indj}{\indi}&=&\textstyle-(-1)^\orderind \themat{\anymatA}{\Sexchanged(\indj)}{\Sexchanged(\indi)}-\Skronecker{\indj}{\thedim}\Skronecker{\indi}{\thedim}\Saction(\themat{\anymatB}{1}{1}+\Skronecker{\orderind}{3}\sum_{\inds=2}^{\thedim-1}\themat{\anymatB}{\inds}{\inds})\IEEEyesnumber\label{eq:homkernel_third_and_higher_orders_02}
  \end{IEEEeqnarray*}
  and
  \begin{IEEEeqnarray*}{rCl}
    \themat{\anymatBT}{1}{1}&=&\textstyle-(-1)^\orderind\Saction\themat{\anymatAT}{1}{1}-(-1)^\orderind\Saction \themat{\anymatA}{\thedim}{\thedim}+\Skronecker{\orderind}{3}\sum_{\inds=2}^{\thedim-1}\themat{\anymatB}{\inds}{\inds}\IEEEyesnumber\label{eq:homkernel_third_and_higher_orders_03}
  \end{IEEEeqnarray*}
  as well as
  \begin{IEEEeqnarray*}{rCl}
    \textstyle\themat{\anymatB}{\thedim}{\thedim}&=&\textstyle-\Skronecker{\orderind}{3}\sum_{\inds=2}^{\thedim-1}\themat{\anymatB}{\inds}{\inds}+(-1)^{\orderind}(\themat{\anymat}{\thedim}{\thedim}+\Skronecker{\orderind}{3}\sum_{\inds=2}^{\thedim-1}\themat{\anymatA}{\inds}{\inds}).\IEEEyesnumber\label{eq:homkernel_third_and_higher_orders_04}
  \end{IEEEeqnarray*}
In fact, \eqref{eq:homkernel_third_and_higher_orders_01}--\eqref{eq:homkernel_third_and_higher_orders_04} for any $(\indj,\indi)\in \SYnumbers{\thedim}^{\Ssetmonoidalproduct 2}$ with $(\indj,\indi)\neq (1,1)$ are all satisfied by each $\anymat\in\thematrices$ as soon as they are all satisfied by an arbitrary one $\anymat\in\thematrices$.
\end{lemma}
\begin{proof}
  The claimed equivalence and the addendum will be proved separately. 
  \par
\emph{Step~1: Equivalence.}  By Lemma~\ref{lemma:homkernel_third_and_higher_orders_helper} it is enough to prove that the conditions from the claim and the demand that for any $\anymat\in\thematrices$ and $(\indj,\indi)\in\SYnumbers{\thedim}^{\Ssetmonoidalproduct 2}$,
    \begin{IEEEeqnarray*}{rCl}
0&=&    \themat{\anymatA}{\indj}{\Sexchanged(\indi)}+(-1)^\orderind\Saction\themat{\anymatBT}{\Sexchanged(\indj)}{\indi}\IEEEyesnumber\label{eq:homkernel_third_and_higher_orders_1}\\
    &&\textstyle{}-\Skronecker{\indj}{\thedim}\Skronecker{\indi}{\thedim}\Saction(\themat{\anymatAT}{1}{\thedim}+(-1)^\orderind\Saction\themat{\anymatB}{\thedim}{1})\\
    &&\textstyle{}+\Skronecker{\indj}{1}\Skronecker{\indi}{\thedim}(-1)^\orderind\Saction(\themat{\anymatB}{1}{1}+\Skronecker{\orderind}{3}\sum_{\inds=2}^{\thedim-1}\themat{\anymatB}{\inds}{\inds})\\
    &&\textstyle{}+\Skronecker{\indj}{\thedim}\Skronecker{\indi}{1}\Saction(\themat{\anymatAT}{1}{1}+\Skronecker{\orderind}{3}\Saction\sum_{\inds=2}^{\thedim-1}\themat{\anymatAT}{\inds}{\inds})
  \end{IEEEeqnarray*}  
  are equivalent. The two implications will be treated separately.
  \par
  \emph{Step~1.1.} First, let  \eqref{eq:homkernel_third_and_higher_orders_01}--\eqref{eq:homkernel_third_and_higher_orders_04} hold for any $\anymat\in\thematrices$ and any $(\indj,\indi)\in\SYnumbers{\thedim}^{\Ssetmonoidalproduct 2}$ with $(\indj,\indi)\neq(1,1)$. For any $\anymat\in\thematrices$ instantiating \eqref{eq:homkernel_third_and_higher_orders_01} at $(\indj,\indi)=(1,\thedim)$ produces the identity $\themat{\anymatAT}{1}{\thedim}=-(-1)^\orderind\Saction\themat{\anymatB}{\thedim}{1}$. The latter shows that the second line of \eqref{eq:homkernel_third_and_higher_orders_1} vanishes for arbitrary $\anymat\in\thematrices$ and $(\indj,\indi)\in \SYnumbers{\thedim}^{\Ssetmonoidalproduct 2}$. Now, let $\anymat\in\thematrices$ and $(\indj,\indi)\in \SYnumbers{\thedim}^{\Ssetmonoidalproduct 2}$ be arbitrary. Then \eqref{eq:homkernel_third_and_higher_orders_1} is satisfied by  $(\anymat,\indj,\indi)$ for the following reasons.
  \par
  \emph{Case~1.1.1.} If $(\indj,\indi)\neq (\thedim,1)$, then in addition to the second also the fourth line of \eqref{eq:homkernel_third_and_higher_orders_1} is zero, which means that what needs to be shown is that
      \begin{IEEEeqnarray*}{rCl}
0&=&\textstyle
\themat{\anymatA}{\indj}{\Sexchanged(\indi)}+(-1)^\orderind\Saction\themat{\anymatBT}{\Sexchanged(\indj)}{\indi}+\Skronecker{\indj}{1}\Skronecker{\indi}{\thedim}(-1)^\orderind\Saction(\themat{\anymatB}{1}{1}+\Skronecker{\orderind}{3}\sum_{\inds=2}^{\thedim-1}\themat{\anymatB}{\inds}{\inds}).    
  \end{IEEEeqnarray*}
  Multiplying this identity with $(-1)^\orderind$ shows  it  equivalent to \eqref{eq:homkernel_third_and_higher_orders_02} for $\Sexchanged(\indj)$ in the role $\indj$. Since  $(\Sexchanged(\indj),\indi)\neq (1,1)$ by assumption that proves \eqref{eq:homkernel_third_and_higher_orders_1}  in this case.
  \par
  \emph{Case~1.1.2.} Alternatively, for $(\indj,\indi)=(\thedim,1)$ condition \eqref{eq:homkernel_third_and_higher_orders_1} collapses to the demand that
      \begin{IEEEeqnarray*}{rCl}
0&=&\textstyle
    \themat{\anymatA}{\thedim}{\thedim}+(-1)^\orderind\Saction\themat{\anymatBT}{1}{1}+(\themat{\anymatAT}{1}{1}+\Skronecker{\orderind}{3}\Saction\sum_{\inds=2}^{\thedim-1}\themat{\anymatAT}{\inds}{\inds}).
  \end{IEEEeqnarray*}
  But this is a mere rearrangement of \eqref{eq:homkernel_third_and_higher_orders_01} for $(\anymatAT,\thedim,\thedim)$ in the role of $(\anymat,\indj,\indi)$ there. Hence, \eqref{eq:homkernel_third_and_higher_orders_1} is true in all cases. That proves one half of the asserted equivalence.
  \par
  \emph{Step~1.2.} Conversely, suppose \eqref{eq:homkernel_third_and_higher_orders_1} is true for any $\anymat\in\thematrices$ and $(\indj,\indi)\in\SYnumbers{\thedim}^{\Ssetmonoidalproduct 2}$.  Then, for any $\anymat\in\thematrices$ with $(\anymatBT,1,1)$ in place of $(\anymatA,\indj,\indi)$ equation \eqref{eq:homkernel_third_and_higher_orders_1} gives
  \begin{IEEEeqnarray*}{rCl}
    0=\themat{\anymatBT}{1}{\thedim}+(-1)^\orderind\Saction\themat{\anymat}{\thedim}{1}\IEEEyesnumber\label{eq:homkernel_third_and_higher_orders_3},
  \end{IEEEeqnarray*}
  while instantiating \eqref{eq:homkernel_third_and_higher_orders_1} at $(\anymat,\indj,\indi)=(\anymat,\thedim,\thedim)$ yields
  \begin{IEEEeqnarray*}{rCl}
    0=\themat{\anymatA}{\thedim}{1}+(-1)^\orderind\Saction\themat{\anymatBT}{1}{\thedim}-(\themat{\anymatAT}{1}{\thedim}+(-1)^\orderind\themat{\anymatB}{\thedim}{1}).\IEEEyesnumber\label{eq:homkernel_third_and_higher_orders_4} \end{IEEEeqnarray*}
  For any $\anymat\in\thematrices$ multiplying \eqref{eq:homkernel_third_and_higher_orders_3} with $(-1)^\orderind$ and subtracting \eqref{eq:homkernel_third_and_higher_orders_4} from the resulting identity now proves $    0=\themat{\anymatAT}{1}{\thedim}+(-1)^\orderind\themat{\anymatB}{\thedim}{1}$.
  Hence, the second line of \eqref{eq:homkernel_third_and_higher_orders_1} is actually zero no matter what $(\anymat,\indj,\indi)$ is.\par
  In consequence, for any $\anymat\in\thematrices$ and any $(\indj,\indi)\in\SYnumbers{\thedim}^{\Ssetmonoidalproduct 2}$ with $(\anymat,\Sexchanged(\indj),\indi)$ in place of $(\anymat,\indj,\indi)$  assumption \eqref{eq:homkernel_third_and_higher_orders_1} implies
    \begin{IEEEeqnarray*}{rCl}
0&=&
    \themat{\anymatA}{\Sexchanged(\indj)}{\Sexchanged(\indi)}+(-1)^\orderind\Saction\themat{\anymatBT}{\indj}{\indi}\IEEEyesnumber\label{eq:homkernel_third_and_higher_orders_5}\\
&&\textstyle{}+\Skronecker{\indj}{\thedim}\Skronecker{\indi}{\thedim}(-1)^\orderind\Saction(\themat{\anymatB}{1}{1}+\Skronecker{\orderind}{3}\sum_{\inds=2}^{\thedim-1}\themat{\anymatB}{\inds}{\inds})\\
    &&\textstyle{}+\Skronecker{\indj}{1}\Skronecker{\indi}{1}\Saction(\themat{\anymatAT}{1}{1}+\Skronecker{\orderind}{3}\Saction\sum_{\inds=2}^{\thedim-1}\themat{\anymatAT}{\inds}{\inds}).
  \end{IEEEeqnarray*}
  Multiplying \eqref{eq:homkernel_third_and_higher_orders_5} with $(-1)^\orderind$ and solving for $\themat{\anymatBT}{\indj}{\indi}$ produces for any $\anymat\in\thematrices$ and any $(\indj,\indi)\in\SYnumbers{\thedim}^{\Ssetmonoidalproduct 2}$ the identity
    \begin{IEEEeqnarray*}{rCl}
\themat{\anymatBT}{\indj}{\indi}&=&-(-1)^\orderind\Saction\themat{\anymatA}{\Sexchanged(\indj)}{\Sexchanged(\indi)}\IEEEyesnumber\label{eq:homkernel_third_and_higher_orders_6}\\
&&\textstyle{}-\Skronecker{\indj}{\thedim}\Skronecker{\indi}{\thedim}(\themat{\anymatB}{1}{1}+\Skronecker{\orderind}{3}\sum_{\inds=2}^{\thedim-1}\themat{\anymatB}{\inds}{\inds})\\
    &&\textstyle{}-\Skronecker{\indj}{1}\Skronecker{\indi}{1}(-1)^\orderind\Saction\Saction(\themat{\anymatAT}{1}{1}+\Skronecker{\orderind}{3}\Saction\sum_{\inds=2}^{\thedim-1}\themat{\anymatAT}{\inds}{\inds}).
  \end{IEEEeqnarray*}
  Evidently, for any  $\anymat\in\thematrices$  any $(\indj,\indi)\in\SYnumbers{\thedim}^{\Ssetmonoidalproduct 2}$ with $(\indj,\indi)\neq (1,1)$ there is no difference between \eqref{eq:homkernel_third_and_higher_orders_6} and \eqref{eq:homkernel_third_and_higher_orders_02}, which proves the latter.
  \par
  But replacing $\anymat$ by $\anymatB$ transforms \eqref{eq:homkernel_third_and_higher_orders_02} into \eqref{eq:homkernel_third_and_higher_orders_01}, which has thus been verified as well for any  $\anymat\in\thematrices$  any $(\indj,\indi)\in\SYnumbers{\thedim}^{\Ssetmonoidalproduct 2}$ with $(\indj,\indi)\neq (1,1)$. It only remains to prove \eqref{eq:homkernel_third_and_higher_orders_03} and \eqref{eq:homkernel_third_and_higher_orders_04}.
  \par
  Applying \eqref{eq:homkernel_third_and_higher_orders_01} to $(\inds,\inds)$ in the role of $(\indj,\indi)$ for any $\anymat\in\thematrices$ and $\inds\in\SYnumbers{\thedim}$ with $\inds\neq 1$ shows $    \themat{\anymatAT}{\inds}{\inds}=-(-1)^\orderind\themat{\anymatB}{\Sexchanged(\inds)}{\Sexchanged(\inds)}$ and thus in particular for any $\anymat\in\thematrices$,
  \begin{IEEEeqnarray*}{rCl}
    \textstyle\Skronecker{\orderind}{3}\sum_{\inds=2}^{\thedim-1}\themat{\anymatAT}{\inds}{\inds}&=&\textstyle-(-1)^\orderind\Skronecker{\orderind}{3}\sum_{\inds=2}^{\thedim-1}\themat{\anymatB}{\inds}{\inds}.\IEEEyesnumber\label{eq:homkernel_third_and_higher_orders_7}
  \end{IEEEeqnarray*}
  Inserting \eqref{eq:homkernel_third_and_higher_orders_7} into \eqref{eq:homkernel_third_and_higher_orders_6} evaluated at $(\indj,\indi)=(1,1)$ now proves \eqref{eq:homkernel_third_and_higher_orders_03} for any $\anymat\in\thematrices$.
  \par
  Multiplying \eqref{eq:homkernel_third_and_higher_orders_03} with $(-1)^\orderind$ and rearranging yields 
  \begin{IEEEeqnarray*}{rCl}
    \themat{\anymatAT}{1}{1}&=&\textstyle-\themat{\anymat}{\thedim}{\thedim}-(-1)^\orderind\Saction\themat{\anymatBT}{1}{1}+(-1)^\orderind\Skronecker{\orderind}{3}\sum_{\inds=2}^{\thedim-1}\themat{\anymatB}{\inds}{\inds}\IEEEyesnumber\label{eq:homkernel_third_and_higher_orders_8}
  \end{IEEEeqnarray*}
  for any $\anymat\in\thematrices$. Replacing $\anymat$ by $\anymatB$ in \eqref{eq:homkernel_third_and_higher_orders_8} thus shows for any $\anymat\in\thematrices$
  \begin{IEEEeqnarray*}{rCl}
    \themat{\anymatBT}{1}{1}&=&\textstyle-\themat{\anymatB}{\thedim}{\thedim}-(-1)^\orderind\Saction\themat{\anymatAT}{1}{1}+(-1)^\orderind\Skronecker{\orderind}{3}\sum_{\inds=2}^{\thedim-1}\themat{\anymatA}{\inds}{\inds}.\IEEEyesnumber\label{eq:homkernel_third_and_higher_orders_9}
  \end{IEEEeqnarray*}
  For any $\anymat\in\thematrices$, equating the right-hand sides of \eqref{eq:homkernel_third_and_higher_orders_03} and \eqref{eq:homkernel_third_and_higher_orders_9} and then adding $(-1)^\orderind\themat{\anymatAT}{1}{1}$ on both sides of the resulting equation shows
  \begin{IEEEeqnarray*}{rCl}    \textstyle-(-1)^\orderind\Saction\themat{\anymatA}{\thedim}{\thedim}+\Skronecker{\orderind}{3}\sum_{\inds=2}^{\thedim-1}\themat{\anymatB}{\inds}{\inds}&=&\textstyle-\themat{\anymatB}{\thedim}{\thedim}+(-1)^\orderind\Skronecker{\orderind}{3}\sum_{\inds=2}^{\thedim-1}\themat{\anymatA}{\inds}{\inds},\IEEEyesnumber\label{eq:homkernel_third_and_higher_orders_10}
  \end{IEEEeqnarray*}
  which is equivalent to \eqref{eq:homkernel_third_and_higher_orders_04}. Hence, the asserted equivalence  of the conditions in the claim with  $\thehom\Scomposition\thedifferential{\orderind}=0$ is true. 
  \par
  \emph{Step~2: Addendum.} The proof of the addendum is split into two steps.
  \par
  \emph{Step~2.1.} For the following reasons, given any $\othermat\in\thematrices$, if the conditions in the claim are satisfied by $\anymat=\othermat$, then they are also true for $\anymat=\othermatB$.
  In the cases of \eqref{eq:homkernel_third_and_higher_orders_01} and \eqref{eq:homkernel_third_and_higher_orders_02} this intermediate claim is clearly true because replacing $\anymatA$ by $\anymatB$  transforms \eqref{eq:homkernel_third_and_higher_orders_01} into \eqref{eq:homkernel_third_and_higher_orders_02} and vice versa. Since \eqref{eq:homkernel_third_and_higher_orders_04} is equivalent to
    \begin{IEEEeqnarray*}{rCl}
    \textstyle\themat{\anymatB}{\thedim}{\thedim}+\Skronecker{\orderind}{3}\sum_{\inds=2}^{\thedim-1}\themat{\anymatB}{\inds}{\inds}&=&\textstyle(-1)^{\orderind}(\themat{\anymat}{\thedim}{\thedim}+\Skronecker{\orderind}{3}\sum_{\inds=2}^{\thedim-1}\themat{\anymatA}{\inds}{\inds})\IEEEyesnumber\label{eq:homkernel_third_and_higher_orders_04alt1}
  \end{IEEEeqnarray*}
and since  multiplication with $(-1)^\orderind$ has the same effect on \eqref{eq:homkernel_third_and_higher_orders_04alt1} as replacing $\anymatA$ by $\anymatB$ identity \eqref{eq:homkernel_third_and_higher_orders_04} indeed holds for any $\anymatA=\othermatB$ as soon as it does for $\anymatA=\othermatA$.
\par
As noted before, \eqref{eq:homkernel_third_and_higher_orders_04} being true for $\anymat=\othermat$ implies that also  \eqref{eq:homkernel_third_and_higher_orders_10} holds for $\anymat=\othermat$. Inserting \eqref{eq:homkernel_third_and_higher_orders_10} into the assumption \eqref{eq:homkernel_third_and_higher_orders_03} for $\anymat=\othermat$ produces the identity (read for $\anymat=\othermat$)
\begin{IEEEeqnarray*}{rCl}
  \themat{\anymatBT}{1}{1}&=&\textstyle-(-1)^\orderind\Saction\themat{\anymatAT}{1}{1}-\themat{\anymatB}{\thedim}{\thedim}+(-1)^\orderind\Skronecker{\orderind}{3}\sum_{\inds=2}^{\thedim-1}\themat{\anymat}{\inds}{\inds},
\end{IEEEeqnarray*}
which on multiplication with $(-1)^\orderind$ becomes a mere rearrangement of
  \begin{IEEEeqnarray*}{rCl}
    \themat{\anymatAT}{1}{1}&=&\textstyle-(-1)^\orderind\Saction\themat{\anymatBT}{1}{1}-(-1)^\orderind\Saction \themat{\anymatB}{\thedim}{\thedim}+\Skronecker{\orderind}{3}\sum_{\inds=2}^{\thedim-1}\themat{\anymatA}{\inds}{\inds},
  \end{IEEEeqnarray*}
  i.e.,  \eqref{eq:homkernel_third_and_higher_orders_03} for $\anymatA=\othermatB$. Hence, \eqref{eq:homkernel_third_and_higher_orders_01}--\eqref{eq:homkernel_third_and_higher_orders_04} for any $(\indj,\indi)\in\SYnumbers{\thedim}^{\Ssetmonoidalproduct 2}$ with $(\indj,\indi)\neq(1,1)$ automatically also hold for  $\anymat=\othermatB$ if they do for $\anymat=\othermatA$.
  \par
  \emph{Step~2.2.} What is left to prove after Step~2.1 is that, if the conditions in the claim are satisfied by some $\anymat=\othermat\in\thematrices$, then so they are by $\anymat=\othermatAT$.
  \par
  It has already been noted that \eqref{eq:homkernel_third_and_higher_orders_01} being true  for $\anymat=\othermatA$ and any $(\indj,\indi)\in \SYnumbers{\thedim}^{\Ssetmonoidalproduct 2}$ with $(\indj,\indi)\neq (1,1)$ implies \eqref{eq:homkernel_third_and_higher_orders_7} for this $\anymat=\othermatA$. Analogously, summing up   \eqref{eq:homkernel_third_and_higher_orders_02} for $(\inds,\inds)$ in place of $(\indj,\indi)$ over all $\inds\in\SYnumbers{\thedim}$ with $\inds\notin\{1,\thedim\}$ shows that \eqref{eq:homkernel_third_and_higher_orders_7} is also valid for $\anymatA=\othermatB$. In fact, since multiplying with $-(-1)^\orderind$ has the same effect on \eqref{eq:homkernel_third_and_higher_orders_7} as replacing $\anymatA$ by $\anymatBT$ it has thus been shown that \eqref{eq:homkernel_third_and_higher_orders_7} is true for any $\anymat\in\thematrices$.
  \par
  Instantiating \eqref{eq:homkernel_third_and_higher_orders_01} respectively \eqref{eq:homkernel_third_and_higher_orders_02} at $(\anymat,\indj,\indi)=(\othermat,\thedim,\thedim)$ yields that
  \begin{IEEEeqnarray*}{rCl}
    \themat{\anymatAT}{\thedim}{\thedim}&=&\textstyle-(-1)^\orderind\Saction\themat{\anymatB}{1}{1}-\themat{\anymatA}{1}{1}-\Skronecker{\orderind}{3}\sum_{\inds=2}^{\thedim-1}\themat{\anymatA}{\inds}{\inds}
  \end{IEEEeqnarray*}
  is true for $\anymatA=\othermatA$ and $\anymatA=\othermatB$. Inserting into this  \eqref{eq:homkernel_third_and_higher_orders_7} for $\anymatA=\othermatAT$ respectively  $\anymatA=\othermatBT$ proves
    \begin{IEEEeqnarray*}{rCl}
    \themat{\anymatAT}{\thedim}{\thedim}&=&\textstyle-(-1)^\orderind\Saction\themat{\anymatB}{1}{1}-\themat{\anymatA}{1}{1}+(-1)^\orderind\Skronecker{\orderind}{3}\sum_{\inds=2}^{\thedim-1}\themat{\anymatBT}{\inds}{\inds}
  \end{IEEEeqnarray*}
  for $\anymatA=\othermatA$ and $\anymatA=\othermatB$. By multiplying with $(-1)^\orderind$ and solving for $\themat{\anymatB}{1}{1}$ it follows
  \begin{IEEEeqnarray*}{rCl}
    \themat{\anymatB}{1}{1}&=&\textstyle-(-1)^\orderind\Saction\themat{\anymat}{1}{1}-(-1)^\orderind\Saction\themat{\anymatAT}{1}{1}+\Skronecker{\orderind}{3}\sum_{\inds=2}^{\thedim-1}\themat{\anymatBT}{\inds}{\inds}
  \end{IEEEeqnarray*}
  for $\anymatA=\othermatA$ and $\anymatA=\othermatB$, respectively, i.e., exactly \eqref{eq:homkernel_third_and_higher_orders_03}
  for $\anymatA=\othermatAT$ and $\anymatA=\othermatBT$. Hence, \eqref{eq:homkernel_third_and_higher_orders_03} is true for any $\anymat\in\thematrices$ by Step~2.1.
  \par
  It has already been seen in Step~1.2 that \eqref{eq:homkernel_third_and_higher_orders_03} being true for any $\anymat\in\thematrices$ implies that \eqref{eq:homkernel_third_and_higher_orders_04} holds for all $\anymat\in\thematrices$, in particular for $\anymat=\othermatAT$.
  \par
  Thus, all that still needs proving at this point is that \eqref{eq:homkernel_third_and_higher_orders_01} and \eqref{eq:homkernel_third_and_higher_orders_02} are true for $\anymatA=\othermatAT$ and any $(\indj,\indi)\in\SYnumbers{\thedim}^{\Ssetmonoidalproduct 2}$. For $(\indj,\indi)\neq (\thedim,\thedim)$ this is immediately clear by replacing $(\indj,\indi)$  by $(\Sexchanged(\indj),\Sexchanged(\indi))$ in \eqref{eq:homkernel_third_and_higher_orders_02} respectively \eqref{eq:homkernel_third_and_higher_orders_01} and then multiplying with $-(-1)^\orderind$.
  \par
  Multiplying \eqref{eq:homkernel_third_and_higher_orders_03} for $\anymatA=\othermatA$ respectively $\anymatA=\othermatB$ with $-(-1)^\orderind$ and solving for $\themat{\anymat}{\thedim}{\thedim}$ shows
  \begin{IEEEeqnarray*}{rCl}
    \themat{\anymat}{\thedim}{\thedim}&=&\textstyle-(-1)^\orderind\Saction\themat{\anymatBT}{1}{1}-(\themat{\anymatAT}{1}{1}-(-1)^\orderind\Skronecker{\orderind}{3}\sum_{\inds=2}^{\thedim-1}\themat{\anymatB}{\inds}{\inds}),
  \end{IEEEeqnarray*}
whence it follows by using \eqref{eq:homkernel_third_and_higher_orders_7} for $\anymatA=\othermatB$ respectively $\anymatA=\othermatA$ that
  \begin{IEEEeqnarray*}{rCl}
    \themat{\anymat}{\thedim}{\thedim}&=&\textstyle-(-1)^\orderind\Saction\themat{\anymatBT}{1}{1}-(\themat{\anymatAT}{1}{1}+\Skronecker{\orderind}{3}\sum_{\inds=2}^{\thedim-1}\themat{\anymatAT}{\inds}{\inds})
  \end{IEEEeqnarray*}
  for $\anymatA=\othermatA$ respectively $\anymatA=\othermatB$. Since this is precisely \eqref{eq:homkernel_third_and_higher_orders_01} respectively \eqref{eq:homkernel_third_and_higher_orders_02} for $(\anymat,\indj,\indi)=(\othermatAT,\thedim,\thedim)$, the proof is complete.
\end{proof}
With Lemma~\ref{lemma:homkernel_third_and_higher_orders} at hand it is time to distinguish the cases $\orderind=3$ and $4\leq \orderind$.
\subsubsection{Third order}
In the case of $\ShomOX{\thealgebra}{\thedifferential{3}}{\themodule}$ there is still a little more work to do in order to determine the defect and rank.
\begin{lemma}
  \label{lemma:homkernel_third_order}
Any $\thehom\in\ShomOX{\thealgebra}{\thechainsmodule{2}}{\themodule}$ satisfies $\thehom\Scomposition\thedifferential{3}=0$ if and only if for any $\anymat\in\thematrices$, all in conjunction, for any $(\indj,\indi)\in\SYnumbers{\thedim}^{\Ssetmonoidalproduct 2}$ with $(\indj,\indi)\notin\{ (1,1),(\thedim,\thedim)\}$,
  \begin{IEEEeqnarray*}{rCl}
    \themat{\anymatAT}{\indj}{\indi}&=&\textstyle \themat{\anymatB}{\Sexchanged(\indj)}{\Sexchanged(\indi)}\IEEEyesnumber\label{eq:homkernel_third_order_01},    \\ \themat{\anymatBT}{\indj}{\indi}&=&\textstyle \themat{\anymatA}{\Sexchanged(\indj)}{\Sexchanged(\indi)}\IEEEyesnumber\label{eq:homkernel_third_order_02}
  \end{IEEEeqnarray*}
  and
  \begin{IEEEeqnarray*}{rCl}
    \themat{\anymatBT}{1}{1}&=&\textstyle\themat{\anymatAT}{1}{1}+ \themat{\anymatA}{\thedim}{\thedim}+\sum_{\inds=2}^{\thedim-1}\themat{\anymatB}{\inds}{\inds}\IEEEyesnumber\label{eq:homkernel_third_order_03}
  \end{IEEEeqnarray*}
  as well as
  \begin{IEEEeqnarray*}{rCl}
    \textstyle\themat{\anymatB}{\thedim}{\thedim}&=&\textstyle-\sum_{\inds=2}^{\thedim-1}\themat{\anymatB}{\inds}{\inds}-\sum_{\inds=2}^{\thedim}\themat{\anymatA}{\inds}{\inds}\IEEEyesnumber\label{eq:homkernel_third_order_04}
  \end{IEEEeqnarray*}
  and
  \begin{IEEEeqnarray*}{rCl}
    \themat{\anymatB}{1}{1}=\textstyle\sum_{\inds=1}^{\thedim}\themat{\anymatA}{\inds}{\inds}.\IEEEyesnumber\label{eq:homkernel_third_order_05}
  \end{IEEEeqnarray*}
In fact, \eqref{eq:homkernel_third_order_01}--\eqref{eq:homkernel_third_order_05} for any $(\indj,\indi)\in \SYnumbers{\thedim}^{\Ssetmonoidalproduct 2}$ with $(\indj,\indi)\notin\{ (1,1),(\thedim,\thedim)\}$ are all satisfied by each $\anymat\in\thematrices$ as soon as they are all satisfied by an arbitrary one $\anymat\in\thematrices$.
\end{lemma}
\begin{proof}
Both the claimed equivalence and the addendum can be proved simultaneously by showing that for any  fixed $\anymat\in\thematrices$ the conditions \eqref{eq:homkernel_third_order_01}--\eqref{eq:homkernel_third_order_05} are satisfied for any $(\indj,\indi)\in \SYnumbers{\thedim}^{\Ssetmonoidalproduct 2}$ with $(\indj,\indi)\notin\{ (1,1),(\thedim,\thedim)\}$ if and only if the conditions \eqref{eq:homkernel_third_and_higher_orders_01}--\eqref{eq:homkernel_third_and_higher_orders_04} from Lemma~\ref{lemma:homkernel_third_and_higher_orders} are satisfied for $\orderind=3$ and any $(\indj,\indi)\in \SYnumbers{\thedim}^{\Ssetmonoidalproduct 2}$ with $(\indj,\indi)\neq (1,1)$.
\par
Obviously, for $\orderind=3$ and $(\indj,\indi)\notin\{(1,1),(\thedim,\thedim)\}$ the conditions \eqref{eq:homkernel_third_and_higher_orders_01} and \eqref{eq:homkernel_third_order_01} are identical, as are \eqref{eq:homkernel_third_and_higher_orders_02} and \eqref{eq:homkernel_third_order_02}. Likewise,  $\orderind=3$ provided, \eqref{eq:homkernel_third_and_higher_orders_03} and \eqref{eq:homkernel_third_order_03} agree, and so do \eqref{eq:homkernel_third_and_higher_orders_04} and \eqref{eq:homkernel_third_order_04}. Thus, all that remains to be shown is that, assuming all of these, the equations \eqref{eq:homkernel_third_and_higher_orders_01} for $(\indj,\indi)=(\thedim,\thedim)$, i.e.,
  \begin{IEEEeqnarray*}{rCl}
    \themat{\anymatAT}{\thedim}{\thedim}&=&\textstyle \themat{\anymatB}{1}{1}-\sum_{\inds=1}^{\thedim-1}\themat{\anymatA}{\inds}{\inds},\IEEEyesnumber\label{eq:homkernel_third_orders_1}
  \end{IEEEeqnarray*}
and \eqref{eq:homkernel_third_and_higher_orders_02} for $(\indj,\indi)=(\thedim,\thedim)$, i.e.,
\begin{IEEEeqnarray*}{rCl}
  \themat{\anymatBT}{\thedim}{\thedim}&=&\textstyle \themat{\anymatA}{1}{1}-\sum_{\inds=1}^{\thedim-1}\themat{\anymatB}{\inds}{\inds},\IEEEyesnumber\label{eq:homkernel_third_orders_2}
\end{IEEEeqnarray*}
are together equivalent to \eqref{eq:homkernel_third_order_05}.
\par
Because $\Sbvfull{\anymatAT}{2}{\thedim}{\thedim}=\Sbvfull{\anymatA}{2}{\thedim}{\thedim}$ and $\Sbvfull{\anymatBT}{2}{\thedim}{\thedim}=\Sbvfull{\anymatB}{2}{\thedim}{\thedim}$ by Lemma~\ref{lemma:evil_identities}  the assumption $\orderind=3$ requires $\themat{\anymatAT}{\thedim}{\thedim}=\themat{\anymatA}{\thedim}{\thedim}$ and $\themat{\anymatBT}{\thedim}{\thedim}=\themat{\anymatB}{\thedim}{\thedim}$. Because the first of the latter two identities reveals that \eqref{eq:homkernel_third_orders_1} and \eqref{eq:homkernel_third_order_05} are equivalent, it is  clear that \eqref{eq:homkernel_third_orders_1} and \eqref{eq:homkernel_third_orders_2} imply \eqref{eq:homkernel_third_order_05} and that it suffices to prove that \eqref{eq:homkernel_third_order_05}  requires  \eqref{eq:homkernel_third_orders_2}.
\par
And, indeed, by using \eqref{eq:homkernel_third_order_04} in the second and \eqref{eq:homkernel_third_order_05} in the fourth step it can be concluded that
\begin{IEEEeqnarray*}{rCl}
  \themat{\anymatBT}{\thedim}{\thedim}&=&  \themat{\anymatB}{\thedim}{\thedim}\\
  &=&\textstyle-\sum_{\inds=2}^{\thedim-1}\themat{\anymatB}{\inds}{\inds}-\sum_{\inds=2}^{\thedim}\themat{\anymatA}{\inds}{\inds}   \\
&=&\textstyle\themat{\anymatA}{1}{1}-\sum_{\inds=2}^{\thedim-1}\themat{\anymatB}{\inds}{\inds}-\sum_{\inds=1}^{\thedim}\themat{\anymatA}{\inds}{\inds}   \\  
  &=&\textstyle\themat{\anymatA}{1}{1}-\sum_{\inds=1}^{\thedim-1}\themat{\anymatB}{\inds}{\inds},
\end{IEEEeqnarray*}
which thus concludes the proof.
\end{proof}
\begin{lemma}
  \label{lemma:cohomology_third_order}
Both defect and rank of $\ShomOX{\thealgebra}{\thedifferential{3}}{\themodule}$ are $2\thedim^2-1$.
\end{lemma}
\begin{proof}
  The domain  $\ShomOX{\thealgebra}{\thechainsmodule{2}}{\themodule}=\ShomOX{\thealgebra}{\thefield\thechains{2}\SmonoidalproductC{\thefield}\thealgebra}{\themodule}\cong \thefield\thechains{2}$ of $\ShomOX{\thealgebra}{\thedifferential{3}}{\themodule}$ has dimension $|\thechains{2}|=4\thedim^2-2$ by Lemma~\ref{lemma:evil_identities}. For a fixed $\anymatA\in\thematrices$, say $\anymatA=\theuniA$,  the (addendum  to) Lemma~\ref{lemma:homkernel_third_order} characterizes membership of an arbitrary $\thehom\in\ShomOX{\thealgebra}{\thechainsmodule{2}}{\themodule}$ in the kernel of $\ShomOX{\thealgebra}{\thedifferential{3}}{\themodule}$ by the equations  \eqref{eq:homkernel_third_order_01}--\eqref{eq:homkernel_third_order_05} for $(\indj,\indi)\in \SYnumbers{\thedim}^{\Ssetmonoidalproduct 2}$ with $(\indj,\indi)\notin\{ (1,1),(\thedim,\thedim)\}$.
  \par 
  These conditions are linearly independent for the following reasons. The variables $\themat{\anymatAT}{\indj}{\indi}$ and $\themat{\anymatBT}{\indj}{\indi}$ for $(\indj,\indi)\in \SYnumbers{\thedim}^{\Ssetmonoidalproduct 2}$ with $(\indj,\indi)\notin\{(1,1),(\thedim,\thedim)\}$ appear only in the  equations \eqref{eq:homkernel_third_order_01} and \eqref{eq:homkernel_third_order_02}, which are  obviously linearly independent. Hence, the question  is reduced to whether \eqref{eq:homkernel_third_order_03}--\eqref{eq:homkernel_third_order_05} are linearly independent. But \eqref{eq:homkernel_third_order_03} features the variables $\themat{\anymatAT}{1}{1}$ and $\themat{\anymatBT}{1}{1}$ which  appear neither in \eqref{eq:homkernel_third_order_04} nor in \eqref{eq:homkernel_third_order_05}. And, likewise, \eqref{eq:homkernel_third_order_04} and \eqref{eq:homkernel_third_order_05} each have a variable, namely $\themat{\anymatB}{\thedim}{\thedim}$ respectively $\themat{\anymatB}{1}{1}$, which occurs in no other equation. 
\par
  Since there are  $(\thedim^2-2)+(\thedim^2-2)+1+1+1=2\thedim^2-1$ independent equations for $4\thedim^2-2$ variables, the claim follows by the dimension formula.
\end{proof}
\subsubsection{Fourth and higher orders}
For orders $\orderind\in\Sintegersnn$ with $4\leq \orderind$ the defect and rank of $\ShomOX{\thealgebra}{\thedifferential{\orderind}}{\themodule}$ is now easy to compute.
\begin{lemma}
  \label{lemma:cohomology_higher_orders}
Both defect and rank of $\ShomOX{\thealgebra}{\thedifferential{\orderind}}{\themodule}$ are $2\thedim^2$ for each $\orderind\in\Sintegersnn$ with $4\leq \orderind$.
\end{lemma}
\begin{proof}
  The domain  $\ShomOX{\thealgebra}{\thechainsmodule{\orderind-1}}{\themodule}=\ShomOX{\thealgebra}{\thefield\thechains{\orderind-1}\SmonoidalproductC{\thefield}\thealgebra}{\themodule}\cong \thefield\thechains{\orderind-1}$ of $\ShomOX{\thealgebra}{\thedifferential{\orderind}}{\themodule}$ has dimension $|\thechains{\orderind-1}|=4\thedim$ by Lemma~\ref{lemma:evil_identities}. For a fixed $\anymatA\in\thematrices$, say $\anymatA=\theuniA$,  the (addendum  to) Lemma~\ref{lemma:homkernel_third_and_higher_orders} characterizes membership of an arbitrary $\thehom\in\ShomOX{\thealgebra}{\thechainsmodule{\orderind-1}}{\themodule}$ in the kernel of $\ShomOX{\thealgebra}{\thedifferential{\orderind}}{\themodule}$ by the equations  \eqref{eq:homkernel_third_and_higher_orders_01}--\eqref{eq:homkernel_third_and_higher_orders_04} for $(\indj,\indi)\in \SYnumbers{\thedim}^{\Ssetmonoidalproduct 2}$ with $(\indj,\indi)\neq(1,1)$.
  \par 
  These conditions are linearly independent. Indeed, for any  $(\indj,\indi)\in \SYnumbers{\thedim}^{\Ssetmonoidalproduct 2}$ with $(\indj,\indi)\neq(1,1)$ the variable   $\themat{\anymatAT}{\indj}{\indi}$ appears only in the corresponding \eqref{eq:homkernel_third_and_higher_orders_01} and nowhere else in the system of equations. Likewise, for such $(\indj,\indi)$ the variable $\themat{\anymatAT}{\indj}{\indi}$ is found only in the respective \eqref{eq:homkernel_third_and_higher_orders_02} and nowhere else. Furthermore, in the form of $\themat{\anymatBT}{1}{1}$ and $\themat{\anymatB}{\thedim}{\thedim}$, each of \eqref{eq:homkernel_third_and_higher_orders_03} and \eqref{eq:homkernel_third_and_higher_orders_04} respectively exhibits a  variable which appears in no other equation.
\par
  Since there are  $(\thedim^2-1)+(\thedim^2-1)+1+1=2\thedim^2$ independent equations for $4\thedim^2$ variables, the claim follows by the dimension formula.
\end{proof}
\subsection{Synthesis}
\label{section:cohomology_synthesis}
Of course, the result of the computation was already known in advance from \cite{DasFranzKulaSkalski2018} and \cite{BaraquinFranzGerholdKulaTobolski2023}.
\begin{proposition}
  For any $\orderind\in \Sintegersp$,
  \begin{IEEEeqnarray*}{rCl}
    \dim_\thefield(\SextX{\thealgebra}{\themodule}{\themodule}{\orderind})=
    \begin{cases}
      \thedim^2&\Scase \orderind=1\\
      \thedim^2-1&\Scase \orderind=2\\
      1 &\Scase\orderind=3\\
      0&\Scase4\leq \orderind.
    \end{cases}
  \end{IEEEeqnarray*}
\end{proposition}
\begin{proof}
  \newcommand{\dummydefect}[1]{d_{#1}}
  \newcommand{\dummyrank}[1]{r_{#1}}
  Because   $\SextX{\thealgebra}{\themodule}{\themodule}{\orderind}$ is the quotient vector space of the kernel of $\ShomOX{\thealgebra}{\thedifferential{\orderind+1}}{\themodule}$ by the image of $\ShomOX{\thealgebra}{\thedifferential{\orderind}}{\themodule}$,
  the dimension of $\SextX{\thealgebra}{\themodule}{\themodule}{\orderind}$ is given by the difference $\dummydefect{\orderind+1}-\dummyrank{\orderind}$ of the defect $\dummydefect{\orderind+1}$ of $\ShomOX{\thealgebra}{\thedifferential{\orderind+1}}{\themodule}$ and the rank $\dummyrank{\orderind}$ of  $\ShomOX{\thealgebra}{\thedifferential{\orderind}}{\themodule}$. Because  $\dummyrank{1}=0$ and $\dummydefect{2}=\dummyrank{2}=\thedim^2$ and $\dummydefect{3}=\dummyrank{3}=2\thedim^2-1$ and $\dummydefect{\orderind}=\dummyrank{\orderind}=2\thedim^2$ if $4\leq \orderind$ by Lemmata~\ref{lemma:cohomology_first_order}, \ref{lemma:cohomology_second_order}, \ref{lemma:cohomology_third_order} and \ref{lemma:cohomology_higher_orders}  that proves the claim.
\end{proof}
}

\section{Discussion}
\label{section:discussion}
The present article is a case study in computing homological invariants of easy quantum groups.
\begin{itemize}
\item Reminder on easy quantum groups, invariants -- Section~\ref{section:on_free_unitary_quantum_group_and_easy_quantum_groups}
\item Literature context and motivation  -- Section~\ref{section:motivation}
  \item Conclusions and outlook  -- Section~\ref{section:conclusions}
  \item Various other comments and remarks -- Section~\ref{section:additional_remarks}
\end{itemize}
\subsection{Quantum groups background}
\label{section:on_free_unitary_quantum_group_and_easy_quantum_groups}
One main approach to quantizing the notion of groups was developed by Woronowicz in \cite{Woronowicz1987b, Woronowicz1991, Woronowicz1998}, culminating in the definition of \emph{compact quantum groups} (see Sec\-tion~\ref{section:compact_quantum_groups}). While there are no free compact quantum groups, the \emph{universal compact matrix quantum groups} $U_n^+(Q)$ discovered by Van Daele and Wang in \cite{VanDaeleWang1996a,Wang1995b} come close in spirit (see Sec\-tion~\ref{section:the_free_unitary_quantum_group}). A special case of those is the free unitary quantum group $U_n^+$, the  subject of the present article. It arises from the unitary group $U_n$ by a process termed \emph{liberation} by Banica and Speicher in \cite{BanicaSpeicher2009}, which makes $U_n^ +$ a so-called \emph{easy quantum groups} (see Sec\-tion~\ref{section:easy_quantum_groups}).  
The findings of this article have promising implications for the feasibility of computing the quantum group co\-ho\-mo\-lo\-gy (see Sec\-tion~\ref{section:quantum_group_cohomology}) of the discrete duals of easy quantum groups (see Sections~\ref{section:motivation} and \ref{section:conclusions}).

\subsubsection{Compact quantum groups}
\label{section:compact_quantum_groups}
{
  \newcommand{\anystate}{h}
  \newcommand{\anyelement}{a}
  \newcommand{\anygroup}{G}  
Compact quantum groups as defined by Woronowicz in  \cite{Woronowicz1987b, Woronowicz1991, Woronowicz1998} are (formally dual to) -- not necessarily commutative --  $C^\ast$-algebras with the same kind of extra structure as the group law of a compact group $\anygroup$ would induce on the commutative $C^\ast$-algebra $C(\anygroup)$ of continuous  complex-valued functions over it. The exact definition is not important for the purposes of the present article. It is enough to know that there exists a full and essentially surjective-on-objects contravariant functor  $\mathcal{O}$ from the category of compact quantum groups to the category of  \emph{CQG algebras} introduced by Dijkhuizen and Koornwinder in \cite{DijkhuizenKoornwinder1994}.
\par
A CQG algebra is a  Hopf $\ast$-algebra which admits a positive faithful integral. And, a \emph{Hopf $\ast$-algebra} is a complex Hopf algebra equipped with a conjugate-linear anti-multiplicative co-multiplicative involution, the \emph{star} $\ast$. An integral $\anystate$ on a Hopf $\ast$-algebra $\thealgebra$ is positive and faithful if and only if $h(\anyelement\SstarP\anyelement)>0$ for any $\anyelement\in\thealgebra$ with $\anyelement\neq 0$. In particular, by forgetting (the star and) the co-multiplication but not its co-unit $\thecounit$ any CQG algebra can be interpreted as an augmented algebra and comes with a natural (right) module structure $\themodulecomps$ on $\Scomps$.
That is the only  part of the compact quantum group structure that has appeared explicitly so far in the present article. 
}
\subsubsection{Universal compact matrix quantum groups}
\label{section:the_free_unitary_quantum_group}
{
  \newcommand{\indi}{i}
  \newcommand{\indj}{j}
  \newcommand{\inds}{s}
  \newcommand{\indp}{p}
  \newcommand{\indt}{t}
  Particular objects of the category of compact quantum groups are Van Daele and Wang's  universal compact $\thedim\times\thedim$-matrix quantum groups $U_\thedim^+(Q)$ defined in \cite{VanDaeleWang1996a,Wang1995b}. The latter can be understood as limits in appropriate subcategories, which motivates the name. The indexing parameters $\thedim$ and $Q$ can be any positive integer respectively any invertible $\thedim\times\thedim$\-/matrix with complex entries. The augmented algebra underlying the CQG algebra $\mathcal{O}(U_\thedim^+(Q))$ is the universal complex algebra with $2\thedim^2$ generators $\thegens=\{\thewuniX{\indj}{\indi},\thebuniX{\indj}{\indi}\}_{\indi,\indj=1}^\thedim$ subject to the relations consisting of for any $(\indj,\indi)\in \SYnumbers{\thedim}^{\Ssetmonoidalproduct 2}$ the polynomials
  \begin{IEEEeqnarray*}{ll}
    \begin{IEEEeqnarraybox}[][c]{s.s}
$\textstyle\sum_{\inds=1}^\thedim \thewuniX{\indj}{\inds}\thebuniX{\indi}{\inds}-\Skronecker{\indj}{\indi}\Saction\theone$, \hspace{12pt}&$\textstyle\sum_{\indp,\inds,\indt=1}^\thedim Q_{\indt,\inds}Q\SinverseP_{\indp,\indi}\Saction \thewuniX{\indt}{\indj}\thebuniX{\inds}{\indp}-\Skronecker{\indj}{\indi}\Saction \theone$,
  \\
    $\textstyle\sum_{\inds=1}^\thedim \thebuniX{\inds}{\indj}\thewuniX{\inds}{\indi}-\Skronecker{\indj}{\indi}\Saction\theone$, \hspace{12pt}&$\textstyle  \sum_{\indp,\inds,\indt=1}^\thedim Q_{\indj,\indt}Q\SinverseP_{\inds,\indp}\Saction\thebuniX{\indt}{\inds}\thewuniX{\indi}{\indp}-\Skronecker{\indj}{\indi}\Saction \theone$      
    \end{IEEEeqnarraybox}\IEEEyesnumber\label{eq:universal_matrix_quantum_group_relations}
\end{IEEEeqnarray*}
together with the morphism of complex algebras to $\Scomps$ with $\thewuniX{\indj}{\indi}\mapsto \Skronecker{\indj}{\indi}$ and $\thebuniX{\indj}{\indi}\mapsto \Skronecker{\indj}{\indi}$ for any $(\indj,\indi)\in \SYnumbers{\thedim}^{\Ssetmonoidalproduct 2}$. In the special case that $Q$  is the identity $\thedim\times\thedim$\-/matrix, $U_\thedim^+(Q)$ is simply denoted by $U_\thedim^+$. Then, for any $(\indj,\indi)\in \SYnumbers{\thedim}^{\Ssetmonoidalproduct 2}$ the second polynomials in the the two rows of   \eqref{eq:universal_matrix_quantum_group_relations} become $\textstyle\sum_{\inds=1}^\thedim  \thewuniX{\inds}{\indj}\thebuniX{\inds}{\indi}-\Skronecker{\indj}{\indi}\Saction \theone$ and $\textstyle  \sum_{\inds=1}^\thedim \thebuniX{\indj}{\inds}\thewuniX{\indi}{\inds}-\Skronecker{\indj}{\indi}\Saction \theone$ respectively. Hence, in that case the relations are precisely $\therels$ and the augmentation precisely $\thecounit$  from the \hyperref[main-result]{Main result}.

  }
\subsubsection{Easy quantum groups}
\label{section:easy_quantum_groups}
{
  \newcommand{\upp}[1]{{}^{\scalebox{.4}{$\square$}}#1}
  \newcommand{\lop}[1]{{}_{\scalebox{.4}{$\square$}}#1}
  \newcommand{\djpin}{{}^{\scalebox{.3}{$\square$}}{}_{\scalebox{.3}{$\square$}}}  
  {
  \newcommand{\indj}{j}
  \newcommand{\indi}{i}
  \newcommand{\inds}{s}
  \newcommand{\indt}{t}
  \newcommand{\incol}{\mathfrak{c}}
  \newcommand{\incolx}[1]{\incol_{#1}}
  \newcommand{\outcol}{\mathfrak{d}}
  \newcommand{\outcolx}[1]{\outcol_{#1}}
  \newcommand{\inind}{a}
  \newcommand{\outind}{b}
  \newcommand{\inlen}{k}
  \newcommand{\outlen}{\ell}
  \newcommand{\anyblock}{\mathsf{B}}  
  \newcommand{\anypartition}{p}
  Adding to the definition of $\mathcal{O}(U_\thedim^+)$ the relations  $\theuniX{\indt}{\inds}\theuniX{\indj}{\indi}-\theuniX{\indj}{\indi}\theuniX{\indt}{\inds}$ for any $\{(\indt,\inds),(\indj,\indi)\}\subseteq \SYnumbers{\thedim}^{\Ssetmonoidalproduct 2}$ produces the CQG algebra of the unitary group $U_\thedim$. Of course it must have been  by the reverse process, removing those relations from the presentation of $\mathcal{O}(U_\thedim)$, that Van Daele and Wang arrived at the definition of $U_\thedim^+$  in the first place. In \cite{BanicaSpeicher2009}, Banica and Speicher gave this method of quantization by straight-up removing rather than deforming relations the name \emph{liberation}. With the modifications by Tarrago and Weber \cite{TarragoWeber2018} Banica and Speicher's construction provides a way of exhibiting new examples of quantum subgroups of $U_\thedim^+$, called \emph{easy}, by finding solutions to special combinatorics problems. 
  \par
   More precisely, if $\{\upp{1},\upp{2},\ldots\}$ and $\{\lop{1},\lop{2},\ldots\}$ are two disjoint copies of $\Sintegersp$, the \emph{upper} respectively \emph{lower points}, then for any $(\inlen,\outlen)\in \Sintegersnn$ a \emph{partition} $\anypartition$ from any string  $\incol\in\Scolors^{\Ssetmonoidalproduct \inlen}$ of \emph{upper colors} to any \emph{lower colors} $\outcol\in\Scolors^{\Ssetmonoidalproduct \outlen}$  is the set of equivalence classes, its \emph{blocks}, of an equivalence relation on  $\{\upp{\inind},\lop{\outind}\Ssetbuilder\inind\in \SYnumbers{\inlen},\outind\in \SYnumbers{\outlen}\}$, its \emph{points}.
  }
  
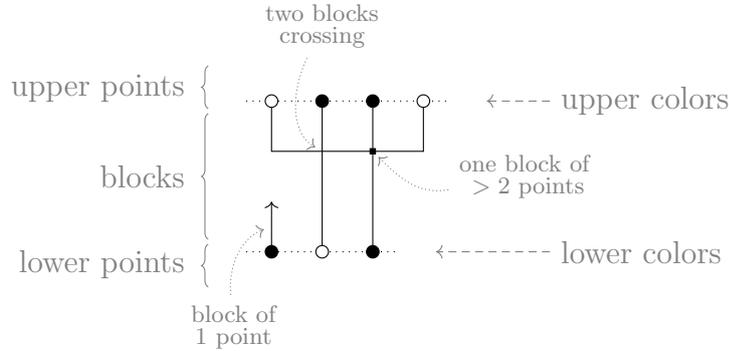
\begin{center}
          \hspace{-3em}
\begin{tikzpicture}[baseline=0.91cm]
    \def\scp{0.666}
    \def\linksize{\scp*0.075cm}
    \def\pointsize{\scp*0.25cm}
    \def\dd{\scp*0.5cm}
    \def\dx{\scp*1cm}
    \def\cx{\scp*0.3cm}
    \def\txu{3*\dx}    
    \def\txl{2*\dx}
    \def\dy{\scp*1cm}
    \def\cy{\scp*0.3cm}
    \def\ty{3*\dy}
    \tikzset{whp/.style={circle, inner sep=0pt, text width={\pointsize}, draw=black, fill=white}}
    \tikzset{blp/.style={circle, inner sep=0pt, text width={\pointsize}, draw=black, fill=black}}
    \tikzset{lk/.style={regular polygon, regular polygon sides=4, inner sep=0pt, text width={\linksize}, draw=black, fill=black}}
    \draw[dotted] ({0-\dd},{0}) -- ({\txl+\dd},{0});
    \draw[dotted] ({0-\dd},{\ty}) -- ({\txu+\dd},{\ty});
    \coordinate (l1) at ({0+0*\dx},{0+0*\ty}) {};    
    \coordinate (l2) at ({0+1*\dx},{0+0*\ty}) {};
    \coordinate (l3) at ({0+2*\dx},{0+0*\ty}) {};
    \coordinate (u1) at ({0+0*\dx},{0+1*\ty}) {};
    \coordinate (u2) at ({0+1*\dx},{0+1*\ty}) {};
    \coordinate (u3) at ({0+2*\dx},{0+1*\ty}) {};
    \coordinate (u4) at ({0+3*\dx},{0+1*\ty}) {};
    \node[lk] at ({2*\dx},{2*\dy}) {};
    \draw (u1) -- ++ ({0*\dx},{-1*\dy}) -| (u4);
    \draw[->] (l1) -- ++({0*\dx},{1*\dy});
    \draw (l2) -- (u2);
    \draw (l3) -- (u3);    
    \node[blp] at (l1) {};
    \node[whp] at (l2) {};
    \node[blp] at (l3) {};
    \node[whp] at (u1) {};
    \node[blp] at (u2) {};
    \node[blp] at (u3) {};
    \node[whp] at (u4) {};
    \node[anchor=west, gray] (upper-colors) at ({5.5*\dx},{3*\dy}) {upper colors};
    \node[anchor=east, gray] (upper-points) at ({-1.5*\dx},{3.25*\dy}) {upper points};
    \node[anchor=west, gray] (lower-colors) at ({5.5*\dx},{0*\dy}) {lower colors};    
    \node[anchor=east, gray] (lower-points) at ({-1.5*\dx},{-0.25*\dy}) {lower points};
    \node[anchor=east, gray] (blocks) at ({-1.5*\dx},{1.5*\dy}) {blocks};
    \node[align=center, anchor=center, gray] (crossing-blocks) at ({1*\dx},{4.5*\dy}) {\scriptsize two blocks\\[-0.5em] \scriptsize crossing}; 
    \node[align=center, anchor=center, gray] (big-block) at ({5*\dx},{1.5*\dy}) {\scriptsize one block of\\[-0.5em] \scriptsize ${}>2$ points};
    \node[align=center, anchor=center, gray] (singleton-block) at ({-0.75*\dx},{-1.5*\dy}) {\scriptsize block of\\[-0.5em] \scriptsize $1$ point};
    \draw[->, densely dotted, out=95, in=200, shorten >= 3pt, gray] (singleton-block) to ({-0*\dx},{0.5*\dy});
    \draw[->, densely dotted, out=190, in=315, shorten >= 3pt, gray] (big-block) to ({2*\dx},{2*\dy});
    \draw[->, densely dotted, out=245, in=150, shorten >= 3pt, gray] (crossing-blocks) to ({1*\dx},{2*\dy});
    \draw[xshift={-0.75*\dx}, decorate, decoration={brace}, gray] ({-0.5*\dx},{0.25*\dy}) to ({-0.5*\dx},{2.75*\dy});
    \draw[xshift={-0.75*\dx}, decorate, decoration={brace}, gray] ({-0.5*\dx},{-0.7*\dy}) to ({-0.5*\dx},{0.15*\dy});
    \draw[xshift={-0.75*\dx}, decorate, decoration={brace}, gray] ({-0.5*\dx},{2.85*\dy}) to ({-0.5*\dx},{3.7*\dy});
    \draw[->, densely dashed, out=180, in=0, shorten >= 0pt, gray] (upper-colors) to ({3*\dx+2.5*\dd},{3*\dy});    
    \draw[->, densely dashed, out=180, in=0, shorten >= 0pt, gray] (lower-colors) to ({2*\dx+2.5*\dd},{0*\dy});        
  \end{tikzpicture}
\end{center}
{
  \newcommand{\incol}{\mathfrak{c}}
  \newcommand{\outcol}{\mathfrak{d}}
  \newcommand{\anypartition}{p}
  \newcommand{\anyblock}{\mathsf{B}}  
            \newcommand{\inind}{a}
            \newcommand{\outind}{b}
            Given any partition $\Sxfromto{\anypartition}{\incol}{\outcol}$, its \emph{adjoint} $\Sxfromto{\anypartition\SstarP}{\outcol}{\incol}$ is obtained by exchanging the roles of upper and lower points, $\anypartition\SstarP=\{ \{\upp{\outind}\Ssetbuilder\lop{\outind}\in \anyblock\}\cup\{\lop{\inind}\Ssetbuilder\upp{\inind}\in\anyblock\}\}_{\anyblock\in\anypartition}$.
            }
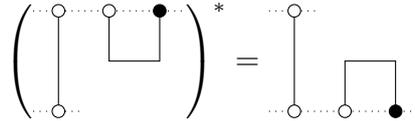
\begin{center}
  $\left(
  \begin{tikzpicture}[vadjust={2}]
    \def\scp{0.666}
    \def\linksize{\scp*0.075cm}
    \def\pointsize{\scp*0.25cm}
    \def\dd{\scp*0.5cm}
    \def\dx{\scp*1cm}
    \def\cx{\scp*0.3cm}
    \def\txu{2*\dx}    
    \def\txl{0*\dx}
    \def\dy{\scp*1cm}
    \def\cy{\scp*0.3cm}
    \def\ty{2*\dy}
    \tikzset{whp/.style={circle, inner sep=0pt, text width={\pointsize}, draw=black, fill=white}}
    \tikzset{blp/.style={circle, inner sep=0pt, text width={\pointsize}, draw=black, fill=black}}
    \tikzset{lk/.style={regular polygon, regular polygon sides=4, inner sep=0pt, text width={\linksize}, draw=black, fill=black}}
    \draw[dotted] ({0-\dd},{0}) -- ({\txl+\dd},{0});
    \draw[dotted] ({0-\dd},{\ty}) -- ({\txu+\dd},{\ty});
    \coordinate (l1) at ({0+0*\dx},{0+0*\ty}) {};    
    \coordinate (u1) at ({0+0*\dx},{0+1*\ty}) {};
    \coordinate (u2) at ({0+1*\dx},{0+1*\ty}) {};
    \coordinate (u3) at ({0+2*\dx},{0+1*\ty}) {};
    \draw (u2) -- ++ ({0*\dx},{-1*\dy}) -| (u3);
    \draw (l1) -- (u1);
    \node[whp] at (l1) {};
    \node[whp] at (u1) {};
    \node[whp] at (u2) {};
    \node[blp] at (u3) {};
  \end{tikzpicture}
  \right)^\ast=
                \begin{tikzpicture}[vadjust={2}]
    \def\scp{0.666}
    \def\linksize{\scp*0.075cm}
    \def\pointsize{\scp*0.25cm}
    \def\dd{\scp*0.5cm}
    \def\dx{\scp*1cm}
    \def\cx{\scp*0.3cm}
    \def\txu{0*\dx}    
    \def\txl{2*\dx}
    \def\dy{\scp*1cm}
    \def\cy{\scp*0.3cm}
    \def\ty{2*\dy}
    \tikzset{whp/.style={circle, inner sep=0pt, text width={\pointsize}, draw=black, fill=white}}
    \tikzset{blp/.style={circle, inner sep=0pt, text width={\pointsize}, draw=black, fill=black}}
    \tikzset{lk/.style={regular polygon, regular polygon sides=4, inner sep=0pt, text width={\linksize}, draw=black, fill=black}}
    \draw[dotted] ({0-\dd},{0}) -- ({\txl+\dd},{0});
    \draw[dotted] ({0-\dd},{\ty}) -- ({\txu+\dd},{\ty});
    \coordinate (l1) at ({0+0*\dx},{0+0*\ty}) {};    
    \coordinate (l2) at ({0+1*\dx},{0+0*\ty}) {};
    \coordinate (l3) at ({0+2*\dx},{0+0*\ty}) {};
    \coordinate (u1) at ({0+0*\dx},{0+1*\ty}) {};
    \draw (l2) -- ++ ({0*\dx},{1*\dy}) -| (l3);
    \draw (l1) -- (u1);
    \node[whp] at (l1) {};
    \node[whp] at (l2) {};
    \node[blp] at (l3) {};
    \node[whp] at (u1) {};
  \end{tikzpicture}        $
\end{center}
{
  \newcommand{\anyblock}{\mathsf{B}}  
              \newcommand{\inlenI}[1]{k_{#1}}
              \newcommand{\outlenI}[1]{\ell_{#1}}
              \newcommand{\incolI}[1]{\mathfrak{c}_{#1}}
            \newcommand{\outcolI}[1]{\mathfrak{d}_{#1}}
            \newcommand{\partitionI}[1]{p_{#1}}
            \newcommand{\inind}{a}
            \newcommand{\outind}{b}
            \newcommand{\twoind}{t}
For any $\{\inlenI{\twoind},\outlenI{\twoind}\}\subseteq \Sintegersnn$, any $\incolI{\twoind}\in \Scolors^{\Ssetmonoidalproduct \inlenI{\twoind}}$, any $\outcolI{\twoind}\in \Scolors^{\Ssetmonoidalproduct \inlenI{\twoind}}$ and any partition $\Sxfromto{\partitionI{\twoind}}{\incolI{\twoind}}{\outcolI{\twoind}}$ for each $\twoind\in\Sonetwo$ the \emph{tensor product} $\partitionI{1}\Smonoidalproduct\partitionI{2}$ is the partition 
$\Sfromto{\incolI{1}\incolI{2}}{\outcolI{1}\outcolI{2}}$          obtained by horizontal concatenation, $\partitionI{1}\Smonoidalproduct\partitionI{2}\Seqpd\partitionI{1}\cup \{ \{\upp{(\inlenI{1}+\inind)}\Ssetbuilder \inind\in\SYnumbers{\inlenI{2}}\Sand \upp{\inind}\in \anyblock\}\cup \{\lop{(\outlenI{1}+\outind)}\Ssetbuilder \outind\in\SYnumbers{\outlenI{2}}\Sand \lop{\outind}\in \anyblock\}\}_{\anyblock\in\partitionI{2}}$.
}
 \begin{center}
   $    
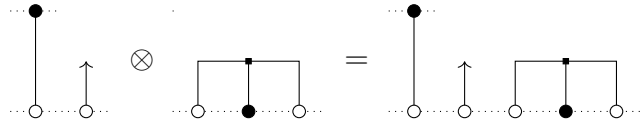
\begin{tikzpicture}[baseline=0.575cm]
    \def\scp{0.666}
    \def\linksize{\scp*0.075cm}
    \def\pointsize{\scp*0.25cm}
    \def\dd{\scp*0.5cm}
    \def\dx{\scp*1cm}
    \def\cx{\scp*0.3cm}
    \def\txu{0*\dx}    
    \def\txl{1*\dx}
    \def\dy{\scp*1cm}
    \def\cy{\scp*0.3cm}
    \def\ty{2*\dy}
    \tikzset{whp/.style={circle, inner sep=0pt, text width={\pointsize}, draw=black, fill=white}}
    \tikzset{blp/.style={circle, inner sep=0pt, text width={\pointsize}, draw=black, fill=black}}
    \tikzset{lk/.style={regular polygon, regular polygon sides=4, inner sep=0pt, text width={\linksize}, draw=black, fill=black}}
    \draw[dotted] ({0-\dd},{0}) -- ({\txl+\dd},{0});
    \draw[dotted] ({0-\dd},{\ty}) -- ({\txu+\dd},{\ty});
    \coordinate (l1) at ({0+0*\dx},{0+0*\ty}) {};    
    \coordinate (l2) at ({0+1*\dx},{0+0*\ty}) {};
    \coordinate (u1) at ({0+0*\dx},{0+1*\ty}) {};
    \draw[->] (l2) -- ++({0*\dx},{1*\dy});
    \draw (l1) -- (u1);
    \node[whp] at (l1) {};
    \node[whp] at (l2) {};
    \node[blp] at (u1) {};
  \end{tikzpicture}
\ {}  \otimes{}\
              \begin{tikzpicture}[baseline=0.575cm]
    \def\scp{0.666}
    \def\linksize{\scp*0.075cm}
    \def\pointsize{\scp*0.25cm}
    \def\dd{\scp*0.5cm}
    \def\dx{\scp*1cm}
    \def\cx{\scp*0.3cm}
    \def\txu{-1*\dx}    
    \def\txl{2*\dx}
    \def\dy{\scp*1cm}
    \def\cy{\scp*0.3cm}
    \def\ty{2*\dy}
    \tikzset{whp/.style={circle, inner sep=0pt, text width={\pointsize}, draw=black, fill=white}}
    \tikzset{blp/.style={circle, inner sep=0pt, text width={\pointsize}, draw=black, fill=black}}
    \tikzset{lk/.style={regular polygon, regular polygon sides=4, inner sep=0pt, text width={\linksize}, draw=black, fill=black}}
    \draw[dotted] ({0-\dd},{0}) -- ({\txl+\dd},{0});
    \draw[dotted] ({0-\dd},{\ty}) -- ({\txu+\dd},{\ty});
    \coordinate (l1) at ({0+0*\dx},{0+0*\ty}) {};    
    \coordinate (l2) at ({0+1*\dx},{0+0*\ty}) {};
    \coordinate (l3) at ({0+2*\dx},{0+0*\ty}) {};
    \node[lk] at ({1*\dx},{1*\dy}) {};
    \draw (l1) -- ++ ({0*\dx},{1*\dy}) -| (l2);
    \draw  ({1*\dx},{1*\dy}) -| (l3);
    \node[whp] at (l1) {};
    \node[blp] at (l2) {};
    \node[whp] at (l3) {};
  \end{tikzpicture}
  \ {}={} \
               \begin{tikzpicture}[baseline=0.575cm]
    \def\scp{0.666}
    \def\linksize{\scp*0.075cm}
    \def\pointsize{\scp*0.25cm}
    \def\dd{\scp*0.5cm}
    \def\dx{\scp*1cm}
    \def\cx{\scp*0.3cm}
    \def\txu{0*\dx}    
    \def\txl{4*\dx}
    \def\dy{\scp*1cm}
    \def\cy{\scp*0.3cm}
    \def\ty{2*\dy}
    \tikzset{whp/.style={circle, inner sep=0pt, text width={\pointsize}, draw=black, fill=white}}
    \tikzset{blp/.style={circle, inner sep=0pt, text width={\pointsize}, draw=black, fill=black}}
    \tikzset{lk/.style={regular polygon, regular polygon sides=4, inner sep=0pt, text width={\linksize}, draw=black, fill=black}}
    \draw[dotted] ({0-\dd},{0}) -- ({\txl+\dd},{0});
    \draw[dotted] ({0-\dd},{\ty}) -- ({\txu+\dd},{\ty});
    \coordinate (l1) at ({0+0*\dx},{0+0*\ty}) {};    
    \coordinate (l2) at ({0+1*\dx},{0+0*\ty}) {};
    \coordinate (l3) at ({0+2*\dx},{0+0*\ty}) {};
    \coordinate (l4) at ({0+3*\dx},{0+0*\ty}) {};
    \coordinate (l5) at ({0+4*\dx},{0+0*\ty}) {};    
    \coordinate (u1) at ({0+0*\dx},{0+1*\ty}) {};
    \node[lk] at ({3*\dx},{1*\dy}) {};
    \draw[->] (l2) -- ++({0*\dx},{1*\dy});
    \draw (l1) -- (u1);
    \draw (l3) -- ++ ({0*\dx},{1*\dy}) -| (l4);
    \draw  ({3*\dx},{1*\dy}) -| (l5);    
    \node[whp] at (l1) {};
    \node[whp] at (l2) {};
    \node[whp] at (l3) {};
    \node[blp] at (l4) {};
    \node[whp] at (l5) {};
    \node[blp] at (u1) {};
  \end{tikzpicture}$
\end{center}
{           \newcommand{\inlen}{k}
          \newcommand{\midlen}{\ell}
          \newcommand{\outlen}{m}
            \newcommand{\incol}{\mathfrak{c}}
            \newcommand{\midcol}{\mathfrak{d}}
            \newcommand{\outcol}{\mathfrak{e}}          
          \newcommand{\firstpartition}{p}
          \newcommand{\secondpartition}{q}
          \newcommand{\midsup}{s}
          \newcommand{\inblock}{\mathsf{A}}
          \newcommand{\midblock}{B}
          \newcommand{\outblock}{\mathsf{C}}
          \newcommand{\inind}{i}
          \newcommand{\outind}{j}
          Finally, for any $\{\inlen,\midlen,\outlen\}\subseteq \Sintegersnn$, any $\incol\in\Scolors^{\Ssetmonoidalproduct \inlen}$, $\midcol\in\Scolors^{\Ssetmonoidalproduct \midlen}$, and $\outcol\in\Scolors^{\Ssetmonoidalproduct \outcol}$ and any partitions $\Sxfromto{\firstpartition}{\incol}{\midcol}$ and $\Sxfromto{\secondpartition}{\midcol}{\outcol}$
          the \emph{composition} $\Sxfromto{\secondpartition\firstpartition}{\incol}{\outcol}$ of $(\secondpartition,\firstpartition)$ is the vertical concatenation under forgetting of loops, $\secondpartition\firstpartition\Seqpd \{\inblock\in\firstpartition\Sand \inblock\subseteq \{\upp{1},\upp{2},\ldots\}\}\cup \{\outblock\in\secondpartition\Sand \outblock\subseteq \{\lop{1},\lop{2},\ldots\}\}\cup \{ \bigcup\{\inblock\cap\{\upp{1},\upp{2},\ldots\}\Ssetbuilder \inblock\in \firstpartition\Sand \exists \outind\in\midblock\Ssuchthat\lop{\outind}\in\inblock \}\cup \bigcup\{\outblock\cap \{\lop{1},\lop{2},\ldots\}\Ssetbuilder \outblock\in\secondpartition\Sand \exists \inind\in\midblock\Ssuchthat\upp{\inind}\in\outblock\}\}_{\midblock\in\midsup}\backslash \{\emptyset\} $, where 
          $\midsup$ is the set of equivalence classes of the finest equivalence relation on $\SYnumbers{\midlen}$  coarser than the equivalence relation with the classes $\{\{\outind\in\SYnumbers{\midlen}\Sand \lop{\outind}\in \inblock\}\}_{\inblock\in\firstpartition}\backslash\{\emptyset\}$ and the equivalence relation with the classes $\{\{\inind\in\SYnumbers{\midlen}\Sand \upp{\inind}\in \outblock\}\}_{\outblock\in\secondpartition}\backslash\{\emptyset\}$.
          \begin{center}
            $  \begin{tikzpicture}[baseline=0.91cm]
    \def\scp{0.666}
    \def\linksize{\scp*0.075cm}
    \def\pointsize{\scp*0.25cm}
    \def\dd{\scp*0.5cm}
    \def\dx{\scp*1cm}
    \def\cx{\scp*0.3cm}
    \def\txu{1*\dx}    
    \def\txl{0*\dx}
    \def\dy{\scp*1cm}
    \def\cy{\scp*0.3cm}
    \def\ty{3*\dy}
    \tikzset{whp/.style={circle, inner sep=0pt, text width={\pointsize}, draw=black, fill=white}}
    \tikzset{blp/.style={circle, inner sep=0pt, text width={\pointsize}, draw=black, fill=black}}
    \tikzset{lk/.style={regular polygon, regular polygon sides=4, inner sep=0pt, text width={\linksize}, draw=black, fill=black}}
    \draw[dotted] ({0-\dd},{0}) -- ({\txl+\dd},{0});
    \draw[dotted] ({0-\dd},{\ty}) -- ({\txu+\dd},{\ty});
    \coordinate (l1) at ({0+0*\dx},{0+0*\ty}) {};    
    \coordinate (u1) at ({0+0*\dx},{0+1*\ty}) {};
    \coordinate (u2) at ({0+1*\dx},{0+1*\ty}) {};
    \draw (u1) -- ++ ({0*\dx},{-1*\dy}) -| (u2);
    \draw[->] (l1) -- ++({0*\dx},{1*\dy});
    \node[whp] at (l1) {};
    \node[whp] at (u1) {};
    \node[blp] at (u2) {};
  \end{tikzpicture}
  \cdot
  \begin{tikzpicture}[baseline=0.575cm]
    \def\scp{0.666}
    \def\linksize{\scp*0.075cm}
    \def\pointsize{\scp*0.25cm}
    \def\dd{\scp*0.5cm}
    \def\dx{\scp*1cm}
    \def\cx{\scp*0.3cm}
    \def\txu{2*\dx}    
    \def\txl{1*\dx}
    \def\dy{\scp*1cm}
    \def\cy{\scp*0.3cm}
    \def\ty{2*\dy}
    \tikzset{whp/.style={circle, inner sep=0pt, text width={\pointsize}, draw=black, fill=white}}
    \tikzset{blp/.style={circle, inner sep=0pt, text width={\pointsize}, draw=black, fill=black}}
    \tikzset{lk/.style={regular polygon, regular polygon sides=4, inner sep=0pt, text width={\linksize}, draw=black, fill=black}}
    \draw[dotted] ({0-\dd},{0}) -- ({\txl+\dd},{0});
    \draw[dotted] ({0-\dd},{\ty}) -- ({\txu+\dd},{\ty});
    \coordinate (l1) at ({0+0*\dx},{0+0*\ty}) {};    
    \coordinate (l2) at ({0+1*\dx},{0+0*\ty}) {};
    \coordinate (u1) at ({0+0*\dx},{0+1*\ty}) {};
    \coordinate (u2) at ({0+1*\dx},{0+1*\ty}) {};
    \coordinate (u3) at ({0+2*\dx},{0+1*\ty}) {};
    \node[lk] at ({1*\dx},{1*\dy}) {};
    \draw ({1*\dx},{1*\dy}) -| (u3);
    \draw (l1) -- (u1);
    \draw (l2) -- (u2);    
    \node[whp] at (l1) {};
    \node[blp] at (l2) {};
    \node[whp] at (u1) {};
    \node[whp] at (u2) {};
    \node[blp] at (u3) {};
  \end{tikzpicture}
  =\begin{tikzpicture}[baseline=0.91cm]
    \def\scp{0.666}
    \def\linksize{\scp*0.075cm}
    \def\pointsize{\scp*0.25cm}
    \def\dd{\scp*0.5cm}
    \def\dx{\scp*1cm}
    \def\cx{\scp*0.3cm}
    \def\txu{2*\dx}    
    \def\txl{0*\dx}
    \def\dy{\scp*1cm}
    \def\cy{\scp*0.3cm}
    \def\ty{3*\dy}
    \tikzset{whp/.style={circle, inner sep=0pt, text width={\pointsize}, draw=black, fill=white}}
    \tikzset{blp/.style={circle, inner sep=0pt, text width={\pointsize}, draw=black, fill=black}}
    \tikzset{lk/.style={regular polygon, regular polygon sides=4, inner sep=0pt, text width={\linksize}, draw=black, fill=black}}
    \draw[dotted] ({0-\dd},{0}) -- ({\txl+\dd},{0});
    \draw[dotted] ({0-\dd},{\ty}) -- ({\txu+\dd},{\ty});
    \coordinate (l1) at ({0+0*\dx},{0+0*\ty}) {};    
    \coordinate (u1) at ({0+0*\dx},{0+1*\ty}) {};
    \coordinate (u2) at ({0+1*\dx},{0+1*\ty}) {};
    \coordinate (u3) at ({0+2*\dx},{0+1*\ty}) {};
    \node[lk] at ({1*\dx},{2*\dy}) {};
    \draw (u1) -- ++ ({0*\dx},{-1*\dy}) -| (u2);
    \draw ({1*\dx},{2*\dy}) -| (u3);    
    \draw[->] (l1) -- ++({0*\dx},{1*\dy});
    \node[whp] at (l1) {};
    \node[whp] at (u1) {};
    \node[whp] at (u2) {};
    \node[blp] at (u3) {};
  \end{tikzpicture}
    {\color{gray}\leftarrow
  \begin{tikzpicture}[baseline=1.5cm]
    \def\scp{0.666}
    \def\linksize{\scp*0.075cm}
    \def\pointsize{\scp*0.25cm}
    \def\dd{\scp*0.5cm}
    \def\dx{\scp*1cm}
    \def\cx{\scp*0.3cm}
    \def\txu{2*\dx}
    \def\txm{1*\dx}
    \def\txl{0*\dx}
    \def\dy{\scp*1cm}
    \def\cy{\scp*0.3cm}
    \def\tyl{3*\dy}
    \def\tyu{2*\dy}    
    \tikzset{whp/.style={circle, inner sep=0pt, text width={\pointsize}, draw=gray, fill=white}}
    \tikzset{blp/.style={circle, inner sep=0pt, text width={\pointsize}, draw=gray, fill=gray}}
    \tikzset{lk/.style={regular polygon, regular polygon sides=4, inner sep=0pt, text width={\linksize}, draw=gray, fill=gray}}
    \draw[dotted] ({0-\dd},{0}) -- ({\txl+\dd},{0});
    \draw[dotted] ({0-\dd},{\tyl}) -- ({\txm+\dd},{\tyl});
    \draw[dotted] ({0-\dd},{\tyu+\tyl}) -- ({\txu+\dd},{\tyu+\tyl});    
    \coordinate (l1) at ({0+0*\dx},{0+0*\tyl+0*\tyu}) {};    
    \coordinate (m1) at ({0+0*\dx},{0+1*\tyl+0*\tyu}) {};
    \coordinate (m2) at ({0+1*\dx},{0+1*\tyl+0*\tyu}) {};
    \coordinate (u1) at ({0+0*\dx},{0+1*\tyl+1*\tyu}) {};
    \coordinate (u2) at ({0+1*\dx},{0+1*\tyl+1*\tyu}) {};
    \coordinate (u3) at ({0+2*\dx},{0+1*\tyl+1*\tyu}) {};
    \node[lk] at ({1*\dx},{1*\dy+1*\tyl+0*\tyu}) {};
    \draw ({1*\dx},{1*\dy+1*\tyl+0*\tyu}) -| (u3);
    \draw (m1) -- ({0*\dx},{-1*\dy+1*\tyl+0*\tyu}) -| (m2);    
    \draw[->] (l1) -- ++({0*\dx},{1*\dy});
    \draw (m1) -- (u1);
    \draw (m2) -- (u2);    
    \node[whp] at (l1) {};
    \node[whp] at (m1) {};
    \node[blp] at (m2) {};    
    \node[whp] at (u1) {};
    \node[whp] at (u2) {};
    \node[blp] at (u3) {};
  \end{tikzpicture}}$
          \end{center}
        }
        With those definitions, a \emph{category of (two-colored) partitions} is any set of partitions which is closed under forming adjoints, tensor products and (where defined)  composition and which contains the partitions
        \begin{IEEEeqnarray*}{C}
                    \begin{tikzpicture}[baseline=0]
    \def\scp{0.666}
    \def\linksize{\scp*0.075cm}
    \def\pointsize{\scp*0.25cm}
    \def\dd{\scp*0.5cm}
    \def\dx{\scp*1cm}
    \def\cx{\scp*0.3cm}
    \def\txu{0*\dx}    
    \def\txl{0*\dx}
    \def\dy{\scp*1cm}
    \def\cy{\scp*0.3cm}
    \def\ty{2*\dy}
    \tikzset{whp/.style={circle, inner sep=0pt, text width={\pointsize}, draw=black, fill=white}}
    \tikzset{blp/.style={circle, inner sep=0pt, text width={\pointsize}, draw=black, fill=black}}
    \tikzset{lk/.style={regular polygon, regular polygon sides=4, inner sep=0pt, text width={\linksize}, draw=black, fill=black}}
    \draw[dotted] ({0-\dd},{0}) -- ({\txl+\dd},{0});
    \draw[dotted] ({0-\dd},{\ty}) -- ({\txu+\dd},{\ty});
    \coordinate (l1) at ({0+0*\dx},{0+0*\ty}) {};    
    \coordinate (u1) at ({0+0*\dx},{0+1*\ty}) {};
    \draw (l1) -- (u1);
    \node[whp] at (l1) {};
    \node[whp] at (u1) {};
  \end{tikzpicture},
          \begin{tikzpicture}[baseline=0]
    \def\scp{0.666}
    \def\linksize{\scp*0.075cm}
    \def\pointsize{\scp*0.25cm}
    \def\dd{\scp*0.5cm}
    \def\dx{\scp*1cm}
    \def\cx{\scp*0.3cm}
    \def\txu{0*\dx}    
    \def\txl{0*\dx}
    \def\dy{\scp*1cm}
    \def\cy{\scp*0.3cm}
    \def\ty{2*\dy}
    \tikzset{whp/.style={circle, inner sep=0pt, text width={\pointsize}, draw=black, fill=white}}
    \tikzset{blp/.style={circle, inner sep=0pt, text width={\pointsize}, draw=black, fill=black}}
    \tikzset{lk/.style={regular polygon, regular polygon sides=4, inner sep=0pt, text width={\linksize}, draw=black, fill=black}}
    \draw[dotted] ({0-\dd},{0}) -- ({\txl+\dd},{0});
    \draw[dotted] ({0-\dd},{\ty}) -- ({\txu+\dd},{\ty});
    \coordinate (l1) at ({0+0*\dx},{0+0*\ty}) {};    
    \coordinate (u1) at ({0+0*\dx},{0+1*\ty}) {};
    \draw (l1) -- (u1);
    \node[blp] at (l1) {};
    \node[blp] at (u1) {};
  \end{tikzpicture},  
          \begin{tikzpicture}[baseline=0]
    \def\scp{0.666}
    \def\linksize{\scp*0.075cm}
    \def\pointsize{\scp*0.25cm}
    \def\dd{\scp*0.5cm}
    \def\dx{\scp*1cm}
    \def\cx{\scp*0.3cm}
    \def\txu{1*\dx}    
    \def\txl{0*\dx}
    \def\dy{\scp*1cm}
    \def\cy{\scp*0.3cm}
    \def\ty{2*\dy}
    \tikzset{whp/.style={circle, inner sep=0pt, text width={\pointsize}, draw=black, fill=white}}
    \tikzset{blp/.style={circle, inner sep=0pt, text width={\pointsize}, draw=black, fill=black}}
    \tikzset{lk/.style={regular polygon, regular polygon sides=4, inner sep=0pt, text width={\linksize}, draw=black, fill=black}}
    \draw[dotted] ({0-\dd},{\ty}) -- ({\txu+\dd},{\ty});
    \coordinate (u1) at ({0+0*\dx},{0+1*\ty}) {};
    \coordinate (u2) at ({0+1*\dx},{0+1*\ty}) {};
    \draw (u1) -- ++ ({0*\dx},{-1*\dy}) -| (u2);
    \node[whp] at (u1) {};
    \node[blp] at (u2) {};
  \end{tikzpicture},
  \begin{tikzpicture}[baseline=0]
    \def\scp{0.666}
    \def\linksize{\scp*0.075cm}
    \def\pointsize{\scp*0.25cm}
    \def\dd{\scp*0.5cm}
    \def\dx{\scp*1cm}
    \def\cx{\scp*0.3cm}
    \def\txu{1*\dx}    
    \def\txl{0*\dx}
    \def\dy{\scp*1cm}
    \def\cy{\scp*0.3cm}
    \def\ty{2*\dy}
    \tikzset{whp/.style={circle, inner sep=0pt, text width={\pointsize}, draw=black, fill=white}}
    \tikzset{blp/.style={circle, inner sep=0pt, text width={\pointsize}, draw=black, fill=black}}
    \tikzset{lk/.style={regular polygon, regular polygon sides=4, inner sep=0pt, text width={\linksize}, draw=black, fill=black}}
    \draw[dotted] ({0-\dd},{\ty}) -- ({\txu+\dd},{\ty});
    \coordinate (u1) at ({0+0*\dx},{0+1*\ty}) {};
    \coordinate (u2) at ({0+1*\dx},{0+1*\ty}) {};
    \draw (u1) -- ++ ({0*\dx},{-1*\dy}) -| (u2);
    \node[blp] at (u1) {};
    \node[whp] at (u2) {};
  \end{tikzpicture},
  \begin{tikzpicture}[baseline=0]
    \def\scp{0.666}
    \def\linksize{\scp*0.075cm}
    \def\pointsize{\scp*0.25cm}
    \def\dd{\scp*0.5cm}
    \def\dx{\scp*1cm}
    \def\cx{\scp*0.3cm}
    \def\txu{0*\dx}    
    \def\txl{1*\dx}
    \def\dy{\scp*1cm}
    \def\cy{\scp*0.3cm}
    \def\ty{2*\dy}
    \tikzset{whp/.style={circle, inner sep=0pt, text width={\pointsize}, draw=black, fill=white}}
    \tikzset{blp/.style={circle, inner sep=0pt, text width={\pointsize}, draw=black, fill=black}}
    \tikzset{lk/.style={regular polygon, regular polygon sides=4, inner sep=0pt, text width={\linksize}, draw=black, fill=black}}
    \draw[dotted] ({0-\dd},{0}) -- ({\txl+\dd},{0});
    \coordinate (l1) at ({0+0*\dx},{0+0*\ty}) {};
    \coordinate (l2) at ({0+1*\dx},{0+0*\ty}) {};
    \draw (l1) -- ++ (0,{1*\dy}) -| (l2);
    \node[whp] at (l1) {};
    \node[blp] at (l2) {};
  \end{tikzpicture},
    \begin{tikzpicture}[baseline=0]
    \def\scp{0.666}
    \def\linksize{\scp*0.075cm}
    \def\pointsize{\scp*0.25cm}
    \def\dd{\scp*0.5cm}
    \def\dx{\scp*1cm}
    \def\cx{\scp*0.3cm}
    \def\txu{0*\dx}    
    \def\txl{1*\dx}
    \def\dy{\scp*1cm}
    \def\cy{\scp*0.3cm}
    \def\ty{2*\dy}
    \tikzset{whp/.style={circle, inner sep=0pt, text width={\pointsize}, draw=black, fill=white}}
    \tikzset{blp/.style={circle, inner sep=0pt, text width={\pointsize}, draw=black, fill=black}}
    \tikzset{lk/.style={regular polygon, regular polygon sides=4, inner sep=0pt, text width={\linksize}, draw=black, fill=black}}
    \draw[dotted] ({0-\dd},{0}) -- ({\txl+\dd},{0});
    \coordinate (l1) at ({0+0*\dx},{0+0*\ty}) {};
    \coordinate (l2) at ({0+1*\dx},{0+0*\ty}) {};
    \draw (l1) -- ++ (0,{1*\dy}) -| (l2);
    \node[blp] at (l1) {};
    \node[whp] at (l2) {};
  \end{tikzpicture}.\IEEEyesnumber\label{eq:minimal_partitions}
\end{IEEEeqnarray*}
\par
{
  \newcommand{\indj}{j}
  \newcommand{\indi}{i}
  \newcommand{\incol}{\mathfrak{c}}
  \newcommand{\outcol}{\mathfrak{d}}
  \newcommand{\incolx}[1]{\mathfrak{c}_{#1}}
  \newcommand{\outcolx}[1]{\mathfrak{d}_{#1}}  
  \newcommand{\inlen}{k}
  \newcommand{\outlen}{\ell}
  \newcommand{\infree}{g}
  \newcommand{\outfree}{j}
  \newcommand{\inbound}{h}
  \newcommand{\outbound}{i}
  \newcommand{\inind}{a}
  \newcommand{\outind}{b}
  \newcommand{\infreex}[1]{g_{#1}}
  \newcommand{\outfreex}[1]{j_{#1}}
  \newcommand{\inboundx}[1]{h_{#1}}
  \newcommand{\outboundx}[1]{i_{#1}}  
  \newcommand{\incounter}{a}
  \newcommand{\outcounter}{b}
  \newcommand{\thepartition}{p}
  \newcommand{\anynum}{z}
For each category of two-colored partitions there is an associated quantum subgroup of $U_\thedim^+$ whose CQG algebra is obtained by adding to the definition of $\mathcal{O}(U_\thedim^+)$ for any  $\{\inlen,\outlen\}\subseteq \Sintegersnn$, any $\incol\in\Scolors^{\Ssetmonoidalproduct\inlen}$ and $\outcol\in\Scolors^{\Ssetmonoidalproduct\outlen}$, any partition $\Sxfromto{\thepartition}{\incol}{\outcol}$ of the category and any $\infree\in\SYnumbers{\thedim}^{\Ssetmonoidalproduct\inlen}$ and $\outfree\in\SYnumbers{\thedim}^{\Ssetmonoidalproduct\outlen}$ the relation
\begin{IEEEeqnarray*}{l}
\textstyle\sum_{\outbound\in\SYnumbers{\thedim}^{\Ssetmonoidalproduct \outlen}}
\Szetafunction{\thepartition}{\Sker(\infree\djpin\outbound)}\;
\thecuniX{\outcolx{1}}{\outfreex{1}}{\outboundx{1}}\thecuniX{\outcolx{2}}{\outfreex{2}}{\outboundx{2}}\hdots \thecuniX{\outcolx{\outlen}}{\outfreex{\outlen}}{\outboundx{\outlen}}\IEEEeqnarraynumspace\IEEEyesnumber\label{eq:relation_of_partition}\\
\textstyle\hspace{8em}{}-\sum_{\inbound\in\SYnumbers{\thedim}^{\Ssetmonoidalproduct \inlen}}\Szetafunction{\thepartition}{\Sker(\inbound\djpin\outfree)}\;
\thecuniX{\incolx{1}}{\inboundx{1}}{\infreex{1}}\thecuniX{\incolx{2}}{\inboundx{2}}{\infreex{2}}\hdots \thecuniX{\incolx{\inlen}}{\inboundx{\inlen}}{\infreex{\inlen}},
\end{IEEEeqnarray*}
where for any $\infree\in\SYnumbers{\thedim}^{\Ssetmonoidalproduct\inlen}$ and $\outfree\in\SYnumbers{\thedim}^{\Ssetmonoidalproduct\outlen}$ the number $\Szetafunction{\thepartition}{\Sker(\infree\djpin\outfree)}$ is defined as $1$ if and only if the equivalence relation embodied by $\thepartition$ is finer than the equivalence relation with the classes $\{\{\upp{\inind}\Ssetbuilder\inind\in\SYnumbers{\inlen}\Sand \infreex{\inind}=\anynum \}\cup\{\lop{\outind}\Ssetbuilder\outind\in\SYnumbers{\outlen}\Sand \outfreex{\outind}=\anynum\}\}_{\anynum\in\SYnumbers{\thedim}}\backslash \{\emptyset\}$ and as $0$ otherwise. 
}
\par
In particular,  $U_\thedim^+$ itself corresponds exactly to the smallest of all possible categories of partitions. The non\-/zero ones among the polynomials \eqref{eq:relation_of_partition} of the partitions \eqref{eq:minimal_partitions} are (up to a sign) the relations \eqref{eq:universal_matrix_quantum_group_relations} (for $Q$ being the identity).
\par
While many other categories of partitions have  been found in a collective effort by various authors (see \cite{BanicaSpeicher2009, BanicaCurranSpeicher2010, Weber2013, RaumWeber2014, RaumWeber2016a, RaumWeber2016b,TarragoWeber2018, Gromada2018, MangWeber2020, MangWeber2021a, MangWeber2021b, MangWeber2021c, Maassen2021}), the full classification is still open. And, in the vast majority of cases little to nothing is known about the easy quantum groups corresponding to these categories. Critically, it is unclear when two easy quantum groups are isomorphic, especially across different $\thedim$.  Homological invariants can help with making these distinctions.
\subsubsection{Quantum group cohomology}
\label{section:quantum_group_cohomology}
{
  \newcommand{\anyqg}{G}
  One invariant which can possibly discriminate between at least some easy quantum groups is the quantum group co\-ho\-mo\-lo\-gy of their  duals (see \cite{Bichon2017} for an introduction). The fact that in terminology this co\-ho\-mo\-lo\-gy is attributed not to the easy quantum groups but the discrete quantum groups (see \cite{VanDaele1996b}) arising as their generalized Pontryagin duals is immaterial. A coefficient module which can be chosen  for any  discrete quantum group is the complex numbers. It provides the broadest possible range for comparisons. Given any compact quantum group $\anyqg$, the \emph{quantum group co\-ho\-mo\-lo\-gy  with trivial coefficients of its discrete dual} is defined as
  \begin{IEEEeqnarray*}{rCl}
    H^{\Sargph}(\widehat{\anyqg})=\SextX{\mathcal{O}(\anyqg)}{\themodulecomps}{\themodulecomps}{\Sargph},
  \end{IEEEeqnarray*}
  i.e., as the derived functor at $\themodulecomps$  of the contravariant external hom functor $\ShomO{\mathcal{O}(\anyqg)}(\Sargph,\themodulecomps)$  from the category of (right) $\mathcal{O}(\anyqg)$\-/modules to the category of complex vector spaces. Thus, in order to compute these invariants of compact quantum groups it is necessary to have a projective resolution of  $\themodulecomps$ in terms of right $\mathcal{O}(\anyqg)$\-/modules.
}

\subsection{Motivation}
\label{section:motivation}
{
  \newcommand{\anyqg}{G}
    \newcommand{\otherqg}{H}
  \newcommand{\somehopf}{\mathcal{A}}  
The present article provides a projective resolution of the trivial module of the CQG algebra of the easy quantum group $U_\thedim^+$ and uses it to compute  the co\-ho\-mo\-lo\-gy of the discrete dual of $U_\thedim^+$ (see Section~\ref{section:on_free_unitary_quantum_group_and_easy_quantum_groups} for definitions). Of course, before the work that led to this article had even begun, Baraquin, Franz, Gerhold, Kula and Tobolski had already found in \cite{BaraquinFranzGerholdKulaTobolski2023} such a resolution and thus managed to  determine the third and higher orders of said co\-ho\-mo\-lo\-gy after Das, Franz, Kula and Skalski had previously computed the first two orders in \cite{DasFranzKulaSkalski2018,DasFranzKulaSkalski2021}. In fact,  in much greater  generality than the present article, the results of \cite{BaraquinFranzGerholdKulaTobolski2023} apply, to the co\-ho\-mo\-lo\-gy of the dual of the universal compact $\thedim\times\thedim$\-/matrix quantum group $U_\thedim^+(Q)$ or, technically, the more general universal co\-/sovereign Hopf algebra $\mathcal{H}(F)$ from \cite{Bichon2001} for generic $F$, i.e., any invertible scalar matrix $F$ such that $\mathrm{tr}(F)=0$  if and only if  $\mathrm{tr}(F\SinverseP)=0$  and such that there exist no roots of unity $q$ of order greater than $2$ with $q^2-\sqrt{\mathrm{tr}(F)\mathrm{tr}(F\SinverseP)}\Saction q+1=0$. Actually, transcending the results of the present article even further, \cite{BaraquinFranzGerholdKulaTobolski2023}  determines in addition the bi\-/algebra co\-ho\-mo\-lo\-gy  of  $\mathcal{H}(F)$, also for generic $F$ (although their results hold in greater generality by \cite{Bichon2023}), and Hochschild co\-ho\-mo\-lo\-gy with arbitrary one\-/dimensional coefficients. Finally, it must be emphasized that Ba\-ra\-quin, Franz, Ger\-hold, Ku\-la and To\-bols\-ki's resolution is not only finite but in fact minimal, namely zero  from the fourth order onwards, whereas the one presented here is infinite. Given that what the current article provides can be achieved in much greater generality with much better results, why even write it in the first place?
\par
The article grew out of the attempt to continue the work of \cite{Mang2023arxiv} and compute after the first also the second co\-ho\-mo\-lo\-gy of the discrete duals of all easy quantum groups $\anyqg$. As demonstrated by Bichon, Franz and Gerhold in \cite{BichonFranzGerhold2017} and then Das, Franz, Gerhold, Kula and Skalski in \cite{DasFranzKulaSkalski2018,DasFranzKulaSkalski2021}, at least for certain $\anyqg$,  the standard resolution is good enough for this purpose (see Appendix~\ref{section:known_method} for an explanation of their method). However, after finding myself consistently failing with this strategy, even for relatively uncomplicated $\anyqg$, I decided to  look for smaller resolutions. An analysis of the techniques used for obtaining resolutions for the very few $\anyqg$ for which this has been achieved, left me clueless as to how to adapt those to arbitrary $\anyqg$. 
\par
Researching instead  general strategies for going from the presentation of an augemented algebra in terms of generators and relations to a projective resolution of the trivial module I learned about Anick resolutions (see \cite{Anick1986}) from \cite{MussonLeymarie2023}. I wondered whether it was feasible to use those for my purpose.  There is a theorem guaranteeing the existence of an Anick resolution for each augmented algebra and monomial order. That means, in principle, this is a resolution which could be used for \emph{any} $\anyqg$.  Of course, there is a price. Its differentials are defined recursively. Hence, with those resolutions (just like with the standard resolution) the difficulty lies in computing the co\-ho\-mo\-lo\-gy from the resolution. In contrast, with the computations of quantum group co\-ho\-mo\-lo\-gies in the literature the challenge is to find a resolution in the first place and prove its exactness. Computing the co\-ho\-mo\-lo\-gy from an Anick resolution would presumably get easier if the recursion could be solved, i.e., if there was a closed-form expression for the differentials. As far as I was able to ascertain, it is not known whether this is always possible. Eventually, I decided to try. Having no idea what to expect, it seemed prudent to test on the easy $\anyqg$ for which the set of relations $\mathcal{O}(\anyqg)$ is the smallest, $U_\thedim^+$. And for $U_\thedim^+$, it turns out, there is a non\-/recursive formula for the differential, making the co\-ho\-mo\-lo\-gy computation simple.
\par
The present article is effectively me reporting on the outcome of that experiment. Although it computes a co\-ho\-mo\-lo\-gy that is already known, it is ultimately aimed at determining new, unknown ones. 
While it provides an explicit resolution for a quantum group where a better one is already available, it does so by very different means.
 At this point in time the only easy quantum groups $G$ the full co\-ho\-mo\-lo\-gies of whose duals  have been determined are  $U_\thedim^+$, the free orthogonal quantum group $O_\thedim^+$ defined in \cite{VanDaele1996b} and the free symmetric quantum group $S_\thedim^+$ of \cite{Wang1998} (see Appendix~\ref{section:other_methods} for an analysis of the three  computations in the literature).
And, whereas it is unclear to me where to even start with the methods used in those cases, the strategy employed here is, at least in principle, immediately applicable to arbitrary easy $G$ (see Section~\ref{section:conclusions}).
}
\label{section:crux}
\subsection{Conclusions}
\label{section:conclusions}
{
  \newcommand{\anyqg}{G}
  \newcommand{\orderind}{\ell}  
     \newcommand{\indi}{i}
     \newcommand{\indj}{j}
  The present article is a case study, a single data point, anecdotal evidence. As such, very few rigorous inferences are possible. That is not its purpose. Rather, it is meant to generate ideas and suggest future research directions.
  \par
On the object level, it shows that for each $\thedim\in\Sintegers$ with $2\leq \thedim$ there exists a monomial order on the set $\thegens$ of generators in the standard presentation of the CQG algebra $\mathcal{O}(U_\thedim^+)$ of the easy quantum group $U_\thedim^+$ with respect to which the Anick resolution of the co\-/unit admits a closed\-/form, non\-/recursive expression. (This resolution is different from the one obtained by Baraquin, Franz, Gerhold, Kula and Tobolski in  \cite{BaraquinFranzGerholdKulaTobolski2023}.) And it shows that the Anick resolution is good enough to compute from it the (previously known) quantum group co\-ho\-mo\-lo\-gy of the dual of $U_\thedim^+$  in all orders. 
\par
On the meta level, the article suggests the following.
\begin{enumerate}[wide, label=\arabic*)]
\item The most important conclusion to be drawn from these findings is that Anick resolutions might be a way to reliably compute the co\-ho\-mo\-lo\-gies of the discrete duals of easy quantum groups $\anyqg$.  For any  $\anyqg$ the trivial $\mathcal{O}(\anyqg)$\-/module has such a  resolution. If a Gröbner basis of $\mathcal{O}(\anyqg)$ can be found, there is, in principle, nothing preventing the investigation of the co\-ho\-mo\-lo\-gy by the method used here. The only question is whether there is a closed formula and whether the number of recursion steps needed to prove it is managable or not.
\item In particular, if one is just interested in the second co\-ho\-mo\-lo\-gy, this approach might be worth considering. The alternative method employed so far (see Appendix~\ref{section:known_method}) uses the standard resolution, which  is likely much larger than the Anick resolution already in low degrees.
\item The formulas for the differentials $(\thedifferential{\orderind})_{\orderind\in\Sintegersnn}$ become periodic from $\orderind=4$ onwards. Additionally, the $(\thedifferential{\orderind})_{\orderind\in\Sintegersnn}$ depend on $\thedim$ only in a very benign way. Probably it would have been possible to guess the formula for general $\thedim$ and $\orderind$ knowing just the orders, say, $\orderind\leq 6$ in the special cases $\thedim\leq 5$. Perhaps the Anick resolutions of $\themodule$  for other $\anyqg$ behave similarly. In that case it might be quicker to use a computer algebra system to compute the resolution for small $\orderind$ and $\thedim$ first, then form a hypothesis for general $\orderind$ and $\thedim$ and try to use the Gröbner basis to prove exactness, rather than solve the recursion.
\end{enumerate}
}
\subsection{Additional remarks}
\label{section:additional_remarks}
The final section offers some context for the choices made in the article and discusses potential modifications.
\subsubsection{On the choice of monomial order}
{
  \newcommand{\indj}{j}
  \newcommand{\indi}{i}
  \newcommand{\indl}{\ell}
  \newcommand{\indk}{k}
  \newcommand{\indt}{t}
  \newcommand{\inds}{s}
  \newcommand{\colc}{\mathfrak{c}}
  \newcommand{\cold}{\mathfrak{d}}
  When working with $U_n^+$ or with Kac-type compact $\thedim\times\thedim$\-/matrix quantum groups in general it is customary to consider besides the fundamental representation matrix $\theuni$ its transpose $\theuni\StransposeP$, its conjugate  $\SdualX{\theuni}$ and its  adjoint $\theuni\SdaggerP$. In contrast, rather than the usual $\{\theuni,\theuni\StransposeP,\SdualX{\theuni},\theuni\SdaggerP\}$ the \hyperref[main-result]{Main result} (and the entire article) only references  $\thematrices=\{\theuni,\theuniAT,\theuniB,\theuniBT\}$. The switch  $(\SdualX{\theuni},\theuni\SdaggerP)\leadsto  (\theuniB,\theuniBT)$  is due to the particular choice of monomial order for $\thegens=\{\theuniAX{\indj}{\indi},\theuniAX{\indj}{\indi}{}\SstarP\}_{\indi,\indj=1}^\thedim$.
  \par
More precisely, it is the fact that $\{\theuniAX{\indj}{\indi}\}_{\indi,\indj=1}^\thedim=\{\thewuniX{\indj}{\indi}\}_{\indi,\indj=1}^\thedim$ was given the lexicographic order but $\{\theuniAX{\indj}{\indi}{}\SstarP\}_{\indi,\indj=1}^\thedim=\{\thebuniX{\indj}{\indi}\}_{\indi,\indj=1}^\thedim$ the \emph{opposite} of the lexicographic order, that motivates this perspective. This choice may or may not seem natural. But,  importantly, equipping both $\{\thewuniX{\indj}{\indi}\}_{\indi,\indj=1}^\thedim$ and $\{\thebuniX{\indj}{\indi}\}_{\indi,\indj=1}^\thedim$  with the lexicographic order would not have worked.

  \begin{remark}
     $\therels$ is \emph{no} Gröbner basis of $\theideal$ with respect to the degreewise lexicographic extension of the total order on $\thegens$ defined by for any
     $\{\colc,\cold\}\subseteq \Scolors$ and any $\{(\indj,\indi),(\indt,\inds)\}\subseteq \SYnumbers{\thedim}^{\Ssetmonoidalproduct 2}$, 
    \begin{IEEEeqnarray*}{rClCrCl}
      \thecuniX{\cold}{\indt}{\inds}&\leq& \thecuniX{\colc}{\indj}{\indi} &\Siff& (\cold,\colc)&=&{\mathnormal\Sblack\Swhite}\\
&&&&      {}\Sor(\cold,\colc)&\in&\{\mathnormal\Sblack\Sblack,\mathnormal\Swhite\Swhite\} \Sand (\indt,\inds)<(\indj,\indi).
\end{IEEEeqnarray*}
In particular, $\therels$ is not a universal Gröbner basis of $\theideal$.
\end{remark}
\begin{proof}
\newcommand{\theevil}{b}
  \newcommand{\leftred}{\varphi_1}
  \newcommand{\rightred}{\varphi_2}    
Actually, Proposition~\ref{proposition:bergman} is only an excerpt of Bergman's diamond lemma from \cite{Bergman1978}. The full version  states, among other things, (although not in these words) that any orbit (in the sense of monoid actions) of the reductions by a Gröbner basis has exactly one fixed point. In particular, in order to refute that $\therels$ is a Gröbner basis, it suffices to exhibit a polynomial $\theevil\in\thefreealg$ and reductions $\leftred$ and $\rightred$ by $\therels$ with $\leftred(\theevil)\neq\rightred(\theevil)$ such that both $\leftred(\theevil)$ and $\rightred(\theevil)$ are fixed points of the reductions by $\therels$.
  \par
If $\theevil\Seqpd \thewuniX{\thedim}{\thedim}\thebuniX{\thedim}{\thedim}\thewuniX{\thedim}{\thedim}$, then, on the one hand, reduction of $\theevil$ by $\thewuniX{\thedim}{\thedim}\thebuniX{\thedim}{\thedim}+\sum_{\indt=1}^{\thedim-1}\thewuniX{\indt}{\thedim}\thebuniX{\indt}{\thedim}-\theone$ produces
  \begin{IEEEeqnarray*}{rCl}
\textstyle-\sum_{\indt=1}^{\thedim-1}\thewuniX{\indt}{\thedim}\thebuniX{\indt}{\thedim}\thewuniX{\thedim}{\thedim}+\thewuniX{\thedim}{\thedim},
\end{IEEEeqnarray*}
which when reduced by $\thewuniX{\indt}{\thedim}\thebuniX{\indt}{\thedim}+\sum_{\inds=1}^{\thedim-1}\thewuniX{\indt}{\inds}\thebuniX{\indt}{\inds}-\theone$ for each $\indt\in\SYnumbers{\thedim-1}$ becomes
\begin{IEEEeqnarray*}{rCl}
\textstyle\sum_{\inds,\indt=1}^{\thedim-1}\thewuniX{\indt}{\inds}\thebuniX{\indt}{\inds}\thewuniX{\thedim}{\thedim}-(\thedim-2)\thewuniX{\thedim}{\thedim}.\IEEEyesnumber\label{eq:monomial_order_justification_1}
\end{IEEEeqnarray*}
  On the other hand, if reduced by $\thebuniX{\thedim}{\thedim}\thewuniX{\thedim}{\thedim}+\sum_{\indt=1}^{\thedim-1}\thebuniX{\indt}{\thedim}\thewuniX{\indt}{\thedim}-\theone$ the polynomial $\theevil$ takes on the form
  \begin{IEEEeqnarray*}{rCl}
\textstyle-\sum_{\indt=1}^{\thedim-1}\thewuniX{\thedim}{\thedim}\thebuniX{\indt}{\thedim}\thewuniX{\indt}{\thedim}+\thewuniX{\thedim}{\thedim},
\end{IEEEeqnarray*}
which after reduction by
$\thebuniX{\indt}{\thedim}\thewuniX{\indt}{\thedim}+\sum_{\inds=1}^{\thedim-1}\thebuniX{\indt}{\inds}\thewuniX{\indt}{\inds}-\theone$ for each $\indt\in\SYnumbers{\thedim-1}$ reads
\begin{IEEEeqnarray*}{rCl}
\textstyle\sum_{\inds,\indt=1}^{\thedim-1}\thewuniX{\thedim}{\thedim}\thebuniX{\indt}{\inds}\thewuniX{\indt}{\inds}-(\thedim-2)\thewuniX{\thedim}{\thedim}.\IEEEyesnumber\label{eq:monomial_order_justification_2}
\end{IEEEeqnarray*}
Since \eqref{eq:monomial_order_justification_1} and \eqref{eq:monomial_order_justification_2} are distinct and fixed under reductions by $\therels$ the set $\therels$ cannot be a Gröbner basis. 
\end{proof}
Of course, there may still exist monomial orders which are different from the one in the \hyperref[main-result]{Main result} and which exhibit even more desirable properties. But, if one is interested in quantum groups, an order which makes the set $\therels$ encoding an essential part of what makes $U_n^+$ a quantum \enquote{group} seemed a reasonable choice.
}
\subsubsection{Left modules, bi-modules}
{
  \newcommand{\thegensI}[1]{\thegens_{#1}}
  \newcommand{\therelsI}[1]{\therels_{#1}}
  \newcommand{\anybimod}{X}
  \newcommand{\countind}{k}
\newcommand{\extanybimod}{X''}
\newcommand{\toranybimod}{X'}
\newcommand{\anypoly}{p}
\newcommand{\anyrel}{r}
\newcommand{\anygen}{e}
\newcommand{\anygenI}[1]{e_{#1}}
While the co\-/unit of $\thealgebra=\mathcal{O}(U_\thedim^+)$ is resolved as a right module in the \hyperref[main-result]{Main result}, with the obvious modifications the same method could produce a resolution in terms of left modules. However, a different question is  whether there is also a version for bi\-/modules. While that is not directly addressed by Anick in \cite{Anick1986}, there is of course the option of interpreting bi\-/modules over $\thealgebra$ as, say,  right modules over the enveloping algebra $\thealgebra\SoppositeP\Smonoidalproduct\thealgebra$ of $\thealgebra$. The enveloping algebra can also be expressed as a universal algebra over $\thefield$. Any disjoint union $\thegensI{1}\cup\thegensI{2}$ of two copies $\thegensI{1}$ and $\thegensI{2}$ of $\thegens$ could serve as a generating set. The relations $\bigcup_{\countind=0}^2\therelsI{\countind}$ would consist of one copy $\therelsI{2}=\{\anyrel((\anygenI{2})_{\anygenI{2}\in\thegensI{2}})\Ssetbuilder\anyrel\in\therels\}$ of $\therels\subseteq\thefreealg$ imposed on $\thegensI{2}$, another copy, but reversed, $\therelsI{1}=\{\anyrel\SoppositeP((\anygenI{1})_{\anygenI{1}\in\thegensI{1}})\Ssetbuilder\anyrel\in\therels\}$ of $\therels$ for $\thegensI{1}$, and the demand $\therelsI{0}=\{\anygenI{1}-\anygenI{2}\Ssetbuilder (\anygenI{1},\anygenI{2})\in\thegensI{1}\Ssetmonoidalproduct\thegensI{2}\}$ that $\thegensI{1}$ and $\thegensI{2}$ commute. One may want to try to as monomial order for $\thegensI{1}\cup\thegensI{2}$ the same order as in the \hyperref[main-result]{Main Theorem} on $\thegensI{2}$ and the same or the opposite of that on $\thegensI{1}$ together with the prescription that the maximum of $\thegensI{1}$ is less than the minimum of $\thegensI{2}$. Probably there is a then a way of obtaining a Gröbner basis of $\thealgebra\SoppositeP\Smonoidalproduct\thealgebra$ from that of $\thealgebra$. Perhaps a similar connection exists between the sets of chains of $\thealgebra\SoppositeP\Smonoidalproduct\thealgebra$ and $\thealgebra$.
\par
Be that as it may, depending on the intended application, this might not be the best approach, for example, to computing the Hochschild co\-ho\-mo\-lo\-gy of $\thealgebra$. That is because $\thealgebra$ is not just an algebra but has the structure of a Hopf algebra.
\begin{remark}
  \newcommand{\indi}{i}
  \newcommand{\indj}{j}
  \newcommand{\inds}{s}
  There exists a unique algebra morphism $\Sxfromto{\thecomultiplication}{\thealgebra}{\thealgebra\Smonoidalproduct\thealgebra}$  such that for any $\anymat\in\thematrices$ and $(\indj,\indi)\in\SYnumbers{\thedim}^{\Ssetmonoidalproduct 2}$,
  \begin{IEEEeqnarray*}{rCl}
    \thecomultiplication(\anymatAX{\indj}{\indi})&=&\sum_{\inds=1}^\thedim
    \begin{cases}
\anymatAX{\indj}{\inds}\Smonoidalproduct\anymatAX{\inds}{\indi}      &\Scase\anymat\in\{\theuniA,\theuniB\}\\
\anymatAX{\inds}{\indi}\Smonoidalproduct\anymatAX{\indj}{\inds}        &\Scase\anymat\in\{\theuniAT,\theuniBT\}.
    \end{cases}
  \end{IEEEeqnarray*}  
  $\thecomultiplication$ turns $\thealgebra$ into a Hopf algebra with co\-/unit $\thecounit$ and with antipode $\theantipode$ given by the unique algebra morphism $\Sfromto{\thealgebra}{\thealgebra\SoppositeP}$ with for any $\anymat\in\thematrices$ and $(\indj,\indi)\in\SYnumbers{\thedim}^{\Ssetmonoidalproduct 2}$,
  \begin{IEEEeqnarray*}{rCl}
    \theantipode(\anymatAX{\indj}{\indi})=\anymatBTX{\Sexchanged(\indj)}{\Sexchanged(\indi)}.
  \end{IEEEeqnarray*}
\end{remark}
\begin{proof}
  Well-knonw consequence of \cite{VanDaeleWang1996a}.
\end{proof}
In consequence (see \cite[Proposition~2.1]{Bichon2013}), there are natural isomorphisms between the Hochschild co\-ho\-mo\-lo\-gy of $\thealgebra$ with coefficients in a given bi\-/module $\anybimod$ and $\StorX{\thealgebra}{\themodule}{\toranybimod}{\Sargph}$ or  $\SextX{\thealgebra}{\themodule}{\extanybimod}{\Sargph}$ for the left $\thealgebra$\-/module $\toranybimod$ and the right $\thealgebra$\-/module $\extanybimod$ with actions given by
  \begin{IEEEeqnarray*}{C}
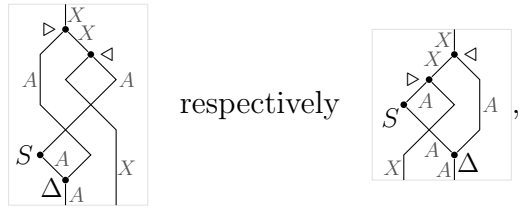
     \begin{tikzpicture}[strdiag={-1.5}{.}{2.}{4.}]
       \coordinate (i1) at ({.*\dx},{.*\dy});
       \coordinate (i2) at ($(i1)+({1.*\dx},{.*\dy})$);
       \node[lmorph={n1}{($(i1)+({.*\dx},{.5*\dy})$)}{\(\thecomultiplication\)}{210}{6pt}];
       \node[lmorph={n2}{($(n1)+({-.5*\dx},{.5*\dy})$)}{\(\theantipode\)}{180}{6pt}];       
       \coordinate (s1) at ($(i2)+({.*\dx},{1.5*\dy})$);
       \coordinate (s3) at ($(s1)+({.*\dx},{.5*\dy})$);
       \coordinate (s2) at ($(s1)+({.*\dx},{1.*\dy})$);       
       \node[lmorph={n3}{($(s1)+({-.5*\dx},{1.5*\dy})$)}{\(\Srightaction\)}{0}{6pt}];
       \node[lmorph={n4}{($(n3)+({-.5*\dx},{.5*\dy})$)}{\(\Sleftaction\)}{180}{6pt}];
       \coordinate (o1) at ($(n4)+({.*\dx},{.5*\dy})$);
       \draw (i1) to[lobj={\(\thealgebra\)}{.5}{0}{6pt}] (n1);
       \draw (n1) to[lobj={\(\thealgebra\)}{.5}{45}{7pt}] (n2);
       \draw (n1) -- (n3|-n2) -- (n2|-s3)  to[lobj={\(\thealgebra\)}{.5}{180}{6pt}] (n2|-n3) -- (n4);
       \draw (n2)  to[lobj={\(\thealgebra\)}{1.}{0}{6pt}] (i2|-s2) -- (n3);
       \draw (i2) to[lobj={\(\anybimod\)}{.5}{0}{6pt}] (i2|-s1) -- (o1|-s2) -- (n3);
       \draw (n3)  to[lobj={\(\anybimod\)}{.5}{45}{7pt}] (n4);
       \draw (n4) to[lobj={\(\anybimod\)}{.5}{0}{6pt}] (o1);
     \end{tikzpicture}
     \hspace{1em}\text{respectively}\hspace{1em}
          \begin{tikzpicture}[strdiag={-1.}{.}{2.5}{3.}]
       \coordinate (i1) at ({.*\dx},{.*\dy});
       \coordinate (i2) at ($(i1)+({1.*\dx},{.*\dy})$);
       \node[lmorph={n1}{($(i2)+({.*\dx},{.5*\dy})$)}{\(\thecomultiplication\)}{330}{6pt}];
       \node[lmorph={n2}{($(i1)+({.*\dx},{1.5*\dy})$)}{\(\theantipode\)}{225}{7pt}];
       \coordinate (s0) at ($(i1)+({.*\dx},{.5*\dy})$);
       \coordinate (s1) at ($(n1)+({.5*\dx},{.5*\dy})$);
       \node[lmorph={n3}{($(n2)+({.5*\dx},{.5*\dy})$)}{\(\Sleftaction\)}{180}{6pt}];
       \coordinate (s2) at ($(n2)+({1.*\dx},{.*\dy})$);
       \node[lmorph={n4}{($(n3)+({.5*\dx},{.5*\dy})$)}{\(\Srightaction\)}{0}{6pt}];
       \coordinate (o1) at ($(n4)+({.*\dx},{.5*\dy})$);
       \draw (i1) to[lobj={\(\anybimod\)}{.5}{180}{6pt}] (s0) -- (s2) -- (n3);
       \draw (i2) to[lobj={\(\thealgebra\)}{.5}{180}{6pt}] (n1);
       \draw (n1) to[lobj={\(\thealgebra\)}{.25}{225}{7pt}] (n2);
       \draw (n1) -- (s1) to[lobj={\(\thealgebra\)}{.5}{0}{6pt}] (s1|-n3) -- (n4);
       \draw (n2) to[lobj={\(\thealgebra\)}{.5}{315}{7pt}]  (n3);
       \draw (n3) to[lobj={\(\anybimod\)}{.5}{135}{7pt}]  (n4);
       \draw (n4) to[lobj={\(\anybimod\)}{.5}{180}{6pt}] (o1);
     \end{tikzpicture},
   \end{IEEEeqnarray*}
where $\Sleftaction$ and $\Srightaction$ are the left respectively right $\thealgebra$\-/actions on $\anybimod$.
}
\subsubsection{Yetter-Drinfeld modules}
{
  \newcommand{\anycomod}{V}
  \newcommand{\Srightcoaction}{\gamma}
  \newcommand{\orderind}{\ell}
  \newcommand{\anyspace}{T}
  \newcommand{\inputcoaction}{\kappa}
  \newcommand{\domcoaction}{\kappa}
  \newcommand{\codcoaction}{\kappa'}
  \newcommand{\domspace}{T}
  \newcommand{\codspace}{T'}
  \newcommand{\anymorph}{d}
  \newcommand{\indi}{i}
  \newcommand{\indj}{j}
  \newcommand{\inds}{s}
  \newcommand{\indt}{t}
  The chain complex constructed by Ba\-ra\-quin, Franz, Gerhold, Kula and Skalski is \cite{BaraquinFranzGerholdKulaTobolski2023} resolves $\themodule$ not just as an $\mathcal{O}(U_\thedim^+)$\-/module but in fact as a Yetter-Drinfeld module with respect to the co-action given by the unit $\theone$ of $\mathcal{O}(U_\thedim^+)$. It is only natural to ask whether the modules $(\thechainsmodule{\orderind})_{\orderind\in\Sintegersnn}$ from the \hyperref[main-result]{Main result} can be equipped with co-actions which also make $(\thedifferential{\orderind})_{\orderind\in\Sintegersnn}$ a resolution in co\-/modules.
  \par
  A  vector space  $\anycomod$ together with a right action $\Srightaction$ and a right co\-/action $\Srightcoaction$ of a Hopf algebra $\thealgebra$ with multiplication $\themultiplication$, co\-/multiplication $\thecomultiplication$ and antipode $\theantipode$ is a right \emph{Yetter-Drinfeld module} if
    \begin{IEEEeqnarray*}{rCl}
     \begin{tikzpicture}[strdiag={-1.}{.}{2.}{1.75}]
       \coordinate (i1) at ({.*\dx},{.*\dy});
       \coordinate (i2) at ($(i1)+({1.*\dx},{.*\dy})$);
       \node[lmorph={n1}{($(i1)+({.5*\dx},{.5*\dy})$)}{\(\Srightaction\)}{20}{6pt}];
       \node[lmorph={n2}{($(n1)+({.*\dx},{.75*\dy})$)}{\(\Srightcoaction\)}{200}{6pt}];       
       \coordinate (o1) at ($(n2)+({-.5*\dx},{.5*\dy})$);
       \coordinate (o2) at (i2|-o1);
       \draw (i1) to[lobj={\(\anycomod\)}{.25}{135}{7pt}] (n1);
       \draw (i2) to[lobj={\(\thealgebra\)}{.25}{45}{7pt}] (n1);
       \draw (n1) to[lobj={\(\anycomod\)}{.666}{0}{6pt}] (n2);
       \draw (n2) to[lobj={\(\anycomod\)}{.75}{225}{7pt}] (o1);
       \draw (n2) to[lobj={\(\thealgebra\)}{.75}{315}{7pt}] (o2);
     \end{tikzpicture}
     &=&
     \begin{tikzpicture}[strdiag={-1.5}{.}{3.}{3.}]
       \coordinate (i1) at ({.*\dx},{.*\dy});
       \coordinate (i2) at ($(i1)+({1.*\dx},{.*\dy})$);
       \node[lmorph={n1}{($(i1)+({.*\dx},{.5*\dy})$)}{\(\Srightcoaction\)}{210}{7pt}];
       \node[lmorph={n2}{(i2|-n1)}{\(\thecomultiplication\)}{330}{7pt}];
       \coordinate (s1) at ($(n1)+({-.5*\dx},{.5*\dy})$);
       \node[lmorph={n3}{($(n2|-s1)+({.5*\dx},{.*\dy})$)}{\(\thecomultiplication\)}{350}{7pt}];
       \coordinate (s2) at ($(n3)+({.5*\dx},{.5*\dy})$);
       \node[lmorph={n4}{(i1|-s2)}{\(\theantipode\)}{240}{7pt}];       
       \node[lmorph={n5}{($(n3|-s2)+({.*\dx},{.5*\dy})$)}{\(\themultiplication\)}{10}{7pt}];
       \node[lmorph={n6}{($(i1|-n5)+({.*\dx},{.5*\dy})$)}{\(\Srightaction\)}{150}{7pt}];
       \node[lmorph={n7}{(i2|-n6)}{\(\themultiplication\)}{30}{7pt}];       
       \coordinate (o1) at ($(i1|-n7)+({.*\dx},{.5*\dy})$);
       \coordinate (o2) at (i2|-o1);
       \draw (i1) to[lobj={\(\anycomod\)}{.5}{0}{6pt}] (n1);
       \draw (i2) to[lobj={\(\thealgebra\)}{.5}{180}{6pt}] (n2);
       \draw (n1) -- (s1) to[lobj={\(\anycomod\)}{.5}{180}{6pt}] (s1|-n5) -- (n6);
       \draw (n1) to[lobj={\(\thealgebra\)}{.875}{135}{7pt}] (n5);
       \draw (n2) to[lobj={\(\thealgebra\)}{.25}{225}{7pt}] (n4);
       \draw (n2) to[lobj={\(\thealgebra\)}{.75}{315}{7pt}] (n3);
       \draw (n3) to[lobj={\(\thealgebra\)}{.125}{225}{7pt}] (n6);
       \draw (n3) to[lobj={\(\thealgebra\)}{1.}{0}{6pt}] (s2) -- (n5);
       \draw (n4) to[lobj={\(\thealgebra\)}{.75}{135}{7pt}] (n7);
       \draw (n5) to[lobj={\(\thealgebra\)}{.25}{45}{7pt}] (n7);
       \draw (n6) to[lobj={\(\anycomod\)}{.5}{0}{6pt}] (o1);
       \draw (n7) to[lobj={\(\thealgebra\)}{.5}{180}{6pt}] (o2);       
     \end{tikzpicture}.
   \end{IEEEeqnarray*}
   In particular, the Hopf algebra $\thealgebra$ itself becomes the \emph{regular Yetter-Drinfeld module} if equipped with the co-action
   \begin{IEEEeqnarray*}{C}
          \begin{tikzpicture}[strdiag={-1.5}{.}{2.}{2.5}]
       \coordinate (i1) at ({.*\dx},{.*\dy});
       \node[lmorph={n1}{($(i1)+({.*\dx},{.5*\dy})$)}{\(\thecomultiplication\)}{330}{7pt}];
       \node[lmorph={n2}{($(n1)+({-.5*\dx},{.5*\dy})$)}{\(\theantipode\)}{240}{7pt}];
       \node[lmorph={n3}{($(n1|-n2)+({.5*\dx},{.*\dy})$)}{\(\thecomultiplication\)}{340}{7pt}];
       \coordinate (s1) at ($(n3)+({.5*\dx},{.5*\dy})$);
       \node[lmorph={n4}{($(n3)+({.*\dx},{1.*\dy})$)}{\(\themultiplication\)}{20}{7pt}];              
       \coordinate (o1) at ($(n2)+({.*\dx},{1.5*\dy})$);
       \coordinate (o2) at (n4|-o1);       
       \draw (i1) to[lobj={\(\thealgebra\)}{.5}{180}{6pt}] (n1);
       \draw (n1) to[lobj={\(\thealgebra\)}{.5}{45}{7pt}] (n2);
       \draw (n1) to[lobj={\(\thealgebra\)}{.75}{315}{7pt}] (n3);
       \draw (n3) to[lobj={\(\thealgebra\)}{1.}{0}{6pt}] (s1) --  (n4);
       \draw (n3) -- (n2|-n4) to[lobj={\(\thealgebra\)}{.5}{0}{6pt}] (o1);
       \draw (n4) to[lobj={\(\thealgebra\)}{.5}{180}{6pt}] (o2);
       \draw (n2) to[lobj={\(\thealgebra\)}{.75}{315}{7pt}] (n4);
     \end{tikzpicture}.\IEEEyesnumber\label{eq:regular_yetter_drinfeld_module}
   \end{IEEEeqnarray*}
   More generally, given any right co\-/module $\anyspace$ over $\thealgebra$ with co\-/action $\inputcoaction$,  the free right $\thealgebra$-module $\anyspace\Smonoidalproduct\thealgebra$ over $\anyspace$ can be turned into the \emph{free Yetter\-/Drinfeld module} over $(\anyspace,\inputcoaction)$ by means of the co\-/action
   \begin{IEEEeqnarray*}{C}
          \begin{tikzpicture}[strdiag={-1.5}{.}{3.}{3.}]
       \coordinate (i1) at ({.*\dx},{.*\dy});
       \coordinate (i2) at ($(i1)+({1.*\dx},{.*\dy})$);
       \node[lmorph={n1}{($(i1)+({.*\dx},{.5*\dy})$)}{\(\inputcoaction\)}{210}{7pt}];
       \node[lmorph={n2}{(i2|-n1)}{\(\thecomultiplication\)}{330}{7pt}];
       \coordinate (s1) at ($(n1)+({-.5*\dx},{.5*\dy})$);
       \node[lmorph={n3}{($(n2|-s1)+({.5*\dx},{.*\dy})$)}{\(\thecomultiplication\)}{350}{7pt}];
       \coordinate (s2) at ($(n3)+({.5*\dx},{.5*\dy})$);
       \node[lmorph={n4}{(i1|-s2)}{\(\theantipode\)}{240}{7pt}];       
       \node[lmorph={n5}{($(n3|-s2)+({.*\dx},{.5*\dy})$)}{\(\themultiplication\)}{10}{7pt}];
       \coordinate (s3) at ($(i1|-n5)+({.*\dx},{.5*\dy})$);
       \node[lmorph={n7}{(i2|-s3)}{\(\themultiplication\)}{30}{7pt}];       
       \coordinate (o1) at ($(i1|-n7)+({.*\dx},{.5*\dy})$);       
       \coordinate (o2) at (i2|-o1);
       \coordinate (o0) at (s1|-o1);       
       \draw (i1) to[lobj={\(\anyspace\)}{.5}{0}{6pt}] (n1);
       \draw (i2) to[lobj={\(\thealgebra\)}{.5}{180}{6pt}] (n2);
       \draw (n1) -- (s1) to[lobj={\(\anyspace\)}{.5}{180}{6pt}] (o0);
       \draw (n1) to[lobj={\(\thealgebra\)}{.875}{135}{7pt}] (n5);
       \draw (n2) to[lobj={\(\thealgebra\)}{.25}{225}{7pt}] (n4);
       \draw (n2) to[lobj={\(\thealgebra\)}{.75}{315}{7pt}] (n3);
       \draw (n3) -- (s3) to[lobj={\(\thealgebra\)}{.5}{0}{6pt}] (o1);
       \draw (n3) to[lobj={\(\thealgebra\)}{1.}{0}{6pt}] (s2) -- (n5);
       \draw (n4) to[lobj={\(\thealgebra\)}{.75}{135}{7pt}] (n7);
       \draw (n5) to[lobj={\(\thealgebra\)}{.25}{45}{7pt}] (n7);
       \draw (n7) to[lobj={\(\thealgebra\)}{.5}{180}{6pt}] (o2);       
     \end{tikzpicture}.\IEEEyesnumber\label{eq:free_yetter_drinfeld_module}
   \end{IEEEeqnarray*}
   With respect to this co\-/action on $\anyspace\Smonoidalproduct\thealgebra$ any morphism $\Sxfromto{\anymorph}{\anyspace\Smonoidalproduct\thealgebra}{\anycomod}$ of right $\thealgebra$\-/modules into any  right Yetter-Drinfeld module $\anycomod$ with co\-/action $\Srightcoaction$ is already a morphism of co\-/modules if
     \begin{IEEEeqnarray*}{C}
     \begin{tikzpicture}[strdiag={-2.}{.}{1.5}{2.5}]
       \coordinate (i1) at ({.*\dx},{.*\dy});
       \node[lmorph={n1}{($(i1)+({.*\dx},{.5*\dy})$)}{\(\inputcoaction\)}{200}{7pt}];
       \coordinate (s1) at ($(n1)+({.5*\dx},{.5*\dy})$);
       \coordinate (s2) at ($(n1)+({-1.*\dx},{1.*\dy})$);
       \node[lmorph={n2}{(n1|-s2)}{\(\theone\)}{270}{7pt}];
       \node[lmorph={n3}{($(s2)+({.5*\dx},{.5*\dy})$)}{\(\anymorph\)}{160}{7pt}];       
       \coordinate (o1) at ($(n3)+({.*\dx},{.5*\dy})$);
       \coordinate (o2) at (s1|-o1);       
       \draw (i1) to[lobj={\(\anyspace\)}{.5}{0}{6pt}] (n1);
       \draw (n1) -- (s1) to[lobj={\(\thealgebra\)}{.5}{0}{6pt}] (o2) ;
       \draw (n1) to[lobj={\(\anyspace\)}{.75}{225}{7pt}] (s2) -- (n3);
       \draw (n2) to[lobj={\(\thealgebra\)}{.5}{45}{7pt}]  (n3);
       \draw (n3) to[lobj={\(\anycomod\)}{.5}{0}{6pt}]  (o1);
     \end{tikzpicture}
     =
     \begin{tikzpicture}[strdiag={-1.}{.}{2.}{2.25}]
       \coordinate (i1) at ({.*\dx},{.*\dy});
       \coordinate (s1) at ($(i1)+({.*\dx},{.5*\dy})$);
       \node[lmorph={n1}{($(s1)+({1.*\dx},{.*\dy})$)}{\(\theone\)}{340}{7pt}];
       \node[lmorph={n2}{($(n1)+({-.5*\dx},{.5*\dy})$)}{\(\anymorph\)}{160}{6pt}];
       \node[lmorph={n3}{($(n2)+({.*\dx},{.75*\dy})$)}{\(\Srightcoaction\)}{200}{6pt}];       
       \coordinate (o1) at ($(n3)+({-.5*\dx},{.5*\dy})$);
       \coordinate (o2) at ($(n3|-o1)+({.5*\dx},{.*\dy})$);
       \draw (i1) to[lobj={\(\anyspace\)}{.5}{0}{6pt}] (s1) -- (n2);
       \draw (n1) to[lobj={\(\thealgebra\)}{.5}{45}{7pt}] (n2);
       \draw (n2) to[lobj={\(\anycomod\)}{.5}{0}{6pt}] (n3);
       \draw (n3) to[lobj={\(\anycomod\)}{.75}{225}{7pt}] (o1);
       \draw (n3) to[lobj={\(\thealgebra\)}{.75}{315}{7pt}] (o2);
     \end{tikzpicture}.\IEEEyesnumber\label{eq:co-module_morphism_out_of_free_yetter_drinfeld_module}
   \end{IEEEeqnarray*}
   In particular, the co\-/unit  $\thecounit$ of the Hopf algebra $\thealgebra$ is a morphism of co\-/modules with respect to the Yetter-Drinfeld co\-/action \eqref{eq:regular_yetter_drinfeld_module} on $\thealgebra$. Beyond that the above proves the following about the objects of the \hyperref[main-result]{Main result}. 
   \par
   \begin{remark}
     For any family $(\thechainscoaction{\orderind})_{\orderind\in\Sintegersp}$  such that for each $\orderind\in\Sintegersp$  the mapping $\thechainscoaction{\orderind}$ is a right co\-/action of $\thealgebra$ on  $\thefield\thechains{\orderind}$, such that
  \begin{IEEEeqnarray*}{rCl}
     \begin{tikzpicture}[strdiag={-2.}{.}{1.5}{2.5}]
       \coordinate (i1) at ({.*\dx},{.*\dy});
       \node[lmorph={n1}{($(i1)+({.*\dx},{.5*\dy})$)}{\(\thechainscoaction{1}\)}{340}{8pt}];
       \coordinate (s1) at ($(n1)+({.5*\dx},{.5*\dy})$);
       \coordinate (s2) at ($(n1)+({-1.*\dx},{1.*\dy})$);
       \node[lmorph={n2}{(n1|-s2)}{\(\theone\)}{270}{7pt}];
       \node[lmorph={n3}{($(n2)+({-.5*\dx},{.5*\dy})$)}{\(\thedifferential{1}\)}{160}{8pt}];
       \coordinate (o1) at ($(n3)+({.*\dx},{.5*\dy})$);
       \coordinate (o2) at (s1|-o1);
       \draw (i1) to[lobj={\(\thefield\thechains{1}\)}{.5}{180}{13pt}] (n1);
       \draw (n1) to[lobj={\(\thefield\thechains{1}\)}{.75}{225}{13pt}] (s2) -- (n3);
       \draw (n1) -- (s1) to[lobj={\(\thealgebra\)}{.5}{0}{6pt}] (o2);
       \draw (n2) to[lobj={\(\thealgebra\)}{.5}{45}{7pt}] (n3);
       \draw (n3) to[lobj={\(\thealgebra\)}{.5}{0}{6pt}] (o1);
     \end{tikzpicture}
     &=&
     \begin{tikzpicture}[strdiag={-1.}{.}{2.5}{3.5}]
       \coordinate (i1) at ({.*\dx},{.*\dy});
       \node[lmorph={n1}{($(i1)+({1.*\dx},{.5*\dy})$)}{\(\theone\)}{340}{6pt}];
       \node[lmorph={n2}{($(n1)+({-.5*\dx},{.5*\dy})$)}{\(\thedifferential{1}\)}{180}{10pt}];
       \node[lmorph={n3}{($(n2)+({.*\dx},{.5*\dy})$)}{\(\thecomultiplication\)}{330}{6pt}];
       \node[lmorph={n4}{($(n3)+({-.5*\dx},{.5*\dy})$)}{\(\theantipode\)}{180}{6pt}];
       \node[lmorph={n5}{($(n3)+({.5*\dx},{.5*\dy})$)}{\(\thecomultiplication\)}{340}{7pt}];
       \coordinate (s1) at ($(n5)+({.5*\dx},{.5*\dy})$);
       \node[lmorph={n6}{($(s1)+({-.5*\dx},{.5*\dy})$)}{\(\themultiplication\)}{20}{6pt}];       
       \coordinate (o1) at ($(i1|-n6)+({.*\dx},{.5*\dy})$);
       \coordinate (o2) at (n6|-o1);       
       \draw (i1) to[lobj={\(\thefield\thechains{1}\)}{.5}{0}{13pt}] (i1|-n1) -- (n2);
       \draw (n1) to[lobj={\(\thealgebra\)}{.25}{45}{7pt}] (n2);
       \draw (n2) to[lobj={\(\thealgebra\)}{.5}{180}{6pt}] (n3);
       \draw (n3) to[lobj={\(\thealgebra\)}{.5}{45}{7pt}] (n4);
       \draw (n3) to[lobj={\(\thealgebra\)}{.75}{315}{7pt}] (n5);
       \draw (n4) to[lobj={\(\thealgebra\)}{.75}{315}{7pt}] (n6);
       \draw (n5) -- (i1|-n6) to[lobj={\(\thealgebra\)}{.5}{0}{6pt}] (o1);
       \draw (n5) to[lobj={\(\thealgebra\)}{1.}{0}{6pt}] (s1) -- (n6);
       \draw (n6) to[lobj={\(\thealgebra\)}{.5}{180}{6pt}] (o2);
     \end{tikzpicture}\IEEEyesnumber\label{eq:remark_yetter_drinfeld_01}
   \end{IEEEeqnarray*}
     and such that for any $\orderind\in\Sintegersp$ with $2\leq \orderind$  
  \begin{IEEEeqnarray*}{rCl}
     \begin{tikzpicture}[strdiag={-2.5}{.}{1.5}{2.5}]
       \coordinate (i1) at ({.*\dx},{.*\dy});
       \node[lmorph={n1}{($(i1)+({.*\dx},{.5*\dy})$)}{\(\thechainscoaction{\orderind}\)}{340}{8pt}];
       \coordinate (s1) at ($(n1)+({.5*\dx},{.5*\dy})$);
       \coordinate (s2) at ($(n1)+({-1.*\dx},{1.*\dy})$);
       \node[lmorph={n2}{(n1|-s2)}{\(\theone\)}{270}{7pt}];
       \node[lmorph={n3}{($(n2)+({-.5*\dx},{.5*\dy})$)}{\(\thedifferential{\orderind}\)}{0}{8pt}];
       \coordinate (o1) at ($(s2|-n3)+({.*\dx},{.5*\dy})$);
       \coordinate (o2) at (n2|-o1);
       \coordinate (o3) at (s1|-o1);
       \draw (i1) to[lobj={\(\thefield\thechains{\orderind}\)}{.5}{180}{13pt}] (n1);
       \draw (n1) to[lobj={\(\thefield\thechains{\orderind}\)}{.75}{225}{13pt}] (s2) -- (n3);
       \draw (n1) -- (s1) to[lobj={\(\thealgebra\)}{.5}{0}{6pt}] (o3);
       \draw (n2) to[lobj={\(\thealgebra\)}{.5}{225}{7pt}] (n3);
       \draw (n3) to[lobj={\(\thefield\thechains{\orderind-1}\)}{1.}{225}{13pt}] (o1);
       \draw (n3) to[lobj={\(\thealgebra\)}{1.}{315}{7pt}] (o2);
     \end{tikzpicture}
     &=&
     \begin{tikzpicture}[strdiag={-2.25}{.}{3.}{4.}]
       \coordinate (i1) at ({.*\dx},{.*\dy});
       \node[lmorph={n1}{($(i1)+({1.*\dx},{.5*\dy})$)}{\(\theone\)}{320}{6pt}];
       \node[lmorph={n2}{($(n1)+({-.5*\dx},{.5*\dy})$)}{\(\thedifferential{\orderind}\)}{0}{9pt}];
       \node[lmorph={n3}{($(n2)+({-.5*\dx},{.5*\dy})$)}{\(\thechainscoaction{\orderind-1}\)}{180}{12pt}];
       \node[lmorph={n4}{($(n2)+({.5*\dx},{.5*\dy})$)}{\(\thecomultiplication\)}{330}{7pt}];
       \coordinate (s1) at ($(n3)+({-.5*\dx},{.5*\dy})$);
       \node[lmorph={n5}{($(n4)+({.5*\dx},{.5*\dy})$)}{\(\thecomultiplication\)}{330}{7pt}];
       \coordinate (s2) at ($(n5)+({.5*\dx},{.5*\dy})$);
       \node[lmorph={n6}{($(s1)+({.5*\dx},{.5*\dy})$)}{\(\theantipode\)}{120}{7pt}];
       \node[lmorph={n7}{($(s2)+({-.5*\dx},{.5*\dy})$)}{\(\themultiplication\)}{20}{6pt}];
       \node[lmorph={n8}{($(n7)+({-.5*\dx},{.5*\dy})$)}{\(\themultiplication\)}{30}{6pt}];       
       \coordinate (o1) at ($(s1|-n8)+({.*\dx},{.5*\dy})$);
       \coordinate (o2) at (n6|-o1);
       \coordinate (o3) at (n8|-o1);       
       \draw (i1) to[lobj={\(\thefield\thechains{\orderind}\)}{.5}{180}{13pt}] (i1|-n1) -- (n2);
       \draw (n1) to[lobj={\(\thealgebra\)}{.5}{225}{7pt}] (n2);
       \draw (n2) to[lobj={\(\thefield\thechains{\orderind-1}\)}{.5}{200}{17pt}] (n3);
       \draw (n2) to[lobj={\(\thealgebra\)}{.5}{135}{7pt}] (n4);
       \draw (n4) to[lobj={\(\thealgebra\)}{.625}{315}{7pt}] (n5);
       \draw (n4) to[lobj={\(\thealgebra\)}{.75}{45}{7pt}] (n6);
       \draw (n6) to[lobj={\(\thealgebra\)}{.75}{315}{7pt}] (n8);
       \draw (n7) to[lobj={\(\thealgebra\)}{.375}{45}{7pt}] (n8);
       \draw (n8) to[lobj={\(\thealgebra\)}{.5}{180}{6pt}] (o3);
       \draw (n3) -- (s1) to[lobj={\(\thefield\thechains{\orderind-1}\)}{.5}{180}{17pt}] (o1);
       \draw (n3) to[lobj={\(\thealgebra\)}{.16161616}{135}{7pt}] (n7);
       \draw (n5) -- (i1|-n8) to[lobj={\(\thealgebra\)}{.5}{0}{6pt}] (o2);
       \draw (n5) to[lobj={\(\thealgebra\)}{1.}{0}{6pt}] (s2) -- (n7);       
     \end{tikzpicture}\IEEEyesnumber\label{eq:remark_yetter_drinfeld_02}
   \end{IEEEeqnarray*}
     the family $(\thedifferential{\orderind})_{\orderind\in\Sintegersnn}$ becomes a resolution in co\-/modules if  $\thechainsmodule{0}$ is considered the regular Yetter\-/Drinfeld module and $\thechainsmodule{\orderind}$ the free Yetter\-/Drinfeld module over $(\thefield\thechains{\orderind},\thechainscoaction{\orderind})$ for any $\orderind\in\Sintegersp$.
   \end{remark}
   \begin{proof}
     As mentioned above, $\thedifferential{0}=\thecounit$ is a co\-/module morphism  
since $\thechainsmodule{0}$ has the co\-/action \eqref{eq:regular_yetter_drinfeld_module}. For the same reason     condition \eqref{eq:remark_yetter_drinfeld_01} is precisely what \eqref{eq:co-module_morphism_out_of_free_yetter_drinfeld_module} means for $\thedifferential{1}$. Likewise, \eqref{eq:remark_yetter_drinfeld_02} spells out \eqref{eq:co-module_morphism_out_of_free_yetter_drinfeld_module} for $\thedifferential{\orderind}$ for $\orderind\in\Sintegersp$ with $2\leq \orderind$ because both $\thechainsmodule{\orderind}$ and $\thechainsmodule{\orderind-1}$ are equipped with the co\-/action \eqref{eq:free_yetter_drinfeld_module}.
\end{proof}
Perhaps a recursive algorithm can be developed for computing such a family $(\thechainscoaction{\orderind})_{\orderind\in\Sintegersp}$ like for $(\thedifferential{\orderind})_{\orderind\in\Sintegersnn}$ itself.
}
  \clearpage  
\appendix
\pagenumbering{roman}
\section{Relationship to the known resolution}

\label{section:relationship}
  \newcommand{\Tzerobv}[1]{z^{#1}}
  \newcommand{\Tonefirstbv}[3]{b^{#1}_{#2,#3}}
  \newcommand{\Tonesecondbv}{y}
  \newcommand{\Ttwobv}[3]{a^{#1}_{#2,#3}}
  \newcommand{\Tthreebv}[1]{x^{#1}}
  \newcommand{\Tbvs}[1]{C'_{#1}}
  \newcommand{\Tmod}[1]{P'_{#1}}
  \newcommand{\Tdiff}[1]{d'_{#1}}
{
  \newcommand{\anypoly}{p}
  This section answers the question of how the resolution from the \hyperref[main-result]{Main result} relates to the one previously found by Baraquin, Franz, Gerhold, Kula, Tobolski in \cite{BaraquinFranzGerholdKulaTobolski2023}. I.e., an explicit quasi-isomorphism is given.
As before, for any $\anypoly\in\thefreealg$ the element $\anypoly+\theideal$ of $\thealgebra$ will be referred to as $\anypoly$. And this convention is extended to matrices of elements of  $\thealgebra$.
  }
On p.~20 of \cite{BaraquinFranzGerholdKulaTobolski2023} the following  result about the universal co-sovereign Hopf algebra $\mathcal{H}(F)$ can be found verbatim.
\begin{quotation}
\enquote{\textbf{Theorem~5.2.} \textit{Let $F=E^{t}E^{-1}\in \mathrm{GL}_n(\mathbb{C})$ be a generic asymmetry and $\mathcal{H}=\mathcal{H}(F)$. Then the following sequence
  \begin{IEEEeqnarray*}{C}
    \begin{tikzpicture}
      \node[inner sep=2.5pt] (p0)  at (0,0) {$0$};
      \node[inner sep=2.5pt] (p1) [left = .25cm of p0] {$\mathbb{C}_\varepsilon$};
      \node[inner sep=1.5pt] (p2) [left = .5cm of p1] {$
        \begin{pmatrix}
          \mathcal{H}\\
          \mathcal{H}
        \end{pmatrix}
        $};
      \node[inner sep=1.5pt] (p3) [left = .5cm of p2] {$
        \begin{pmatrix}
          M_n(\mathcal{H})\\
          M_n(\mathcal{H})\\
          \mathcal{H}
        \end{pmatrix}
        $};
      \node[inner sep=1.5pt] (p4) [left = .5cm of p3] {$
        \begin{pmatrix}
          M_n(\mathcal{H})\\
          M_n(\mathcal{H})
        \end{pmatrix}$};
      \node[inner sep=1.5pt] (p5) [left = .5cm of p4] {$
        \begin{pmatrix}
          \mathcal{H}\\
          \mathcal{H}
        \end{pmatrix}$};      
      \node[inner sep=2.5pt] (p6) [left = .25cm of p5] {$0$};
      \draw[->] (p6) to (p5);
      \draw[->] (p5) to node [above] {$\Phi_3^{\mathcal{H}}$} (p4);
      \draw[->] (p4) to node [above] {$\Phi_2^{\mathcal{H}}$} (p3);
      \draw[->] (p3) to node [above] {$\Phi_1^{\mathcal{H}}$} (p2);
      \draw[->] (p2) to node [above] {$\varepsilon$} (p1);
      \draw[->] (p1) to (p0);
    \end{tikzpicture}
  \end{IEEEeqnarray*}
  with the mappings
  \begin{IEEEeqnarray*}{rCl}
    \Phi_3^{\mathcal{H}}
    \begin{pmatrix}
      a\\
      b
    \end{pmatrix}
    &=&
    \begin{pmatrix}
      -Fa+(EuE^{-t})b\\
      F^{-1}va-Fb
    \end{pmatrix}
    ,
    \\
    \Phi_2^{\mathcal{H}}
    \begin{pmatrix}
      A\\
      B
    \end{pmatrix}
    &=&
    \begin{pmatrix}
      A+(E^{-1}u^{t}BE)^{t}\\
      B+(v^{t}E^{-1}AE)^{t}\\
      0
    \end{pmatrix}
    ,
    \\
    \Phi_1^{\mathcal{H}}
    \begin{pmatrix}
      A\\
      B\\
      c
    \end{pmatrix}
    &=&
    \begin{pmatrix}
      \mathrm{tr}(-A+u^{t}B)+c\\
      \mathrm{tr}(Ev^{t}E^{-1}A-B)-c
    \end{pmatrix}
  \end{IEEEeqnarray*}
  for $a, b, c\in\mathcal{H}$, $A,B\in M_n(\mathcal{H})$, yields a projective resolution of the counit of $\mathcal{H}$.}}
\end{quotation}
{
  \newcommand{\indi}{i}
  \newcommand{\indj}{j}
  \newcommand{\indt}{t}
  \newcommand{\inds}{s}
  \newcommand{\indk}{k}
  \newcommand{\indp}{p}  
  \newcommand{\indl}{\ell}
  \newcommand{\orderind}{\ell}
  \newcommand{\elmatrix}[2]{m^{#1,#2}}
  \newcommand{\elmatrixX}[4]{m^{#1,#2}_{#3,#4}}
  \newcommand{\anypoly}{p}
  \newcommand{\anycolor}{\mathfrak{c}}  
  \newcommand{\Tiso}{\Psi}
  \newcommand{\TisoX}[1]{\Tiso_{#1}}
  \newcommand{\anymatX}[2]{v_{#1,#2}}
Expressed in a way closer to the language of the present article a special case of \cite[Theorem~5.2]{BaraquinFranzGerholdKulaTobolski2023}   might read as follows.
  \begin{proposition}
  If $\thealgebra$, $\thecounit$ and $\themodule$  are as in the \hyperref[main-result]{Main result}, if  $\SdualX{\Swhite}\Seqpd \Sblack$ and $\SdualX{\Sblack}\Seqpd \Swhite$, if
\begin{IEEEeqnarray*}{rCl}
  \Tbvs{0}&=&\{\Tzerobv{\Swhite},\Tzerobv{\Sblack}\}, \\ \Tbvs{1}&=&\{\Tonefirstbv{\Swhite}{\indj}{\indi},\Tonefirstbv{\Sblack}{\indj}{\indi},\Tonesecondbv\}_{\indi,\indj=1}^\thedim,\\
  \Tbvs{2}&=&\{\Ttwobv{\Swhite}{\indj}{\indi},\Ttwobv{\Sblack}{\indj}{\indi}\}_{\indi,\indj=1}^\thedim,\\ \Tbvs{3}&=&\{\Tthreebv{\Swhite},\Tthreebv{\Sblack}\}
  \end{IEEEeqnarray*}
   are such that $|\Tbvs{0}|=|\Tbvs{3}|=2$ and $|\Tbvs{1}|=2\thedim^2+1$ and $|\Tbvs{2}|=2\thedim^2$, if for any $\orderind\in\Sintegers$ with $-1\leq \orderind$,
  \begin{IEEEeqnarray*}{rCl}
    \Tmod{\orderind}&\Seqpd&
    \begin{cases}
      \themodule&\Scase \orderind=-1\\
      \thefield\Tbvs{\orderind}\Smonoidalproduct\thealgebra&\Scase 0\leq \orderind\leq 3\\
      0&\Scase 4\leq \orderind
    \end{cases},
  \end{IEEEeqnarray*}
  if  $(\Tdiff{\orderind})_{\orderind\in\Sintegersnn}$ is such that for each $\orderind\in\Sintegersnn$ the mapping $\Tdiff{\orderind}$ is a morphism $\Sfromto{\Tmod{\orderind}}{\Tmod{\orderind-1}}$ of right $\thealgebra$-modules and such that for any $(\indj,\indi)\in\SYnumbers{\thedim}^{\Ssetmonoidalproduct 2}$ and any $\anycolor\in\Scolors$,
  \begin{IEEEeqnarray*}{rCl}
    \Tdiff{0}(\Tzerobv{\anycolor}\Smonoidalproduct\theone)&=&1,\\ \Tdiff{1}(\Tonefirstbv{\anycolor}{\indj}{\indi}\Smonoidalproduct\theone)&=&\Tzerobv{\SdualX{\anycolor}}\Smonoidalproduct\thecuniX{\SdualX{\anycolor}}{\indj}{\indi}-\Skronecker{\indj}{\indi}\Saction\Tzerobv{\anycolor}\Smonoidalproduct\theone,\\
    \Tdiff{1}(\Tonesecondbv\Smonoidalproduct\theone)&=&(\Tzerobv{\Swhite}-\Tzerobv{\Sblack})\Smonoidalproduct\theone,\\
    \Tdiff{2}(\Ttwobv{\anycolor}{\indj}{\indi}\Smonoidalproduct\theone)&=&\textstyle\sum_{\indt=1}^\thedim\Tonefirstbv{\SdualX{\anycolor}}{\indi}{\indt}\Smonoidalproduct\thecuniX{\SdualX{\anycolor}}{\indj}{\indt}+\Tonefirstbv{\anycolor}{\indj}{\indi}\Smonoidalproduct\theone,\\
    \Tdiff{3}(\Tthreebv{\anycolor}\Smonoidalproduct\theone)&=&\textstyle\sum_{\inds,\indt=1}^\thedim\Ttwobv{\SdualX{\anycolor}}{\indt}{\inds}\Smonoidalproduct\thecuniX{\SdualX{\anycolor}}{\indt}{\inds}-\sum_{\inds=1}^\thedim\Ttwobv{\anycolor}{\inds}{\inds}\Smonoidalproduct\theone
  \end{IEEEeqnarray*}  
  and such that $\Tdiff{\orderind}$ is the zero morphism  of right $\thealgebra$-modules for each $\orderind\in\Sintegersp$ with $4\leq \orderind$, then the complex
    \begin{IEEEeqnarray*}{C}
  \begin{tikzpicture}
    \node (p0) at (0,0) {$\themodule$};
    \node (p1) [left = .5cm of p0] {$\thefield\Tmod{0}\Smonoidalproduct\thealgebra$};
    \node (p2) [left = .5cm of p1] {$\thefield\Tmod{1}\Smonoidalproduct\thealgebra$};
    \node (p3) [left = .5cm of p2] {$\thefield\Tmod{2}\Smonoidalproduct\thealgebra$};
    \node (p4) [left = .5cm of p3] {$\thefield\Tmod{3}\Smonoidalproduct\thealgebra$};
    \node (p5) [left = .5cm of p4] {$0$};    
    \node (p6) [left = .5cm of p5] {$\hdots$};
    \draw [->] (p6) to  (p5);    
    \draw [->] (p5) to  (p4);    
    \draw [->] (p4) to node [above] {$\Tdiff{3}$} (p3);
    \draw [->] (p3) to node [above] {$\Tdiff{2}$} (p2);
    \draw [->] (p2) to node [above] {$\Tdiff{1}$} (p1);
    \draw [->] (p1) to node [above] {$\Tdiff{0}$} (p0);
  \end{tikzpicture}  
\end{IEEEeqnarray*}
  of right $\thealgebra$-modules is exact, meaning that $(\Tmod{\orderind},\Tdiff{\orderind})_{\orderind\in\Sintegersnn}$ is a resolution of $\themodule$ in terms of free right $\thealgebra$-modules.
\end{proposition} 

\newcommand{\herethere}{f}
\newcommand{\herethereD}[1]{\herethere_{#1}}
\newcommand{\therehere}{f'}
\newcommand{\therehereD}[1]{\therehere_{#1}}
\newcommand{\herehomotopy}{D}
\newcommand{\herehomotopyD}[1]{\herehomotopy_{#1}}
\newcommand{\therehomotopy}{D'}
\newcommand{\therehomotopyD}[1]{\therehomotopy_{#1}}
It can then be checked that a chain map from the chain complex in the \hyperref[main-result]{Main result} can be defined as follows.
\begin{proposition}
  If
  $(\herethereD{\orderind})_{\orderind\in\Sintegersnn}$ is such that for any $\orderind\in \Sintegersnn$ the mapping $\herethereD{\orderind}$ is a morphism of right $\thealgebra$-modules $\Sfromto{\thechainsmodule{\orderind}}{\Tmod{\orderind}}$, such that for any $(\indj,\indi)\in \SYnumbers{\thedim}^{\Ssetmonoidalproduct 2}$
  \begin{IEEEeqnarray*}{rCl}
    \herethereD{0}(\theone)&=&\textstyle\frac{1}{2}\Saction(\Tzerobv{\Swhite}+\Tzerobv{\Sblack})\Smonoidalproduct\theone
  \end{IEEEeqnarray*}
  and
  \begin{IEEEeqnarray*}{rCl}
    \herethereD{1}(\Sbv{\theuniA}{\indj}{\indi}\Smonoidalproduct\theone)&=&\textstyle\Tonefirstbv{\Sblack}{\indj}{\Sexchanged(\indi)}\Smonoidalproduct\theone-\frac{1}{2}\Saction\Tonesecondbv\Smonoidalproduct(\thewuniX{\indj}{\Sexchanged(\indi)}+\Skronecker{\indj}{\Sexchanged(\indi)}\theone),\\
\herethereD{1}(\Sbv{\theuniB}{\indj}{\indi}\Smonoidalproduct\theone)&=&\textstyle\Tonefirstbv{\Swhite}{\Sexchanged(\indj)}{\indi}\Smonoidalproduct\theone+\frac{1}{2}\Saction\Tonesecondbv\Smonoidalproduct(\thebuniX{\Sexchanged(\indj)}{\indi}+\Skronecker{\Sexchanged(\indj)}{\indi}\theone)
\end{IEEEeqnarray*}
and
\begin{IEEEeqnarray*}{rCl}  \herethereD{2}(\Sbv{\theuniA}{\indj}{\indi}\Smonoidalproduct\theone)&=&\textstyle\Ttwobv{\Swhite}{\indi}{\indj}\Smonoidalproduct\theone -\Skronecker{\indj}{\thedim}\Skronecker{\indi}{\thedim}\sum_{\inds=1}^{\thedim-1}\sum_{\indt=1}^\thedim \Ttwobv{\Sblack}{\indt}{\inds}\Smonoidalproduct\thebuniX{\indt}{\inds},\\
  \herethereD{2}(\Sbv{\theuniAT}{\indj}{\indi}\Smonoidalproduct\theone)&=&\textstyle\sum_{\inds=1}^\thedim \Ttwobv{\Sblack}{\inds}{\indj}\Smonoidalproduct\thebuniX{\inds}{\indi}-\Skronecker{\indj}{\thedim}\Skronecker{\indi}{\thedim}\sum_{\inds=1}^{\thedim-1}\Ttwobv{\Swhite}{\inds}{\inds}\Smonoidalproduct\theone,\\
\herethereD{2}(\Sbv{\theuniB}{\indj}{\indi}\Smonoidalproduct\theone)&=&\textstyle\Ttwobv{\Sblack}{\Sexchanged(\indi)}{\Sexchanged(\indj)}\Smonoidalproduct\theone -\Skronecker{\indj}{\thedim}\Skronecker{\indi}{\thedim}\sum_{\inds=2}^{\thedim}\sum_{\indt=1}^\thedim \Ttwobv{\Swhite}{\indt}{\inds}\Smonoidalproduct\thewuniX{\indt}{\inds},\\
  \herethereD{2}(\Sbv{\theuniBT}{\indj}{\indi}\Smonoidalproduct\theone)&=&\textstyle\sum_{\inds=1}^\thedim \Ttwobv{\Swhite}{\inds}{\Sexchanged(\indj)}\Smonoidalproduct\thewuniX{\inds}{\Sexchanged(\indi)}-\Skronecker{\indj}{\thedim}\Skronecker{\indi}{\thedim}\sum_{\inds=2}^{\thedim}\Ttwobv{\Sblack}{\inds}{\inds}\Smonoidalproduct\theone,
\end{IEEEeqnarray*}
  and such that $\herethereD{\orderind}=0$ for any $\orderind \in\Sintegersnn$ with $3\leq \orderind$, then    $(\herethereD{\orderind})_{\orderind\in\Sintegersnn}$ is a chain map $\Sfromto{(\thechainsmodule{\orderind},\thedifferential{\orderind})_{\orderind\in\Sintegersnn}x}{(\Tmod{\orderind},\Tdiff{\orderind})_{\orderind\in\Sintegersnn}}$ of complexes of right  $\thealgebra$-modules.  
\end{proposition}
A chain map in the converse direction is for example the following. 
\begin{proposition}
  If
  $(\therehereD{\orderind})_{\orderind\in\Sintegersnn}$ is such that for any $\orderind\in \Sintegersnn$ the mapping $\therehereD{\orderind}$ is a morphism of right $\thealgebra$-modules $\Sfromto{\Tmod{\orderind}}{\thechainsmodule{\orderind}}$, such that for any $(\indj,\indi)\in \SYnumbers{\thedim}^{\Ssetmonoidalproduct 2}$
  \begin{IEEEeqnarray*}{rCl}
    \therehereD{0}(\Tzerobv{\Swhite}\Smonoidalproduct\theone)&=&\theone,\\
    \therehereD{0}(\Tzerobv{\Sblack}\Smonoidalproduct\theone)&=&\theone
  \end{IEEEeqnarray*}
  and
  \begin{IEEEeqnarray*}{rCl}
    \therehereD{1}(\Tonefirstbv{\Swhite}{\indj}{\indi}\Smonoidalproduct\theone)&=&\Sbv{\theuniB}{\Sexchanged(\indj)}{\indi},\\
    \therehereD{1}(\Tonefirstbv{\Sblack}{\indj}{\indi}\Smonoidalproduct\theone)&=&\Sbv{\theuniA}{\indj}{\Sexchanged(\indi)},\\
\therehereD{1}(\Tonesecondbv\Smonoidalproduct\theone)&=&0
\end{IEEEeqnarray*}
and
\begin{IEEEeqnarray*}{rCl}
  \therehereD{2}(\Ttwobv{\Swhite}{\indj}{\indi}\Smonoidalproduct\theone)&=&\textstyle\Sbv{\theuniA}{\indi}{\indj}\Smonoidalproduct\theone+\Skronecker{\indj}{\thedim}\Skronecker{\indi}{\thedim}\sum_{\inds=1}^{\thedim-1}\Sbv{\theuniAT}{\inds}{\inds}\Smonoidalproduct\theone,\\
  \therehereD{2}(\Ttwobv{\Sblack}{\indj}{\indi}\Smonoidalproduct\theone)&=&\textstyle\Sbv{\theuniB}{\Sexchanged(\indi)}{\Sexchanged(\indj)}\Smonoidalproduct\theone+\Skronecker{\indj}{1}\Skronecker{\indi}{1}\sum_{\inds=1}^{\thedim-1}\Sbv{\theuniBT}{\inds}{\inds}\Smonoidalproduct\theone
\end{IEEEeqnarray*}
and
\begin{IEEEeqnarray*}{rCl}
  \therehereD{3}(\Tthreebv{\Swhite}\Smonoidalproduct\theone)&=&\textstyle\sum_{\inds=1}^\thedim \Sbv{\theuniB}{\Sexchanged(\inds)}{\inds},\\
  \therehereD{3}(\Tthreebv{\Sblack}\Smonoidalproduct\theone)&=&\textstyle\sum_{\inds=1}^\thedim \Sbv{\theuniA}{\Sexchanged(\inds)}{\inds},
\end{IEEEeqnarray*}
  and such that $\therehereD{\orderind}=0$ for any $\orderind \in\Sintegersnn$ with $4\leq \orderind$, then    $(\therehereD{\orderind})_{\orderind\in\Sintegersnn}$ is a chain map $\Sfromto{(\Tmod{\orderind},\Tdiff{\orderind})_{\orderind\in\Sintegersnn}}{(\thechainsmodule{\orderind},\thedifferential{\orderind})_{\orderind\in\Sintegersnn}}$ of complexes of right  $\thealgebra$-modules.  
\end{proposition}
Composing the two chain maps in the two possible ways produces the maps below.
\begin{proposition}
  For any $\orderind\in\Sintegersnn$ and any $(\indj,\indi)\in\SYnumbers{\thedim}^{\Ssetmonoidalproduct 2}$,
  \begin{IEEEeqnarray*}{rCl}
    (\herethereD{0}\Scomposition\therehereD{0})(\Tzerobv{\Swhite}\Smonoidalproduct\theone)=(\herethereD{0}\Scomposition\therehereD{0})(\Tzerobv{\Sblack}\Smonoidalproduct\theone)=\textstyle\frac{1}{2}\Saction(\Tzerobv{\Swhite}+\Tzerobv{\Sblack})\Smonoidalproduct\theone
  \end{IEEEeqnarray*}
  and
  \begin{IEEEeqnarray*}{rCl}
    (\herethereD{1}\Scomposition\therehereD{1})(\Tonefirstbv{\Swhite}{\indj}{\indi}\Smonoidalproduct\theone)&=&\textstyle\Tonefirstbv{\Swhite}{\indj}{\indi}\Smonoidalproduct\theone+\frac{1}{2}\Saction\Tonesecondbv\Smonoidalproduct(\thebuniX{\indj}{\indi}+\Skronecker{\indj}{\indi}\theone),\\
    (\herethereD{1}\Scomposition\therehereD{1})(\Tonefirstbv{\Sblack}{\indj}{\indi}\Smonoidalproduct\theone)&=&\textstyle\Tonefirstbv{\Sblack}{\indj}{\indi}\Smonoidalproduct\theone-\frac{1}{2}\Saction\Tonesecondbv\Smonoidalproduct(\thewuniX{\indj}{\indi}+\Skronecker{\indj}{\indi}\theone),\\
    (\herethereD{1}\Scomposition\therehereD{1})(\Tonesecondbv\Smonoidalproduct\theone)&=&0
  \end{IEEEeqnarray*}
  and $\herethereD{2}\Scomposition\therehereD{2}=\Sidentity$ and $\herethereD{\orderind}\Scomposition\therehereD{\orderind}=0$ for any $\orderind\in\Sintegersnn$ with $3\leq \orderind$.
\end{proposition}

\begin{proposition}
  For any $\orderind\in\Sintegersnn$, if $\orderind\leq 1$, then $\herethereD{\orderind}\Scomposition\therehereD{\orderind}=\Sidentity$, if $\orderind=2$, then for any  $(\indj,\indi)\in\SYnumbers{\thedim}^{\Ssetmonoidalproduct 2}$, if $\anymat\in\{\theuniA,\theuniB\}$, then 
  \begin{IEEEeqnarray*}{rCl}
    \IEEEeqnarraymulticol{3}{l}{
      (\herethereD{2}\Scomposition\therehereD{2})(\Sbv{\anymat}{\indj}{\indi}\Smonoidalproduct\theone)
    }\\
    \hspace{.em}&=&\textstyle
      \Sbv{\anymatA}{\indj}{\indi}\Smonoidalproduct\theone -\Skronecker{\indj}{\thedim}\Skronecker{\indi}{\thedim}(\sum_{\inds=2}^\thedim\sum_{\indt=1}^\thedim \Sbv{\anymatB}{\inds}{\indt}\Smonoidalproduct\anymatBX{\indt}{\inds}+\sum_{\inds=1}^{\thedim-1}\Sbv{\anymatBT}{\inds}{\inds}\Smonoidalproduct(\anymatBTX{\thedim}{\thedim}-1))
  \end{IEEEeqnarray*}
  and, if $\anymat\in\{\theuniAT,\theuniBT\}$, then 
  \begin{IEEEeqnarray*}{rCl}
    \IEEEeqnarraymulticol{3}{l}{
      (\herethereD{2}\Scomposition\therehereD{2})(\Sbv{\anymat}{\indj}{\indi}\Smonoidalproduct\theone)
    }\\
    \hspace{.em}&=&
\textstyle\sum_{\inds=1}^\thedim \Sbv{\anymatBT}{\Sexchanged(\indj)}{\inds}\Smonoidalproduct\anymatBTX{\inds}{\Sexchanged(\indi)}+\Skronecker{\indj}{1}\sum_{\inds=1}^{\thedim-1}\Sbv{\anymatB}{\inds}{\inds}\Smonoidalproduct\anymatBX{\Sexchanged(\indi)}{\thedim}     -\Skronecker{\indj}{\thedim}\Skronecker{\indj}{\thedim}\sum_{\inds=1}^{\thedim-1}\anymatATX{\inds}{\inds}\Smonoidalproduct\theone
  \end{IEEEeqnarray*}
  and, if $3\leq \orderind$, then $\herethereD{\orderind}\Scomposition\therehereD{\orderind}=0$.
\end{proposition}
Finally, choices of chain homotopies from the compositions of the chain maps to the identity chain maps are given by the below maps.
\begin{proposition}
  If $(\therehomotopyD{\orderind})_{\orderind\in\Sintegersnn}$ is such that for any $\orderind\in \Sintegersnn$ the mapping $\therehomotopyD{\orderind}$ is a morphism of right $\thealgebra$\-/modules $\Sfromto{\Tmod{\orderind-1}}{\Tmod{\orderind}}$, such that $\therehomotopyD{0}=0$, such that
  \begin{IEEEeqnarray*}{rCl}
    \therehomotopyD{1}(\Tzerobv{\Swhite}\Smonoidalproduct\theone)&=&\textstyle-\frac{1}{2}\Saction\Tonesecondbv\Smonoidalproduct\theone,\\
    \therehomotopyD{1}(\Tzerobv{\Sblack}\Smonoidalproduct\theone)&=&\textstyle\frac{1}{2}\Saction\Tonesecondbv\Smonoidalproduct\theone,
  \end{IEEEeqnarray*}
  and such $\therehomotopyD{\orderind}=0$ for any $\orderind\in\Sintegersnn$ with $2\leq \orderind$, then $(\therehomotopyD{\orderind})_{\orderind\in\Sintegersnn}$ is a chain homotopy $\Sfromto{(\herethereD{\orderind}\Scomposition\therehereD{\orderind})_{\orderind\in\Sintegersnn}}{\Sidentity}$ of chain maps of complexes of right $\thealgebra$\-/modules.
\end{proposition}

\begin{proposition}
  If $(\herehomotopyD{\orderind})_{\orderind\in\Sintegersnn}$ is such that for any $\orderind\in \Sintegersnn$ the mapping $\herehomotopyD{\orderind}$ is a morphism of right $\thealgebra$\-/modules $\Sfromto{\thechainsmodule{\orderind-1}}{\thechainsmodule{\orderind}}$, such that $\therehomotopyD{0}=\therehomotopyD{1}=\therehomotopyD{2}=0$ and such that for any $\orderind\in\Sintegersnn$ with $3\leq \orderind$, any $\anymat\in\thematrices$  and any $(\indj,\indi)\in\SYnumbers{\thedim}^{\Ssetmonoidalproduct 2}$, if $\anymat\in\{\theuniA,\theuniB\}$, then 
  \begin{IEEEeqnarray*}{rCl}
    \herehomotopyD{\orderind}(\Sbv{\anymatA}{\indj}{\indi}\Smonoidalproduct\theone)&=&\textstyle(-1)^\orderind\Skronecker{\indj}{\thedim}\Skronecker{\indi}{\thedim}\Saction(\Sbv{\anymatB}{\thedim}{1}+\Skronecker{\orderind}{3}\sum_{\inds=2}^{\thedim-1}\Sbv{\anymatB}{\Sexchanged(\inds)}{\inds})\Smonoidalproduct\theone
  \end{IEEEeqnarray*}
  and, if $\anymat\in\{\theuniAT,\theuniBT\}$, then
\begin{IEEEeqnarray*}{rCl}
    \herehomotopyD{\orderind}(\Sbv{\anymatA}{\indj}{\indi}\Smonoidalproduct\theone)&=&\textstyle(-1)^\orderind\Saction(\Sbv{\anymatBT}{\Sexchanged(\indj)}{\indi}+\Skronecker{\indj}{1}\Skronecker{\indi}{\thedim}\Sbv{\anymatB}{1}{1})\Smonoidalproduct\theone,
  \end{IEEEeqnarray*}
then $(\herehomotopyD{\orderind})_{\orderind\in\Sintegersnn}$ is a chain homotopy $\Sfromto{(\therehereD{\orderind}\Scomposition\herethereD{\orderind})_{\orderind\in\Sintegersnn}}{\Sidentity}$ of chain maps of complexes of right $\thealgebra$\-/modules.
\end{proposition}
Taken together the chain maps and  chain homotopies above thus constitute a  quasi-isomorphism between the two chain complexes.
}

\section{On using the standard resolution}
\label{section:known_method}
{
  \newcommand{\anydegree}{k}
  \newcommand{\anyqg}{G}
  \newcommand{\anycocycle}{\gamma}
  \newcommand{\anychain}{\psi}
  \newcommand{\stddiff}{\partial_{\mathrm{HS}}}
  \newcommand{\stddiffD}[1]{\stddiff^{#1}}
  \newcommand{\anyelement}{a}
  \newcommand{\countvar}{i}
  \newcommand{\thebimodulecomps}{{}_\thecounit\Scomps_\thecounit}  
  \newcommand{\anyelementX}[1]{\anyelement_{#1}}
  \newcommand{\anybimod}{X}
  \newcommand{\Sleftact}{}
  \newcommand{\Srightact}{}
  \newcommand{\abimodel}{x}
  \newcommand{\hsz}[3]{Z^{#3}_{\mathrm{HS}}(#1,#2)}
  \newcommand{\hsb}[3]{B^{#3}_{\mathrm{HS}}(#1,#2)}
  \newcommand{\hsh}[3]{H^{#3}_{\mathrm{HS}}(#1,#2)}
  For the interested reader this section explains the strategy used by Bichon, Franz and Gerhold in  \cite{BichonFranzGerhold2017} and Das, Franz, Kula and Skalski in \cite{DasFranzKulaSkalski2018,DasFranzKulaSkalski2021} to compute the orders $1$ and $2$ of the quantum group co\-ho\-mo\-lo\-gy of the discrete duals of certain easy quantum groups, including $U_\thedim^+$, from the standard resolution.  (Those authors investigated these quantum groups not with the objective of discerning isomorphism classes of easy quantum groups. Rather, they were led by an interest in the Calabi-Yau property of \cite{Ginzburg2006}, which generalizes Poincaré duality, and the AC property of \cite{FranzGerholdThom2015}, which guarantees the existence of an associated Schürmann triple.)
  \par
\subsection{Standard resolution}  Using the standard resolution in order to compute the quantum group co\-ho\-mo\-lo\-gy of the dual of a compact quantum group $\anyqg$ is equivalent to determining the \emph{Hochschild co\-ho\-mo\-lo\-gy} (see \cite{Hochschild1956}) of $\mathcal{O}(\anyqg)$ with coefficients in the bi\-/module  $\thebimodulecomps$ obtained by letting $\thecounit$ induce both actions on $\Scomps$ (see, e.g.,  \cite{Bichon2013}). 
  Given any bi\-/module $\anybimod$
over any algebra $\thealgebra$ over any field $\thefield$, if for any $\anydegree\in\Sintegersnn$ the differential  $\stddiffD{\anydegree}$ is the linear map
$\Sfromto{\ShomOX{\thefield}{\thealgebra^{\Smonoidalproduct\anydegree}}{\anybimod}}{\ShomOX{\thefield}{\thealgebra^{\Smonoidalproduct(\anydegree+1)}}{\anybimod}}$ with for any $\anychain\in \ShomOX{\thefield}{\thealgebra^{\Smonoidalproduct\anydegree}}{\anybimod}$ and any $(\anyelementX{0},\anyelementX{1},\ldots,\anyelementX{\anydegree})\in \thealgebra^{\Ssetmonoidalproduct \anydegree+1}$,
\begin{IEEEeqnarray*}{rCl}  \IEEEeqnarraymulticol{3}{l}{\stddiffD{\anydegree}\anychain(\anyelementX{0}\Smonoidalproduct\anyelementX{1}\Smonoidalproduct\hdots\Smonoidalproduct\anyelementX{\anydegree})
  }\\
  \hspace{1em}&=&\sum_{\countvar=0}^{\anydegree+1}(-1)^\countvar\begin{cases} \anyelementX{0}\Sleftact\anychain(\anyelementX{1}\Smonoidalproduct\anyelementX{2}\Smonoidalproduct\hdots\Smonoidalproduct\anyelementX{\anydegree})&\Scase \countvar=0\\
                                                              \begin{aligned}
&                                                                \anychain(\anyelementX{0}\Smonoidalproduct\anyelementX{1}\Smonoidalproduct\hdots\Smonoidalproduct\anyelementX{\countvar-2}\Smonoidalproduct\anyelementX{\countvar-1}\anyelementX{\countvar}\\
&                                                              \hspace{5.25em}{}\Smonoidalproduct\anyelementX{\countvar+1}\Smonoidalproduct\anyelementX{\countvar+2}\Smonoidalproduct\hdots\Smonoidalproduct\anyelementX{\anydegree})
                                                              \end{aligned}                     &\bigg\vert\, 0<\countvar<\anydegree\\                                                                                              \anychain( \anyelementX{0}\Smonoidalproduct\anyelementX{1}\Smonoidalproduct\hdots\Smonoidalproduct\anyelementX{\anydegree-1})\Srightact\anyelementX{\anydegree}&\Scase \countvar=\anydegree,
\end{cases}
\end{IEEEeqnarray*}
then the $\anydegree$\-/th \emph{Hochschild co\-ho\-mo\-lo\-gy}  $\hsh{\thealgebra}{\anybimod}{\anydegree}$ of $\thealgebra$ with coefiicients in $\anybimod$ is the quotient vector space of the space
\begin{IEEEeqnarray*}{rCl}
  \hsz{\thealgebra}{\anybimod}{\anydegree}&=&\{\anycocycle\in\ShomOX{\thefield}{\thealgebra^{\Smonoidalproduct\anydegree}}{\anybimod}\Sand \stddiffD{\anydegree}\anycocycle=0\}
\end{IEEEeqnarray*}
of \emph{$\anydegree$\-/co\-/cycles} by the space  $\hsb{\thealgebra}{\anybimod}{\anydegree}$ of \emph{$\anydegree$\-/co-boundaries} given by $0$ if $\anydegree=0$ and otherwise by
\begin{IEEEeqnarray*}{rCl}
\hsb{\thealgebra}{\anybimod}{\anydegree}&=&\{\stddiffD{\anydegree-1}\anychain\Ssetbuilder\anychain\in\ShomOX{\thefield}{\thealgebra^{\Smonoidalproduct(\anydegree-1)}}{\anybimod}\},
\end{IEEEeqnarray*}
where $\thealgebra^{\Smonoidalproduct 0}$ means $\thefield$ and where $\ShomO{\thefield}$ is the internal hom functor of vector spaces. Specializing to $\thealgebra=\mathcal{O}(\anyqg)$ and $\anybimod=\thebimodulecomps$ gives the quantum group co\-ho\-mo\-lo\-gy with trivial coefficients of the dual of $\anyqg$.
\par
 \newcommand{\leftop}{a_1}
 \newcommand{\rightop}{a_2}
 \newcommand{\leftvect}{x_1}
 \newcommand{\rightvect}{x_2}
 \newcommand{\otherop}{b}
 \newcommand{\bleftact}{\mathop{\triangleright}}
 \newcommand{\brightact}{\mathop{\triangleleft}}
 \newcommand{\hombimod}{\thealgebra\vert\anybimod}
 \newcommand{\algstr}[1]{\thealgebra\boxplus_{#1}\anybimod}
 \newcommand{\algstrx}[2]{\thealgebra\boxplus_{#1}#2} 
 \newcommand{\anygen}{e}
 \newcommand{\anyrel}{r}
 \newcommand{\defrels}{S}
 \newcommand{\defspace}{V}
 \newcommand{\defmap}{\Delta}
 \newcommand{\leftcoc}{\eta}
 \newcommand{\rightcoc}{\theta}
 \newcommand{\leftcocx}[1]{\eta_{#1}}
 \newcommand{\rightcocx}[1]{\theta_{#1}}
 \newcommand{\anynum}{m}
 \newcommand{\anyop}{a} 
 \newcommand{\combococ}{\varphi}
 \newcommand{\indi}{i}
 \newcommand{\exhmap}{\Xi}
 \newcommand{\exhspace}{W}
 \newcommand{\tenhomadj}{\Lambda}    
\subsection{First order}
 The space $\hsb{\thealgebra}{\anybimod}{1}$ is $0$ then and computing the first co\-ho\-mo\-lo\-gy thus boils down to determining $\hsz{\thealgebra}{\anybimod}{1}$ alone.
That is possible to do for all easy quantum groups more or less simultaneously, as shown in  \cite{Mang2023arxiv}. The key is to use the following folk theorem (see, e.g.,  \cite[Lemma~1.9]{KyedRaum2017}). On the direct sum $\thealgebra\Sdirectsum\anybimod$ an algebra structure  $\algstr{0}$ with a useful property is defined by letting $(\leftop,\leftvect)(\rightop,\rightvect)\Seqpd(\leftop\rightop,\leftop\Sleftact\rightvect+\leftvect\Srightact\rightop)$ for any $\{(\leftop,\leftvect),(\rightop,\rightvect)\}\subseteq \thealgebra\Sdirectsum\anybimod$. Any $\anychain\in\ShomOX{\thefield}{\thealgebra}{\anybimod}$ satisfies $\stddiffD{1}\anychain=0$ if and only if the linear map  $(\Sidentity,\anychain)$ with $\anyelement\mapsto(\anyelement,\anycocycle(\anyelement))$ for any $\anyelement\in \thealgebra$ is an algebra morphism $\Sfromto{\thealgebra}{\algstr{0}}$. If $\thealgebra$ is a universal algebra with generators $\thegens$ and relations $\therels\subseteq\thefreealg$, the map $(\Sidentity,\anychain)$ is such an algebra morphism if and only if   $\anyrel((\anygen,\anychain(\anygen))_{\anygen\in\thegens})=0$ in $\algstr{0}$ for any $\anyrel\in\therels$. That gives a characterization of $\hsz{\thealgebra}{\anybimod}{1}$ in terms of linear equations induced by $\therels$. Because in case $\thealgebra=\mathcal{O}(\anyqg)$ the polynomials $\therels$ are all of the form \eqref{eq:relation_of_partition} the resulting equations can be predicted well enough to solve them in general, even for those easy quantum groups whose categories of partitions are still unknown.
 \par
\subsection{Second order} Already in order $2$ however, the space of co\-/boundaries is not only non\-/zero but  infinite\-/dimensional in general. Nonetheless the authors of  \cite{BichonFranzGerhold2017} and \cite{DasFranzKulaSkalski2018,DasFranzKulaSkalski2021} managed to compute the co\-ho\-mo\-lo\-gy. Three insights are necessary to understand their strategy. First, if any $\anycocycle\in\hsz{\thealgebra}{\anybimod}{2}$ is called \emph{normalized} if $\anycocycle(\theone\Smonoidalproduct\theone)=0$, then $\hsh{\thealgebra}{\anybimod}{2}$  is also the quotient vector space of the normalized $2$\-/co\-/cycles by the normalized  $2$\-/co\-/boundaries. Second, the authors of \cite{BichonFranzGerhold2017} provided in their Lemma~5.4 (or \cite[Lemma~4.1]{DasFranzKulaSkalski2018}) a key tool, which, if generalized, turns out to extend the above folk theorem: Given any normalized $\anycocycle\in \hsz{\thealgebra}{\anybimod}{2}$ the rule that $(\leftop,\leftvect)(\rightop,\rightvect)\Seqpd(\leftop\rightop,\leftop\Sleftact\rightvect+\leftvect\Srightact\rightop-\anycocycle(\leftop\Smonoidalproduct\rightop))$ for any $\{(\leftop,\leftvect),(\rightop,\rightvect)\}\subseteq \thealgebra\Sdirectsum\anybimod$ also defines an algebra structure $\algstr{\anycocycle}$ on $\thealgebra\Sdirectsum\anybimod$. Now, for any $\anychain\in\ShomOX{\thefield}{\thealgebra}{\anybimod}$ such that $\stddiffD{1}\anychain$ is normalized (which means $\anychain(\theone)=0$) the map $(\Sidentity,\anychain)$ is an algebra morphism $\Sfromto{\thealgebra}{\algstr{\anycocycle}}$ if and only if $\stddiffD{1}\anychain=\anycocycle$. Thus, if $\thealgebra$ is a universal algebra with generators $\thegens$ and relations $\therels$ any normalized $\anycocycle\in \hsz{\thealgebra}{\anybimod}{2}$ is a co\-/boundary if and only if there exists $\anychain\in\ShomOX{\thefield}{\thealgebra}{\anybimod}$ such that $\anychain(\theone)=0$ and  $\anyrel((\anygen,\anychain(\anygen))_{\anygen\in\thegens})=0$ in $\algstr{\anycocycle}$ for any $\anyrel\in\therels$. In fact, it can be seen that any such $\anychain$ is already uniquely determined by its values on $\thegens$ (generalizing \cite[pp.23--25, after $(\Leftarrow)$]{DasFranzKulaSkalski2018}). Solving these equations characterizing when a normalized $2$\-/co\-/cycle is a $2$\-/co\-/boundary can be understood as finding a vector space $\defspace$ and what Das, Franz, Kula and Skalski call a \emph{defect map}, a linear map $\defmap$ from the space of normalized $2$\-/co\-/cycles into $\defspace$ whose kernel is the space of normalized  $2$\-/co\-/boundaries (see the proof of \cite[Theorem~4.6]{DasFranzKulaSkalski2018}). Computing $\hsh{\thealgebra}{\anybimod}{2}$ is then equivalent to determining the image of $\defmap$. The third part of the strategy employed in \cite{BichonFranzGerhold2017} and \cite{DasFranzKulaSkalski2018} consists in an ansatz for achieving just that. In the case where $\anybimod$ is additionally equipped with an algebra structure compatible with the $\thealgebra$\-/actions, for any $(\leftcoc,\rightcoc)\in \hsz{\thealgebra}{\anybimod}{1}^{\Ssetmonoidalproduct 2}$  a normalized $2$\-/co\-/cycle is given by the so-called \emph{cup product} $\leftcoc\cup\rightcoc$ defined by $\leftop\Smonoidalproduct\rightop\mapsto \leftcoc(\leftop)\rightcoc(\rightop)$ for any $\{\leftop,\rightop\}\subseteq \thealgebra$. If any normalized $2$\-/co\-/cycle is co\-/homologous to a linear combination of cup products, the image of $\defmap$ coincides with the span of the images of cup products under $\defmap$. In the case $\thealgebra=\mathcal{O}(\anyqg)$ and $\anybimod=\thebimodulecomps$ for the $\anyqg$ investigated in \cite{BichonFranzGerhold2017} and \cite{DasFranzKulaSkalski2018}, in particular for $\anyqg=U_\thedim^+$, this turns out to be true (see \cite[p.~26]{DasFranzKulaSkalski2018}).
 \par
\subsection{Additional tools for the second order} This strategy succeeds also for certain other  easy  quantum groups $\anyqg$ than the ones from \cite{BichonFranzGerhold2017} and \cite{DasFranzKulaSkalski2018}, as demonstrated by \cite{Wendel2020}.
 When trying to carry it out for all easy $\anyqg$ simultaneously, it may be useful to extend the strategy by two additional parts. 
 First, the folk theorem already used twice can be employed a third time to address the question whether for a given $\anyqg$ any normalized $\anycocycle\in\hsz{\thealgebra}{\anybimod}{2}$ is co\-/homologous to a linear combination of cup products. It namely implies that this is the case  if and only if there exist $\anychain\in\ShomOX{\thefield}{\thealgebra}{\anybimod}$ with $\anychain(\theone)=0$ and $\anynum\in\Sintegersnn$ and $\{\leftcocx{\indi},\rightcocx{\indi}\}_{\indi=1}^\anynum\subseteq \hsz{\thealgebra}{\anybimod}{1}$ such that, if $\combococ\Seqpd \anycocycle-\sum_{\indi=1}^\anynum \leftcocx{\indi}\cup\rightcocx{\indi}$, then $(\Sidentity,\anychain)$ is an algebra morphism $\Sfromto{\thealgebra}{\algstr{\combococ}}$. As before, if $\thealgebra$ is generated by $\thegens$ subject to relations $\therels$, this means one has to check whether $\anyrel((\anygen,\anychain(\anygen))_{\anygen\in\thegens})=0$ in $\algstr{\combococ}$ for any $\anyrel\in\therels$.
 Solving those conditions likely requires a lot of knowledge about $\hsz{\thealgebra}{\anybimod}{2}$. Therefore, the second piece one might want to add to the strategy from \cite{BichonFranzGerhold2017} and \cite{DasFranzKulaSkalski2018} is the following. The space $\{\anychain\in\ShomOX{\thefield}{\thealgebra}{\anybimod}\Sand \anychain(\theone)=0\}$ can be turned into an $\thealgebra$\-/bi\-/module $\hombimod$ via the left action $\bleftact$ and right action $\brightact$ defined by $(\anyop\bleftact\anychain)(\otherop)\Seqpd \anyop\Sleftact\anychain(\otherop)$ and $(\anychain\brightact\anyop)(\otherop)\Seqpd \anychain(\otherop\anyop)-\anychain(\otherop)\Srightact\anyop$ for any element $\anychain$ and any $\{\anyop,\otherop\}\subseteq \thealgebra$. By restricting the natural isomorphism $\Sxfromto{\tenhomadj}{\ShomOX{\thefield}{\thealgebra\Smonoidalproduct\thealgebra}{\anybimod}}{\ShomOX{\thefield}{\thealgebra}{\ShomOX{\thefield}{\thealgebra}{\anybimod}}}$ induced by the tensor hom adjunction it can be seen that the space of normalized elements of $\hsz{\thealgebra}{\anybimod}{2}$ is isomorphic to $\hsz{\thealgebra}{\hombimod}{1}$ (see \cite{Hochschild1956}). Hence, the folk theorem can be applied a fourth and last time to conclude that for any normalized element $\anycocycle$ of $\hsz{\thealgebra}{\anybimod}{2}$ the map $(\Sidentity,\tenhomadj(\anycocycle))$ is an algebra morphism $\Sfromto{\thealgebra}{\algstrx{0}{\hombimod}}$. For $\thealgebra$ universal with generators $\thegens$ and relations $\therels$ that means $\anyrel((\anygen,\tenhomadj(\anycocycle))_{\anygen\in\thegens})=0$ in $\algstrx{0}{\hombimod}$ for any element $\anyrel$ of the ideal generated by $\therels$.
}

{
   \newcommand{\anyqg}{G}
    \newcommand{\otherqg}{H}
  \newcommand{\somehopf}{\mathcal{A}}  
\section{Other methods}
\label{section:other_methods}
At this point in time the only easy quantum groups the cohomologies of whose duals have been determined are  $O_\thedim^+$, $S_\thedim^+$ and $U_\thedim^+$. For the interested reader this section summarizes how this was achieved in each case and what would be required in order to adapt the strategies to other easy quantum groups.
\subsection{\texorpdfstring{Free orthogonal quantum group $O_\thedim^+$}{Free orthogonal quantum group}}
The free orthogonal quantum group  $O_\thedim^+$ introduced by Wang and Van Daele in \cite{VanDaeleWang1996a} is easy by \cite{BanicaSpeicher2009}. Its associated category of partitions is  generated by the partitions \eqref{eq:minimal_partitions} and
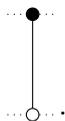
\begin{IEEEeqnarray*}{rCl}
  \begin{tikzpicture}[baseline=0]
    \def\scp{0.666}
    \def\linksize{\scp*0.075cm}
    \def\pointsize{\scp*0.25cm}
    \def\dd{\scp*0.5cm}
    \def\dx{\scp*1cm}
    \def\cx{\scp*0.3cm}
    \def\txu{0*\dx}    
    \def\txl{0*\dx}
    \def\dy{\scp*1cm}
    \def\cy{\scp*0.3cm}
    \def\ty{2*\dy}
    \tikzset{whp/.style={circle, inner sep=0pt, text width={\pointsize}, draw=black, fill=white}}
    \tikzset{blp/.style={circle, inner sep=0pt, text width={\pointsize}, draw=black, fill=black}}
    \tikzset{lk/.style={regular polygon, regular polygon sides=4, inner sep=0pt, text width={\linksize}, draw=black, fill=black}}
    \draw[dotted] ({0-\dd},{0}) -- ({\txl+\dd},{0});
    \draw[dotted] ({0-\dd},{\ty}) -- ({\txu+\dd},{\ty});
    \coordinate (l1) at ({0+0*\dx},{0+0*\ty}) {};    
    \coordinate (u1) at ({0+0*\dx},{0+1*\ty}) {};
    \draw (l1) -- (u1);
    \node[whp] at (l1) {};
    \node[blp] at (u1) {};
  \end{tikzpicture}.
\end{IEEEeqnarray*}
The co\-ho\-mo\-lo\-gy of the dual of $O_\thedim^+$ was computed by Collins, Härtel and Thom in \cite{CollinsHaertelThom2009} directly from a projective resolution of the co-unit of $\mathcal{O}(O_\thedim^+)$. And the latter was a hard\-/won achievement. According to Bemerkung 3.1.1 in Härtel's (German-language) PhD thesis \cite{Haertel2008} it was found by \enquote{computer-assisted guessing, using, besides lots of custom-made software, the applications Plural, Gap and Magma} (my translation).
\par
In other words, it would presumably be difficult to reproduce this remarkable feat for each of the many easy quantum groups whose co\-ho\-mo\-lo\-gies want to be computed. At least it seems quite unclear how and where to start.
\par
An alternative way to come up with the same resolution was later found by Bichon in \cite{Bichon2013}. In fact, he provided a free resolution in terms of Yetter-Drinfeld modules of the trivial module over $\mathcal{B}(E)$, the  symmetry hopf algebra of a non\-/degenerate bi\-/linear form introduced by Dubois-Violette and Launer in \cite{DuboisVioletteLauner1990}. 
The key argument is to use monoidal equivalence of co\-/module categories between $\mathcal{B}(E)$ and $\mathcal{O}(SL_2(q))$ for specific $q$. That allows one to transport over a non\-/commutative Koszul resolution of the trivial $\mathcal{O}(SL_2(q))$\-/module found by Hadfield and Krähmer in \cite{HadfieldKraehmer2005}. The monoidal equivalence used was discovered by Bichon in \cite{Bichon2003b}. It is obtained by constructing an explicit Hopf bi\-/Galois extension between $\mathcal{O}(SL_2(q))$ and $\mathcal{B}(E)$ for specific $q$, which suffices according to a theorem by Schauenburg from \cite{Schauenburg1996}. If one is just interested in $\mathcal{O}(O_\thedim^+)$, i.e., trivial $E$, one can also use the equivalence Banica constructed in \cite{Banica1996} by determining the simple co\-/modules  of $\mathcal{O}(O_\thedim^+)$. In particular, both proofs depend on the  detailed understanding of $SL_2(q)$ that is available.
\par
Adapting this strategy to easy quantum groups $\anyqg$ other than $O_\thedim^+$ thus has three prerequisites. First, a resolution of the trivial module over some hopf algebra $\somehopf$ as well as knowledge of, second, $\somehopf$, and, third, $\mathcal{O}(\anyqg)$ sufficient to construct a monoidal equivalence or Hopf bi\-/Galois object. Unfortunately, besides the presentation of $\mathcal{O}(\anyqg)$ in terms of generators and relations granted by the associated category of partitions, nothing is known about $\mathcal{O}(\anyqg)$ in almost all cases $\anyqg$. Moreover, there is no reason why their co\-/module categories are necessary equivalent to a known category of co\-/modules of some $\somehopf$ or that, if so, a resolution is available of the trivial $\somehopf$\-/module.
\subsection{\texorpdfstring{Free symmetric quantum group $S_\thedim^+$}{Free symmetric quantum group}}
In \cite{BanicaSpeicher2009} it is recognized that the free symmetric quantum group $S_\thedim^+$ or quantum permutation group defined by Wang in \cite{Wang1998} is easy, namely corresponds to the category of partitions generated by the partitions \eqref{eq:minimal_partitions}  and
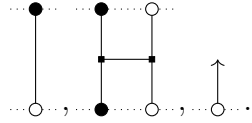
\begin{IEEEeqnarray*}{rCl}
\begin{tikzpicture}[baseline=0]
    \def\scp{0.666}
    \def\linksize{\scp*0.075cm}
    \def\pointsize{\scp*0.25cm}
    \def\dd{\scp*0.5cm}
    \def\dx{\scp*1cm}
    \def\cx{\scp*0.3cm}
    \def\txu{0*\dx}    
    \def\txl{0*\dx}
    \def\dy{\scp*1cm}
    \def\cy{\scp*0.3cm}
    \def\ty{2*\dy}
    \tikzset{whp/.style={circle, inner sep=0pt, text width={\pointsize}, draw=black, fill=white}}
    \tikzset{blp/.style={circle, inner sep=0pt, text width={\pointsize}, draw=black, fill=black}}
    \tikzset{lk/.style={regular polygon, regular polygon sides=4, inner sep=0pt, text width={\linksize}, draw=black, fill=black}}
    \draw[dotted] ({0-\dd},{0}) -- ({\txl+\dd},{0});
    \draw[dotted] ({0-\dd},{\ty}) -- ({\txu+\dd},{\ty});
    \coordinate (l1) at ({0+0*\dx},{0+0*\ty}) {};    
    \coordinate (u1) at ({0+0*\dx},{0+1*\ty}) {};
    \draw (l1) -- (u1);
    \node[whp] at (l1) {};
    \node[blp] at (u1) {};
  \end{tikzpicture},
  \begin{tikzpicture}[baseline=0]
    \def\scp{0.666}
    \def\linksize{\scp*0.075cm}
    \def\pointsize{\scp*0.25cm}
    \def\dd{\scp*0.5cm}
    \def\dx{\scp*1cm}
    \def\cx{\scp*0.3cm}
    \def\txu{1*\dx}    
    \def\txl{1*\dx}
    \def\dy{\scp*1cm}
    \def\cy{\scp*0.3cm}
    \def\ty{2*\dy}
    \tikzset{whp/.style={circle, inner sep=0pt, text width={\pointsize}, draw=black, fill=white}}
    \tikzset{blp/.style={circle, inner sep=0pt, text width={\pointsize}, draw=black, fill=black}}
    \tikzset{lk/.style={regular polygon, regular polygon sides=4, inner sep=0pt, text width={\linksize}, draw=black, fill=black}}
    \node[lk] at ($(l1)+(.,{1*\dy})$)  {};
    \node[lk] at ($(l2)+(.,{1*\dy})$)  {};
    \draw[dotted] ({0-\dd},{0}) -- ({\txl+\dd},{0});
    \draw[dotted] ({0-\dd},{\ty}) -- ({\txu+\dd},{\ty});    
    \coordinate (l1) at ({0+0*\dx},{0+0*\ty}) {};
    \coordinate (l2) at ({0+1*\dx},{0+0*\ty}) {};
    \coordinate (u1) at ({0+0*\dx},{0+1*\ty}) {};
    \coordinate (u2) at ({0+1*\dx},{0+1*\ty}) {};        
    \draw (l1) --  (u1);
    \draw (l2)  -- (u2);
    \draw (l1)  ++ (.,{1*\dy}) -- ++({1*\dx},.);
    \node[blp] at (l1) {};
    \node[whp] at (l2) {};
    \node[blp] at (u1) {};
    \node[whp] at (u2) {};    
  \end{tikzpicture},  
  \begin{tikzpicture}[baseline=0]
    \def\scp{0.666}
    \def\linksize{\scp*0.075cm}
    \def\pointsize{\scp*0.25cm}
    \def\dd{\scp*0.5cm}
    \def\dx{\scp*1cm}
    \def\cx{\scp*0.3cm}
    \def\txu{0*\dx}    
    \def\txl{0*\dx}
    \def\dy{\scp*1cm}
    \def\cy{\scp*0.3cm}
    \def\ty{2*\dy}
    \tikzset{whp/.style={circle, inner sep=0pt, text width={\pointsize}, draw=black, fill=white}}
    \tikzset{blp/.style={circle, inner sep=0pt, text width={\pointsize}, draw=black, fill=black}}
    \tikzset{lk/.style={regular polygon, regular polygon sides=4, inner sep=0pt, text width={\linksize}, draw=black, fill=black}}
    \draw[dotted] ({0-\dd},{0}) -- ({\txl+\dd},{0});
    \coordinate (l1) at ({0+0*\dx},{0+0*\ty}) {};
    \draw[->] (l1) -- ++ (.,{1*\dy});
    \node[whp] at (l1) {};
  \end{tikzpicture}.
\end{IEEEeqnarray*}
The co\-ho\-mo\-lo\-gy of the dual of $S_\thedim^+$ was determined in \cite{BichonFranzGerhold2017} by Bichon, Franz and Gerhold. Key to their proof  is to establish that $\mathcal{O}(S_\thedim^+)$ is Calabi\-/Yau of dimension $3$  in the sense of \cite{Ginzburg2006}, provided $\thedim\leq 4$. In fact, this is proved in greater generality for the quantum symmetry algebras of  semi\-/simple measured algebras of finite dimension greater than $3$  with normalizable tracial measure. And the proof of that crucially relies on a monoidal equivalence of co\-/module categories between the quantum symmetry algebra and $\mathcal{O}(PSL_2(q))$ proved by Mrozinski in \cite{Mrozinski2015}. That in turn required detailed knowledge of the simple co\-/modules of both the quantum symmetry algebra and $\mathcal{O}(SO_3(q))$ for at least certain $q$.
\par
Knowing that  $\mathcal{O}(S_\thedim^+)$ is Calabi\-/Yau enables Bichon, Franz and Gerhold to conclude that the cohomological dimension is $3$. Hence, they only need to compute the co\-ho\-mo\-lo\-gy up to the third order. The orders $1$ and $2$ are computed directly from the standard resolution. For the third order they use a duality between  co\-ho\-mo\-lo\-gy and homology implied  by the Calabi\-/Yau property. In consequence, rather than compute the co\-ho\-mo\-lo\-gy in the third order, it suffices to determine the homology in the first. This again is achieved  via the standard resolution.
\par
Bichon, Franz and Gerhold's method thus strongly relies on a particular property of $S_\thedim^+$. How many further easy quantum groups $\anyqg$ can be regarded as quantum symmetry groups of measured algebras is unclear. Of course, even if $\anyqg$ is not a quantum symmetry group, it could still be Calabi-Yau. Though, verifying that might then require a further generalization of Mrozinski's result. Such an extension seems likely to demand significant knowledge about the simple co-modules of $\mathcal{O}(\anyqg)$. Sadly, though, besides a combinatorial approach to computing partial fusion rules in \cite{FreslonWeber2016}, figuring out the simple co-modules of $\mathcal{O}(\anyqg)$ for general easy $\anyqg$ remains an open problem.
\subsection{\texorpdfstring{Free unitary quantum group $U_\thedim^+$}{Free unitary quantum group}}
\label{section:previous_proof}
As was already mentioned, in \cite{BaraquinFranzGerholdKulaTobolski2023}, Baraquin, Franz, Gerhold, Kula and Tobolski not only compute  the co\-ho\-mo\-lo\-gy of the dual of the easy quantum group $U_\thedim^+$. They give a free resolution of $\mathcal{H}(F)$ for generic $F$ and determine its Hochschild co\-ho\-mo\-lo\-gy with one-dimensional coefficients and its bialgebra co\-ho\-mo\-lo\-gy.
\par
Their proof builds on the finite free resolution of the co\-/unit of $\mathcal{B}(E)$  found by  Bichon in \cite{Bichon2013}.
Baraquin, Franz, Gerhold, Kula and Tobolski prove that, if $F=E\StransposeP E$, then  $\mathcal{H}(F)$ appears as a Hopf subalgebra of the co\-/product $\mathcal{B}(E)\ast\Scomps\Sintegers_2$ of  $\mathcal{B}(E)$ and the group algebra of the cyclic group $\Sintegers_2$ of order two. Namely, $\mathcal{H}(F)$ is precisely the so-called glued free product  $\mathcal{B}(E)\mathop{\widetilde{\ast}}\Scomps\Sintegers_2$ in the sense of \cite{TarragoWeber2016} of $\mathcal{B}(E)$ and  $\Scomps\Sintegers_2$. 
Already in \cite{Banica1997} Banica had observed (although not in those terms) that  $\mathcal{O}(U_n^+)$ is the glued free product $\mathcal{O}(O_\thedim^+)\mathop{\widetilde{\ast}}\Scomps\Sintegers$ of  $\mathcal{O}(O_n^+)$ and the group algebra of the integers.  Gromada later refined this result by narrowing $\Sintegers$ to $\Sintegers_2$ in \cite{Gromada2022c}. It is also his study of the glued product that when combined with the monoidal equivalence between the co\-/modules of $\mathcal{B}(E)$ and $\mathcal{O}(SL_2(q))$ from \cite{Bichon2003b} is the basis of the proof that $\mathcal{H}(F)\cong\mathcal{B}(E)\mathop{\widetilde{\ast}}\Scomps\Sintegers_2$.
\par
The authors of \cite{BaraquinFranzGerholdKulaTobolski2023} then combine these two ingredients, the resolution for $\mathcal{B}(E)$ and the isomorphism $\mathcal{H}(F)\cong\mathcal{B}(E)\mathop{\widetilde{\ast}}\Scomps\Sintegers_2$ as follows. They prove that a projective resolution of the trivial module over the co\-/product of two Hopf algebras can be obtained as the direct sum of projective resolutions of the co\-/factors. In that way they construct from the projective resolution for $\mathcal{B}(E)$ and the known  one for $\Scomps\Sintegers
_2$ a resolution for $\mathcal{B}(E)\ast\Scomps\Sintegers_2$. The fact that co\-/semi\-/simple Hopf algebras are faithfully flat over Hopf subalgebras, proved by Chirvasitu in \cite{Chirvasitu2014}, enables them to apply the restriction of scalars functor to turn that free resolution for  $\mathcal{B}(E)\ast\Scomps\Sintegers_2$ into a projective resolution for $\mathcal{B}(E)\mathop{\widetilde{\ast}}\Scomps\Sintegers_2$ and thus for $\mathcal{H}(F)$. In fact, it turns out that the resolution so obtained consists of  free modules which can be further upgraded to free Yetter-Drinfeld modules. Finally, by using the theory of Yetter-Drinfeld modules Baraquin, Franz, Gerhold, Kula and Tobolski manage to exchange the premise that $F=E\StransposeP E$ for the assumption that $F$ is generic.
\par
If taken as a template for easy $\anyqg$ other than $U_\thedim^+$, the strategy would be to express $\mathcal{O}(\anyqg)$ as a Hopf subalgebra of $\mathcal{O}(\otherqg)$ for some compact quantum group $\otherqg$ for which a projective resolution is available  and then applying the restriction of scalars functors. More narrowly, $\mathcal{O}(\otherqg)$ could be a co\-/product of two Hopf algebras with known resolutions and $\mathcal{O}(\anyqg)$ could be a glued free product. In particular, such a presentation exists if the degree of reflection (in the sense of \cite{Gromada2018,Gromada2022c}) of the category associated with $\anyqg$ is non\-/zero. Because the majority of already all so\-/called non\-/hy\-per\-oc\-ta\-he\-dral categories of partitions, which make up \enquote{three quarters} of all categories,  have degree of reflection zero by \cite{MangWeber2021c}, a more general product construction would likely be in order for arbitrary $\anyqg$. The issue is that establishing such a subalgebra relationship in the first place requires detailed knowledge of both co\-/module categories. And, again, for general $\anyqg$, with very few exceptions, that is not available. 
\par
In summary, for any $\anyqg\in \{O_\thedim^+,S_\thedim^+, U_\thedim^+\}$ the known resolution of the trivial $\mathcal{O}(\anyqg)$\-/module  either was guessed with computer assistance, is the standard resolution (in low degrees) or was found by establishing a monoidal equivalence of co\-/modules between $\mathcal{O}(G)$ and a Hopf algebra with known resolution. Of course, it was never an objective of any of the authors to  compute the quantum group co\-ho\-mo\-lo\-gy of the duals of arbitrary easy quantum groups $\anyqg$.  If, however, one is interested in that, there would be significant difficulties to overcome trying to adapt these methods, most importantly, the lack of knowledge of the simple co\-/modules of $\mathcal{O}(\anyqg)$.
}
  }

\printbibliography{}
\end{document}